\documentclass[10pt,reqno]{amsart}
\usepackage{amsthm, amsmath,amsfonts,amssymb,euscript,hyperref,graphics,color,slashed}
\usepackage{graphicx}
\usepackage{comment}
\usepackage{import}
\usepackage{tikz}
\usepackage{latexsym}
\usepackage{xcolor}

\definecolor{deepgreen}{cmyk}{1,0,1,0.5}

\def\ub{{\underline{u}}}
\def\Lb{\underline{L}}
\def\Cb{{\underline{C}}}
\def\Eb{\underline{E}}
\def\Fb{\underline{F}}
\def\gslash{\mbox{$g \mkern -8mu /$ \!}}

\def\nablaslash{\mbox{$\nabla \mkern -13mu /$ \!}}
\def\laplacianslash{\mbox{$\Delta \mkern -13mu /$ \!}}

\newtheorem*{MainTheorem}{Main Theorem}
\newtheorem{theorem}{Theorem}[section]
\newtheorem{lemma}[theorem]{Lemma}
\newtheorem{proposition}[theorem]{Proposition}
\newtheorem{corollary}[theorem]{Corollary}
\newtheorem{definition}[theorem]{Definition}
\newtheorem{remark}[theorem]{Remark}
\newtheorem*{Main Theorem}{Main Theorem}

\numberwithin{equation}{section}

\def\Et{\widetilde{E}}
\def\Ebt{\widetilde{\underline{E}}}
\def\lot{{\text{l.o.t.}}}
\def\Ntop{{N_{\text{top}}}}
\def\Nmu{N_{\mu}}
\def\Ninfty{N_{\infty}}
\def\gt{{\widetilde{g}}}
\def\ds{{\slashed{d}}}
\def\gslash{{\slashed{g}}}
\def\ub{\underline{u}}
\def\Lb{\underline{L}}
\def\Cb{\underline{C}}
\def\Eb{\underline{E}}
\def\Fb{\underline{F}}
\def\tr{{\textrm{tr}}}
\def\dmug{{d \mu_{\slashed{g}}}}
\def\alphab{\underline{\alpha}}

\def\etab{\underline{\eta}}
\def\zetab{\underline{\zeta}}

\def\O{\mathcal{O}}

\def\slashedLhat{{\widehat{\slashed{\mathcal{L}}}}}
\def\slashedL{{\slashed{\mathcal{L}}}}

\def\slashedLRi{{\slashed{\mathcal{L}}_{R_i}}}

\def\chib{\underline{\chi}}
\def\chibh{\widehat{\underline{\chi}}}
\def\chit{\widetilde{\chi}}
\def\chibt{\widetilde{\underline{\chi}}}
\def\gslash{\mbox{$g \mkern -8mu /$ \!}}
\def\divslash{\slashed{\textrm{div}}}

\def\nablaslash{\mbox{$\nabla \mkern -13mu /$ \!}}
\def\laplacianslash{\mbox{$\Delta \mkern -13mu /$ \!}}
\def\nablat{{\widetilde{\nabla}}}
\def\Zb{{\underline{Z}}}
\def\Th{\widehat{T}}

\def\Et{{\widetilde{E}}}

\def\Ebt{{\underline{\widetilde{E}}}}
\def\Fbt{{\underline{\widetilde{F}}}}
\newcommand{\leftexp}[2]{{\vphantom{#2}}^{#1}{#2}}

\begin{document}
\title[Shock QLW]{On the formation of shocks for quasilinear wave equations}
\author[Shuang Miao]{Shuang Miao}
\thanks{Department of Mathematics, The University of Michigan,
Ann Arbor MI, U.S.A. shmiao@umich.edu}

\author[Pin Yu]{Pin Yu}
\thanks{Department of Mathematics and Yau Mathematical Sciences Center, Tsinghua University, Beijing, China. pin@math.tsinghua.edu.cn}

\begin{abstract}
The paper is devoted to the study of shock formation of the 3-dimensional quasilinear wave equation
\begin{equation}\label{Main Equation}
 - \big(1+3G^{\prime\prime}(0) (\partial_t\phi)^2\big)\partial^2_t \phi +\Delta\phi=0,\tag{\textbf{$\star$}}
\end{equation}
where $G^{\prime\prime}(0)$ is a non-zero constant. We will exhibit a family of smooth initial data and show that the foliation of the incoming characteristic hypersurfaces collapses. Similar to 1-dimensional conservational laws, we refer this specific type breakdown of smooth solutions as shock formation. Since $(\star)$ satisfies the classical null condition, it admits global smooth solutions for small data. Therefore, we will work with large data (in energy norm). Moreover, no symmetry condition is imposed on the initial datum.

We emphasize the geometric perspectives of shock formations in the proof. More specifically, the key idea is to study the interplay between the following two objects:

(1) the energy estimates of the linearized equations of $(\star)$;

(2) the differential geometry of the Lorentzian metric $g=-\dfrac{1}{\left(1+3G^{\prime\prime}(0) (\partial_t\phi)^2\right)} d t^2+dx_1^2+dx_2^2+dx_3^2$. Indeed, the study of the characteristic hypersurfaces (implies shock formation) is the study of the null hypersurfaces of $g$.

The techniques in the proof are inspired by the work \cite{Ch-Shocks} in which the formation of shocks for $3$-dimensional relativistic compressible Euler equations with small initial data is established. We also use the short pulse method which is introduced in the study of formation of black holes in general relativity in \cite{Ch-BlackHoles} and generalized in \cite{K-R-09}. 
\end{abstract}
\maketitle

\section{Introduction}

This paper is devoted to the study of the following quasilinear wave equation
\begin{equation}\label{main eq}
 - \big(1+3G^{\prime\prime}(0) (\partial_t\phi)^2\big)\partial^2_t \phi +\Delta\phi=0,
\end{equation}
where $G''(0)$ is a \emph{nonzero} constant and $\phi \in C^\infty(\mathbb{R}_t\times\mathbb{R}_{x}^3;\mathbb{R})$ is a smooth solution. We propose a geometric mechanism for shock formations, i.e. how the smoothness of $\phi$ breaks down. We remark that if $G''(0) = 0$, the equation is linear so that no shock is expected. We will see that $(\star)$ can be regarded as the simplest quasilinear wave equations that can be derived from the least action principle. The equation can also be regarded as a model equation for the nonlinear version of Maxwell equations in nonlinear electromagnetic theory, in which the shocks can be observed experimentally. The shock formation in nonlinear electromagnetic theory will be the subject of a forthcoming paper by the authors.

\medskip

The breakdown mechanism is a central object in the theory of quasilinear hyperbolic equations. We give a brief account on the results related to the current work. In \cite{Alinhac_Ann}, Alinhac proved a conjecture of H\"{o}rmander concerning upper bounds of the lifespan for the solutions of $-\partial_t^2 \phi +\Delta \phi = \partial_t \phi \, \partial_t^2 \phi$ on $\mathbb{R}^{2+1}$.
This equation was first introduced by F. John (see the survey paper \cite{John-90} and the references therein). He \cite{John-85} studied the rotationally symmetric cases and obtained upper bounds for the lifespan of the solutions. In \cite{Alinhac_Ann} and \cite{Alinhac_Acta}, without any symmetry assumptions, Alinhac not only shows the solution blows up but also gives a very precise description of the solution near the blow-up point.  Despite the slight different forms of the equations, Alinhac's results are fundamentally different from the current work in the following aspects: (1) He deals with small data problem. We will (and have to) deal with large data problem. (2) He uses Nash-Moser method to recover the loss of derivatives. Based on the variational nature of $(\star)$, we can close the energy estimates with finite many derivatives. (3) Just as Alihnac's work, we can give a detailed account on the behaviors of the solutions near blow-up
points. Moreover, we show that the singularities formed are indeed shocks. (4) We can give a pure geometric interpretation of the shock formation (in terms of certain curvature tensors) and show that the blow-up behavior indeed can be read off from the initial data directly.

\medskip

A major breakthrough in understanding the shock formations for the Euler equations has been made by D. Christodoulou in his monograph \cite{Ch-Shocks}. He considers the relativistic Euler equations for a perfect irrotational fluid with an arbitrary equation of state. Provided certain smallness assumptions on the initial data, he obtained a complete picture of shock formations in \emph{three} dimensions. A similar result for classical Euler's equations has also been obtained by Christodoulou and S. Miao in \cite{Ch-Miao}. The approaches are based on differential geometric methods originally introduced by Christodoulou and Klainerman in their monumental proof \cite{Ch-K} of the nonlinear stability of the Minkowski spacetime in general relativity.  Most recently, based on similar ideas, G. Holzegel, S. Klainerman, J. Speck and W. Wong have obtained remarkable results in understanding the stable mechanism for shock formations for certain types of quasilinear wave equations with small data in three dimensions, 
see their overview paper \cite{HKSW} and Speck's detailed proof \cite{Sp}. We remark that one of the key ideas in \cite{Ch-Shocks} and \cite{Ch-Miao} is to explore the variational structure of Euler's equations. This idea also plays a key r\^{o}le in the current work. We emphasize that \cite{Ch-Shocks} and \cite{Ch-Miao} obtained sharp lower and upper bounds for the lifespan of smooth solutions associated to the given data \emph{without} any symmetry conditions. Prior to \cite{Ch-Shocks} and \cite{Ch-Miao}, most of works on shock waves in fluid are limited to the simplified case of with spherical symmetry assumptions, i.e. essentially the one space dimension case. As an example, we mention \cite{Alinhac_Inven} of Alihnac which studies the singularity formations for the compressible Euler equations on $\mathbb{R}^2$ with rotational symmetry.

\medskip

All the aforementioned  works have the common feature that the initial data are assumed to be small. However, since the nonlinearity in $(\star)$ is cubic. By the classical result of Klainerman \cite{K-80}, for small smooth initial data, the solutions of $(\star)$ are globally regular. In particular, we do not expect shock formation. We will use a special family of large data, so called \emph{short pulse data}, in the current work. It was firstly introduced by D. Christodoulou in a milestone work \cite{Ch-BlackHoles} in understanding the formation of black holes in general relativity. By identifying an open set of initial data without any symmetry assumptions (the short pulse ansatz!), he shows that a trapped surface can form, even in vacuum space-time, from completely dispersed initial configurations and by means of the focusing effect of gravitational waves. Although the data are no longer close to Minkowski data, in other words, the data are no longer small, he is still able to prove a long time existence 
result for these data. This establishes the first result on the long time dynamics in general relativity and paves the way for many new developments on dynamical problems related to black holes. Shortly after Christodoulou's work, Klainerman and Rodnianski extends and significantly simplifies Christodoulou's work, see \cite{K-R-09}. From a pure PDE perspective, the data appeared in the above works are carefully chosen large profiles which can be preserved by the Einstein equations along the evolution. The data in the current work are inspired by these ideas, in particular the idea in \cite{Ch-BlackHoles}. The initial profiles are designed in such a way that the shape of the data will be preserved along the evolution of $(\star)$.

\medskip

As a summary, the data used in the paper are motivated by Christodoulou's work \cite{Ch-BlackHoles} and Klainerman-Rodnianski \cite{K-R-09} on the formation of black holes in general relativity. The ideas of the proof are motivated by Christodoulou's work \cite{Ch-Shocks} and Christodoulou-Miao \cite{Ch-Miao}. We have to overcome all the technical difficulties in the works mentioned above, in particular those in Christodoulou's works. At the same time, we would like to present a clearer geometric picture of the underlying shock formation mechanism.

\subsection{The heuristics for shock formations}
We rewrite $(\star)$ in the so called \emph{geometric form}:
\begin{equation*}
 -\frac{1}{c^2} \partial^2_t \phi +\Delta\phi=0 \quad \cdots \cdots (\star_g),
\end{equation*}
where $c =\big(1+3G^{\prime\prime}(0) (\partial_t\phi)^2\big)^{-\frac{1}{2}}$. Recall that, if $c$ was a constant, $(\star_g)$ would describe the propagation of light in Minkowski space and $c$ was the speed of light. In the current situation, we still regard $c$ (which is \textbf{not} a constant) as the speed of light. But the speed of light  depends on the position $(t,x)$ in spacetime and the solution $\phi$. This is of course the quasilinear nature of the equation. We now briefly review on the basics of shock formations for the inviscid Burgers' equation. The idea is to get a heuristic argument for the main equation $(\star)$ and to motivate the main theorem.

\medskip

The inviscid Burgers' equation can be written as
\begin{equation*}
\partial_t u + u\partial_x u =0 \quad \cdots \cdots (*).
\end{equation*}
We assume that $u \in C^\infty(\mathbb{R}_t\times \mathbb{R}_x; \mathbb{R})$ is a smooth solution. Given smooth initial datum $u(0,x)$(non-zero everywhere for simplicity), $(*)$ can be solved by the method of characteristics. A characteristic is a curve in $\mathbb{R}^2_{t,x}$ defined by the solution $u$. In the case of Burgers' equation, a characteristic is a straight line and it is determined by the initial datum $u(0,x)$ as follows: it is the unique line passing through $(0,x)$ with slope $\frac{1}{u(0,x)}$. The method of characteristics says that $u$ is constant along each of the characteristics.

\smallskip

To make connections to the geometric form $(\star_g)$ of the main equation $(\star)$, we also propose a \emph{geometric form} of the Burgers' equation (we assume that $u\neq 0$ to make the following computation legitimate):
\begin{equation*}
\frac{1}{c}\partial_t u + \partial_x u =0 \quad \cdots \cdots (*_{g}),
\end{equation*}
where $c = u$. Then $c$ is the speed of the characteristics and it depends on the solution $u$.

\smallskip

We now consider two specific characteristics passing through $x_1$ and $x_2$ ($x_1<x_2$). If we choose datum in such a way that $u(0,x_1) > u(0,x_2) >0$, both the characteristics travel towards the right. Moreover, the characteristic on the left (noted as $C_1$) travels with speed $c_1 = u(0,x_1)$ and the characteristic on the right (noted as $C_2$) travels with speed $c_2 = u(0,x_2)$. Since $C_1$ travels faster than $C_2$, $C_1$ will eventually catch up with $C_2$. The collision of two characteristics causes the breakdown on the smoothness of the solution. In summary, we have a geometric perspective on shock formation: a ``faster" characteristic catches up a ``slower" one so that it causes a collapse of characteristics.

\smallskip

The above discussion can also be read off easily from the following picture:

\includegraphics[width = 4.7 in]{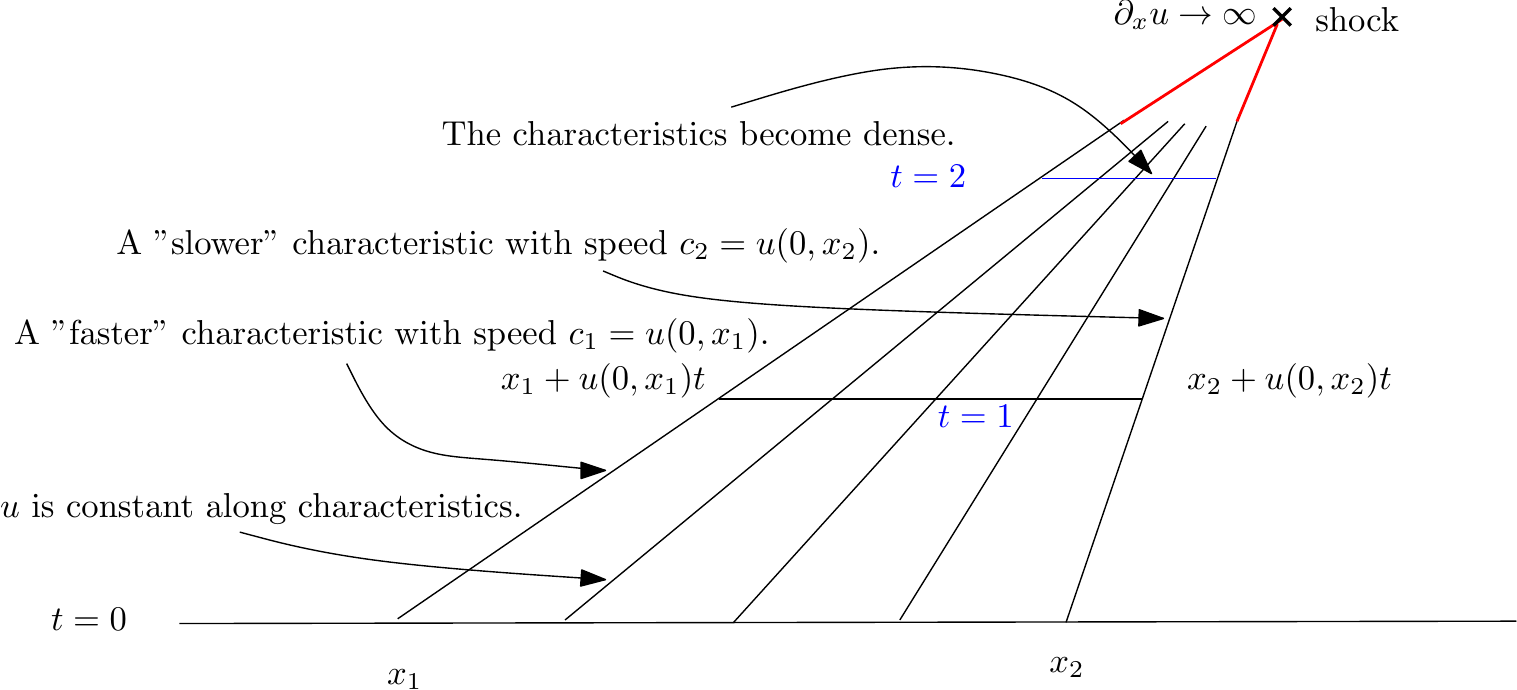}

In reality, in stead of showing that characteristics collapse (which is on the heuristic level),  we show that $|\partial_x u|$  blows up. In stead of being na\"{i}vely a derivative, $|\partial_x u|$ have an important \emph{geometric interpretation}. Recall that the level sets of $u$ are exactly the characteristics and the $(t,x)$-plane is foliated by the characteristics (see the above picture). Therefore, $|\partial_x u|$ is the \emph{density of the foliation by the characteristics}. As a consequence,  we can regard the shock formation as the following geometric picture: \emph{the foliation of characteristics becomes infinitely dense}.

\smallskip

We also recall a standard way to prove the blow-up of $|\partial_x u|$. The remarkable feature of this standard proof is that in three dimensions similar phenomenon happens for the main equation $(\star)$. Let $\underline{L} = \partial_t + u \partial_x$ be the generator vectorfield of the characteristics (for $(\star)$, the corresponding vectorfield are generators of null geodesics on the characteristic hypersurfaces). Therefore, by taking $\partial_x$ derivatives, we obtain
\begin{equation*}
\underline{L}  \partial_x u + (\partial_x u)^2 = 0.
\end{equation*}
This is  a Riccati equation for $\partial_x u$ and the blow-up theory for $\partial_x u$ is standard. However, we would like to understand the blow-up in another way (which is intimately tied to the shock formation for $(\star)$). We define the \emph{inverse density function} $\mu = - (\partial_x u)^{-1}$, therefore, along each characteristic curve, $\mu$ satisfies the following equation:
\begin{equation*}
\underline{L} \mu(t,x) = -1,
\end{equation*}
i.e. $\underline{L} \mu$ is  constant along each characteristic so that it is determined by its initial value. Therefore, $\mu$ will eventually become $0$ which implies that the foliation becomes infinitely dense (For $(\star)$, we will also define an inverse density function $\mu$ for the foliation of characteristic hypersurfaces and show that $\underline{L} \mu(t,x)$ is almost a constant along each generating geodesic of the characteristic hypersurfaces).

\medskip

We return to the main equation in the geometric form $-\frac{1}{c^2} \partial^2_t \phi +\Delta\phi=0$ with $c =\big(1+3G^{\prime\prime}(0) (\partial_t\phi)^2\big)^{-\frac{1}{2}}$. We prescribe initial data $(\phi(-2, \cdot), \partial_t \phi(-2, \cdot))$ on the time slice $\Sigma_{-2}$ defined by $t=-2$. We use $S_r$ to denote the sphere of radius $r$ centered at the origin on $\Sigma_{-2}$ and use $B_2$ to denote the ball of radius $2$ with boundary $S_2$. Therefore, the region enclosed by $S_2$ and $S_{2+\delta}$ (where $\delta$ is a small positive number) is foliated by the $S_r$'s for $2\leq r \leq 2+ \delta$. The following picture may help to illustrate the process.

\ \ \ \ \ \ \ \ \includegraphics[width = 5 in]{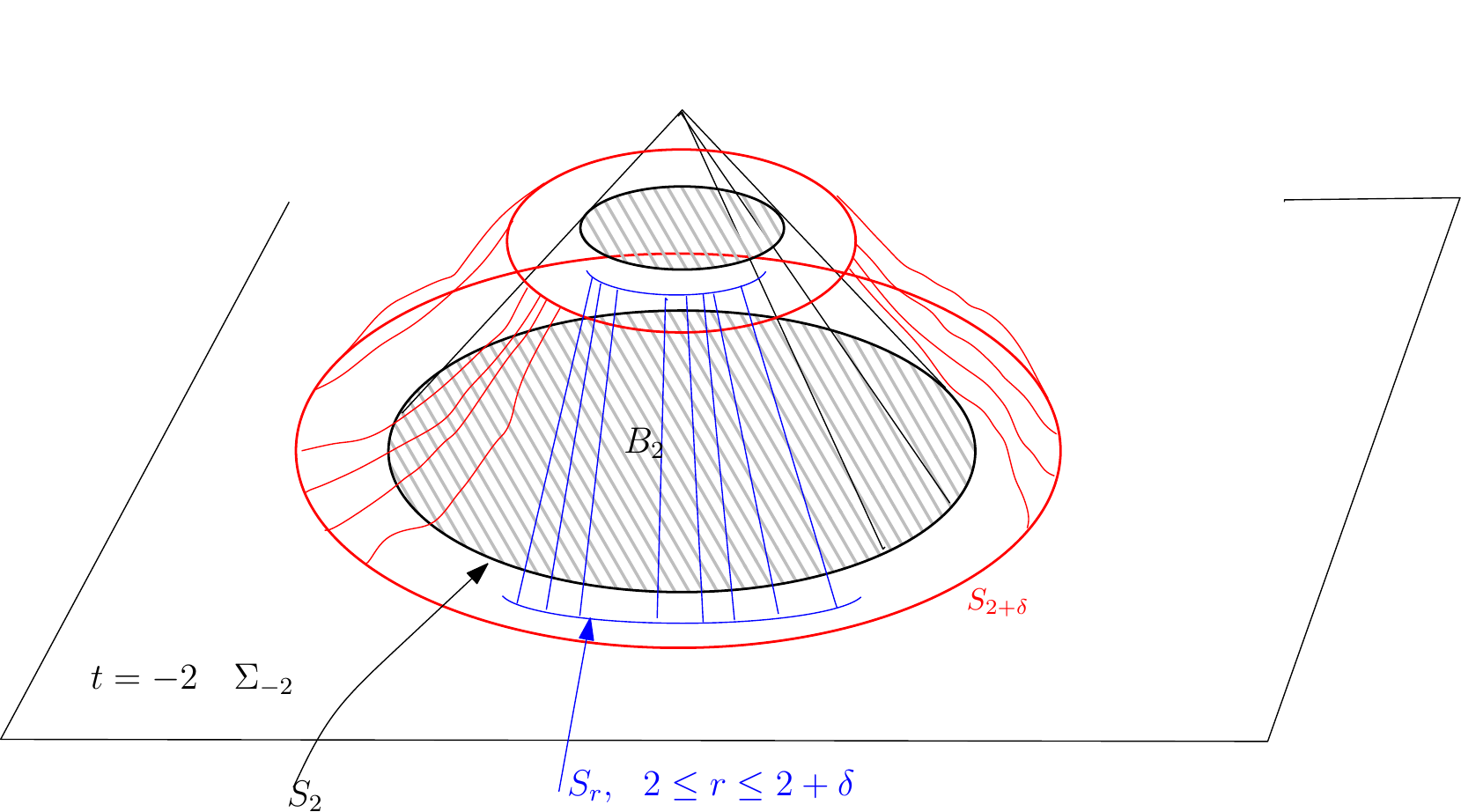}

For each leaf $S_r$ in the foliation, there is a unique incoming characteristic hypersurface, which will be defined more precisely in the next section, emanated from $S_r$. In the picture, we use a blue surface to denote it. The incoming characteristic hypersurfaces emanated from $S_2$ and $S_{2+\delta}$ are drawn in black and red respectively.

\smallskip

(1) Data inside $B_{2}$. We take trivial initial data $\phi(-2, x)\equiv 0$ and $\partial_t \phi(-2, x)\equiv 0$ inside $B_2$.

\smallskip

In view of the Huygens' principle, in the backward solid light cone with bottom $B_2$ (colored in black in the picture),  the solution $\phi$ is identically zero. In particular, for the incoming characteristic hypersurface, which is the cone in black in the picture emanated from the leaf $S_{2}$, since $c =\big(1+3G^{\prime\prime}(0) (\partial_t\phi)^2\big)^{-\frac{1}{2}}$, the incoming speed of this hypersurface, which is noted as $c_1$, can be computed as
\begin{equation*}
c_1= 1.
\end{equation*}

\smallskip

(2) Data on the annulus region between $S_{2}$ and $S_{2+\delta}$. This is the region between the black circle and the red circle in the picture. We require that the size of $\partial_t\phi$ is approximately $\delta^\frac{1}{2}$ on the outermost circle (which is red in the picture) $S_{2+\delta}$, i.e.   $|\partial_t \phi| \sim \delta^{\frac{1}{2}}$ at $S_{2+\delta}$ .

\smallskip

By Taylor expansion, we can compute the speed $c_2$ of the outermost incoming characteristics hypersurface emanated from $S_{2+\delta}$ (which is red in the picture) as follows
\begin{equation*}
c_2= 1 -\frac{3}{2}G''(0)\delta + O(\delta^2).
\end{equation*}

We are now in a situation that resembles the Burgers' picture. The initial distance between the inner most characteristic hypersurface (which is black in the picture) and the outer most characteristic hypersurface (which is red in the picture) is $\delta$. Both hypersurfaces travel towards the center. The difference of the speeds of two characteristic hypersurfaces is $c_1-c_2 \sim \delta$. We also expect the "faster"(outer) characteristic hypersurface catching up the "slower"(inner) one. This catching up process needs approximately $\frac{\text{distance}}{\text{speed}}=\frac{\delta}{c_1-c_2} \sim 1$ amount of time. We also regard the collision of characteristic hypersurfaces as shock formation, we hope that shocks form around $t=-1$.

\smallskip

We would like to point out a serious gap in the above heuristic argument. There is one assumption which seems to be very unreasonable: by the choice of the data, we can make sure that the speed $c_2$ of the outer most incoming characteristic hypersurfaces is of size $1 -\frac{3}{2}G''(0)\delta + O(\delta^2)$, but there is no clear reason that we should believe the speed $c_2$ remaining the same later on. Therefore, the difference of the two speeds $c_1-c_2$ may vary a lot so that the outer most characteristic hypersurface never catch up with the inner one.

\smallskip

The whole point of the paper is to identify a set of initial data so that the profile of the data propagates, i.e. the profile remains almost unchanged. In particular, we can prove that the speeds of the characteristic hypersurfaces remain almost unchanged for later time. Another way to understand this is through energy estimates: we can find a specific set of data so that we can obtain a priori energy estimates. Once we showed that the energy (and its higher order analogue) is almost conserved, we can use Sobolev inequality to show that $\partial_t \phi$ is almost conserved pointwisely along the generators of the incoming characteristic hypersurfaces. According to the formula of $c$, this also implies the speeds are almost conserved.

\smallskip

Finally, we point out that, as in the Burgers' equation case, instead of showing that characteristic hypersurfaces meet, we show that the inverse density $\mu$ of the foliation by the characteristic hypersurfaces becomes $0$, i.e. the foliation turns to being infinitely dense. Similarly, this can be done by showing that $\underline{L} \mu(t,x)$ is almost a constant along each generating geodesic of the characteristic hypersurfaces.

\subsection{The main result}

With motivations from the previous subsection, we are ready to state the main result of the paper. Let $t$ be the time function in Minkowski spacetime. We use $\Sigma_t$ to denote the level sets of $t$ and it is a copy of $\mathbb{R}^3$ for each $t$. We fix $r_0 = 2$ in this paper.\footnote{In a future work, we will consider a more general case for which the data is prescribed at past infinity. Therefore, we have to let $r_0$ go to $\infty$ and the dependence of the estimates (of the current work) on $r_0$ will be crucial.} We also use $\Sigma_{-r_0}^{\delta}$ to denote the following $\delta$-thin annulus:
\begin{equation}\label{delta annulus}
 \Sigma_{-r_{0}}^{\delta}:=\big\{x\in\Sigma_{-r_{0}} \big| r_{0}\leq r(x)\leq r_{0}+\delta\big\},
\end{equation}
where $\delta$ is any given small positive constant.

\smallskip

We recall that the wave speed $c$ is defined as $c=\big(1+3G''(0)(\partial_t\phi)^2\big)^{-\frac{1}{2}}$. Let $\Lb = \partial_t - c\partial_r$ and $L=\partial_t + c\partial_r$. We first introduce a pair of functions $\big(\phi_1(s,\theta),\phi_2(s,\theta)\big) \in C^\infty\big((0,1]\times \mathbb{S}^2\big)$ and we will call it the \emph{seed data}.

\smallskip

The seed data $\big(\phi_1(s,\theta),\phi_2(s,\theta)\big)$ can be freely prescribed and once it is given once forever. In particular, the choice of the seed data is independent of the small parameter $\delta$.

\begin{lemma}\label{lemma constraint}
 Given seed data $(\phi_1,\phi_2)$, there exists a $\delta'>0$ depending only on the seed data, for all $\delta<\delta'$, we can construct another function $\phi_0\in  C^\infty\big((0,1]\times \mathbb{S}^2\big)$ satisfying the following two properties:

(1) For all $k\in \mathbb{Z}_{\geq 0}$, the $C^k$-norm of $\phi_0$ are bounded by a function in the $C^k$-norms of $\phi_1$ and $\phi_2$;

(2) If we pose initial data for $(\star)$ on $\Sigma_{-2}$ in the following way:

For all $x \in \Sigma_{-2}$ with $r(x) \leq 2$, we require $\big(\phi(-2,x),  \partial_t \phi(-2,x)\big) = \big(0,0\big)$;\ For $2 \leq r(x) \leq 2+\delta$, we require that
\begin{equation*}
 \phi(-2,x)=\delta^{3/2}\phi_{0}\left(\frac{r-2}{\delta},\theta\right),\ \ (\partial_{t}\phi)(-2,x)=\delta^{1/2}\phi_{1}\left(\frac{r-2}{\delta},\theta\right).
\end{equation*}
Then we have
\begin{equation}\label{no outgoing radiation}
\|\Lb\phi\|_{L^\infty\left(\Sigma^{\delta}_{-r_{0}}\right)}\lesssim \delta^{3/2},\ \ \|\Lb^{2}\phi\|_{L^\infty\left(\Sigma^{\delta}_{-r_{0}}\right)}\lesssim \delta^{3/2}.
\end{equation}
\end{lemma}
We remark that the condition \eqref{no outgoing radiation} has a clear physical meaning: since $\Lb$ are incoming directions, the waves are initially set to be incoming and the outgoing radiation is very little (controlled by $\delta$).
\begin{definition}
The Cauchy initial data of $(\star)$ constructed in the lemma (satisfying the two properties) are called no-outgoing-radiation short pulse data
\end{definition}

Before we state the main theorem, we prove the lemma hence show the existence of no-outgoing-radiation short pulse data.

\begin{proof}
We recall that  $c=\big(1+3G''(0)(\partial_t\phi)^2\big)^{-\frac{1}{2}}$,$\,\Lb = \partial_t - c\partial_r$ and $L=\partial_t + c\partial_r$.

\smallskip

We first take an arbitrary choice of $\phi_1$ and fix this function. Therefore, $\partial_t \phi$ is given by the formula $(\partial_{t}\phi)(-2,x)=\delta^{1/2}\phi_{1}(\frac{r-2}{\delta},\theta)$. In particular, all the spatial derivatives of $\partial_t \phi$ and $c$ (determined completely by $\partial_t \phi$) are prescribed on $\Sigma^\delta_{-2}$.

\smallskip

Therefore, by definition, we have
$$\Lb^2 \phi = \partial_t^2 \phi -\partial_t c \, \partial_r \phi + c\partial_r c \, \partial_r \phi +c^2\partial_r^2 \phi-2c \partial_r(\partial_t \phi).$$
According to the definition of $c$, we have $\partial_t c = -3 G''(0)c^3 \partial_t \phi\,\partial_t^2 \phi$. Thus, we have
$$\Lb^2 \phi = \big(1-3G''(0)c^3\partial_t\phi\,\partial_r \phi\big)\partial_t^2 \phi + c\partial_r c \, \partial_r \phi +c^2\partial_r^2 \phi-2c \partial_r(\partial_t \phi).$$
By virtue of the main equation, we have $\partial_t^2\phi = c^2\big(\partial_r^2 \phi+\frac{2}{r}\partial_r \phi + \frac{1}{r^{2}}\laplacianslash_{\mathbb{S}^{2}} \phi\big)$, where $\laplacianslash_{\mathbb{S}^{2}} $ is the Laplace operator on $\mathbb{S}^{2}$. Therefore, we obtain
\begin{equation*}
\begin{split}
 \Lb^2 \phi &= \big(2-3G''(0)c^3\partial_t\phi\,\partial_r \phi\big)\big(c^2 \partial_r^2 \phi\big) + \big(1-3G''(0)c^3\partial_t\phi\,\partial_r \phi\big)\big(\frac{2c^2}{r}\partial_r\phi + \frac{c^2}{r^{2}} \laplacianslash_{\mathbb{S}^{2}} \phi \big)\\
 &\ \ \ +c\partial_r c \, \partial_r \phi-2c \partial_r(\partial_t \phi).
\end{split}
\end{equation*}
We now use the fact that $\phi(-2,x)=\delta^{\frac{3}{2}}\phi_{0}(\frac{r-2}{\delta},\theta)$, where $\phi_0$ will be determined later on. Thus, $\Lb^2 \phi$ can be computed as
\begin{equation}
\begin{split}
 \Lb^2 \phi &= \big(2-3G''(0)c^3\phi_1\,\partial_s \phi_0 \,\delta\big)\big(\delta^{-\frac{1}{2}}c^2 \partial_s^2 \phi_0 \big) + \big(1-3G''(0)c^3\phi_1\,\partial_s \phi_0 \,\delta\big)\big(\frac{2c^2}{r}\partial_s\phi_0 \, \delta^{\frac{1}{2}} + \frac{c^2}{r^{2}} \laplacianslash_{\mathbb{S}^{2}} \phi_0 \, \delta^\frac{3}{2} \big)\\
 &\ \ \ +c\partial_r c \, \partial_s \phi_0 \cdot \delta^\frac{1}{2}-2c \partial_s\phi_1 \delta^{-\frac{1}{2}}.
\end{split}
\end{equation}

\smallskip

We claim that we can choose $\phi_0$, which may depend on the choice of $\phi_1$ but is independent of $\delta$, in such a way that $|\Lb^2 \phi| \lesssim \delta^\frac{3}{2}$. 

\smallskip

To see this, we first observe that since $\phi_1$ is given and $\delta$ is small, we have $|c|+|\partial_r c| \lesssim 1$. Indeed, $\partial_r c = -3c^3 G''(0)\phi_1 \partial_s \phi_1$ so the bound on $\partial_r c$ is clear. We make the following ansatz for $\phi_0$:
\begin{equation}\label{ansatz on phi_0}
 |\partial_s \phi_0| +  |\partial^2_\theta \phi_0| \leq C,
\end{equation}
where the constant $C$ may only depend on $\phi_1$ but not on $\delta$.

\smallskip
By the ansatz \eqref{ansatz on phi_0} and by looking at the expansions in $\delta^\frac{1}{2}$, one can ignore all the terms equal to or higher than $\delta^{\frac{3}{2}}$. Therefore, to show $|\Lb^2 \phi| \lesssim \delta^\frac{3}{2}$, it suffices to consider
\begin{equation*}
\begin{split}
 \big(2-3G''(0)c^3\phi_1\,\partial_s \phi_0 \,\delta\big)\big(\delta^{-\frac{1}{2}}c^2 \partial_s^2 \phi_0 \big) +  \frac{2c^2}{r}\partial_s\phi_0 \, \delta^{\frac{1}{2}}    +c\partial_r c \, \partial_s \phi_0 \cdot \delta^\frac{1}{2}-2c \partial_s\phi_1 \delta^{-\frac{1}{2}} = O(\delta^\frac{3}{2}),
\end{split}
\end{equation*}
or equivalently
\begin{equation*}
\begin{split}
 \big(2-3G''(0)c^3\phi_1\,\partial_s \phi_0 \,\delta\big)\big( c^2 \partial_s^2 \phi_0 \big) +  \frac{2c^2}{r}\partial_s\phi_0 \, \delta    +c\partial_r c \, \partial_s \phi_0 \, \delta -2c \partial_s\phi_1  = O(\delta^2).
\end{split}
\end{equation*}
Since $ \big(2-3G''(0)c^3\phi_1\,\partial_s \phi_0 \,\delta\big)^{-1}=\frac{1}{2}+\frac{3}{4}G''(0)c^3\phi_1\,\partial_s \phi_0 \,\delta + O(\delta^2)$, by multiplying both sides of the above identity by $\big(2-3G''(0)c^3\phi_1\,\partial_s \phi_0 \,\delta\big)^{-1}$, it suffices to consider
\begin{equation*}
\begin{split}
 c^2 \partial_s^2 \phi_0   +  \frac{c^2}{r}\partial_s\phi_0 \, \delta+\frac{1}{2}c\partial_r c \, \partial_s \phi_0 \, \delta -\big(1+\frac{3}{2}G''(0)c^3\phi_1\,\partial_s \phi_0 \,\delta \big)c \partial_s\phi_1  = O(\delta^2).
\end{split}
\end{equation*}
Since $c\sim 1$, we finally have
\begin{equation*}
\begin{split}
\partial_s^2 \phi_0   +  \big(\frac{\delta}{r}  
+\frac{\delta}{2c} \partial_r c  -\frac{3\delta}{2}G''(0)c^2\phi_1   \partial_s\phi_1 \big)\partial_s \phi_0-c^{-1}\partial_{s}\phi_1  = O(\delta^2).
\end{split}
\end{equation*}
To solve for $\phi_0$, for $s \in [0,1)$, we consider the following family of (parametrized by a compact set of parameters $\theta \in \mathbb{S}^2$ and the parameter $\delta$) linear ordinary differential equation:
\begin{equation*}
\begin{split}
\partial_s^2 \phi_0   + & \big(\frac{\delta}{r}  +\frac{\delta}{2c} \partial_r c  -\frac{3\delta}{2}G''(0)c^2\phi_1   \partial_s\phi_1 \big)\partial_s \phi_0-c^{-1}\partial_{s}\phi_1  = \delta^2 \phi_{2},\\
&\ \ \ \ \ \ \ \ \ \phi_0(0,\theta)=0, \ \ \partial_s\phi_0(0,\theta) = 0.
\end{split}
\end{equation*}
Since the $C^k$-norms of the solution depends smoothly on the coefficients and the parameter $\theta, \delta$, all $C^k$-norms of $\phi_0$ are of order $O(1)$ and indeed are determined by the solution of 
\begin{equation*}
\begin{split}
&\ \ \  \partial_s^2 \phi_0  -c^{-1}\partial_{s}\phi_1  = 0, \\
&\   \phi_0(0,\theta)=0, \ \partial_s\phi_0(0,\theta) = 0.
\end{split}
\end{equation*}
In particular, this shows that the ansatz \eqref{ansatz on phi_0} holds if we choose $C$ appropriately large in \eqref{ansatz on phi_0} and $\delta$ sufficiently small. Therefore the above construction shows that 
$$|\Lb^2 \phi| \lesssim \delta^\frac{3}{2}.$$
\smallskip
We claim that, by the above choice of intial data, on $\Sigma_{-2}^\delta$, we automatically have
$$|\Lb \phi| \lesssim \delta^\frac{3}{2}.$$
\smallskip
Indeed, by replacing $\partial_t = \Lb +c\partial_r$ in the main equation, we obtain
\begin{equation*}
 \partial_r \Lb\phi = \frac{1}{2c}\Big(-\Lb^2\phi-Lc \, \partial_r\phi+\frac{2c^2}{r}\partial_r\phi+\frac{c^2}{r^{2}}\laplacianslash_{\mathbb{S}^{2}} \phi\Big).
\end{equation*}
By the construction of the data, it is obvious that all the terms on the right hand side are of size $O(\delta^\frac{1}{2})$. By integrating from $2$ to $r$ with $r\in [2,2+\delta)$ and $\Lb \phi(2,\theta)=0$, we have 
$$|\Lb\phi(r,\theta)|\leq \delta \cdot O(\delta^\frac{1}{2})\lesssim \delta^\frac{3}{2}.$$
\end{proof}

The main theorem of the paper is as follows:
\begin{MainTheorem}
For a given constant $G''(0)\neq 0$, we consider
\begin{equation*}
 - \big(1+3G^{\prime\prime}(0) (\partial_t\phi)^2\big)\partial^2_t \phi +\Delta\phi=0.
\end{equation*}
Let $(\phi_1,\phi_2)$ be a pair of seed data and the initial data for the equation is taken to be the no-outgoing-radiation initial data.

If the following condition on $\phi_{1}$ holds for at least one $(r,\theta) \in (0,1] \times \mathbb{S}^2$:
\begin{equation}\label{shock condition}
G''(0)\cdot\partial_r \phi_1(r,\theta) \cdot \phi_1(r,\theta) \leq -\frac{1}{6},
\end{equation}
then there exists a constant $\delta_0$ which depends only on the seed data $(\phi_1,\phi_2)$, so that for all $\delta< \delta_0$, shocks form for the corresponding solution $\phi$ before $t=-1$, i.e. $\phi$ will no longer be smooth.
\end{MainTheorem}

\begin{remark}
The choice of $\phi_1$ in the proof of Lemma \ref{lemma constraint} is arbitrary. In particular, this is \textbf{consistent} with the condition \eqref{shock condition} since $\phi_1$ can be freely prescribed.
\end{remark}

\begin{remark}
\begin{itemize}
\item[(1)] We do \emph{not} assume spherical symmetry on the initial data. Therefore, the theorem is in nature a higher dimensional result. 

\item[(2)] The proof can be applied to a large family of equations derived through action principles. We will discuss this point when we consider the Lagrangian formulation of \eqref{Main Equation}.

\item[(3)] The condition \eqref{shock condition} is \emph{only} needed to create shocks. It is not necessary at all for the a priori energy estimates.

\end{itemize}
\end{remark}

\begin{remark}\label{remark smooth break}
 The smoothness of $\phi$ breaks down in the following sense:
\begin{itemize}
 \item[1)]\ The solution and its first derivative, i.e. $\phi$ and $\partial \phi$, are always bounded. Moreover, $|\partial_t\phi| \lesssim \delta^{\frac{1}{2}}$, therefore \eqref{Main Equation} is always of wave type.

 \item[2)]\ The second derivative of the solution blows up. In fact, when one approaches the shocks, $\nabla\partial_t  \phi$ blows up. See Remark \ref{2nd derivative blow up} for the proof.
\end{itemize}
\end{remark}

\subsection{Lagrangian formulation of the main equation and its relation to nonlinear electromagnetic waves}
We briefly discuss the derivation of the main equation $(\star)$. The linear wave equation in Minkowski spacetime $(\mathbb{R}^{3+1}, m_{\mu\nu})$ can be derived by a variational principle: we take the Lagrangian density $L(\phi)$ to be $\frac{1}{2}\big(-(\partial_t \phi)^2 + |\nabla_x \phi|^2\big)$ and take the action functional $\mathcal{L}(\phi)$ to be $\int_{\mathbb{R}^{3+1}} L(\phi) d \mu_m$ where $d \mu_m$ denotes the volume form of the standard Minkowski metric $m_{\mu\nu}$. The corresponding Euler-Lagrange equation is exactly the linear wave equation  $-\partial^2_t\phi +\Delta \phi = 0$. We observe that the quadratic nature of the Lagrangian density result in the \textit{linearity} of the equation. This simple observation allows one to derive plenty of \textit{nonlinear} wave equations by changing the quadratic nature of the Lagrangian density. In particular, we will change the quadratic term in $\partial_t\phi$ to a quartic term, this will lead to a \textit{quasi-linear} wave equation.

In fact, we consider a perturbation of the Lagrangian density of linear waves:
\begin{equation}\label{Lagrangian density}
 L(\phi)= -\frac{1}{2}G\big((\partial_{t}\phi)^{2}\big)+\frac{1}{2}|\nabla\phi|^{2},
\end{equation}
where $G=G(\rho)$ is a smooth function defined on $\mathbb{R}$ and $\rho = |\partial_t  \phi|^2$. The corresponding Euler-Lagrange equation is
\begin{equation*}
 -\partial_t \big(G^{\prime}(\rho) \partial_t\phi\big)+\Delta\phi=0.
\end{equation*}

 The function $G(\rho)$ as a perturbation of $G_0(\rho)=\rho$ and therefore we can think of the above equation as a perturbation of the linear wave equation. For instance, we can work with a real analytic function $G(\rho)$ with $G(0)=0$ and $G^{\prime}(0)=1$. In particular, we can perturb $G(\rho) = \rho$ in the simplest possible way by adding a quadratic function so that $G(\rho) = \rho +\frac{1}{2}G''(0)\rho^2$. In this situation, we obtain precisely the main equation $(\star)$. It is in this sense that $(\star)$ can be regarded as the simplest quasi-linear wave equation derived from action principles.

The main equation $(\star)$ is also closely tight to electromagnetic waves in a nonlinear dielectric. The Maxwell equations in a homogeneous insulator derived from a Lagrangian $L$ which is a function of the electric field $E$ and the magnetic field $B$. The corresponding displacements $D$ and $H$ are defined through $L$ by $D=-\frac{\partial L}{\partial E}$ and $H=\frac{\partial L}{\partial B}$ respectively. In the case of an isotropic dielectric, $L$ is of the form
\begin{align}\label{Lagrangian Maxwell}
L=-\frac{1}{2}G(|E|^{2})+\frac{1}{2}|B|^2,
\end{align}
hence $H = B$. The fields $E$ and $B$ are derived from the scalar potential $\phi$ and the vector potential $A$ according to $E=-\nabla\phi-\partial_{t}A$ and $B=\nabla\times A$ respectively. This is equivalent to the first pair of Maxwell equations:
\begin{align*}
\nabla\times E+\partial_{t}B=0,\quad \nabla\cdot B=0.
\end{align*}
The potentials are determined only up to a gauge transformation $\phi\mapsto \phi-\partial_{t}f$ and $A\mapsto A+df$, where $f$ is an arbitrary smooth function. The second pair of Maxwell equations
\begin{align*}
\nabla\cdot D=0,\quad \nabla\times H-\partial_{t}D=0
\end{align*}
are the Euler-Lagrange equations, the first resulting from the variation of $\phi$ and the second resulting from the variation of $A$. Fixing the gauge by setting $\phi = 0$, we obtain a simplified model if we neglect the vector character of $A$ replacing it by a scalar function $\phi$. Then the above equations for the fields in terms of the potentials simplify to $E=-\partial_{t}\phi$ and $B=\nabla\phi$.
The Lagrangian \eqref{Lagrangian Maxwell} becomes
\begin{align*}
L=-\frac{1}{2}G((\partial_{t}\phi)^{2})+\frac{1}{2}|\nabla\phi|^{2}
\end{align*}
which is exactly \eqref{Lagrangian density}. Therefore, the main equation $(\star)$ provides a good approximation for shock formation in a natural physical model: the shock formations for nonlinear electromagnetic waves.

\subsection{Main features of the proof}

We now briefly sketch four main ingredients of the proof.

(1) The short pulse method. By rewriting $(\star)$ in the semilinear form $\Box_m \phi = -3 G''(0)(\partial_t\phi)^2 \partial_t^2\phi$, we notice that the nonlinearity is cubic. Therefore the result in \cite{K-80} implies that small smooth initial data lead to global smooth solutions since the classical null condition is satisfied. We are then forced to consider large initial data. According to the choice of data in the Main Theorem, they are supported in the annulus of width $\delta$ and with amplitude $\delta^{\frac{1}{2}}$ which looks like a pulse (the short pulse data). The energy associated to the data is  of size $1$. On the technical level, although the short pulse data is no longer small,
we still have a small parameter $\delta$ coming into play. Therefore, most of the techniques for small data problems can also be applied here.

(2) A Lorentzian geometry defined by the solutions. Since the Lagrangian, therefore $(\star)$ itself, is invariant under the time translation and the isometries of $\mathbb{R}^{3}$, we can linearize the equation via the infinitesimal generators of those actions. The most important feature about the linearized equations is, they are not just linear, they are linear \emph{wave} equations with respect to a special Lorentzian metric defined by the solution. This reflects the Lagrangian nature of the equation $(\star)$: the metric comes from the second derivative of the Lagrangian. In particular, the incoming null hypersurfaces with respect to this metric correspond to the characteristic hypersurfaces of the solution. Recall that the shock formation is the study of the collapsing of the characteristic hypersurfaces, hence the differential geometry of the metric dictates shock formation.

Moreover, we study the energy estimates for the linearized linear wave equations.  The energy estimates on one hand depend heavily on the underlying geometry, e.g. the curvature, the fundamental forms of null foliations, the isoperimetric inequalities, etc.; on the other hand, the energy estimates also control the underlying geometry. Therefore, the study of the linearized equations are more or less equivalent to the study of the underlying geometry. This leads to a natural bootstrap argument.

(3) Coercivity of angular energy near shocks. We use the vector field method to study the energy estimates for the linearized equations. Since we expect shock waves, the function $\mu$, i.e., the inverse density of the characteristic hypersurfaces, may turn to $0$. This will pose a fundamental difficulty for energy estimates (even for linear wave equations!). Roughly speaking, for a free wave $\psi$, for all possible multiplier vectorfields, in the associated energy or flux integrals, the components for the rotational directions all look like $\int \mu |\nablaslash \psi|^2$. But in the error integrals, some $\nablaslash \psi$ components show up \textit{without} a $\mu$ factor. In view of the fact that $\mu\rightarrow 0$ in the shock region, the above disparity in $\mu$ shows that one can not control the error integrals by the energy or flux terms.

This difficulty is of course tied to the formation of shocks. The remarkable thing is, it is also resolved by the formation of shocks. The idea is as follows: initially, the $\mu\sim 1$. If in the future, no shock forms, then the disparity of $\mu$ simply result in an universal constant in the estimates since $\mu$ will be bounded below and above. If shock forms eventually, then along the incoming direction $\mu$ decreases, i.e. $\Lb \mu <0$ where $\Lb$ is the generator of the incoming null geodesics. Although the error integrals contain many terms without factor $\mu$ for $\nablaslash{\psi}$, there is one term has a very special form: it looks like $\int\!\!\!\int \Lb\mu |\nablaslash \psi|^2$. The sign of $\Lb \mu$ in the shock region shows a miraculous coercivity of the energy estimates. This term is just enough to control all the $\nablaslash \psi$ terms appearing in the error terms. This is the major difference between the usual energy estimates and the case where shocks
form. The use of the sign $\Lb \mu$ is the key to the entire argument in the current work.

(4) The descent scheme. The energy estimates on the top order terms may suffer a loss of a factor in $\mu$ and this can be dangerous in the shock region. Indeed, some error integral looks like $\int_{-r_0}^t \mu^{-1}\frac{\partial\mu}{\partial t}E(\tau) d\tau$ where $E(\tau)$ for the energy (it appears also on the lefthand side of the energy identity). If $s^*$ is the time where shock forms, we can show that $\mu$ behaves like $|t-s^*|$ near shocks. Therefore, the presence of $\mu^{-1}$ cause a $\log$ loss in time. The descent scheme is designed to retrieve the loss. The idea is, rather than proving the top order terms are bounded in energy, we prove that the energy may blow up with a specific rate in $\mu$ to some negative power. To illustrate the idea, we do the following formal computations by assuming $\widetilde{E}(t) = \sup_{\tau \leq t}\mu^a E(\tau)$ is bounded for a large positive number $a$. The energy identity, which looks like $E(t) + \cdots \lesssim \int_{-r_0}^t \mu^{-1}\frac{\partial\mu}{\partial t}E(\tau) d\tau + \cdots$, can be rewritten as
\begin{align*}
E(t)+ \cdots 
&\lesssim \big(\int_{-r_0}^t \mu^{-(a+1)}\frac{\partial\mu}{\partial t} d\tau \big)\widetilde{E}(t) + \cdots.
\end{align*}
Since $\mu\sim |t-s^*|$, the term in the parenthesis gives a factor $\frac{1}{a}$ which is small ($a$ is large!). Therefore, the righthand side can be absorbed by the left hand side.

\section{The optical geometry}

\subsection{Optical metrics and linearized equations}
We observe that main equation $(\star)$ is invariant under the following symmetries: space translations, rotations and the time translation. Indeed, the Lagrangian $L(\phi)$ is invariant under these symmetries, hence the Euler-Lagrange equation must be invariant too. We use $A$ to denote any possible choice from $\{\partial_t, \partial_i, \Omega_{ij}= x^i \partial_j-x^j \partial_i\}$ where $i,j=1,2,3$ and $i<j$. These vectorfields correspond to the infinitesimal generators of the symmetries of $(\star)$.

We linearize $(\star)$ according to $A$ by the following procedure: We apply the symmetry generated by $A$ to a solution $\phi$ of $(\star)$ to obtain a family of solutions $\{\phi_{\tau}:\tau\in \mathbb{R}\big| \phi_0 = \phi\}$. Therefore, $ -\frac{1}{c(\phi_\tau)^2} \partial^2_t \phi_\tau +\Delta\phi_\tau=0$ for $\tau \in \mathbb{R}$. We then differentiate in $\tau$ and evaluate at $\tau=0$.  We define the so called \emph{variations} $\psi$ as
\begin{equation}\label{definition for variation}
 \psi:= A\phi = \frac{d\phi_{\tau}}{d\tau}|_{\tau=0}
\end{equation}
By regarding $\phi$ as a fixed function, this procedure produce a linear equation for $\psi$. We call it \emph{the linearized equation of $(\star)$ for the solution $\phi$ with respect to the symmetry $A$}.

In the tangent space at each point in $\mathbb{R}^{3+1}$ where the solution $\phi$ is defined, we introduce a Lorentzian metric $g_{\mu\nu}$ as follows
\begin{equation}\label{optical metric}
g = - c^2 dt \otimes d t +d x^1 \otimes d x^1 + d x^2 \otimes d x^2 + d x^3 \otimes d x^3,
\end{equation}
with $(t,x^{1},x^{2},x^{3})$ being the standard rectangular coordinates in Minkowski spacetime.
Since $c$ depends on the solution $\phi$, $g_{\mu\nu}$ also depends on the solution $\phi$. We also introduce a conformal  metric $\gt_{\mu\nu}$ with the conformal factor $\Omega=\dfrac{1}{c}$
\begin{align}
 \gt_{\mu\nu}=\Omega \cdot g_{\mu\nu} = \frac{1}{c} g_{\mu\nu}.
\end{align}
We refer $g_{\mu\nu}$ and $\gt_{\mu\nu}$ as the \textit{optical metric} and the \textit{conformal optical metric} respectively.

\begin{lemma}
The linearized equation of $(\star)$ for a solution $\phi$ with respect to $A$ can be written as
\begin{equation}\label{linearization conformal}
 \Box_{\tilde{g}}\psi=0,
\end{equation}
where $\Box_{\gt}$ is the wave operator with respect to $\gt$ and $\psi = A \phi$.
\end{lemma}
There are two ways to derive the linearized equations.

To derive \eqref{linearization conformal}, we can directly differentiate $(\star)$. We denote the Christoffel symbols of $\gt$ in the Cartesian coordinates by $\widetilde{\Gamma}_{\alpha\beta}^{\gamma}$. Let $\widetilde{\Gamma} ^\gamma = \gt^{\alpha\beta}\widetilde{\Gamma}_{\alpha\beta}^{\gamma}$, then $\widetilde{\Gamma} ^0 =-{2}{c^{-2}} \partial_t c$ and the other $\widetilde{\Gamma}^\gamma$'s vanish. Hence,
\begin{align*}
\Box_{\tilde{g}}\psi &= \gt^{\mu\nu}\partial_\mu\partial_\nu \psi - \widetilde{\Gamma}^\gamma\partial_\gamma \psi = c\big[\underbrace{\big(-\frac{1}{c^2} \partial^2_t \psi +\Delta\psi\big)}_{T_1} - \underbrace{\partial_\rho\big(\frac{1}{c^2}\big) \partial_t \rho \cdot \partial_t \psi}_{T_2}\big],
\end{align*}
where $\rho = (\partial_t\phi)^2$. We use $A$ to differentiate $(\star_g)$. If $A$ hits the factor $-\frac{1}{c^2}$, it yields $T_2$; since $A$ (the symmetries!) commutes with $\partial_t$ and $\Delta$, the other terms are precisely $T_1$.

There is a more natural proof which is standard in Lagrangian field theory, e.g. see \cite{Christodoulou-Acition Principle}. In fact, the linearized equation of $(\star)$ is the Euler-Lagrange equation of the linearized Lagrangian density $\dot{L}(\psi):=\frac{1}{2}\frac{d^{2}}{d\tau^{2}}\big|_{\tau=0} L(\phi+\tau\psi)$.
Since $G(\rho) = \rho+ \dfrac{1}{2}G''(0)\rho^2$, we have
\begin{align*}
 \dot{L}(\psi)&=-\frac{1}{2}G^{\prime}(0)(\partial_{t}\psi)^{2}-\frac{3}{2}G^{\prime\prime}(0)(\partial_{t}\phi)^{2}(\partial_{t}\psi)^{2}
+\frac{1}{2}|\nabla\psi|^{2}=\frac{1}{2}g^{\mu\nu}\partial_{\mu}\psi\partial_{\nu}\psi.
\end{align*}
Therefore, if $D\subset\mathbb{R}^{3+1}$ is a domain in which the solution $\phi$ is defined, the action corresponding to $\dot{L}$ is $\dot{\mathcal{L}}(\psi) = \frac{1}{2} \int_{D} g^{\mu\nu}\partial_{\mu}\psi\partial_{\nu}\psi \, d \mu_m$. We emphasize that the volume form $d\mu_m$ is defined by the Minkowski metric $m_{\alpha\beta}$. In view of the definitions of $g_{\mu\nu}$ and $\gt_{\mu\nu}$, the action $\dot{\mathcal{L}}(\psi)$ can be written as
\begin{equation*}
\dot{\mathcal{L}}(\psi) = \frac{1}{2} \int_{D} \tilde{g}^{\mu\nu}\partial_{\mu}\psi\partial_{\nu}\psi\, d \mu_{\tilde{g}}.
\end{equation*}
At this stage, it is clear that the linearized equation is the free wave equation with respect to $\Box_{\gt}$.

\subsection{Lorentzian geometry of the maximal development}
\subsubsection{The maximal development}
We define a function $\ub$ on $\Sigma_{-2}$ as follows:
\begin{align}\label{ub initial}
\ub:=r-2.
\end{align}
The level sets of $\ub$ in $\Sigma_{-2}$ are denoted by $S_{-2,\ub}$ and they are round spheres of radii $\ub+2$.  The annular region $\Sigma_{-2}^\delta$ defined in \eqref{delta annulus} is foliated by $S_{-2,\ub}$ as
\begin{align}\label{foliation initial}
\Sigma_{-2}^{\delta}:=\bigcup_{\ub\in[0,\delta]}S_{-2,\ub}.
\end{align}

Given an initial data set $(\phi,\partial_t \phi)\big |_{t=-2}$ defined on $B_{-2}^{2+\delta}=\bigcup_{\ub\in[-2,\delta]}S_{-2,\ub}$ to the main equation $(\star)$ (as we stated in the Main Theorem), we recall the notion of \emph{the maximal development} or maximal solution with respect to the given data. 

\smallskip

By virtue of the local existence theorem (to $(\star)$ with smooth data), one can claim the existence of a \emph{development} of the given initial data set, namely, the existence of 
\begin{itemize}
\item a domain $\mathcal{D}$ in Minkowski spacetime, whose past boundary is $B_{-2}^{2+\delta}$;
\item a smooth solution $\phi$ to $(\star)$ defined on $\mathcal{D}$ with the given data on $B_{-2}^{2+\delta}$ with following property: For any point $p\in\mathcal{D}$, if an inextendible curve $\gamma: [0,\tau) \rightarrow \mathcal{D}$ satisfies the property that 

(1) $\gamma(0)=p$,

(2) For any $\tau' \in [0,\tau)$, the tangent vector $\gamma'(\tau')$ is past-pointed and causal (i.e., $g(\gamma'(\tau'),\gamma'(\tau'))\leq 0$) with respect to the optical metric $g_{\alpha\beta}$ at the point $\gamma(\tau')$,

\noindent then the curve $\gamma$ 
must terminate at a point of  $B_{-2}^{2+\delta}$.
\end{itemize}

\ \ \ \ \ \ \ \ \includegraphics[width = 4.7 in]{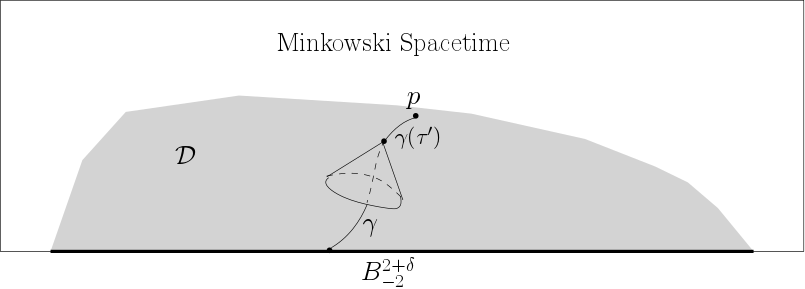}

By the standard terminology of Lorentzian geometry, the above simply says that $B_{-2}^{2+\delta}$ is a Cauchy hypersurface of $\mathcal{D}$.

\smallskip

The local uniqueness theorem asserts that if $(\mathcal{D}_{1},\phi_{1})$ and $(\mathcal{D}_{2},\phi_{2})$ are two developments of the same initial data sets, then $\phi_{1}=\phi_{2}$ in $\mathcal{D}_{1}\bigcap\mathcal{D}_{2}$. Therefore the union of all developments of a given initial data set is itself a development. This is the so called \emph{maximal development} and its corresponding domain is denoted by $W^*$. The corresponding solution is called the \emph{maximal solution}. Sometimes we also identify the development as its corresponding domain when there is no confusion.

\ \ \ \ \ \ \ \ \includegraphics[width = 4.7 in]{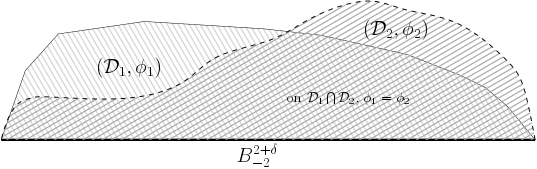}

\subsubsection{Geometric set-up}
Given an initial data set,  we consider a specific family of incoming null hypersurfaces (with respect to the optical metric $g$)  on the maximal development $W^*$. Recall that $\ub$ is defined on $\Sigma_{-2}$ as $r-2$. For any $\ub \in [0,\delta]$, we use $\Cb_{\ub}$ to denote the incoming null hypersurface  emanated from the sphere  $S_{-2,\ub}$.   By definition, we have $\Cb_{\ub} \subset W^*$ and $\Cb_{\ub}\bigcap\Sigma_{-2}=S_{-2,\ub}$. 

We denote the subset of the maximal development of  the given initial data foliated by $\Cb_{\ub}$ with $\ub \in [0,\delta]$ by $W_{\delta}$, i.e.,
\begin{equation}\label{W delta}
W_{\delta} =\bigcup_{\ub\in[0,\delta]}\Cb_{\ub}.
\end{equation}
Roughly speaking, our main estimates will be carried out only on $W_{\delta}$. The reason is as follows: since we assume that the data set is completely trivial for $\ub\leq 0$ on $\Sigma_{-2}$, the uniqueness of smooth solutions for quasilinear wave equations implies that the spacetime in the interior of $\Cb_{0}$ 
is indeed determined by the trivial solution.  Modulo the spherical configurations, the situation can be depicted as

\ \ \ \ \ \ \ \ \includegraphics[width = 4.7 in]{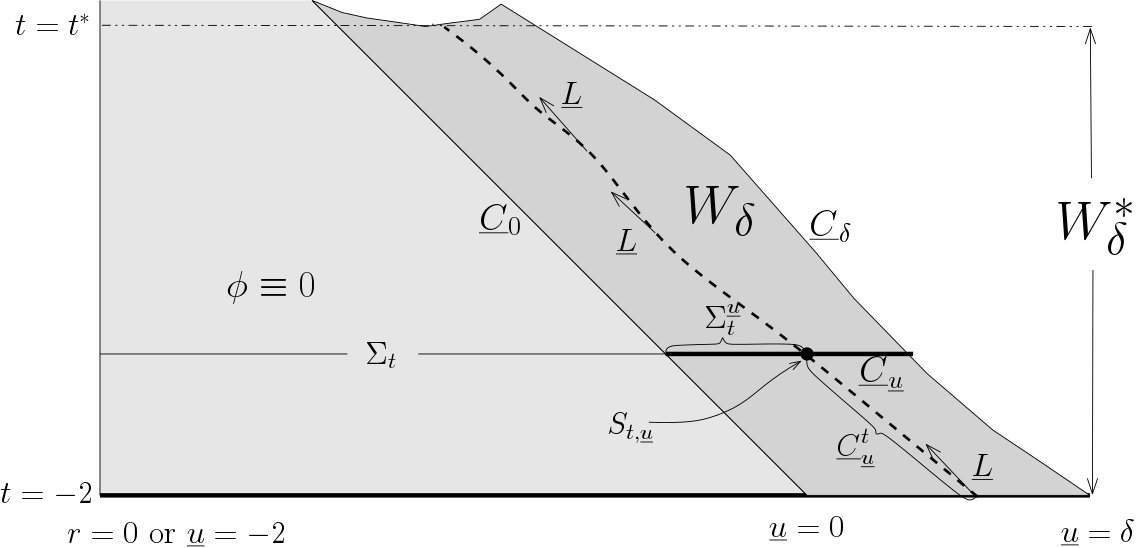}

The grey region is $W_\delta$. In the light gray region, the solution $\phi \equiv 0$ thanks to the construction of the initial data. The solution to \eqref{main eq} then vanishes up to infinite order on $\Cb_{0}$, which is part of the boundary of $W_{\delta}$. In particular, $\Cb_{0}$ is a flat cone in Minkowski spacetime (with respect to the Minkowski metric).

The dashed line denotes a incoming null hypersurface in the above picture. We extend the function $\ub$ to $W_{\delta}$ by requiring that the hypersurfaces $\Cb_{\ub}$ are precisely the level sets of the function $\ub$.  Since $\Cb_{\ub}$ is null with respect to $g_{\alpha\beta}$, the function $\ub$ is then a solution to the equation
\begin{align}\label{eikonal eq}
(g^{-1})^{\alpha\beta}\partial_{\alpha}\ub\partial_{\beta}\ub=0,
\end{align}
where $(g^{-1})^{\alpha\beta}$ is the  inverse of the metric $g_{\alpha\beta}$. 
We call such a function $\ub$ an \emph{optical function}. 

With respect to the affine parameter, the future-directed tangent vectorfield of a null geodesic on $\Cb_{\ub}$ is given by
\begin{align}\label{Lb affine}
\widehat{\underline{L}}:=-(g^{-1})^{\alpha\beta}\partial_{\alpha}\ub\,\partial_{\beta}.
\end{align}
However, for an apparent reason, which will be seen later, instead of using $\widehat{\underline{L}}$, we will work with a renormalized (by the time function $t$) vectorfield $\Lb$ defined through
\begin{align}\label{Lb}
\Lb=\mu\widehat{L},\quad \Lb t=1,
\end{align}
i.e., $\Lb$ is the tangent vectorfield of null geodesics parametrized by $t$. 

The function $\mu$ can be computed as
\begin{align*}
\frac{1}{\mu}=-(g^{-1})^{\alpha\beta}\partial_{\alpha}\ub\partial_{\beta}t.
\end{align*}
We will see later on that the $\mu$ also has a very important geometric meaning: ${\mu}^{-1}$ is the density of the foliation $\bigcup_{\ub\in[0,\delta]}\Cb_{\ub}$.

\medskip

Given $\ub \leq \delta$, to consider the density of null-hypersurface-foliation on $\Sigma_t \bigcap W_\delta$, we define
\begin{align}\label{definition of mu m}
\mu^{\ub}_{m}(t)=\min\left(\inf_{(\ub',\theta)\in[0,\ub]\times \mathbb{S}^{2}}\mu(t,\ub',\theta), 1\right).
\end{align}
For $\ub=\delta$, we  define 
$$s_{*} = \sup \big\{t\big | t\geq -2 \ \text{and} \ \mu_m^\delta(t)>0 \big\}.$$ 
From the PDE perspective, for the given initial data to $(\star)$ (as constructed in Lemma \ref{lemma constraint}), we also define 

\begin{align*}
t_{*} =  \sup \big\{&\tau \big | \tau \geq -2 \ \text{such that the smooth solution exists for all} \ (t,\ub)\in[-2,\tau)\times[0,\delta] \ \text{and} \ \theta\in\mathbb{S}^{2}\big\}.
\end{align*}

Finally, we define
\begin{align}\label{geodesic upper bound}
s^{*}=\min\{s_{*},-1\},\ \ t^{*}=\min\{t_{*},s^{*}\}.
\end{align}
We remark that we will exhibit data in such a way that the solution breaks down before $t=-1$. This is the reason we take $-1$ in the definition of $s^*$.

\smallskip

In the sequel, we will work in a further confined spacetime domain $W^{*}_{\delta}\subset W_\delta \subset W^*$ to prove a priori energy estimates. By definition, it consists of all the points in $W_{\delta}$ with time coordinate $t\leq t^{*}$, i.e.,
$$W^*_\delta =W_\delta \bigcap\Big( \bigcup_{-2\leq t \leq t^*} \!\!\!\!\Sigma_t\Big).$$
In the previous picture, the region $W_\delta^{*}$ is the part of the grey region below the horizontal dash-dot line. 

\smallskip

For the purpose of future use, we introduce more notations to describe various geometric objects.

\smallskip

For each $(t,\ub)\in[-2,t^{*})\times[0,\delta]$, we use $S_{t,\ub}$ to denote the closed two dimensional surface
\begin{align}\label{S t ub}
S_{t,\ub}:=\Sigma_{t}\bigcap\Cb_{\ub}.
\end{align}
In particular, we have
\begin{align}\label{W star}
W^{*}_{\delta}=\bigcup_{(t,\ub)\in[-2,t^{*})\times[0,\delta]}S_{t,\ub}.
\end{align}
For each $(t,\ub) \in [-2,t^*)\times [0,\delta]$, we define
\begin{align}\label{various notations}
\begin{split}
\Sigma_{t}^{\ub}:=&\big\{(t,\ub',\theta) \in \Sigma_t \mid 0\leq\ub' \leq \ub \big\},\\ \Cb_{\ub}^{t}:=&\big\{(\tau,\ub,\theta) \in \Cb_{\ub} \mid  -2\leq\tau \leq t\big\},\\
 W^{t}_{\ub}:=&\bigcup_{(t',\ub')\in[-2,t)\times [0,\ub]}S_{t',\ub'}.
 \end{split}
\end{align}
One can consult the previous picture to visualize those objects. In particular, $W_{\ub}^t$ is the grey region bounded by $\Sigma_t^{\ub}$ and $\Cb_{\ub}^t$.

In what follows when working in $W^{t}_{\ub}$, we usually omit the superscript $\ub$ to write $\mu^{\ub}_{m}(t)$ as $\mu_{m}(t)$, whenever there is no confusion. 

\smallskip

We define the vectorfield $T$ in $W^{*}_{\delta}$ by the following three conditions:
\begin{itemize}
\item[1)] $T$ is tangential to $\Sigma_{t}$;
\item[2)] $T$ is orthogonal (with respect to $g$) to $S_{t,\ub}$ for each $\ub\in[0,\delta]$;
\item[3)] $T\ub=1$.
\end{itemize}
The letter $T$ stands for ``transversal" since the vectorfield is transversal to the foliation of null hypersurfaces $\Cb_{\ub}$.

In particular, the point $1)$ implies
\begin{align}\label{T t}
Tt=0.
\end{align}
According to \eqref{eikonal eq}-\eqref{Lb}, we have
\begin{align}\label{Lb ub}
\Lb\ub=0,\quad \Lb t=1.
\end{align}
In view of \eqref{Lb}, \eqref{Lb ub}, \eqref{T t} and the fact $T\ub=1$, we see that the commutator 
\begin{align}\label{commutator Lb T}
\Lambda:=[\Lb,T]
\end{align}
is tangential to $S_{t,\ub}$.

In view of \eqref{eikonal eq}-\eqref{Lb} and the fact $T\ub=1$ we have
\begin{align}\label{Lb T}
g(\Lb, T)=-\mu, \quad g(\Lb,\Lb)=0.
\end{align}

Since $T$ is spacelike with respect to $g$ (indeed, $\Sigma_t$ is spacelike and $T$ is tangential to $\Sigma_t$), we denote
\begin{align}\label{T magnitude}
g(T,T)=\kappa^{2},\quad \kappa>0.
\end{align}

\begin{lemma}\label{lemma mu kappa}
We have the following relations for $\Lb$, $T$, $\mu$ and $\kappa$:
\begin{align}\label{mu kappa}
\mu=c\kappa,\ \ \Lb=\partial_{0}-c\kappa^{-1}T,
\end{align}
where $\partial_0$ is the standard time vectorfield in Minkowski spacetime.
\end{lemma}

\begin{proof}
The vectorfield $\partial_{0}$ is perpendicular to $\Sigma_{t}$ and therefore is perpendicular to $S_{t,\ub}$. Since $\Lb$ and $T$ are two linearly independent vectorfields perpendicular to $S_{t,\ub}$ and $\Lb t=\partial_{0}t=1$, we have
\begin{align*}
\partial_{0}=\Lb+f T
\end{align*}
for some scalar function $f$. On the other hand,  $\partial_{0}$ is perpendicular to $\Sigma_{t}$ hence to $T$, we have
\begin{align*}
0=g(\partial_{0},T)=g(\Lb,T)+fg(T,T)=-\mu+f\kappa^{2}.
\end{align*}
Therefore, $ f=\frac{\mu}{\kappa^{2}}=\frac{c}{\kappa}$ and the second formula in \eqref{mu kappa} follows. 

For the first formula, in view of the defining equation of the optical metric $g$, we have
\begin{align*}
-c^{2}=g(\partial_{0},\partial_{0})=g(\Lb+fT,\Lb+fT)=-2f\mu+f^{2}\kappa^{2}.
\end{align*}
Since $f=\frac{c}{\kappa}$, we can solve for $\mu$ to complete the proof.
\end{proof}

\begin{remark}
On the initial Cauchy surface $\Sigma_{-2}$, since $\ub=r-2$, we have $T=\partial_r$ and $\kappa=1$. Therefore, by using the standard rectangular coordinates, we obtain that 
$$\Lb = \partial_t -c\partial_r.$$
This is coherent with the notations and computations in  Lemma \ref{lemma constraint}.
\end{remark}

\subsubsection{The optical coordinates}

We construct a new coordinate system on $W_{\delta}^*$. If shocks form, the new coordinate system is completely different from the rectangular coordinates. Indeed, we will show that they define two differentiable structures on $W_\delta^*$ when shocks form.

\smallskip

Given $\ub\in[0,\delta]$, the generators of $\Cb_{\ub}$ define a diffeomorphism between $S_{-2,\ub}$ and $S_{t,\ub}$ for each $t\in[-2,t^{*})$. Since $S_{-2,\ub}$ is diffeomorphic to the standard sphere $\mathbb{S}^{2}\subset\mathbb{R}^{3}$ in a natural way. We obtain a natural diffeomorphism between $S_{t,\ub}$ and $\mathbb{S}^{2}$. If local coordinates $(\theta^{1},\theta^{2})$ are chosen on $\mathbb{S}^{2}$, the diffeomorphism then induces local coordinates on $S_{t,\ub}$ for every $(t,\ub)\in[-2,t^{*})\times[0,\delta]$. The local coordinates $(\theta^{1},\theta^{2})$, together with the functions $(t,\ub)$ define a complete system of local coordinates $(t,\ub,\theta^{1},\theta^{2})$ for $W^{*}_{\delta}$. This new coordinates are defined as \emph{the optical coordinates}. 

\smallskip
We now express for $\Lb$, $T$ and the optical metric $g$ in the optical coordinates. 

\smallskip
First of all, the integral curves of $\Lb$ are the lines with constant $\ub$ and $\theta$. Since $\Lb t=1$, therefore in optical coordinates we have
\begin{align}\label{Lb op}
\Lb=\frac{\partial}{\partial t}.
\end{align}
Similarly, since $T\ub=1$ and $T$ is tangential to $\Sigma_{t}$, we have
\begin{align}\label{T op}
T=\frac{\partial}{\partial\ub}-\Xi
\end{align}
with $\Xi$ a vectorfield tangential to $S_{t,\ub}$. Locally, we can express $\Xi$ as
\begin{align}\label{Xi}
\Xi=\sum_{A=1,2}\Xi^{A}\frac{\partial}{\partial\theta^{A}},
\end{align}
The metric $g$ then can be written in the optical coordinates $(t,\ub,\theta^{1},\theta^{2})$ as
\begin{align}\label{metric op}
g=-2\mu dtd\ub+\kappa^{2}d\ub^{2}+\slashed{g}_{AB}\left(d\theta^{A}+\Xi^{A}
d\ub\right)\left(d\theta^{B}+\Xi^{B}d\ub\right)
\end{align}
with
\begin{align}\label{gslash}
\slashed{g}_{AB}=g\left(\frac{\partial}{\partial\theta^{A}},\frac{\partial}{\partial\theta^{B}}\right), \ 1\leq A,B\leq 2.
\end{align}

\medskip

To study the differentiable structure defined by the optical coordinates, we study
the Jacobian $\triangle$ of the transformation from the optical coordinates $(t,\ub,\theta^{1},\theta^{2})$ to the rectangular coordinates $(x^0,x^{1},x^{2},x^{3})$. 

First of all, since $x^0=t$, we have 
\begin{align*}
\frac{\partial x^{0}}{\partial t}=&1,\quad \frac{\partial x^{0}}{\partial\ub}=\frac{\partial x^{0}}{\partial\theta^{A}}=0.
\end{align*}
Secondly, by \eqref{T op}, we can express $T =  T^i \partial_{i}$ in the rectangular coordinates $(x^1,x^2,x^3)$ as
\begin{align*}
T^{i}=\frac{\partial x^{i}}{\partial\ub}-\sum_{A=1,2}\Xi^{A}\frac{\partial x^{i}}{\partial\theta^{A}}
\end{align*}
In view of the fact that $T$ is orthogonal to $\frac{\partial}{\partial\theta^{A}}$ with respect to the Euclidean metric (which is the induced metric of $g$ on $\Sigma_{t}$!), we have

\begin{align*}
\triangle=\det\left(
\begin{array}{ccc}
T^{1}&T^{2}&T^{3}\\
\dfrac{\partial x^{1}}{\partial\theta^{1}}&\dfrac{\partial x^{2}}{\partial\theta^{1}}&\dfrac{\partial x^{3}}{\partial\theta^{1}}\\
\dfrac{\partial x^{1}}{\partial\theta^{2}}&\dfrac{\partial x^{2}}{\partial\theta^{2}}&\dfrac{\partial x^{3}}{\partial\theta^{2}}
\end{array}
\right)=\|T\|\left\|\dfrac{\partial}{\partial\theta^{1}}\wedge\dfrac{\partial}{\partial\theta^{2}}\right\|=c^{-1}\mu\sqrt{\det\slashed{g}},
\end{align*}
where $\|\cdot\|$ measures the magnitude of a vectorfield with respect to the Euclidean metric in $\mathbb{R}^{3}$ (defined by the rectangular coordinates $(x^1,x^2,x^3)$). 

\smallskip

We end the discussion by an important remark. We can also read the conclusions from the following picture:

\ \ \ \ \ \ \ \ \  \includegraphics[width = 5.5 in]{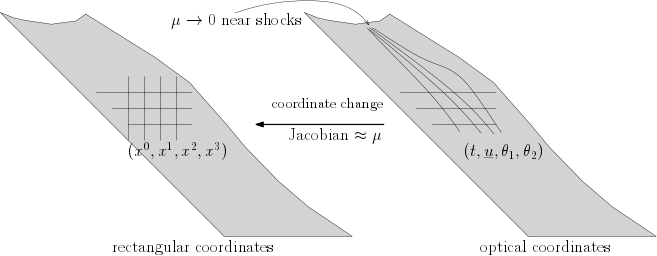}

\begin{remark}[{\bf Geometric meaning of $\mu$}]\label{geometric meaning of mu}
In the sequel, we will show that the wave speed function $c$ will be always approximately equal to $1$ in $W_\delta^*$. Since $\mu = c\kappa$, we may think of $\mu$ being $\kappa$ in a efficient way.

On the other hand, by the definition of $T$, in particular $T\ub=1$,  we know that $\kappa^{-1}$ is indeed the density of the foliation by the $\Cb_{\ub}$'s. This is because $g(T,T)=\kappa^2$. Since the optical metric coincides with the Euclidean metric on each constant time slice $\Sigma_t$, by $\mu \sim \kappa$, we arrive at the following conclusion:
\begin{itemize}
\item  $\mu^{-1}$ measures the foliation of the incoming null hypersurfaces $\Cb_{\ub}$'s.
\end{itemize}
Therefore, by regarding shock formation as the collapsing (i.e. the density blows up) of the characteristics ($\simeq$ the incoming null hypersurfaces), we may say that
\begin{itemize}
\item Shock formation is equivalent to $\mu \rightarrow 0$.
\end{itemize}

By virtue of the formula $\triangle=c^{-1}\mu\sqrt{\det\slashed{g}}$, it is clear (the volume element $\sqrt{\det\slashed{g}}$ will be controlled in the sequel) that if shock forms then the coordinate transformation between the optical coordinates and the rectangular coordinates will fail to be a diffeomorphism. Therefore, we can also say that
\begin{itemize}
\item Shock formation is equivalent to the fact that the optical coordinates on the maximal development defines a different differentiable structure (compared to the usual differentiable structure induced from the Minkowski spacetime).
\end{itemize}
\end{remark}

\subsection{Connection, curvature and structure equations}
We use $\nabla$ to denote the Levi-Civita connection of $g$ and use $X_A$ to denote $\frac{\partial}{\partial \theta^A}$. The $2^{\text{nd}}$ fundamental form of the embedding $S_{t,\ub} \hookrightarrow \Cb_{\ub}$ is
\begin{equation}\label{definition for chib}
\chib_{AB} = g(\nabla_{X_A} \Lb, X_B).
\end{equation}
The trace/traceless part is defined by $\tr\chib = \tr_{\gslash}\chib = \gslash^{AB}\chib_{AB}$ and $\chibh_{AB} = \chib_{AB}-\dfrac{1}{2}\tr\chib \,\gslash_{AB}$. Let $\widehat{T}=c\mu^{-1}T$. Then, $g(\Th,\Th)=1$. The $2^{\text{nd}}$ fundamental form of $S_{t,\ub} \hookrightarrow \Sigma_t$ is
\begin{equation}\label{definition for theta}
\theta_{AB} = g(\nabla_{X_A} \Th, X_B)
\end{equation}
By virtue of \eqref{optical metric}, we have $\underline{\chi}_{AB}=-{c}\,\theta_{AB}$.
Thanks to Gauss' Theorema Egregium, the Gauss curvature $K$ of $S_{t,\ub}$ is
\begin{equation}\label{Gauss equation}
K=\frac{1}{2}(\textrm{tr}_{\slashed{g}}\theta)^{2}-\frac{1}{2}|\theta|^{2}_{\slashed{g}} = \frac{1}{2}c^{-2}\big((\textrm{tr}_{\slashed{g}}\chib)^{2}-|\chib|^{2}_{\slashed{g}}\big).
\end{equation}

We introduce an outgoing null vectorfield
\begin{align}\label{outgoing null}
L = c^{-2}\mu\Lb+2T
\end{align}
 so that $g(L,\Lb) = -2\mu$. The corresponding  $2^{\text{nd}}$ fundamental form is $\chi_{AB} = g(\nabla_{X_A} L, X_B)$. Similarly, we define $\tr\chi = \tr_{\gslash}\chi = \gslash^{AB}\chi_{AB}$ and $\widehat{\chi}_{AB} = \chi_{AB}-\dfrac{1}{2}\tr\chi \,\gslash_{AB}$.

The torsion one forms $\etab_A$ and $\zetab_A$ are defined by $\zetab_A =g(\nabla_{X_A}\Lb,T)$ and $\etab_A=-g(\nabla_{X_A}T,\Lb)$.
They are related to the inverse density $\mu$ by $\etab_A=\zetab_A+X_A(\mu)$ and $\zetab_A=-c^{-1}\mu \, X_A(c)$.

The covariant derivative $\nabla$ is now expressed in the frame $(T,\Lb, X_1,X_2)$ by ($\nablaslash$ is the restriction of $\nabla$ on $S_{t,\ub}$)
\begin{align*}
 \nabla_{\Lb}\Lb&=\mu^{-1}(\Lb\mu)\Lb,\ \ \nabla_{T}\Lb=\etab^{A}X_{A}-{c}^{-1}\Lb(c^{-1}\mu)\Lb,\\
 \nabla_{X_{A}}\Lb&=-\mu^{-1}\zetab_{A}\Lb+\underline{\chi}_{A}{}^{B} X_{B},\ \  \nabla_{\Lb}T=-\zetab^{A}X_{A}-{c}^{-1} \Lb(c^{-1}\mu)\Lb,\\
 \nabla_{T}T &=c^{-3}\mu\big(Tc + \Lb (c^{-1}\mu)\big)\Lb+\Big(c^{-1}\big(Tc+\Lb(c^{-1}\mu)\big)+T\big(\log(c^{-1}\mu)\big)\Big)T-c^{-1}\mu\slashed{g}^{AB}X_{B}(c^{-1}\mu)X_{A},\\
\nabla_{X_{A}}T&=\mu^{-1}\etab_{A}T+c^{-1}\mu\theta_{AB}\slashed{g}^{BC}X_{C},\ \ \nabla_{\Lb}X_{A}=\nabla_{X_{A}}\Lb, \ \ \nabla_{X_{A}}X_{B}=\nablaslash_{X_{A}}X_{B}+\mu^{-1}\underline{\chi}_{AB}T.
\end{align*}
In terms of null frames $(L,\Lb, X_1,X_2)$, we have
\begin{align*}
\nabla_{L}{\Lb}&=-{\Lb}({c}^{-2}\mu){\Lb}+2\underline{\eta}^{A}X_{A}, \ \ \nabla_{{\Lb}}L=-2\underline{\zeta}^{A}X_{A},\\
\nabla_{L}L&=(\mu^{-1}L\mu + {\Lb}({c}^{-2}\mu))L-2\mu X^{A}({c}^{-2}\mu)X_{A}, \ \ \nabla_{X_{A}}L=\mu^{-1}\underline{\eta}_{A}L+\chi_{A}{}^{B}X_{B},\\
\nabla_{X_{A}}X_{B}&=\nablaslash_{X_{A}}X_{B}+\frac{1}{2}\mu^{-1}\underline{\chi}_{AB}L+\frac{1}{2}\mu^{-1}\chi_{AB}{\Lb}.
\end{align*}
In the Cartesian coordinates, the only non-vanishing curvature components are $R_{0i0j}$'s:
\begin{equation*}
R_{0i0j} =\frac{1}{2}\frac{d(c^{2})}{d\rho}\nabla_{i}\nabla_j\rho +\frac{1}{2}\frac{d^{2}(c^{2})}{d\rho^{2}}\nabla_{i}\rho\nabla_{j}\rho
-\frac{1}{4}c^{-2}|\frac{d(c^{2})}{d\rho}|^2\nabla_{i}\rho \nabla_{j}\rho.
\end{equation*}
In the optical coordinates, the only nonzero curvature components are $\alphab_{AB}=R(X_{A},\Lb,X_{B},\Lb)$:
\begin{equation*}
\alphab_{AB}= \frac{1}{2}\frac{d(c^{2})}{d\rho}\slashed{\nabla}^{2}_{X_A,X_B}\rho -\frac{1}{2}\mu^{-1}\frac{d(c^{2})}{d\rho}T(\rho) \chib_{AB}+\frac{1}{2}\big(\frac{d^{2}(c^{2})}{d\rho^{2}}
-\frac{1}{2}c^{-2}\big|\frac{d(c^{2})}{d\rho}\big|^2\big)X_{A}(\rho) X_{B}(\rho).
\end{equation*}
We define $$\alphab'_{AB}= \frac{1}{2}\frac{d(c^{2})}{d\rho}\slashed{\nabla}^{2}_{X_A,X_B}\rho +\frac{1}{2}[\frac{d^{2}(c^{2})}{d\rho^{2}}
-\frac{1}{2}c^{-2}|\frac{d(c^{2})}{d\rho}|^2]X_{A}(\rho) X_{B}(\rho).$$ Therefore, we write 
\begin{equation}\label{alpha expansion in mu and alpha prime}
\alphab_{AB} =  -\dfrac{1}{2}\mu^{-1}\dfrac{d(c^{2})}{d\rho}T(\rho) \chib_{AB}+ \alphab'_{AB}.
\end{equation}

\begin{remark}
As a convention, we say that the first term on the right hand side of \eqref{alpha expansion in mu and alpha prime}  is \emph{singular} in $\mu$ (since $\mu$ may go to zero). The second term $\alphab'_{AB}$ is \emph{regular} in $\mu$.

Indeed, in the course of the proof, we will see that $\alphab'_{AB}$ are bounded and $\alphab_{AB}$ behaves exactly as ${\mu}^{-1}$ in amplitude. Therefore, in addition to two equivalent descriptions of the shock formation in Remark \ref{geometric meaning of mu}, we have another geometric interpretation: 
\begin{itemize}
\item Shock formation is equivalent to the fact that curvature tensor of the optical metric $g$ becomes unbounded.
\end{itemize}
Compared to the one dimensional picture of shock formations in conservation laws, e.g., for inviscid Burgers equation, this new description of shock formation is purely geometric in the following sense: it does not even depend on the choice of characteristic foliation (because the curvature tensor is tensorial!).
\end{remark}

In the frame $(T, \Lb, \frac{\partial}{\partial\theta^A})$, the connection coefficients and the curvature components satisfies the following \emph{structure equations}:
\begin{equation}\label{Structure Equation Lb chibAB}
\Lb(\underline{\chi}_{AB})=\mu^{-1}(\Lb\mu)\underline{\chi}_{AB} +\underline{\chi}_{A}{}^{C}\underline{\chi}_{BC}-\alphab_{AB},
\end{equation}
\begin{equation}\label{Structure Equation div chib}
\divslash\underline{\chi}-\slashed{d}\textrm{tr}\underline{\chi} = -\mu^{-1}(\zetab\cdot\underline{\chi}-\zetab\textrm{tr}\underline{\chi}),
\end{equation}
\begin{equation}\label{Structure Equation T chib}
\slashed{\mathcal{L}}_{T}\underline{\chi}_{AB}=(\nablaslash\widehat{\otimes}\etab)_{AB}+\mu^{-1}(\zetab \hat{\otimes}\etab)_{AB}-{c}^{-1}\Lb(c^{-1}\mu)\underline{\chi}_{AB}+c^{-1}\mu(\theta \widehat{\otimes}\chib)_{AB},
\end{equation}
where $(\zetab\cdot\underline{\chi})_{B}=\slashed{g}^{AC}\zetab_{A}\underline{\chi}_{BC}$, $(\nablaslash\widehat{\otimes}\etab)_{AB} = \dfrac{1}{2}(\nablaslash_{A}\etab_{B}+\nablaslash_{B}\etab_{A})$, $(\zetab \widehat{\otimes}\etab)_{AB}=\frac{1}{2}(\zetab_{A}\etab_{B}+\zetab_{B}\etab_{A})$ and $(\theta \widehat{\otimes}\chib)_{AB}=\frac{1}{2}(\theta_{AC}\underline{\chi}_{B}^{C}+\theta_{BC}\underline{\chi}_{A}^{C})$.
By taking the trace of \eqref{Structure Equation Lb chibAB}, we have
\begin{equation}\label{Structure Equation Lb trchib}
\Lb\textrm{tr}\underline{\chi}=\mu^{-1}(\Lb\mu)\textrm{tr}\underline{\chi}-|\underline{\chi}|^{2}_{\slashed{g}}-\textrm{tr}\alphab.
\end{equation}

The inverse density function $\mu$ satisfies the following transport equation:
\begin{equation}\label{Structure Equation Lb mu}
\Lb\mu=m+\mu e,
\end{equation}
with $m=-\frac{1}{2}\frac{d(c^{2})}{d\rho}T\rho$ and $e=\frac{1}{2c^{2}}\frac{d(c^{2})}{d\rho}\Lb\rho$. With these notations, we have $\alphab_{AB} = \mu^{-1} m \chib_{AB} + \alphab'_{AB}$.

Regarding the regularity in $\mu$, we use \eqref{Structure Equation Lb mu} to replace $\Lb\mu$ in \eqref{Structure Equation Lb chibAB}. This yields
\begin{equation}\label{Structure Equation Lb chibAB nonsingular}
\Lb (\chib_{AB}) = e \chib_{AB} + \chib_{A}{}^{C}\chib_{BC}-\alphab_{AB}'.
\end{equation}
Compared to the original \eqref{Structure Equation Lb chibAB}, the new equation is \textit{regular} $\mu$ in the sense that it has no $\mu^{-1}$ terms.

\subsection{Rotation Vectorfields}\label{section definition for Ri}
Although $g\big|_{\Sigma_t}$ is flat, the foliation $S_{t,\ub}$ is different from the standard spherical foliations. In the Cartesian coordinates on $\Sigma_t$, let $\Omega_{1} = x^2 \partial_3 -x^3 \partial_2$, $\Omega_{2} = x^3 \partial_1 -x^1 \partial_3$ and $\Omega_{3} = x^1 \partial_2 -x^2 \partial_1$ be the standard rotations.  Let $\Pi$ be the orthogonal projection to $S_{t,\ub}$ (embedded in $\Sigma_t$). The rotation vectorfields $R_i \in \Gamma(TS_{t,\ub})$ ($i=1,2,3$) are defined by
\begin{equation}\label{definition for rotation vectorfields}
R_i = \Pi \,\Omega_{i}.
\end{equation}
Let indices $i, j, k \in \{1,2,3\}$. We use the $T^k$, $\Lb^k$ and $X_A^{k}$ to denote the components for $T$, $\Lb$ and $X_A$ in the Cartesian frame $\{\partial_{i}\}$ on $\Sigma_t$ (notice that $\Lb$ has also a $0$ component $L^{0} = 1$). We use $\widehat{T}= c\mu^{-1} T$ is the outward unit normal of $S_{t,\ub}$ in $\Sigma_{t}$. We introduce some functions to measure the difference between the foliations $S_{t,\ub}$ and the standard spherical foliations.

The functions $\lambda_i$'s measure the derivation from $R_i$ to $\Omega_i$:
\begin{equation}\label{definition for lambda_i}
 \lambda_i \widehat{T} = \Omega_i - R_i.
\end{equation}
The functions $y'^k$'s measure the derivation from $\widehat{T}$ to the standard radial vectorfield $\frac{x^i}{r}\partial_i$:
\begin{equation}\label{definition for y'}
 y'^k= \widehat{T}^k-\frac{x^k}{r}.
\end{equation}
We also define (we will show that $|y^k-y'^k|$ is bounded by a negligible small number)
\begin{equation}\label{definition for y}
 y^k= \widehat{T}^k-\frac{x^k}{\ub-t}.
\end{equation}
The functions $z^k$'s measure the derivation of $\Lb$ from $\partial_t - \partial_r$ in Minkowski spacetime:
\begin{equation*}
 z^k= \Lb^k+\frac{x^k}{\ub-t} = -\frac{(c-1)x^k}{\ub-t}-cy^k.
\end{equation*}
Finally, the rotation vectorfields can be expressed as
\begin{equation}\label{explicit expression of R i}
R_i = \Omega_i- \lambda_i \displaystyle\sum_{j=1}^3 \widehat{T}^j \partial_j,\quad \lambda_i= \displaystyle\sum_{j,k,l=1}^3 \varepsilon_{ilk}x^l y^k,
\end{equation}
where $\varepsilon_{ijk}$ is the totally skew-symmetric symbol.

\section{Initial data, bootstrap assumptions and the main estimates}
\subsection{Preliminary estimates on initial data}\label{section short pulse data}
In the Main Theorem, we take the so called short pulse datum for $(\star)$ on $\Sigma_{-r_0}^{\delta}$. Recall that $\phi(-r_0,x)=\delta^{3/2}\phi_{0}(\frac{r-r_{0}}{\delta},\theta)$ and $\partial_{t}\phi(-r_0,x)=\delta^{1/2}\phi_{1}(\frac{r-r_{0}}{\delta},\theta)$, where $\phi_{0}, \phi_{1}\in C^\infty_0\big((0,1]\times \mathbb{S}^2\big)$. The condition $(3)$ in the statement of the Main Theorem reads as
\begin{align*}
\|\Lb\phi\|_{L^\infty\left(\Sigma^{\delta}_{-2}\right)}\lesssim \delta^{3/2},\ \ \|\Lb^{2}\phi\|_{L^\infty\left(\Sigma^{\delta}_{-2}\right)}\lesssim \delta^{3/2}.
\end{align*}
We now derive estimates for $\phi$ and its derivatives on $\Sigma^{\delta}_{-2}$. These estimates also suggest the estimates, e.g. the bootstrap assumptions in next subsection, that one can expect later on.

For $\phi$ and $\psi = A \phi$ where $A \in \{\partial_\alpha\}$, by the form of the data, we clearly have
\begin{equation}\label{initial L infinity estimates for psi}
 \|\phi\|_{L^{\infty}(\Sigma^{\delta}_{-2})}\lesssim\delta^{3/2}, \  \|\psi\|_{L^{\infty}(\Sigma^{\delta}_{-2})}\lesssim\delta^{1/2}.
\end{equation}
We will use $Z$ or $Z_j$ to denote any vector from $\{T, R_i, Q\}$ where $Q= t\Lb$. On $\Sigma^{\delta}_{-2}$, $Z$ is simply $\partial_r$, $\Omega_{i}$ or $-r_0(\partial_{t}-\partial_{r})$, therefore, we have $\|Z_1 \circ Z_2 \circ \cdots \circ Z_m (\psi)\|_{L^{\infty}(\Sigma^{\delta}_{-2})} \lesssim \delta^{1/2-l}$
where $l$ is the number of $T$'s appearing in $\{Z_j\}_{1\leq j\leq m}$. We shall use the following schematic expression
\begin{equation}\label{initial L infinity estimates for commutations on psi}
\|Z^m \psi\|_{L^{\infty}(\Sigma^{\delta}_{-2})} \lesssim \delta^{1/2-l},
\end{equation}
with $l$ is the number of $T$'s and $Z \in \{T,\Omega_i,Q\}$. We remark that in this paper $l \leq 2$ and $Q$ appears at most  twice in the string of $Z$'s. 

We also consider the incoming energy for $Z^m \psi$ on $\Sigma^{\delta}_{-2}$. According to \eqref{initial L infinity estimates for commutations on psi}, we have
\begin{align*}
\|\Lb ( Z^m \psi )\|_{L^{2}(\Sigma^{\delta}_{-2})} + \|\slashed{d}(Z^m \psi)\|_{L^{2}(\Sigma^{\delta}_{-2})} \lesssim \delta^{1-l}, \  \|T(Z^m \psi)\|_{L^{2}(\Sigma^{\delta}_{-2})} \lesssim \delta^{-l}
\end{align*}
where $\ds$ denotes for the exterior differential on $S_{t,\ub}$. In terms of $L$, for $m\in \mathbb{Z}_{\geq 0}$, we obtain
\begin{equation}\label{initial energy estimates for commutations on psi}
\|\Lb ( Z^m \psi )\|_{L^{2}(\Sigma^{\delta}_{-2})} + \|\slashed{d}(Z^m \psi)\|_{L^{2}(\Sigma^{\delta}_{-2})}  \lesssim \delta^{1-l}, \ \ \ \|L (Z^m \psi)\|_{L^{2}(\Sigma^{\delta}_{-2})} \lesssim \delta^{-l}
\end{equation}
where $l$ is the number of $T$'s in $Z$'s.

We also consider the estimates on some connection coefficients on $\Sigma_{-r_0}$. For $\mu$, since we have $g(T,T) = c^{-2}\mu^{2}$ and $T=\partial_r$ on $\Sigma_{-r_0}$, we then have $\mu = c$ on $\Sigma_{r_0}$. Since $c =\big(1+ 2 G''(0)(\partial_t \phi)^2\big)^{-\frac{1}{2}}$, according to \eqref{initial L infinity estimates for psi}, for sufficiently small $\delta$, we obtain
\begin{equation}\label{initial estimates on mu}
\|\mu - 1\|_{L^\infty(\Sigma^{\delta}_{-2})} \lesssim \delta.
\end{equation}
For $\chib_{AB}$, since $\chib_{AB} = -c\theta_{AB} =-\dfrac{c}{r_0}\slashed{g}_{AB}$, we have $\chib_{AB} +\dfrac{1}{r_0}\slashed{g}_{AB}  = (1-c)\dfrac{1}{r_0}\slashed{g}_{AB}$.
Hence,
\begin{equation}\label{initial L infinity estimates for tr chib prime}
\|\chib_{AB} +\dfrac{1}{r_0}\slashed{g}_{AB}\|_{L^{\infty}(\Sigma^{\delta}_{-2})}  \lesssim \delta.
\end{equation}
It measures the difference between the $2^{\text{nd}}$ fundamental form with respect to $g_{\alpha\beta}$ and $m_{\alpha\beta}$.

\subsection{Bootstrap assumptions and the main estimates}\label{section Bootstrap Assumptions}
We expect the estimates \eqref{initial L infinity estimates for psi}, \eqref{initial L infinity estimates for commutations on psi} and \eqref{initial energy estimates for commutations on psi} hold not only for $t=-2$ but also for later time slice in $W^*_\delta$. For this purpose, we will run a bootstrap argument to derive the \emph{a priori} estimates for the $Z^m \psi$'s.

\subsubsection{Conventions}
We first introduce three large positive integers $N_{\text{top}}$, $N_{\mu}$ and $\Ninfty$. They will be determined later on.  We require that $\Nmu = \lfloor \frac{3}{4} \Ntop\rfloor$ and $\Ninfty = \lfloor \frac{1}{2} \Ntop\rfloor + 1$. $\Ntop$ will eventually be the total number of derivatives applied to the linearized equation $\Box_{\gt} \psi = 0$.

To count the number of derivatives, we define the \emph{order} of an object. The solution $\phi$ is considered as an order $-1$ object. The variations $\psi = A \phi$ are of order $0$. The metric $g$ depends only on $\psi$, so it is of order $0$. The inverse density function $\mu$ is of order $0$. The connection coefficients are 1st order derivatives on $g$, hence, of order $1$. In particular, $\chib_{AB}$ is of order $1$. Let $\alpha =(i_1,\cdots,i_{k-1})$ be a multi-index with $i_j$'s from $\{1,2,3\}$. We use $Z^{\alpha} \psi$ as a schematic expression of $Z_{i_1}Z_{i_2}\cdots Z_{i_{k-1}} \psi$. The order of $Z^\alpha\psi$ is $|\alpha|$, where $|\alpha| = k-1$. Similarly, for any tensor of order $|\alpha|$, after taking $m$ derivatives, its order becomes $|\alpha|+m$. The highest order objects in this paper will be of order $\Ntop +1$.

Let $l \in \mathbb{Z}_{\geq 0}$ and $k \in \mathbb{Z}$. We use $\O^l_k$ or $\O^{\leq l}_k$ to denote any term of order $l$ or at most $l$ with estimates
\begin{equation*}
 \|\O^l_k\|_{L^\infty(\Sigma^{\delta}_t)}\lesssim \delta^{\frac{1}{2}k}, \ \|\O^{\leq l}_k\|_{L^\infty(\Sigma^{\delta}_t)}\lesssim \delta^{\frac{1}{2}k}.
\end{equation*}
 Similarly, we use $\Psi^l_k$ or $\Psi^{\leq l}_k$ to denote any term of order $l$ or at most $l$ with estimates
\begin{equation*}
 \|\Psi^l_k\|_{L^\infty(\Sigma^{\delta}_t)}\lesssim \delta^{\frac{1}{2}k}, \ \|\Psi^{\leq l}_k\|_{L^\infty(\Sigma^{\delta}_t)}\lesssim \delta^{\frac{1}{2}k},
\end{equation*}
and moreover, it can be explicitly expressed a function of the variations $\psi$. For example,
$\partial_t \phi \cdot \partial_i\phi \in \Psi^0_{2}$; A term of the form $\prod_{i=1}^{n} Z^{\alpha_i} \psi$
so that $\max|\alpha_i|\leq m$ is $\Psi^{\leq m}_{n-2l}$, where $l$ is the number of $T$ appearing in the derivatives. Note that $\chib$ and $\mu$ can not be expressed explicitly in terms of $\psi$. The $\O^l_k$ terms (or similarly the $\Psi^l_k$ terms) obey the following algebraic rules:
\begin{align*}
 \O^{\leq l}_k + \O^{\leq l'}_{k'} &= \O^{\leq \max(l,l')}_{\min(k,k')},\ \   \O^{\leq l}_k  \O^{\leq l'}_{k'} = \O^{\leq \max(l,l')}_{k+k'}.
\end{align*}

\subsubsection{Bootstrap assumptions on $L^\infty$ norms}

Motivated by \eqref{initial L infinity estimates for commutations on psi}, we make the following bootstrap assumptions (B.1) on $W^*_\delta$: For all $t$ and $2\leq |\alpha|\leq \Ninfty$, \footnote{For a multi-index $\alpha$, the symbol $\alpha-1$ means another multi-index $\beta$ with degree $|\beta| = |\alpha|-1$.}
\begin{align*}\tag{B.1}
\|\psi\|_{L^{\infty}(\Sigma^{\delta}_{t})}+\|\Lb \psi\|_{L^\infty(\Sigma^{\delta}_t)}+\|\ds\psi\|_{L^\infty(\Sigma^{\delta}_t)} + \delta \|T \psi\|_{L^\infty(\Sigma^{\delta}_t)} + \delta^l\|Z^\alpha \psi\|_{L^\infty(\Sigma^{\delta}_t)}&\lesssim \delta^{\frac{1}{2}} M.
\end{align*}
where $l$ is the number of $T$'s appearing in $Z^\alpha$ and $M$ is a large positive constant depending on $\phi$. We will show that if $\delta$ is sufficiently small which may depend on $M$, then we can choose $M$ in such a way that it depends only on the initial datum.

\subsubsection{Energy norms}
Let $\dmug$ be the volume form of $\slashed{g}$, for a function $f(t,\ub,\theta)$, we define
\begin{equation*}
 \int_{\Sigma_{t}^{\ub}} f = \int_{0}^{\ub} \Big(\int_{S_{t,\ub'}}f(t,\ub',\theta)\dmug\Big) d \ub', \ \ \int_{\Cb_{\ub}^t} f= \int_{-r_0}^{t} \Big(\int_{S_{\tau,\ub}}f(\tau,\ub,\theta)\dmug\Big) d \tau.
\end{equation*}

For a function $\Psi(t,\ub,\theta)$, we define the energy flux through the hypersurfaces $\Sigma_{t}^{\ub}$ and $\Cb_{\ub}^{t}$ as
\begin{equation}\label{definition for energy and flux}
 \begin{split}
 {E}(\Psi)(t,\ub)&=\int_{\Sigma_{t}^{\ub}}(L\Psi)^2 + \mu^2 |\slashed{d} \Psi|^2, \ \  {F}(\Psi)(t,\ub)=\int_{\Cb_{\ub}^{t}} \mu |\slashed{d} \Psi|^2,\\
{\Eb}(\Psi)(t,\ub)&= \int_{\Sigma_{t}^{\ub}} \mu (\Lb \Psi)^2 + \mu |\slashed{d}\Psi|^2,\ \ {\Fb}(\Psi)(t,\ub)=\int_{\Cb_{\ub}^{t}} (\Lb \Psi)^2.
 \end{split}
\end{equation}

For each integer $0\leq k\leq \Ntop$, we define
\begin{equation}\label{definition for kth energy}
 \begin{split}
 {E}_{k+1}(t,\ub)=\sum_{\psi}\sum_{|\alpha| = k-1} \delta^{2l}E(Z^{\alpha+1} \psi)(t,\ub),\quad F_{k+1}(t,\ub)=  \sum_{\psi}\sum_{|\alpha| = k-1} \delta^{2l}F(Z^{\alpha+1} \psi)(t,\ub),\\
{\Eb}_{k+1}(t,\ub)=\sum_{\psi}\sum_{|\alpha| = k-1} \delta^{2l}\Eb(Z^{\alpha+1} \psi)(t,\ub),\quad \Fb_{k+1}(t,\ub)= \sum_{\psi}\sum_{|\alpha| = k-1}\delta^{2l}\Fb(Z^{\alpha+1} \psi)(t,\ub),
 \end{split}
\end{equation}
where $l$ is the number of $T$'s appearing in $Z^\alpha$.  The symbol $\displaystyle\sum_{\psi}$ means to sum over all the first order variations $A\phi$ of $\psi$. For the sake of simplicity, we shall omit this sum symbol in the sequel.

For each integer $0\leq k \leq \Ntop$, we assign a nonnegative integer $b_k$ to $k$ in such a way that
\begin{equation}
b_0 = b_1 =\cdots = b_{\Nmu} =0, \ \ b_{\Nmu +1} <  b_{{\Nmu}+2}< \cdots <  b_{\Ntop}.
\end{equation}
We call $b_k$'s \emph{the blow-up indices}. The sequence $(b_k)_{0\leq k \leq \Ntop}$ will be determined later on.

For each integer $0\leq k\leq \Ntop$, we also define the modified energy ${\Et}_k(t,\ub)$ and ${\Ebt}_k(t,\ub)$  as
\begin{equation}\label{definition for modified energy}
\begin{split}
 {\Et}_{k+1}(t,\ub)= \sup_{\tau \in[-r_{0},t]}\{\mu^{\ub}_{m}(\tau)^{2b_{k+1}}E_{k+1}(\tau,\ub)\},
 \quad {\Ebt}_{k+1}(t,\ub)=\sup_{\tau \in[-r_{0},t]}\{\mu^{\ub}_{m}(\tau)^{2b_{k+1}}\Eb_{k+1}(\tau,\ub)\}
\\{\widetilde{F}}_{k+1}(t,\ub)= \sup_{\tau \in[-r_{0},t]}\{\mu^{\ub}_{m}(\tau)^{2b_{k+1}}F_{k+1}(\tau,\ub)\},
 \quad {\widetilde{\Fb}}_{k+1}(t,\ub)=\sup_{\tau \in[-r_{0},t]}\{\mu^{\ub}_{m}(\tau)^{2b_{k+1}}\Fb_{k+1}(\tau,\ub)\}
 \end{split}
 \end{equation}

We now state the main estimates of the paper.
\begin{theorem}\label{theorem main estimates}
There exists a constant $\delta_0$ depending only on the seed data $\phi_0$ and $\phi_1$, so that for all $\delta< \delta_0$, there exist constants $M_0$, $\Ntop$ and $(b_k)_{0 \leq k \leq \Ntop}$ with the following properties
\begin{itemize}
\item $M_0$, $\Ntop$ and $(b_k)_{0 \leq k \leq \Ntop}$ depend only on the initial datum.
\item The inequalities (B.1) holds for all $t < t^*$ with $M=M_0$.
\item Either $t^{*}=-1$ and we have a smooth solution in the time slab $[-2,-1]$; or $t^{*}<-1$ and then $\psi_{\alpha}$'s as well as the rectangular coordinates $x^{i}$'s extend smoothly as functions of the coordinates $(t,\ub,\theta)$ to $t=t^{*}$ and there is at least one point on $\Sigma_{t^{*}}^{\delta}$ where $\mu$ vanishes, thus we have shock formation.
\item If, moreover, the initial data satisfies the largeness condition \eqref{shock condition}, then in fact $t^{*}<-1$.
\end{itemize}
\end{theorem}

Before we start the detailed analysis, it is instructional to provide a 3-step scheme to illustrate the structure of the proof (of Theorem \ref{theorem main estimates}):\\

\ \ \ \ \ \ \ \ \ \ \ \ \ \ \ \  \includegraphics[width = 4.7 in]{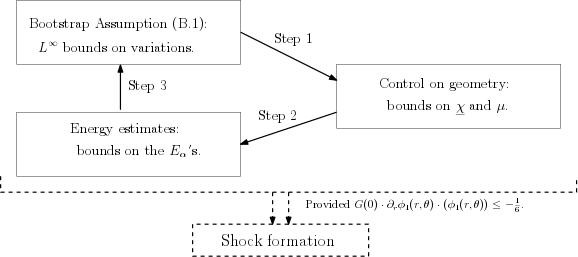}

\subsection{Preliminary results based on (B.1)}

\subsubsection{Estimates on metric and connection}
We start with bounds on $c$.
\begin{lemma} For sufficiently small\footnote{This sentence always means that, there exists $\varepsilon = \varepsilon(M)$ so that for all $\delta \leq \varepsilon$, we have ...} $\delta$, we have
\begin{equation}\label{bound on c}
\begin{split}
\|c-1\|_{L^{\infty}(\Sigma^{\delta}_{t})}&\lesssim \delta M^{2},
\frac{1}{2} \leq c \leq 2,\\
\|\Lb c\|_{L^\infty(\Sigma^{\delta}_t)} + \|{\slashed{d}} c\|_{L^\infty(\Sigma^{\delta}_t)} &\lesssim \delta M^2,\ \|T c\|_{L^\infty(\Sigma^{\delta}_t)} \lesssim M^2.
\end{split}
\end{equation}
\end{lemma}
\begin{proof}
Let $\psi_0 = \partial_t \phi$. Since $c =\big(1+ 3 G''(0)(\partial_t \phi)^2\big)^{-\frac{1}{2}}$, in view of (B.1), we can take $\varepsilon = \frac{1}{4 G''(0)M^2}$. Therefore, the quantity in the parenthesis falls in $[\dfrac{1}{4},\dfrac{7}{4}]$, this implies the bound on $c$.

On the other hand, $\Lb c =-(3/2)G''(0)\big(1+ 3 G''(0)\psi^{2}_{0}\big)^{-\frac{3}{2}} \psi_{0}\cdot \Lb \psi_{0}$. We then use (B.1) exactly in the same way as above, this gives the bound on $\Lb c$. Similarly, we can obtain other bounds in \eqref{bound on c}.
\end{proof}

We now derive estimates on $m$, $e$ and $\mu$:
\begin{lemma}
 For sufficiently small $\delta$, we have
\begin{equation}\label{bound on m}
 \|m\|_{L^\infty(\Sigma^{\delta}_t)} +\|{\slashed{d}} m\|_{L^\infty(\Sigma^{\delta}_t)} \lesssim M^2,\quad  \|T m\|_{L^\infty(\Sigma^{\delta}_t)} \lesssim \delta^{-1} M^2,
\end{equation}
\begin{equation}\label{bound on e}
 \|e\|_{L^\infty(\Sigma^{\delta}_t)}  +  \|{\slashed{d}} e\|_{L^\infty(\Sigma^{\delta}_t)}  \lesssim \delta M^2,\quad  \|T e\|_{L^\infty(\Sigma^{\delta}_t)}  \lesssim  M^2,
\end{equation}
\begin{equation}\label{bound on mu}
 \|\mu\|_{L^\infty(\Sigma^{\delta}_t)} +\|\Lb \mu\|_{L^\infty(\Sigma^{\delta}_t)}\lesssim  M^2.
\end{equation}
\end{lemma}

\begin{proof}
Let $\psi_{0} = \partial_t \phi$. We first bound $m$. Since $m \sim (1\pm \psi_{0}^{2})^{-2}  \psi_{0} \cdot T\psi_{0}$, according to (B.1), we have
\begin{align*}
 \|m\|_{L^\infty(\Sigma^{\delta}_t)} &\lesssim \frac{1}{(1-3G''(0) \|\psi_{0}\|^2_{L^\infty(\Sigma^{\delta}_t)})^2} \|\psi_{0}\|_{L^\infty(\Sigma^{\delta}_t)} \cdot \|T\psi_{0}\|_{L^\infty(\Sigma^{\delta}_t)} \lesssim \frac{1}{(1-3G''(0)\delta M^2)^2} \cdot\delta^{\frac{1}{2}} M \cdot\delta^{-\frac{1}{2}} M.
\end{align*}
This clearly implies the bound on $m$ for sufficiently small $\delta$. For the derivatives of $m$, say $Tm$, it is based on (B.1) and the explicit computation $T m =  \frac{3G''(0)}{(1+3G''(0) \psi_{0}^{2})^2}  \psi_{0} \cdot T^2\psi_{0} +\frac{3G''(0)}{(1+3G''(0) \psi^{2}_{0})^2}(T\psi_{0})^{2} -\frac{36 G''(0)^2}{(1+3G''(0) \psi^{2}_{0})^3}  \psi^{2}_{0}  (T\psi_{0})^2$. Similarly, we can bound $\ds m$.

The estimates on $e$ can be derived exactly in the same way, so we omit the proof.

To bound $\mu$, we integrate the equation $\Lb\mu=m+\mu e$ to derive
\begin{equation}\label{explicit formula for mu}
\mu(t,\ub, \theta)= \exp\big(\int_{-r_{0}}^{t}e(\tau)d\tau\big) \mu(-r_0,\ub,\theta) +\int_{-r_{0}}^{t}\exp(\int_{\tau}^{t}e(\tau^{\prime})d\tau^{\prime}) m(\tau,\ub,\theta)  d \tau.
\end{equation}
Then \eqref{initial estimates on mu}, \eqref{bound on m} and \eqref{bound on e} imply the estimate on $\mu$ immediately. For $\Lb \mu$, we simply combine $\Lb\mu=m+\mu e$, \eqref{bound on m}, \eqref{bound on e} and the bound on $\mu$. This completes the proof.
\end{proof}

We move to the bounds on $T\mu$ and $\slashed{d} \mu$:
\begin{lemma}
For sufficiently small $\delta$, we have
\begin{equation}\label{bound on T mu and slashed d mu}
 \|T\mu\|_{L^\infty(\Sigma^{\delta}_t)} \lesssim \delta^{-1} M^2, \quad \|\slashed{d} \mu\|_{L^\infty(\Sigma^{\delta}_t)} \lesssim M^2.
\end{equation}
\end{lemma}

\begin{proof}
The idea is to commute $T$ and $\slashed{d}$ with $\Lb \mu = m+ e \mu$. For $\slashed{d} \mu$, we have $\slashed{\mathcal{L}}_{\Lb} (\slashed{d} \mu) =  e\slashed{d}\mu +(\slashed{d}m + \mu \slashed{d} e)$.
We have already seen that $\|\slashed{d}m\|_{L^\infty(\Sigma^{\delta}_t)}+ \|\mu \slashed{d} e \|_{L^\infty(\Sigma^{\delta}_t)} \lesssim M^2$. Therefore, we integrate along $\Lb$ to derive the desired bound on $\slashed{d} \mu$.

For $T\mu$, we have $\Lb (T \mu) = eT\mu + \big( Tm  + \mu T e -(\zetab^A+\etab^A) \slashed{d}_A \mu \big)$.
We first show that
\begin{equation}\label{bound on zetab and etab}
\|\zetab\|_{L^\infty(\Sigma^{\delta}_t)} \lesssim \delta M^2, \ \ \|\etab\|_{L^\infty(\Sigma^{\delta}_t)} \lesssim M^2.
\end{equation}

In fact,  $\zetab_A = -c^{-1}\mu \slashed{d}_A (c)$. The bound on $\zetab$ the follows immediately from \eqref{bound on c} and \eqref{bound on mu}. Since $\etab = \zetab + \slashed{d}\mu$, the bound on $\etab$ is also clear because we have just obtained estimates on $\ds \mu$.

Back to the formula for $\Lb(T\mu)$, for small $\delta$, we bound the terms in the parenthesis by
\begin{equation*}
\|\big( Tm  + \mu T e -(\zetab^A+\etab^A) \slashed{d}_A \mu \big)\|_{L^\infty(\Sigma^{\delta}_t)} \lesssim \delta^{-1} M^2.
\end{equation*}
We then integrate to derive the bound for $T\mu$. This completes the proof.
\end{proof}

We now estimate $\chib_{AB}$. For this purpose, we introduce
\begin{equation}\label{definition for chib prime}
\chib'_{AB} = \chib_{AB} + \frac{\slashed{g}_{AB}}{\ub-t}.
\end{equation}
which measures the deviation of $\chib_{AB}$ from the null $2^{\text{nd}}$ fundamental form in Minkowski space. We have
\begin{lemma}
For sufficiently small $\delta$, we have
\begin{equation}\label{precise bound on chib'}
 \| \chib'_{AB}\|_{L^\infty(\Sigma^{\delta}_t)} \lesssim  \delta M^2.
\end{equation}
\end{lemma}
\begin{proof}
According to \eqref{Structure Equation Lb chibAB nonsingular}, we have
\begin{equation}\label{transport equation for chib prime}
\Lb (\chib_{AB}') = e \chib'_{AB} + \chib'_{A}{}^{C}\chib'_{BC}- \frac{e}{\ub - t}\slashed{g}_{AB} - \alphab'_{AB}.
\end{equation}
Hence, $\Lb|\chib'|^2_{\slashed{g}} =2e|\chib'|^2 -2\chib'_{A}{}^{B} \chib'_{B}{}^{C} \chib'_{C}{}^{A}+\frac{2|\chib'|^{2}}{\ub-t}-\frac{2e}{\ub-t}\tr\chib' - 2\chib'^{AB}\alphab_{AB}$.
Therefore, we obtain
\begin{align}\label{a 1}
\Lb((t-\ub)^{2}|\chib'|)\lesssim (t-\ub)^{2}\left((|e|+|\chib'|)|\chib'|+\frac{|e|}{\ub-t}+|\alphab'|\right)
\end{align}
where all norms are defined with respect to $\slashed{g}$. 

Let $\mathcal{P}(t)$ be the property that $\|\chib'\|_{L^{\infty}(\Sigma^{\delta}_{t})}\leq C_{0}\delta M^{2}$ for all $t'\in[-r_{0},t]$. By choosing $C_{0}$ suitably large, according to the assumptions on initial data, we have $\|\chib'\|_{L^{\infty}(\Sigma^{\delta}_{-2})}<C_{0}\delta$. It follows by continuity that $\mathcal{P}(t)$ is true for $t$ sufficiently close to $-r_{0}$. Let $t_{0}$ be the upper
bound of $t\in [-r_{0}, t_0]$ for which $\mathcal{P}(t)$ holds. By continuity, $\mathcal{P}(t_0)$ is true. Therefore, for $t\leq t_0$, we have $|\chib'|+|e|\leq (C_{0}+C_{1})\delta M^{2}$ for a universal constant $C_{1}$. According to the explicit formula of $\alphab'$ and (B.1), for sufficiently small $\delta$, there is a universal constant $C_3$ so that $\|\alphab'\|_{L^\infty(\Sigma^{\delta}_t)} \leq C_3\delta M^{2}$. In view of \eqref{a 1}, there is a universal constant $C_{4}$ so that
\begin{align}\label{chib transport}
\Lb((t-\ub)^{2}|\chib'|)\leq C_{4}(t-\ub)^{2}\left((C_{0}+C_{1})\delta M^{2}|\chib'|+(C_{1}+C_{3})\delta M^{2}\right).
\end{align}
If we define $x(t)=(t-\ub)^{2}|\chib'|$ along the integral curve of $\Lb$, then we can rewrite \eqref{chib transport} as $\frac{dx}{dt}\leq fx+g$, where $f(t)=C_{4}(C_{0}+C_{1})\delta M^{2}$ and $g(t)=C_{4}(C_{1}+C_{3})\delta M^{2}$. By integrating from $-r_{0}$ to $t$, we obtain
\begin{align*}
x(t)\leq e^{\int_{-r_{0}}^{t}f(t')dt'}\big(x(-r_{0})+\int_{-r_{0}}^{t}e^{-\int_{-r_{0}}^{t'}f(t'')dt''}g(t')dt'\big).
\end{align*}
Taking into account the facts that $\int_{-r_{0}}^{t}f(t')dt'\leq C_{4}(C_{0}+C_{1})\delta M^{2}$ and $\int_{-r_{0}}^{t}g(t')dt'\leq C_{4}(C_{1}+C_{3})\delta M^{2}$, since $(t-\ub)\sim 1$ on the support of $\chib'$, for some universal constant $C_{5}$, we have
\begin{align}\label{chib infty 1}
\|\chib'\|_{L^{\infty}(\Sigma^{\delta}_{t})}\leq C_{5}e^{C_{4}(C_{0}+C_{1})\delta M^{2}}\big(\|\chib'\|_{L^{\infty}(\Sigma^{\delta}_{-2})}+C_{4}(C_{1}+C_{3})\delta M^{2}\big).
\end{align}
We then fix $C_{0}$ in such a way that $C_{0}>2C_{5}C_{4}(C_{1}+C_{3})$ and $C_{0}>\frac{\|\chib'\|_{L^{\infty}(\Sigma^{\delta}_{t})}}{4C_{5}}$. Provided $\delta$ satisfying $\delta<\frac{\log 2}{C_{4}(C_{0}+C_{1})M^{2}}$, the estimate \eqref{chib infty 1} implies $\|\chib'\|_{L^{\infty}(\Sigma^{\delta}_{t})}<C_{0}\delta M^{2}$ for all $t\in[-r_{0},t_{0}]$. By continuity, $\mathcal{P}(t)$ holds for some $t>t_{0}$. Hence the lemma follows.\end{proof}

\begin{remark}[\textbf{Estimates related to the conformal optical metric $\gt$}]
 We shall use $\ \widetilde{ } \ $ to indicate the quantities defined with respect to $\gt$, e.g., $\widetilde{\nabla}$ is the Levi-Civita connection of $\gt$ and $\chibt$ is the $2^{\text{nd}}$ fundamental form of $S_{\ub,t} \hookrightarrow \Cb_{\ub}$ with respect to $\Lb$ and $\tilde{g}$: $\chibt_{AB} = \gt(\widetilde{\nabla}_{X_A} \Lb, X_B)$.

We expect the quantities (with $\ \widetilde{} \ $) defined with respect to $\gt$ have the similar estimates as the counterparts (without $\ \widetilde{} \ $) defined with respect to $g$. This is clear: the difference can be explicitly computed in terms of $c$ and hence controlled by the estimates on $c$. For example, the difference between $\chibt'_{AB}$ and $\chib'_{AB}$ is $\chibt_{AB}- \chib_{AB}= -\frac{1}{2c^2} \Lb(c)\slashed{g}_{AB}$, based on \eqref{bound on c} and \eqref{precise bound on chib'}, we have
\begin{equation}\label{precise bound on chibt}
 \| \chibt_{AB} + \frac{\slashed{g}_{AB}}{\ub-t}\|_{L^\infty(\Sigma^{\delta}_t)} \lesssim  \delta M^2.
\end{equation}
\end{remark}

\subsubsection{Estimates on deformation tensors}

We use five vectorfields $Z_1 = T, Z_2 = R_1, Z_3 = R_2, Z_4 =R_3, Z_5 = Q$ as commutation vectorfields in the paper.
We use $Z$ to denote any possible choice of the above $Z_i$'s. The notation $Z^\alpha$ for a multi-index
$\alpha =(i_1,\cdots,i_m)$ means $Z_{i_1} Z_{i_2} \cdots Z_{i_m}$ with $i_j \in \{1,2,3,4,5\}$.

We recall that, for $Z$, the deformation tensor $^{(Z)}{}\pi$ or $^{(Z)}{}\widetilde{\pi}$ with respect to $g$ and $\gt$ is defined by
${}^{(Z)}{}{\pi}_{\alpha\beta} = \nabla_\alpha Z_\beta + \nabla_\beta Z_\alpha$ or ${}^{(Z)}{}\widetilde{\pi}_{\alpha\beta} = \frac{1}{c}{}^{(Z)}{}{\pi}_{\alpha\beta} + Z(\frac{1}{c})g_{\alpha\beta}$.
We now compute the deformation tensors with respect to $\gt$. The deformation tensor of $Z_1 = T$ is given by \footnote{We emphasize that the trace $\tr$ is defined with respect to $g$.}
\begin{equation}\label{deformation tensor of T}
 \begin{split}
  ^{(T)}{}\widetilde{\pi}_{\Lb\Lb}&=0,\ \ ^{(T)}{}\widetilde{\pi}_{LL}= 4 c^{-1}\mu T(c^{-2}\mu),\ \ ^{(T)}{}\widetilde{\pi}_{L\Lb}= -2c^{-1}(T\mu -c^{-1} \mu T(c)),\\
 ^{(T)}{}\widetilde{\pi}_{L A}&=-c^{-3}\mu(\zetab_A+\etab_A),\ \ ^{(T)}{}\widetilde{\pi}_{\Lb A}=-c^{-1}(\zetab_A+\etab_A),\\
^{(T)}{}\widehat{\widetilde{\slashed\pi}}_{A B}&=-2c^{-3}\mu\widehat{\chib}_{AB},\ \ \text{tr}{}^{(T)}\widetilde{\slashed\pi}=-2c^{-3}\mu \text{tr}_{\tilde{\slashed{g}}}\widetilde{\chib}.
 \end{split}
\end{equation}
By virtue of the estimates already derived, for sufficiently small $\delta$, we have
\begin{equation}\label{L infinity etimates on the deformation tensor of T}
 \begin{split}
 \|\mu^{-1}{}^{(T)}{}\widetilde{\pi}_{LL}\|_{L^\infty(\Sigma^{\delta}_t)}&\lesssim \delta^{-1}M^2,\ \ \|{}^{(T)}{}\widetilde{\pi}_{L\Lb}\|_{L^\infty(\Sigma^{\delta}_t)}\lesssim \delta^{-1}M^2,\ \ \|\mu^{-1}{} ^{(T)}{}\widetilde{\pi}_{L A}\|_{L^\infty(\Sigma^{\delta}_t)}\lesssim M^2, \\
 \|{}^{(T)}{}\widetilde{\pi}_{\Lb A}\|_{L^\infty(\Sigma^{\delta}_t)}&\lesssim M^2, \ \ \|\mu^{-1} \leftexp{(T)}{\widehat{\widetilde{\slashed\pi}}}_{A B}\|_{L^\infty(\Sigma^{\delta}_t)}\lesssim \delta M^2,\ \ \|\mu^{-1}\text{tr}{}^{(T)}\widetilde{\slashed\pi}\|_{L^\infty(\Sigma^{\delta}_t)}\lesssim 1.
 \end{split}
\end{equation}
The deformation tensor of $Z_{5}=Q$ is given by
\begin{equation}\label{deformation tensor of Q}
\begin{split}
\leftexp{(Q)}{\tilde{\pi}}_{\Lb\Lb}&=0,\ \leftexp{(Q)}{\tilde{\pi}}_{LL}=4t\Lb(c^{-2}\mu)(c^{-1}\mu)-4c^{-3}\mu^{2},\  \leftexp{(Q)}{\tilde{\pi}}_{L\Lb}=-2t\Lb(c^{-1}\mu)
-2c^{-1}\mu\\
\leftexp{(Q)}{\tilde{\pi}}_{\Lb A}&=0,\  \leftexp{(Q)}{\tilde{\pi}}_{LA}=2tc^{-1}(\zetab_{A}+\etab_{A}), \  \leftexp{(Q)}{\widehat{\tilde{\slashed{\pi}}}}_{AB}
=2tc^{-1}\widehat{\chib}_{AB},\  \text{tr}\leftexp{(Q)}{\tilde{\slashed{\pi}}}=2c^{-1}t\text{tr}_{\tilde{\slashed{g}}}\widetilde{\chib}.
\end{split}
\end{equation}
By virtue of  the estimates already derived, for sufficiently small $\delta$, we have
\begin{equation}\label{Linfty of deformation tensor of Q}
\begin{split}
\|\mu^{-1}\leftexp{(Q)}{\tilde{\pi}}_{LL}\|_{L^{\infty}(\Sigma^{\delta}_{t})}&\lesssim M^{2},\ \|\leftexp{(Q)}{\tilde{\pi}}_{L\Lb}\|_{L^{\infty}(\Sigma^{\delta}_{t})}\lesssim M^{2},\ \|\leftexp{(Q)}{\tilde{\pi}}_{LA}\|_{L^{\infty}(\Sigma^{\delta}_{t})}\lesssim M^{2},\\
\|\leftexp{(Q)}{\widehat{\tilde{\slashed{\pi}}}}_{AB}\|_{L^{\infty}(\Sigma^{\delta}_{t})}&\lesssim\delta M^{2},\ \|\text{tr}\leftexp{(Q)}{\tilde{\slashed{\pi}}}\|_{L^{\infty}(\Sigma^{\delta}_{t})}\lesssim 1.
\end{split}
\end{equation}

Actually the estimate for $\tr\leftexp{(Q)}{\widetilde{\slashed{\pi}}}$ can be improved more precisely. By \eqref{Linfty of deformation tensor of Q} : $\|\text{tr}\leftexp{(Q)}{\tilde{\slashed{\pi}}}\|_{L^{\infty}(\Sigma_{t}^{\ub})}\lesssim 1$ and the equation \eqref{box psi_0 in null frame} together with the bootstrap assumptions, we see the order of the former is also $1$. While the estimate for the latter is a bit more delicate. Let us rewrite the following component of deformation tensor of $Q$:
\begin{align*}
\text{tr}\leftexp{(Q)}{\tilde{\slashed{\pi}}}=2c^{-1}t
\text{tr}_{\tilde{\slashed{g}}}
\widetilde{\chib}=2c^{-1}t
\text{tr}_{\tilde{\slashed{g}}}
\widetilde{\chib}'+4\frac{c^{-1}\ub}{\ub-t}-4(c^{-1}-1)-4
\end{align*}
This tells us:
\begin{align}\label{Improved estimate for trQ}
\|\text{tr}\leftexp{(Q)}{\tilde{\slashed{\pi}}}+4\|_{L^{\infty}(\Sigma_{t}^{\delta})}\lesssim\delta.
\end{align}

We study the deformation tensors of $R_i$'s. Based on \eqref{explicit expression of R i}, we compute $^{(R_i)}{}{\pi}_{\alpha\beta}$ with respect to $g$:
\begin{equation}\label{deformation tensor of Rotational R_i in T,Lb,X_A}
\begin{split}
  ^{(R_i)}{}{\pi}_{\Lb\Lb}&=0,\  ^{(R_i)}{}{\pi}_{TT}= 2c^{-1}\mu \cdot R_i(c^{-1}\mu), \  ^{(R_i)}{}{\pi}_{\Lb T}= -R_i (\mu), \ \ ^{(R_i)}{}{\pi}_{AB}= -2\lambda_i \theta_{AB},\\
^{(R_i)}{}{\pi}_{T A}
&=-c^{-1}\mu (\theta_{AB}-\frac{\slashed{g}_{AB}}{\ub-t})R_i{}^B + c^{-1} \mu \varepsilon_{ikj} y^k X_A{}^{j} + \lambda_i X_A (c^{-1}\mu),\\
^{(R_i)}{}{\pi}_{\Lb A}&=-\chib_{AB}R_i{}^B + \Lb^k \varepsilon_{ikj}X_A{}^{j} + c\mu^{-1}\lambda_i \zetab_A=-(\chib_{AB}+\frac{\slashed{g}_{AB}}{\ub-t})R_i{}^B +  \varepsilon_{ikj}z^k X_A{}^{j} + c\mu^{-1}\lambda_i \zetab_A.
\end{split}
\end{equation}
The Latin indices $i,j,k$ are defined with respect to the Cartesian coordinates on $\Sigma_t$. To bound deformation of $R_{i}$, it suffices to control the $\lambda_i$'s, $y^i$'s and $z^i$'s.

First of all, we have
\begin{equation}\label{estimates on T hat i and Lb i}
|\widehat{T}^i|+|\Lb^i|\lesssim 1.
\end{equation}
The proof is straightforward: $g\big|_{\Sigma^{\delta}_t}$ is flat and $\widehat{T}$ is the unit normal of $S_{t,\ub}$ in $\Sigma^{\delta}_t$, so $|\widehat{T}^i|\leq 1$. In the Cartesian coordinates $(t,x^1,x^2,x^3)$, $\Lb = \partial_t - c\widehat{T}^i \partial_i$, so $|\Lb^i|\lesssim 1$.

Let $r=(\sum_{i=1}^3 x^i)^{\frac{1}{2}}$. Since $Tr =  c^{-1}\mu \sum_{i=1}^3 \frac{x^i \widehat{T}^i}{r}$,
\eqref{estimates on T hat i and Lb i} implies that $|T r| \lesssim M^2$. We then integrate from $0$ to $\ub$,
since $r=-t$ when $\ub = 0$ and $|\ub| \leq \delta$, we obtain $|r+t|\lesssim \delta M^2$. In application, for sufficiently small $\delta$,
we often use $r \sim |t|$. The estimate can also be written as
\begin{equation}\label{estimates on r precise}
\left|\frac{1}{r} - \frac{1}{\ub+|t|}\right|\lesssim \delta M^{2}.
\end{equation}

To control $\lambda_i$, we consider its $\Lb$ derivative. By definition $\lambda_i = g(\Omega_i, \widehat{T})$,
we can write its derivative along $\Lb$ as $\Lb \lambda_i = \sum_{k=1}^3 (\Omega_i){}^k \Lb \widehat{T}^k=\sum_{k=1}^3 (\Omega_i){}^k X^A(c) X_A{}^k$.
As $|t| \sim r$, we have  $|\Omega_i| \lesssim |t| \lesssim 1$, this implies
\begin{equation}\label{estimates on Lb lambda_i}
 \|\Lb \lambda_i\|_{L^\infty(\Sigma^{\delta}_t)} \lesssim \delta M^2.
\end{equation}
Since $\lambda_i = 0$ on $\Sigma^{\delta}_{-r_0}$, we have
\begin{equation}\label{estimates on lambda_i}
 \|\lambda_i\|_{L^\infty(\Sigma^{\delta}_t)} \lesssim \delta M^2.
\end{equation}
To control $y^i$'s and $z^i$'s, let $\overline{y}=(y^1,y^2,y^3)$ and $\overline{x}=(x^1,x^2,x^3)$, we then have
\begin{equation}\label{y estimates 1}
 |\overline{y}-(\frac{1}{r}-\frac{1}{\ub-t})\overline{x}|^2=|(g(\widehat{T},\partial_r)-1)\partial_r|^2+\frac{1}{r^{2}}\sum_{i=1}^{3}\lambda_{i}^{2}
\end{equation}
On the other hand, we have $1-|g(\widehat{T},\partial_{r})|^{2}=\frac{1}{r^{2}}\sum_{i=1}^{3}\lambda_{i}^{2}\lesssim \delta M^{2}$.
While on $S_{t,0}$,
Since $g(\partial_{r},\widehat{T})=1$ on $S_{t,0}$, for sufficiently small $\delta$, the angle between $\partial_{r}$ and $\widehat{T}$ is
less than $\frac{\pi}{2}$, which implies $1+g(\widehat{T},\partial_{r})\geq 1$. Therefore,
\begin{align}\label{angle estimate}
|1-g(\widehat{T},\partial_{r})|\lesssim \delta M^{2}.
\end{align}
Together with
\eqref{estimates on lambda_i} and \eqref{y estimates 1}, this implies
\begin{equation}\label{estimates on y^i}
 |y^i|\lesssim \delta M^2,\quad |y'^{i}|\lesssim \delta M^{2}.
\end{equation}
We then control $z^i$ from its definition
\begin{equation}\label{estimates on z^i}
 |z^i|\lesssim \delta M^2.
\end{equation}
The derivatives of $\lambda_{i}$ on $\Sigma_{t}$ are given by $X_{A}(\lambda_{i})=\left(\theta_{AB}-\frac{\slashed{g}_{AB}}{\ub-t}\right)R^{B}_{i}-\varepsilon_{ikj}y^{k}X^{j}_{A}$ and $T(\lambda_{i})=-R_{i}(c^{-1}\mu)$. Hence,
\begin{align*}
\|\ds\lambda_{i}\|_{L^{\infty}(\Sigma^{\delta}_{t})}\lesssim \delta M^{2},\quad \|T\lambda_{i}\|_{L^{\infty}(\Sigma^{\delta}_{t})}\lesssim M^{2}
\end{align*}
Finally, we obtain the following estimates for the deformation tensor of $R_i$:\footnote{We emphasize that the traceless part of $\widehat{{\slashed\pi}}_{AB}$ is defined with respect to $\slashed{g}$.}
\begin{equation*}
 \begin{split}
  ^{(R_i)}{}{\pi}_{\Lb \Lb}=0, \ \ \|{} ^{(R_i)}&{}{\widetilde{\pi}}_{\Lb T}\|_{L^\infty(\Sigma^{\delta}_t)}\lesssim M^2, \ \  \|\mu^{-1}{}^{(R_i)}{}{{\pi}}_{TT}\|_{L^\infty(\Sigma^{\delta}_t)} \lesssim M^2,\\
\|{}^{(R_i)}{}{{\pi}}_{\Lb A}\|_{L^\infty(\Sigma^{\delta}_t)} &\lesssim \delta M^2,\ \ \|{}^{(R_i)}{}{\pi}_{T A}\|_{L^\infty(\Sigma^{\delta}_t)} \lesssim \delta M^4,\\
\|{}^{(R_i)}{}\widehat{{\slashed\pi}}_{AB}\|_{L^\infty(\Sigma^{\delta}_t)} &\lesssim \delta^2 M^4,\ \ \|\text{tr}{}^{(R_i)}{}{{\slashed\pi}}\|_{L^\infty(\Sigma^{\delta}_t)} \lesssim \delta M^4.
  \end{split}
\end{equation*}
We use the relation $L = c^{-2}\mu \Lb + 2T$ to rewrite the above estimates in null frame as follows:
\begin{equation}\label{estimates on the deformation tensor of Rotational R_i in g metric}
 \begin{split}
 \|{} ^{(R_i)}{}{\pi}_{\Lb L}\|_{L^\infty(\Sigma^{\delta}_t)}&\lesssim M^2, \ \  \|\mu^{-1}{}^{(R_i)}{}{\pi}_{LL}\|_{L^\infty(\Sigma^{\delta}_t)} \lesssim M^2,\\
\|{}^{(R_i)}{}{\pi}_{\Lb A}\|_{L^\infty(\Sigma^{\delta}_t)} &\lesssim  \delta M^2,\ \ \|{}^{(R_i)}{}{\pi}_{L A}\|_{L^\infty(\Sigma^{\delta}_t)} \lesssim \delta M^4,\\
\|{}^{(R_i)}{}\widehat{\slashed\pi}_{AB}\|_{L^\infty(\Sigma^{\delta}_t)} &\lesssim \delta^2 M^4,\ \ \|\text{tr}{}^{(R_i)}{}{\slashed\pi}\|_{L^\infty(\Sigma^{\delta}_t)} \lesssim \delta M^4.
  \end{split}
\end{equation}
The deformation tensors of $R_i$'s with respect to $\widetilde{g}$ are estimated by
\begin{equation}\label{estimates on the deformation tensor of Rotational R_i in g tilde}
 \begin{split}
  ^{(R_i)}{}\widetilde{\pi}_{\Lb \Lb}=0, \ \ \|{} ^{(R_i)}&{}{\widetilde{\pi}}_{\Lb L}\|_{L^\infty(\Sigma^{\delta}_t)}\lesssim M^2, \ \  \|\mu^{-1}{}^{(R_i)}{}{\widetilde{\pi}}_{LL}\|_{L^\infty(\Sigma^{\delta}_t)} \lesssim M^2,\\
\|{}^{(R_i)}{}{\widetilde{\pi}}_{\Lb A}\|_{L^\infty(\Sigma^{\delta}_t)} &\lesssim \delta M^2,\ \ \|{}^{(R_i)}{}\widetilde{\pi}_{L A}\|_{L^\infty(\Sigma^{\delta}_t)} \lesssim \delta M^4,\\
\|{}^{(R_i)}{}\widehat{\widetilde{\slashed\pi}}_{AB}\|_{L^\infty(\Sigma^{\delta}_t)} &\lesssim \delta^2 M^4,\ \ \|\text{tr}{}^{(R_i)}{}{\widetilde{\slashed\pi}}\|_{L^\infty(\Sigma^{\delta}_t)} \lesssim \delta M^4.
  \end{split}
\end{equation}

\subsubsection{Applications} As in \cite{Ch-Shocks} and \cite{Ch-Miao}, we are able to show that the $R_i$ derivatives are equivalent to the $\slashed{d}$ and $\nablaslash$ derivative.
For a 1-form $\xi$ on $S_{t,\ub}$, we have $\sum_{i=1}^3 \xi(R_i)^2 = r^2\Big(|\xi|^2-\big(\xi(y')\big)^2\Big)$.
This is indeed can be derived from the formula $\sum_{i=1}^{3}(R_{i})^{a}(R_{i})^{b}=r^{2}(\delta_{cd}-y^{\prime c}y^{\prime d})\Pi^{a}_{c}\Pi^{b}_{d}$, where $a,b,c,d\in\{1, 2, 3\}$. In view of \eqref{estimates on r precise}, \eqref{estimates on y^i} and the definition of $y^{\prime i}$, for sufficiently small $\delta$,
we have $\sum_{i=1}^3 \xi(R_i)^2 \sim  r^2 |\xi|^2$. Since $r$ is bounded below and above by a constant,
we obtain $\sum_{i=1}^3 \xi(R_i)^2 \sim   |\xi|^2$. Similarly, for a $k$-covariant tensor $\xi$ on $S_{t,\ub}$,
we have $\sum_{i_1,i_2,\cdots, i_k=1}^3 \xi(R_{i_1},R_{i_2},\cdots,R_{i_k})^2 \sim   |\xi|^2$.
In particular, we can take $\xi = \slashed{d}\psi$, therefore, $\sum_{i=1}^3 (R_i\psi)^2 \sim   |\slashed{d}\psi|^2$.
Henceforth, we omit the summation and write schematically as $|R_i\psi| \sim   |\slashed{d}\psi|.$

We can also compare the $R_i$-derivatives with the $\nablaslash$-derivatives for tensors. For $S_{t,\ub}$-tangential 1-form $\xi$ and vectorfield $X$,  let $\slashedLRi \xi$ be the orthogonal projection of the Lie derivative $\mathcal{L}_{R_i} \xi$ onto the surface $S_{t,\ub}$. Since $(\slashedLRi \xi)(X) = (\nablaslash_{R_i}\xi)(X) + \xi(\nablaslash_X R_i)$, we obtain
\begin{equation*}
\sum_{i=1}^3|\slashedLRi \xi|^2 =\sum_{i=1}^3|\xi(R_{i})|^2 + 2 \sum_{i=1}^3 \xi^k (\nablaslash_{R_i} \xi)_a (\nablaslash R_i)^a{}_k + \sum_{i=1}^3 \xi^k \xi^l (\nablaslash R_i)_a {}_k (\nablaslash R_i)^a{}_l.
\end{equation*}
We also have $\sum_{i=1}^3|\nablaslash_{R_i}\xi|^2 = r^2(\delta^{cd}-y'^c y'^d)(\nablaslash \xi)_{ca}(\nablaslash \xi)_d{}^a$. In view of the estimates on $y^{\prime i}$, for sufficiently small $\delta$, we obtain
\begin{equation*}
\sum_{i=1}^3|\nablaslash_{R_i}\xi|^2  \gtrsim |\nablaslash \xi|^2.
\end{equation*}
Let $\varepsilon_{ijk}$ be the volume form on $\Sigma_t$ and $v_i$ be a $S_{t,\ub}$ 1-form with
rectangular components $(v_{i})_{a}=\Pi_{a}^{b}\varepsilon_{ibk}\xi_{k}$. By virtue of the formula $
(\slashed{\nabla}R_{i})^{k}_{l}=\Pi^{m}_{k}\Pi^{n}_{l}
\varepsilon_{inm}-\lambda_{i}\theta_{kl}$, we have
\begin{equation*}
\sum_{i=1}^3 \xi^k \xi^l (\nablaslash R_i)_a {}_k (\nablaslash R_i)^a{}_l =|\xi|^2+ 2c \sum_{i=1}^3 \lambda_i \xi\cdot\chib\cdot v_i + c^2 |\chib \cdot \xi|^2 \sum_{i=1}^3 \lambda_i^2.
\end{equation*}
In view of the estimates on $\lambda_i$, for sufficiently small $\delta$, we have $\sum_{i=1}^3 \xi^k \xi^l (\nablaslash R_i)_a {}_k (\nablaslash R_i)^a{}_l  \sim |\xi|^2$.
Similarly, we have $
\big|\sum_{i=1}^3 \xi^k (\nablaslash_{R_i} \xi)_a (\nablaslash {R_i})^b{}_k \big| \lesssim
|\xi||\nablaslash \xi |$.
Finally, we conclude that
\begin{equation*}
|\xi|^2 + |\nablaslash \xi|^2 \lesssim \sum_{i=1}^3 |\slashedLRi\xi|^2 \lesssim |\xi|^2 + |\nablaslash \xi|^2 .
\end{equation*}
 Henceforth, we omit the summation and write schematically as  $|\slashedLRi\xi| \sim   |\xi|+|\nablaslash\xi|$. Similarly, for a tracefree symmetric 2-tensors $\theta_{AB}$ tangential to $S_{t,\ub}$, we have
 $|\theta|+|\nablaslash \theta| \lesssim |\slashedLRi\theta| \lesssim  |\theta|+|\nablaslash\theta|$. This will be applied to $\theta = \widehat{\chib}_{AB}$ later on.
 
\bigskip

Another application of the pointwise estimates based on the bootstrap assumption is to give an estimate on $\sqrt{\det\slashed{g}}$. On $\Sigma^{\delta}_{-2}$, $S_{-2,\ub}$ is around sphere and $\sqrt{\det\slashed{g}}$ is bounded above and below by positive absolute constants. 
On the other hand, \eqref{definition for chib prime} and \eqref{precise bound on chib'} implies that $\textrm{tr}\chib$ is bounded. This together with the formula
\begin{align}\label{gslash chib}
\Lb\left(\log\sqrt{\det\slashed{g}}\right)=\tr\chib
\end{align}
gives the fact that $\sqrt{\det\slashed{g}}$ is bounded and never vanishes along each null generator.

\subsubsection{Sobolev inequalities and elliptic estimates}
To obtain the Sobolev inequalities on $S_{t,\ub}$, we introduce
$$I(t,\ub) = \displaystyle \sup_{U \in S_{t,\ub}, \atop \partial U \text{ is } C^1} \frac{\min\big(|U|, |S_{t,\ub}-U|\big)}{|\partial U|^2}$$
the isoperimetric constant on $S_{t,\ub}$, where $|U|$, $|S_{t,\ub}-U|$ and $|\partial U|$ are the measures of the corresponding sets with respect to $\slashed{g}$ on $S_{t,\ub}$.
Therefore, in view of the fact that $R_i \sim \nablaslash$, for sufficiently small $\delta$, we have the following Sobolev inequalities:
\begin{align}\label{Sobolev}
\begin{split}
\|f\|_{W^{1,4}(S_{t,\ub})} &\lesssim |I(t,\ub)|^{\frac{1}{4}}|S_{t,\ub}|^{-\frac{3}{4}}\big(\|f\|_{L^2(S_{t,\ub})}+ \|R_i f\|_{L^2(S_{t,\ub})} + \|R_i R_j f\|_{L^2(S_{t,\ub})}\big),\\
\|f\|_{L^{\infty}(S_{t,\ub})} &\lesssim |I(t,\ub)|^{\frac{3}{4}}|S_{t,\ub}|^{-\frac{1}{2}}\big(\|f\|_{L^2(S_{t,\ub})}+ \|R_i f\|_{L^2(S_{t,\ub})} + \|R_i R_j f\|_{L^2(S_{t,\ub})}\big).
\end{split}
\end{align}
where $\|f\|_{W^{1,4}(S_{t,\ub})}$ is defined as $
\|f\|_{W^{1,4}(S_{t,\ub})}=|S_{t,\ub}|^{-1/2}\|f\|_{L^{4}(S_{t,\ub})}+\|\slashed{d}f\|_{L^{4}(S_{t,\ub})}$.
It remains to control the isoperimetric constant $I(t,\ub)$.

We use $T$ to generate a diffeomorphism of $S_{t,\ub}$ to $S_{t,0}$ which maps $U$, $S_{t,\ub}-U$ and $\partial U$ to corresponding sets $U_{\ub}$, $S_{t,0}-U_{\ub}$ and $\partial U_{\ub}$ on $S_{t,0}$. Let $U_{\ub'}, S_{t,\ub'}-U_{\ub'}, \partial U_{\ub'}$ be the inverse images of these on each $S_{t,\ub'}$ for $\ub'\in[0,\ub]$. Since $\slashed{\mathcal{L}}_T \gslash_{AB} = 2c^{-1}\mu\theta = -2c^{-2}\mu\chib_{AB}$, for $\ub' \in [0,\ub]$, we obtain
\begin{align*}
\frac{d}{d\ub'}\big(|U_{\ub'}|\big) = -\int_{U_{\ub'}}c^{-2}\mu \tr\chib d\mu_{\gslash}, \ \ \frac{d}{d\ub'}\big(|\partial U_{\ub'}|\big) = -\int_{U_{\ub'}}c^{-2}\mu \chib(\nu,\nu) ds,
\end{align*}
where $\nu$ is the unit normal of $\partial U_{\ub'}$ in $S_{t,0}$ and $ds$ the element of arc length of $\partial U_{\ub'}$.  In view of the estimates on $\chib$ and $\mu$ derived before,  for sufficiently small $\delta$, we have
\begin{align*}
\frac{d}{d \ub'}\big(|U_{\ub'}|\big) \lesssim |U_{\ub'}|, \ \ \frac{d}{d\ub'}\big(|\partial U_{\ub'}|\big) \lesssim |\partial U_{\ub'}|.
\end{align*}
Therefore, by integrating from $0$ to $\ub$, we have
\begin{align*}
|U| \sim |U_{\ub}|, \ \ |\partial U| \sim |\partial U_{\ub}|.
\end{align*}
Hence, $I(t,\ub) \sim I(t,0) \sim 1$. Finally, since $|S_{t,\ub}|\sim 1$ (This is seen by the fact $d\mu_{\slashed{g}(t,\ub)}\sim d\mu_{\slashed{g}(-r_{0},0)}$, which can be shown by calculations in \cite{Ch-Shocks}.), we conclude that
\begin{equation}\label{Sobolev on S_t ub}
\begin{split}
\|f\|_{W^{1,4}(S_{t,\ub})} + \|f\|_{L^{\infty}(S_{t,\ub})} &\lesssim  \|f\|_{L^2(S_{t,\ub})}+ \|R_i f\|_{L^2(S_{t,\ub})} + \|R_i R_j f\|_{L^2(S_{t,\ub})}.
\end{split}
\end{equation}
We remark that, similarly, we have
\begin{equation}\label{Sobolev on S_t ub L4}
\|f\|_{L^{4}(S_{t,\ub})} \lesssim  \|f\|_{L^2(S_{t,\ub})}+ \|R_i f\|_{L^2(S_{t,\ub})}.
\end{equation}

We also have the following elliptic estimates for traceless two-tensors.
\begin{lemma}
If $\delta$ is sufficiently small, for any traceless 2-covariant symmetric tensor $\theta_{AB}$ on $S_{t,\ub}$, we have
\begin{equation}\label{elliptic estimates}
\int_{S_{t,\ub}} \mu^2 \big(|\nablaslash \theta|^2 + |\theta|^2\big)d\mu_{\slashed{g}}\lesssim \int_{S_{t,\ub}} \mu^2 |\divslash \theta|^2 + |\ds\mu|^2 |\theta|^2 \slashed{d}\mu_{\slashed{g}}
\end{equation}
\end{lemma}
\begin{proof}
Let $J^A = \theta^{B}{}_{C}\nablaslash_B \theta^{AC} - \theta^{A}{}_{C}(\divslash \theta)^C$, then $|J| \lesssim |\theta|\big(|\slashed{\nabla}\theta| + |\divslash \theta|\big)$. The B$\hat{\text{o}}$chner formula says
\begin{equation*}
|\nablaslash \theta|^2+ 2K |\theta|^2=2|\divslash \theta|^2 + \divslash J,
\end{equation*}
where $K$ is the Gauss curvature. According to \eqref{Gauss equation} and the estimates on $\chib_{AB}$,
for sufficiently small $\delta$, we know that $|K|\sim 1$. Therefore, we have
\begin{equation*}
|\nablaslash \theta|^2+  |\theta|^2 \sim 2|\divslash \theta|^2 + \divslash J,
\end{equation*}
We then multiply both sides by $\mu^2$ and integrate on $S_{t,\ub}$. The Cauchy-Schwarz inequality together with the above estimates on $|J|$ as well as the divergence theorem yields the inequality.
\end{proof}

\section{The behavior of the inverse density function}
The behavior of the inverse density function $\mu$ plays an dominant r\^{o}le in this paper. The method of obtaining estimates on $\mu$ is
to relate $\mu$ to its initial value on $\Sigma_{t=-r_0}$. Since the metric $g$ depends only on $\psi_0 = \partial_t \phi$, $\mu$
is also determined by $\psi_0$. This leads naturally to the study of the wave equation $\Box_{\widetilde{g}} \psi_0 =0$. We can rewrite it in the null frame as
\begin{equation}\label{box psi_0 in null frame}
\Lb(L\psi_0) + \frac{1}{2}\tr\chibt\cdot L\psi_0 +\big(-\mu\laplacianslash \psi_0 +\frac{1}{2}\tr\widetilde{\chi}\cdot \Lb\psi_0 + 2 \zetab \cdot \slashed{d}\psi_0 + \mu \slashed{d}\log(c)\cdot \slashed{d}\psi_0\big)=0.
\end{equation}

\subsection{The asymptotic expansion for $\mu$}
We start with a lemma which relates $L\psi_0(t,\ub,\theta)$ to its initial value on $\Sigma_{-r_0}$.
\begin{lemma}
For sufficiently small $\delta$, we have
\begin{equation}\label{expansion of L psi_0}
\big||t| L\psi_0(t,\ub, \theta)-r_0 L \psi_0(-r_0,\ub, \theta)\big| \lesssim \delta^{\frac{1}{2}} M^3.
\end{equation}
\end{lemma}

\begin{proof}
We regard \eqref{box psi_0 in null frame} as a transport equation for $L\psi_0$. According to (B.1) and the estimates from previous sections,
the $L^\infty$ norm of the terms in the big parenthesis in \eqref{box psi_0 in null frame} is bounded by $\delta^{\frac{1}{2}}M^{3}$. Hence,
\begin{equation*}
|\Lb(L\psi_0)(t,\ub, \theta) + \frac{1}{2}\tr\chibt(t,\ub, \theta)\cdot L\psi_0(t,\ub,\theta) | \lesssim  \delta^\frac{1}{2}M^{3}.
\end{equation*}
By virtue of \eqref{precise bound on chibt}, this implies $|\Lb\big(L\psi_0\big)(t,\ub, \theta) -\frac{1}{\ub-t}  L\psi_0(t,\ub,\theta) | \lesssim \delta^\frac{1}{2}M^{3}$. Therefore, we obtain
\begin{equation*}
|\Lb\big((\ub-t)L\psi_0\big)(t,\ub, \theta)| \lesssim \delta^\frac{1}{2}M^{3}.
\end{equation*}
Since $|\ub| \leq \delta$, we integrate from $-r_0$ to $t$ and this yields the desired estimates.
\end{proof}

\begin{remark}
The estimates \ref{expansion of L psi_0} also hold for $R_{i}L\psi_{0}$ or $R_{i}R_{j}\psi_{0}$, e.g., see \eqref{expansion of L R_i psi_0}. To derive these estimates, we commute $R_{i}$'s with \eqref{box psi_0 in null frame} and follow the same way as in the above proof.
\end{remark}

Since $L = c^{-2}\mu \Lb + 2 T$, as a corollary, we have
\begin{corollary}
For sufficiently small $\delta$, we have
\begin{equation}\label{expansion of T psi_0}
||t| T\psi_0(t,\ub, \theta)-r_0 T \psi_0(-r_0,\ub, \theta)| \lesssim \delta^{\frac{1}{2}} M^3,
\end{equation}
\begin{equation}\label{expansion of psi_0}
||t| \psi_0(t,\ub, \theta)-r_0 \psi_0(-r_0,\ub, \theta)| \lesssim \delta^{\frac{3}{2}} M^3.
\end{equation}
\end{corollary}

\begin{proof}
For \eqref{expansion of psi_0}, we integrate \eqref{expansion of T psi_0} for $\ub'$ from $0$ to $\ub$ and use the fact that $\psi_0(t, 0, \theta)=0$.
\end{proof}

We turn to the behavior of $\Lb \mu$.
\begin{lemma}\label{lemma on the expansion of Lb mu}
For sufficiently small $\delta$, we have
\begin{equation}\label{expansion of Lb mu}
||t|^2\Lb \mu(t,\ub, \theta)-r_0^2\Lb \mu(-r_0,\ub, \theta)|\lesssim \delta M^4.
\end{equation}
\end{lemma}
\begin{proof}
According to \eqref{Structure Equation Lb mu}, we write $|t|^2\Lb \mu(t,\ub, \theta)-r_0^2\Lb \mu(-r_0,\ub, \theta)$ as
\begin{align*}
\big(|t|^2 m(t,\ub, \theta)-r_0^2 m(-r_0,\ub, \theta)\big) +\big[ |t|^2\ (\mu \cdot e)(t,\ub, \theta)-r_0^2(\mu \cdot e)(-r_0,\ub, \theta)\big].
\end{align*}
In view of \eqref{bound on e}, we bound the terms in the bracket by $\delta M^4$ up to a universal constant. Therefore,
\begin{align*}
|t|^2\Lb \mu(t,\ub, \theta)-r_0^2\Lb \mu(-r_0,\ub, \theta) &=\big(|t|^2 m(t,\ub, \theta)-r_0^2 m(-r_0,\ub, \theta)\big) +O( \delta M^4).
\end{align*}
Since
\begin{equation*}
|t|^2 m(t,\ub, \theta)-r_0^2 m(-r_0,\ub, \theta)=\tau^2 \frac{3G''(0)(\psi_0(\tau,\ub, \theta) \cdot T \psi_0(\tau,\ub, \theta))}{(1+3G''(0)\psi_0^2(\tau,\ub, \theta))^{2}}\bigg|_{\tau=-r_{0}}^{\tau=t}.
\end{equation*}
It is clear that the estimates follow immediately after \eqref{expansion of T psi_0} and \eqref{expansion of psi_0}.
\end{proof}

We are now able to prove an accurate estimate on $\mu$.
\begin{proposition}\label{proposition on expansion for mu}
For sufficiently small $\delta$, we have
\begin{equation}\label{expansion of mu}
\big|\mu(t,\ub, \theta)-1 + r_0^2(\frac{1}{t}+\frac{1}{r_0})\Lb\mu(-r_0,\ub, \theta) \big|\lesssim \delta M^4.
\end{equation}
In particular, we have $\mu \leq C_0$ where $C_0$ is a universal constant depending only on the initial data.
\end{proposition}
\begin{proof}
According to the previous lemma, we integrate $\Lb \mu$:
\begin{align*}
\mu(t,\ub,\theta)-\mu(-r_0,\ub,\theta)&=\int_{-r_0}^t \Lb \mu(\tau,\ub, \theta)d\tau = \int_{-r_0}^t \frac{\tau^2\Lb \mu(\tau,\ub, \theta)}{\tau^2}d\tau\\
&\stackrel{\eqref{expansion of Lb mu}}{=}\int_{-r_0}^t \frac{r_0^2\Lb \mu(-r_0,\ub, \theta)}{\tau^2} + \frac{O (\delta M^4)}{\tau^{2}}d\tau\\
&=-(\frac{1}{t}+\frac{1}{r_0})r_0^2\Lb\mu(-r_0,\ub, \theta) +O(\delta M^4).
\end{align*}
Therefore, we can use \eqref{initial estimates on mu} to conclude.
\end{proof}

We are ready to derive two key properties of the inverse density function $\mu$. The first asserts that the shock wave region is
trapping for $\mu$, namely, once $p \in W_{shock}$, then all the points after $p$ along the incoming null geodesic stay in $W_{shock}$.
\begin{proposition}\label{Proposition C3}
For sufficiently small $\delta$ and for all $(t,\ub, \theta) \in W_{shock}$, we have
\begin{equation}\label{C3}
\Lb \mu (t,\ub,\theta) \leq -\frac{1}{4|t|^{2}} \lesssim -1.
\end{equation}
\end{proposition}
\begin{proof}
For $(t,\ub, \theta) \in W_{shock}$, we have $\mu(t,\ub, \theta) < \frac{1}{10}$. In view of \eqref{expansion of mu},
we claim that $r_0^2\Lb\mu (-r_0,\ub, \theta)<0$. Otherwise, since $\frac{1}{t}+\frac{1}{r_0}<0$, we would have $\mu(t,\ub, \theta) \geq 1 + O(\delta M^2)>\frac{1}{10}$,
provided $\delta$ is sufficiently small. This contradicts the fact that $\mu(t,\ub, \theta) < \frac{1}{10}$.

We can also use this argument to show that $(\frac{1}{t}+\frac{1}{r_0})r_0^2\Lb\mu(-r_0,\ub, \theta) \geq \frac{1}{2}$. Otherwise, for sufficiently small $\delta$, we would have $\mu(t,\ub, \theta) \geq \frac{1}{2} + O(\delta M^2)>\frac{1}{10}$.

Therefore, we obtain $r_0^2\Lb\mu(-r_0,\ub, \theta) \leq \frac{1}{2}\frac{r_0 t}{r_0+t}$. In view of \eqref{expansion of Lb mu}, we have
\begin{equation*}
t^2 \Lb\mu(t,\ub,\theta)\leq  \frac{1}{2}\frac{r_0 t}{r_0+t}+O(\delta M^2).
\end{equation*}
Here we write $M^{2}$ instead of $M^{4}$ because we already know that $\mu\leq C_{0}$, where $C_{0}$ depends only on initial data.
By taking a sufficiently small $\delta$ and noticing that $\frac{r_{0}t}{r_{0}+t}$ is bounded from above by a negative number, this yields the desired estimates.
\end{proof}

\begin{remark}\label{2nd derivative blow up}
In the case when shock forms, i.e. $\mu\rightarrow0$, by the previous proposition and \eqref{Structure Equation Lb mu}, $m=-\frac{1}{2}\frac{dc^{2}}{d\rho}T\rho\lesssim-1$. In other words, $$T\rho\geq c_{0}>0$$ for some absolute constant $c_{0}$. On the other hand, $\widehat{T}\rho=c\mu^{-1}T\rho$ and $\|\widehat{T}\|=1$, therefore as $\mu\rightarrow0$, $\nabla\rho$ blows up, so does $\nabla\partial_{t}\phi$.
\end{remark}

\begin{remark}
We compare the estimates \eqref{bound on mu} and \eqref{expansion of Lb mu}. \eqref{bound on mu} is rough:
$|\Lb \mu| \lesssim M^2$; \eqref{expansion of Lb mu} is precise: $|\Lb \mu| \leq C_0 + \delta M^2$,
where $C_0$ depends only on the initial data. The improvement comes from integrating the wave equation $\Box_{\widetilde{g}}\psi_0 =0$
or equivalently \eqref{box psi_0 in null frame}.
\end{remark}

\subsection{The asymptotic expansion for derivatives of $\mu$}
We start with an estimate on derivatives of $\tr\chibt$.
\begin{lemma}
For sufficiently small $\delta \leq \varepsilon$, we have
\begin{equation}\label{bound on L trchibt}
 \|L \tr_{\widetilde{g}}\widetilde{\chib} \|_{L^{\infty}} \lesssim  M^2,
\end{equation}
\begin{equation}\label{bound on d trchibt}
 \|\ds \tr \widetilde{\chib} \|_{L^{\infty}} \lesssim \delta  M^2.
\end{equation}
\end{lemma}

\begin{proof}
We derive a transport equation for $R_i \chib'_{AB}$ by commuting $R_i$ with \eqref{transport equation for chib prime}:
\begin{equation*}
\Lb (R_i \chib'_{AB}) = [\Lb ,R_i] \chib'_{AB} + e R_i \chib'_{AB} + 2 \chib'_{A}{}^{C}R_i \chib'_{BC}+(R_i e) \cdot \chib'_{AB} - \frac{e}{\ub-t}R_i \slashed{g}_{AB} - R_i\big(\frac{e}{\ub-t}\big) \slashed{g}_{AB} -R_i\alphab'_{AB}.
\end{equation*}
Since $[\Lb,R_i]^A = {}^{(R_i)}\pi_{\Lb}{}^A$, the commutator term $[\Lb ,R_i] \chib'_{AB}$ can be bounded by the estimates
on the deformation tensors. We then multiply both sides by $R_{i} \chib'_{AB}$ and repeat the procedure that we used to
 derive \eqref{precise bound on chib'}. Since it is routine, we omit the details and only give the final result
\begin{equation}\label{estimate for L chib}
 \|R_i(\chib_{AB}+\frac{\slashed{g}_{AB}}{\ub-t})\|_{L^{\infty}} \lesssim \delta M^2.
\end{equation}
In particular, this yields $\|R_i \tr\chib \|_{L^{\infty}} \lesssim \delta  M^2$ which is equivalent to \eqref{bound on d trchibt}.

We derive a transport equation for $L\chib'_{AB}$ by commuting $L$ with \eqref{transport equation for chib prime}:
\begin{align*}
\Lb (L \chib'_{AB}) &= e L\chib'_{AB} + 2 \chib'_{A}{}^{C}L \chib'_{BC}+ L e \chib'_{AB} - \frac{e}{\ub-t}L \slashed{g}_{AB}\\
 &\ \ + L\big(\frac{e}{\ub-t}\big) \slashed{g}_{AB} -L \alphab'_{AB}+\Lb(c^{-2}\mu)\Lb\chib'_{AB}+\slashed{g}^{CD}(\zetab_{C}+\etab_{D})X_{C}(\chib'_{AB}).
\end{align*}
We then use Gronwall to derive
\begin{equation}\label{expansion for L chib}
 \|L(\chib_{AB}+\frac{\slashed{g}_{AB}}{\ub-t})\|_{L^{\infty}} \lesssim M^2.
\end{equation}
In particular, this yields $\|L \tr\chib \|_{L^{\infty}} \lesssim  M^2$. This is equivalent to \eqref{bound on L trchibt}.
\end{proof}

We now derive estimates for $R_i\psi_{0}$.
\begin{lemma}
For sufficiently small $\delta$, we have
\begin{equation}\label{expansion of L R_i psi_0}
\big||t| L R_i \psi_0(t,\ub, \theta)-r_0 LR_i \psi_0(-r_0,\ub, \theta)\big| \lesssim \delta^{\frac{1}{2}} M^3.
\end{equation}
\end{lemma}
\begin{proof}
We commute $R_i$ with \eqref{box psi_0 in null frame} and we obtain that $\Lb(R_i L  \psi_0) + \frac{1}{2}\tr\chibt\cdot R_i L \psi_0 = N$ with
\begin{align*}
N=  R_i\big(\mu\laplacianslash \psi_0 -\frac{1}{2}\tr\widetilde{\chi}\cdot \Lb \psi_0 - 2 \zetab \cdot \slashed{d} \psi_0 - \mu \slashed{d}\log(c)\cdot \slashed{d}\psi_0\big)  - \frac{1}{2}R_i \tr\chibt \cdot  L \psi_0+ [\Lb,R_i] L  \psi_0.
\end{align*}
 According to  (B.1) and the previous lemma, $N$ is bounded by $\delta^{\frac{1}{2}} M^3$. Hence, $\big|\Lb(R_i L  \psi_0) + \frac{1}{2}\tr\chibt\cdot R_i L  \psi_0\big| \lesssim  \delta^{\frac{1}{2}}M^3$.
 We then integrate to derive
\begin{equation*}
\big||t| R_i L \psi_0(t,\ub, \theta)-r_0 R_i L \psi_0(-r_0,\ub, \theta)\big| \lesssim \delta^{\frac{1}{2}} M^3.
\end{equation*}
The commutator $[R_i, \Lb] \psi_0$ is bounded by $\delta M^3$ thanks to the estimates on deformation tensors. This completes the proof.
\end{proof}

We can obtain a better estimate for $\ds\mu$. The idea is to bound for $R_i \mu$ and use the fact that
 $|\ds \mu| \sim |R_i \mu|$.
\begin{lemma}
For sufficiently small $\delta$, we have
\begin{equation}\label{precise bound on d mu}
\|\ds\mu\|_{L^\infty(\Sigma_t)} \lesssim 1 + \delta M^4.
\end{equation}
\end{lemma}
\begin{proof}
We commute $R_i$ with $\Lb \mu= m+ e\mu$ to derive
\begin{equation*}
\Lb(R_i \mu) = R_i m + \big(e R_i \mu+\mu R_i e +[\Lb,R_i]\mu\big).
\end{equation*}
According to (B.1) and the estimates on $L R_i \psi_{0}$ (needed to bound $R_i m$) from the previous lemma,
it is straightforward to bound the terms in the parenthesis by $\delta M^2$. Similar to \eqref{expansion of Lb mu}, we obtain
\begin{equation}\label{expansion of Lb R_i mu}
||t|^2\Lb\big(R_i\mu \big)(t,\ub,\theta)-|r_0|^2\Lb\big(R_i\mu \big)(-r_0,\ub,\theta)| \lesssim \delta M^4.
\end{equation}
Since $\|[\Lb,R_{i}]\|_{L^{\infty}}\lesssim\delta M^{2}$, we bound $R_i \mu$ as
\begin{align*}
\quad R_i\mu(t,\ub,\theta)-R_i\mu(-r_0,\ub,\theta)&\stackrel{\eqref{expansion of Lb R_i mu}}{=}\int_{-r_0}^t \frac{r_0^2\Lb\big( R_i\mu(\-r_0,\ub, \theta)\big)}{\tau^2} + O ( \delta M^4)d\tau\\
&=-(\frac{1}{t}+\frac{1}{r_0})r_0^2\Lb\big(R_i\mu(-r_0,\ub, \theta)\big) +O(\delta M^4).
\end{align*}
This inequality yields \eqref{precise bound on d mu} for sufficiently small $\delta$.
\end{proof}

We can also obtain a better estimate for $L\mu$.
\begin{lemma}
For sufficiently small $\delta$, we have
\begin{equation}\label{precise bound on L mu}
\|L\mu\|_{L^\infty(\Sigma_t)} \lesssim \delta^{-1} + M^4.
\end{equation}
\end{lemma}
\begin{proof}
By commuting $L$ with $\Lb \mu= m+ e\mu$, we obtain
\begin{equation*}
\Lb(L\mu) =Lm + \big[-2(\zetab^A+\etab^A)X_A(\mu)+\Lb(c^{-2}\mu)\Lb\mu + eL\mu + \mu L e\big].
\end{equation*}
According to (B.1), we can bound the terms in the bracket by $M^4$. Hence,
\begin{align*}
&\quad |t|^2\Lb\big(L\mu \big)(t,\ub,\theta)-|r_0|^2\Lb\big(L\mu \big)(-r_0,\ub,\theta) =|t|^2 L m(t,\ub,\theta)-|r_0|^2L m(-r_0,\ub,\theta) + O(M^4).
\end{align*}
By the explicit formula of $m$, we can proceed exactly as in Lemma \ref{lemma on the expansion of Lb mu} and we obtain
\begin{equation}\label{expansion of Lb L mu}
||t|^2\Lb\big(L\mu \big)(t,\ub,\theta)-|r_0|^2\Lb\big(L\mu \big)(-r_0,\ub,\theta)| \lesssim M^4.
\end{equation}
We then integrate along $\Lb$ and we have
\begin{align*}
L\mu(t,\ub,\theta)-L\mu(-r_0,\ub,\theta)&= \int_{-r_0}^t \frac{\tau^2\Lb\big( L\mu(\tau,\ub, \theta)\big)}{\tau^2}d\tau\stackrel{\eqref{expansion of Lb L mu}}{=}\int_{-r_0}^t \frac{r_0^2\Lb\big( L\mu(\-r_0,\ub, \theta)\big)}{\tau^2} + \frac{O ( M^4)}{\tau^{2}}d\tau\\
&=-(\frac{1}{t}+\frac{1}{r_0})r_0^2\Lb\big(L\mu(-r_0,\ub, \theta)\big) +O(M^4).
\end{align*}
For sufficiently small $\delta$, this implies \eqref{precise bound on L mu}.
\end{proof}

We now relate $L^2 \psi_0(t,\ub,\theta)$ to its initial value.
\begin{lemma}
For sufficiently small $\delta$, we have
\begin{equation}\label{expansion of L L psi_0}
\big||t| L^2\psi_0(t,\ub, \theta)-r_0 L^2 \psi_0(-r_0,\ub, \theta)\big| \lesssim \delta^{-\frac{1}{2}} M^3.
\end{equation}
\end{lemma}

\begin{proof}
We commute $L$ with \eqref{box psi_0 in null frame} and we obtain the following transport equation for $L^2 \psi_0$:
\begin{equation}\label{box L psi_0 in null frame}
\begin{split}
 \Lb(L^2\psi_0) + \frac{1}{2}\tr\chibt\cdot L^2\psi_0  &=-\frac{1}{2}L\big(\tr\chibt\big) \cdot L\psi_0 + [\Lb,L]L\psi_0 \\
&+ L\Big(\mu\laplacianslash \psi_0 -\frac{1}{2}\tr\widetilde{\chi}\cdot \Lb\psi_0 - 2 \zetab \cdot \slashed{d}\psi_0 -\mu \slashed{d}\log(c)\cdot \slashed{d}\psi_0 \Big).
\end{split}
\end{equation}
The righthand side of the above equation can be expanded as
\begin{equation*}
\begin{split}
&L\tr\chibt L\psi_0 + \Lb(c^{-2}\mu)\Lb L\psi_0 + (\etab+\zetab)\slashed{d}L\psi_0 +\Big( L\mu \laplacianslash \psi_0 +\mu L \laplacianslash \psi_0 +L\tr\chibt \Lb \psi_0 + \tr\chibt L\Lb \psi_0 \\
& + L\zetab \slashed{d}\psi_0 + \zetab L\slashed{d}\psi_0 + L\mu \cdot \slashed{d}\log(c)\cdot \slashed{d}\psi_0 + \mu L\big(\slashed{d}\log(c)\big)\cdot \slashed{d}\psi_0 + \mu \slashed{d}\log(c)\cdot L\slashed{d}\psi_0\Big).
\end{split}
\end{equation*}
Since the exact numeric constants and signs for the coefficients are irrelevant for estimates, we replace all of them by $1$ in the above expressions.

Since $\zetab_A = -c^{-1}\mu X_A(c)$, by applying $L$ and using \eqref{precise bound on L mu}, we obtain $|L\zetab| \lesssim M^2$. Therefore, according \eqref{bound on L trchibt}, (B.1) and the estimates derived previously in this section, we can bound all the terms on the right hand side and we obtain
\begin{equation*}
|\Lb(L^2\psi_0) + \frac{1}{2}\tr\chibt\cdot L^2\psi_0| \lesssim \delta^{-\frac{1}{2}}M^3.
\end{equation*}
We then integrate from $-r_0$ to $t$ to obtain \eqref{expansion of L L psi_0}.
\end{proof}

We can commute $L$ twice with \eqref{transport equation for chib prime} or with \eqref{box psi_0 in null frame} to obtain following estimates.
\begin{lemma}
For sufficiently small $\delta$, we have
\begin{equation}\label{bound on L L trchibt}
 \|L^2 \tr_{\widetilde{g}}\widetilde{\chib} \| \lesssim  M^2\delta^{-1},
\end{equation}
\begin{equation}\label{expansion of L L L psi_0}
||t| L^3 \psi_0(t,\ub, \theta)-r_0 L^3 \psi_0(-r_0,\ub, \theta)| \lesssim \delta^{-\frac{3}{2}} M^3.
\end{equation}
\end{lemma}
We omit the proof since it is routine. Similarly, we commute $L$ twice with $\Lb \mu= m+ e\mu$, we can use \eqref{bound on L L trchibt} and \eqref{expansion of L L L psi_0} to obtain
\begin{lemma}
There exists $\varepsilon = \varepsilon(M)$ so that for all $\delta \leq \varepsilon$, we have
\begin{equation}\label{precise bound on L L mu}
\|L^2\mu\|_{L^\infty(\Sigma_t)} \lesssim \delta^{-2} + \delta^{-1}M^2,
\end{equation}
\begin{equation}\label{precise bound on T T mu}
\|T^2\mu\|_{L^\infty(\Sigma_t)} \lesssim \delta^{-2} + \delta^{-1}M^2.
\end{equation}
\end{lemma}

We turn to another key property of $\mu$ which reflects the behavior of $\mu^{-1}T\mu$. 
\eqref{expansion of Lb mu} suggests that if shock forms before $t=-1$, $\Lb \mu$ to behave as a constant near $s^*$. Hence, $\mu$ is proportional to $|t-s^*|$. In view of \eqref{bound on T mu and slashed d mu}, we expect $\mu^{-1} T\mu$ behaves as ${|t-s^*|}^{-1}\delta^{-1}$ near shocks (It is \emph{not} integrable in $t$). The next proposition suggest a much better bound for $\mu^{-1}(T\mu)_{+}$ by affording one more derivative in $T$ and we can improve ${|t-s^*|}^{-1}$ to ${|t-s^*|^{-\frac{1}{2}}}$.

\begin{proposition}\label{Proposition C2}
For $p=(t,\ub, \theta) \in W_\delta$, 
let $\big(\mu^{-1}T \mu\big)_+$ be the nonnegative part of $\mu^{-1}T \mu$. For sufficiently small $\delta $ and for all $p \in W_{shock}$, we have
\begin{equation}\label{C2}
\big(\mu^{-1}T \mu\big)_+ (t,\ub,\theta) \lesssim \frac{1}{|t-s^*|^{\frac{1}{2}}}\delta^{-1}.
\end{equation}
\end{proposition}
\begin{proof}
We use a maximal principle type argument. Let $\gamma: [0, \delta] \rightarrow W^*_\delta$ be the integral curve of $T$ through the point $p$. 
We may choose the $\theta$ coordinates to be constant along $\gamma$ on the given $\Sigma_{t}$ and study the function $f(\ub) = T(\log(\mu))(\ub)$. We assume $f(\ub)$ attains its maximum at a point $\ub^* \in [0, \delta]$. We may also assume that this maximum is positive (Otherwise, \eqref{C2} is automatically true). Since $\ub^*$ is a maximum point, we have $\dfrac{d}{d \ub}(f)(\ub^*) \geq 0$. In other words, $T(T(\log(\mu)))(\ub^*) \geq 0$. Therefore, at the point $(t,\ub^*, \theta)$, we have $\mu^{-1}T^2 \mu - \frac{1}{\mu^2}(T\mu)^2\geq 0$. Hence,
\begin{equation}\label{eq 1}
\|\mu^{-1}(T\mu)_+ \|_{L^\infty([0,\delta])}\leq \sqrt{\frac{\|T^2\mu\|_{L^\infty([0,\delta])}}{\displaystyle \inf_{\ub \in[0,\delta]}\mu(\ub)}}.
\end{equation}
It suffices to bound the denominator and the numerator on the righthand side. For $T^2\mu$, according to \eqref{precise bound on T T mu}, we have
\begin{equation}\label{eq 2}
\|T^2\mu\|_{L^\infty([0,\delta])} \lesssim \frac{1}{|t|}\delta^{-2}\lesssim \delta^{-2}.
\end{equation}
For $\inf \mu$, we can assume $(t,\ub,\theta)\in W_{shock}$. Otherwise we have the lower bound $\mu\geq\frac{1}{10}$. According to \eqref{C3}, the condition $(t,\ub,\theta)\in W_{shock}$ implies $|t|^{2}(\Lb\mu)(t,\ub,\theta)\leq -\frac{1}{4}$. Therefore, by \eqref{expansion of Lb mu}, we have
\begin{align*}
r^{2}_{0}\Lb \mu(-r_0,\ub, \theta)=t^{2}\Lb \mu(t,\ub, \theta)+O(\delta M^{2})\leq -\frac{1}{4}+O(\delta M^{2})\leq -\frac{1}{6}
\end{align*}
provided that $\delta$ is sufficiently small. We now integrate from $t$ to $s^*$ to derive
\begin{align*}
\mu(t,\ub, \theta) &= \mu (s^*, \ub,\theta)- \int_{t}^{s^*}\Lb \mu(\tau,\ub, \theta)\geq - \int_{t}^{s^*}\Lb \mu(\tau,\ub, \theta)=- \int_{t}^{s^*}\left(\frac{r_0^2 \Lb \mu(-r_0,\ub, \theta)}{|\tau|^2} + O( \delta M^2)\right)\\
&\geq -r_0^2 \Lb \mu(-r_0,\ub, \theta)\frac{1}{|s^*||t|}|t-s^*| +O(\delta M^2)|t-s^{*}|.
\end{align*}
Therefore, for sufficiently small $\delta$, we obtain $|t-s^*|\lesssim \mu$. Together with \eqref{eq 1} and \eqref{eq 2}, this completes the proof.
\end{proof}

\section{The mechanism for shock formations}
In this section, we assume Theorem \ref{theorem main estimates} stated in Section \ref{section Bootstrap Assumptions} and we use knowledge on $\mu$ from last section to analyze a mechanism of shock formations. In particular, we have to use condition \eqref{shock condition} in the \textbf{Main Theorem}. This is the only place in the paper where we use \eqref{shock condition}.

The transport equation $\Lb \mu = m + \mu e$ is responsible for the shock formations. We first give precise bounds on each term on the righthand side. Since $m=-\frac{1}{2}\frac{d}{d\rho} (c^{2})T\rho$, we have $m =  \frac{3G''(0)}{\big(1+3G''(0)\psi^2_0\big)^2}\psi_0 \cdot T \psi_0$. In view of \eqref{expansion of T psi_0}, up to an error of size $\delta^{\frac{1}{2}}$, we can replace $T\psi_0(t,\ub, \theta)$ by $\frac{r_0}{|t|} T \psi_0(-r_0,\ub, \theta)$; in view of  \eqref{expansion of psi_0}, up to an error of size $\delta^{\frac{3}{2}}$, we can replace $\psi_0(t,\ub, \theta)$ by $\frac{r_0}{|t|} \psi_0(-r_0,\ub, \theta)$. Therefore, we obtain
\begin{align*}
m &=  3G''(0) \frac{r_0 ^2}{|t|^2} \psi_0(-r_0,\ub, \theta) T \psi_0(-r_0,\ub, \theta) + O(\delta)= 3 \frac{r_0 ^2}{|t|^2} \Big(G''(0)\phi_1(\frac{r-r_0}{\delta}, \theta) \partial_{r} \phi_1(\frac{r-r_0}{\delta}, \theta)\Big) + O(\delta).
\end{align*}
Since $e$ is of size $\delta$ and $|\mu|\lesssim 1$, we regard $\mu e$ as an error term. We then obtain
\begin{align*}
\Lb \mu (t,\ub, \theta)= 3 \frac{r_0 ^2}{|t|^2} \Big(G''(0)\phi_1(\frac{r-r_0}{\delta}, \theta) \partial_{r} \phi_1(\frac{r-r_0}{\delta}, \theta)\Big) + O(\delta).
\end{align*}
We then integrate this equation and we obtain
\begin{align*}
\mu (t,\ub, \theta) - \mu (-r_0,\ub, \theta) &= 3 \int_{-r_0}^t \frac{r_0 ^2}{|\tau|^2} d\tau \cdot \Big(G''(0)\phi_1(\frac{r-r_0}{\delta}, \theta) \partial_{r} \phi_1(\frac{r-r_0}{\delta}, \theta)\Big) + O(\delta)\\
&= 3 r_0^2\big(\frac{1}{|t|}-\frac{1}{r_0}\big) \Big(G''(0)\phi_1(\frac{r-r_0}{\delta}, \theta) \partial_{r} \phi_1(\frac{r-r_0}{\delta}, \theta)\Big) + O(\delta)
\end{align*}
Since $|\mu (-r_0,\ub, \theta)-1| \lesssim \delta$, according to condition \eqref{shock condition},
we have (recall that $r_0=2$)
\begin{align*}
\mu (t,\ub, \theta)  &\leq 1 - 3 \cdot 2^2 \big(\frac{1}{|t|}-\frac{1}{2}\big)\frac{1	}{6} + O(\delta)=1-2\left(\frac{1}{|t|}-\frac{1}{2}\right)+O(\delta)
\end{align*}
Therefore, for sufficiently small $\delta$, $t$ can not be greater than $-1$, otherwise $\mu$ would be negative. In other words, shock forms before $t=-1$.
\begin{corollary}
If we introduce the vectorfield $\widehat{T}:=c\mu^{-1}T$, then when shock forms, the second derivative of $\phi$, $\widehat{T}^{i}\partial_{i}\partial_{t}\phi$, blows up.
\end{corollary}
\begin{proof}
When shock forms, $\mu\longrightarrow0$, which means $\mu<\frac{1}{10}$. Then by \eqref{C3}, $\Lb\mu$ is negative and bounded from above. In other words, there is an absolute positive constant $C$ such that $|\Lb\mu|\geq C$. While from the propagation equation $\Lb\mu=m+\mu e$ and the pointwise estimates $|e|\lesssim\delta, |\mu|\lesssim 1$, the estimate $|m|\geq C$ follows if $\delta$ is sufficiently small. By the definition $m=-\frac{1}{2}\frac{d(c^{2})}{d\rho}T\rho$ and $\widehat{T}=c\mu^{-1}T$, the derivative of $\rho$, $\widehat{T}\rho=c\mu^{-1}T\rho$, blows up when $\mu\longrightarrow 0$. Since $|\psi_{0}|\lesssim \delta^{1/2}$, 
$\widehat{T}\psi_{0}=\widehat{T}^{i}\partial_{i}\partial_{t}\phi$ blows up when $\mu\longrightarrow0$ from the definition $\rho=\psi_{0}^{2}$. 
\end{proof}
\begin{corollary}
When shock forms, the only nonzero curvature component in optical coordinates, $\underline{\alpha}_{AB}$, blows up.
\end{corollary}
\begin{proof}
In the expression:
\begin{align*}
\alphab_{AB}= \frac{1}{2}\frac{d(c^{2})}{d\rho}\slashed{\nabla}^{2}_{X_A,X_B}\rho -\frac{1}{2}\mu^{-1}\frac{d(c^{2})}{d\rho}T(\rho) \chib_{AB}+\frac{1}{2}\big(\frac{d^{2}(c^{2})}{d\rho^{2}}
-\frac{1}{2}c^{-2}\big|\frac{d(c^{2})}{d\rho}\big|^2\big)X_{A}(\rho) X_{B}(\rho),
\end{align*}
the terms $\frac{1}{2}\frac{d(c^{2})}{d\rho}\slashed{\nabla}^{2}_{X_A,X_B}\rho$, $\frac{1}{2}\big(\frac{d^{2}(c^{2})}{d\rho^{2}}-\frac{1}{2}c^{-2}\big|\frac{d(c^{2})}{d\rho}\big|^2\big)X_{A}(\rho) X_{B}(\rho)$, being derivatives of $\rho$ in optical coordinates, are bounded pointwisely. While by \eqref{C3} and \eqref{precise bound on chib'}, the pointwise norm of $\frac{d(c^{2})}{d\rho}T(\rho)\chib_{AB}$ is bounded from below, therefore the term $\mu^{-1}\frac{d(c^{2})}{d\rho}T(\rho)\chib_{AB}$ blows up when $\mu\longrightarrow0$.
\end{proof}
The focus of the rest of the paper is to prove Theorem \ref{theorem main estimates}.

\section{Energy estimates for linear wave equations}
We study energy estimates for the linear wave equation
\begin{equation}\label{model linear wave equation}
 \Box_{\widetilde{g}} \psi = \rho
\end{equation}
where $\rho$ is a smooth function defined on $W_\delta$. The energy momentum tensors for a solution $\psi$ of \eqref{model linear wave equation} are the same for both $g$ and $\gt$, i.e.,
\begin{equation*}
 \widetilde{T}_{\mu\nu} = T_{\mu\nu} = \partial_\mu \psi \partial_\nu \psi - \frac{1}{2}g_{\mu\nu}\partial^\alpha \psi \partial_\alpha \psi.
\end{equation*}
In the null frame $(X_1, X_2, L, \Lb)$, $\widetilde{T}_{\mu\nu}$ is decomposed as
\begin{equation}\label{energy momentum tensor in null frame}
 \begin{split}
  T_{LL} &= (L\psi)^2,\ T_{\Lb\Lb} = (\Lb\psi)^2, \ T_{\Lb L} = \mu |\slashed{d}\psi|^2, \ T_{LA} = L\psi \cdot X_A(\psi),\\
T_{\Lb A} &= \Lb \psi \cdot X_A(\psi),\ T_{AB} = X_A(\psi)X_B(\psi)-\frac{1}{2}\slashed{g}_{AB}(-\mu^{-1}L\psi \Lb \psi + |\slashed{d} \psi|^2).
 \end{split}
\end{equation}
We use two \emph{multiplier vectorfields} $K_0 = L$ and $K_1 = \Lb$. The associated energy currents are defined by
\begin{equation}\label{definition for energy currents of K_0 and K_1}
 \widetilde{P}_0{}^\mu = -\widetilde{T}^{\mu}{}_{\nu}K_0{}^\nu, \ \ \widetilde{P}_1{}^\mu = -\widetilde{T}^{\mu}{}_{\nu}K_1{}^\nu.
\end{equation}
The corresponding deformation tensors are denoted by $\widetilde{\pi}_0 = \,{}^{(K_0)}\widetilde{\pi}$ and $\widetilde{\pi}_1 = \,{}^{(K_1)}\widetilde{\pi}$ respectively. Since the divergence of $\widetilde{T}_{\mu\nu}$ is $\nablat^\mu \widetilde{T}_{\mu\nu} = \rho\cdot \partial_\nu \psi$, we have
\begin{equation}\label{divergence of currents}
 \begin{split}
  \nablat_\mu \widetilde{P}_0{}^\mu&=\widetilde{Q}_0 := -\rho\cdot K_0 \psi-\frac{1}{2}\widetilde{T}^{\mu\nu}\widetilde{\pi}_{0,\mu\nu},\\
  \nablat_\mu \widetilde{P}_1{}^\mu&=\widetilde{Q}_1 :=-\rho\cdot K_1 \psi-\frac{1}{2}\widetilde{T}^{\mu\nu}\widetilde{\pi}_{1,\mu\nu}.
 \end{split}
\end{equation}
Let  $W_t^{\ub}$ be the spacetime region enclosed by $\Sigma_{-r_0}^{\ub}$, $\Cb_{0}^t$, $\Sigma_{t}^{\ub}$ and $\Cb_{\ub}^t$. We integrate \eqref{divergence of currents} on $W_t^{\ub}$ under the condition that $\psi$ and its derivatives vanish on $\Cb_{0}^t$ (This is always the case in later applications), we obtain
\begin{equation}\label{energy identities}
 \begin{split}
  E^{0}(t,\ub)-E^{0}(-r_0,\ub)+F^{0}(t,\ub) &= \int_{W_t^{\ub}} c^{-2} \widetilde{Q_0},\\
 E^{1}(t,\ub)-E^{1}(-r_0,\ub)+F^{1}(t,\ub) &= \int_{W_t^{\ub}} c^{-2} \widetilde{Q_1}.
 \end{split}
\end{equation}
where the associated energy $E^{i}(t,\ub)$ and flux $F^i(t,\ub)$ are defined (naturally from the Stokes formula) as
\begin{equation}\label{definition for energy}
 \begin{split}
  E^0(t,\ub)&=\int_{\Sigma_{t}^{\ub}}\frac{1}{2c}\Big((L\psi)^2 + c^{-2}\mu^2 |\slashed{d} \psi|^2\Big),\ \  F^0(t,\ub)=\int_{\Cb_{\ub}{}^{t}}c^{-1}\mu |\slashed{d} \psi|^2,\\
E^1(t,\ub)&=\int_{\Sigma_{t}^{\ub}}\frac{1}{2c}\Big( c^{-2}\mu(\Lb \psi)^2 + \mu |\slashed{d}\psi|^2\Big),\ \ F^1(t,\ub)=\int_{\Cb_{\ub}{}^{t}} c^{-1}(\Lb \psi)^2.
 \end{split}
\end{equation}
We emphasize that the integral on $W_t^{\ub}$ defined as below on the spacetime slab contains a factor $\mu$:
\begin{equation*}
\int_{W_t^{\ub}}f  = \int_{-r_0}^{t} \int_{0}^{\ub} \left(\int_{S_{\tau,\ub'}} \mu \cdot f(\tau,\ub',\theta)\dmug\right) d\ub' d\tau.
\end{equation*}

We remark that, by \eqref{definition for energy and flux} and \eqref{bound on c}, we have
\begin{equation}\label{comparisons between energy and reduced energy}
\begin{split}
 E^0(t,\ub)&\sim {E}(\psi)(t,\ub),\ \  F^0(t,\ub)\sim {F}(\psi)(t,\ub), \\
 E^1(t,\ub) &\sim {\Eb}(\psi)(t,\ub),\ \ F^1(t,\ub)\sim{\Fb}(\psi)(t,\ub).
\end{split}
\end{equation}
For the sake of simplicity, we use $ {E}(t,\ub)$, ${F}(t,\ub)$, ${\Eb}(t,\ub)$ and ${\Fb}(t,\ub)$ as shorthand notations for $ {E}(\psi)(t,\ub)$, ${F}(\psi)(t,\ub)$, ${\Eb}(\psi)(t,\ub)$ and ${\Fb}(\psi)(t,\ub)$ in the rest of the section.

We need to compute the so called \emph{error integrals} or \emph{error terms}, i.e., $\displaystyle\int_{W_t^{\ub}} \widetilde{Q_0}$ and $\displaystyle \int_{W_t^{\ub}} \widetilde{Q_1}$ in \eqref{energy identities}. This requires an explicit formula for $\widetilde{\pi}_{0,\mu\nu}$ or $\widetilde{\pi}_{1,\mu\nu}$. The deformation tensor $\widetilde{\pi}_{0,\mu\nu}$ is given by
\begin{equation}\label{deformation tensor for K_0}
 \begin{split}
  \widetilde{\pi}_{0,LL}&=0,\ \ \widetilde{\pi}_{0,\Lb \Lb}=0, \ \ \widetilde{\pi}_{0,L\Lb}=-2c^{-1}\mu\Big(\mu^{-1}L \mu  + L\log(c^{-1}) + \Lb(c^{-2}\mu)\Big),\\
\widetilde{\pi}_{0,LA}&=-{2c^{-1}\mu} X_A(c^{-2}\mu),\ \ \widetilde{\pi}_{0,\Lb A} =-{2}{c}^{-1}(\zetab_A + \etab_A),\ \ \widetilde{\pi}_{0,AB}= 2\widetilde{\chi}_{AB}.
 \end{split}
\end{equation}
The deformation tensor $\widetilde{\pi}_{1,\mu\nu}$ is given by
\begin{equation}\label{deformation tensor modified for K_1}
 \begin{split}
 \widetilde{\pi}_{1,\Lb \Lb}&=0,\ \ 	\widetilde{\pi}_{1,LL}={4c^{-1}\mu}\Lb(c^{-2}\mu), \ \  \widetilde{\pi}_{1,L\Lb}=-{2c^{-1}\mu}\Big(\mu^{-1}\Lb \mu+ \Lb\log(c^{-1})\Big),\\
\widetilde{\pi}_{1,\Lb A} &=0, \ \  \widetilde{\pi}_{1,LA}={2}{c}^{-1}\big(\zetab_A+\etab_A\big),\ \ \widetilde{\pi}_{1,AB}=2 \chibt_{AB}.
 \end{split}
\end{equation}
We also need an explicit formula for the energy momentum tensor $T^{\mu\nu}$:
\begin{equation}\label{energy momentum tensor in null frame upper indices}
 \begin{split}
  T^{LL} &= \frac{(\Lb\psi)^2}{4\mu^2},\ T^{\Lb\Lb} = \frac{(L\psi)^2}{4\mu^2}, \ T^{\Lb L} = \frac{(\slashed{d}\psi)^2}{4\mu}, \ T^{LA} =-\frac{ \Lb\psi X_A(\psi)}{2\mu},\\
T^{\Lb A} &= -\frac{L \psi X_A(\psi)}{2\mu},\ T^{AB} = \slashed{g}^{AC}\slashed{g}^{BD}X_C(\psi)X_D(\psi)-\frac{1}{2}\slashed{g}^{AB}(-\frac{L\psi \Lb \psi}{\mu} + |\slashed{d} \psi|^2).
 \end{split}
\end{equation}
Finally, we can compute the integrands $\widetilde{Q_0}$ and $\widetilde{Q_1}$ explicitly. For $\widetilde{Q_0}$, we have
\begin{equation}\label{Q_0}
\begin{split}
 {c}^{-2} \widetilde{Q}_0 &= -{c}^{-2}\rho\cdot K_0 \psi-\frac{1}{2}{T}^{\mu\nu}\widetilde{\pi}_{0,\mu\nu}= Q_{0,0} + Q_{0,1}+Q_{0,2}+Q_{0,3}+ Q_{0,4}\\
&=-{c}^{-2}\rho \cdot K_0\psi  -T^{L\Lb}\widetilde{\pi}_{0,L\Lb}-T^{LA}\widetilde{\pi}_{0,LA}-T^{\Lb A}\widetilde{\pi}_{0,\Lb A}-\frac{1}{2}T^{AB}\widetilde{\pi}_{0,AB}.
\end{split}
\end{equation}
The $Q_{0,i}$'s are given explicitly by
\begin{equation}\label{Q_0,i}
\begin{split}
  Q_{0,1}& =   \frac{1}{2c}\Big(\mu^{-1}L\mu+L\log(c^{-1}) + \Lb(c^{-2}\mu)\Big)|\slashed{d}\psi|^2, \\
Q_{0,2}&=-{c}^{-1}X_A(c^{-2}\mu) \Lb\psi \cdot X_A(\psi),\ \ Q_{0,3}=-{c^{-1} \mu^{-1}}\Big(\zetab_A+\etab_A\Big)L\psi X_A(\psi),\\
Q_{0,4}&=-\frac{1}{2}\big(\widehat{\chit}_{AB}X^A \psi X^B \psi +\frac{c}{2} \mu^{-1}\tr_{\widetilde{g}}\chit\, {L\psi \cdot \Lb \psi}\big).
\end{split}
\end{equation}
For $\widetilde{Q_1}$, we have
\begin{equation}\label{Q_1}
\begin{split}
  {c}^{-2} \widetilde{Q}_1 &= -{c}^{-2}\rho\cdot K_1 \psi-\frac{1}{2}{T}^{\mu\nu}\widetilde{\pi}_{1,\mu\nu}= Q_{1,0} + Q_{1,1}+Q_{1,2}+Q_{1,3}+ Q_{1,4}\\
&=-{c}^{-2}\rho \cdot K_1 \psi-\frac{1}{2}T^{LL}\widetilde{\pi}_{1,LL}-T^{L\Lb}\widetilde{\pi}_{1,L\Lb} -T^{LA}\widetilde{\pi}_{1,LA} -\frac{1}{2}T^{AB}\widetilde{\pi}_{1,AB}.
\end{split}
\end{equation}
The $Q_{1,i}$'s are given explicitly by
\begin{equation}\label{Q_1,i}
\begin{split}
Q_{1,1}&=-\frac{1}{2c\mu}\Lb(c^{-2}\mu)|\Lb \psi|^2, \ \ Q_{1,2}=\frac{1}{2c}\Big(\mu^{-1}\Lb\mu+ \Lb\log(c^{-1}) \Big)|\slashed{d}\psi|^2,\\
Q_{1,3}&=\frac{1}{c \mu}\big(\zetab_A+\etab_A\big)\Lb\psi \cdot X^A \psi,\ \ Q_{1,4}=-\frac{1}{2}\big(\widehat{\chibt}_{AB}X^A \psi X^B \psi +\frac{c}{2} \mu^{-1}\tr_{\widetilde{g}}\chibt \, {L\psi \cdot \Lb \psi} \big).
\end{split}
\end{equation}

\subsection{Estimates on $Q_{1,2}$: the coercivity of energy estimates in shock region}

We separate the principal terms and lower order terms of $Q_{1,2}$ as follows
\begin{equation*}
\begin{split}
 Q_{1,2}&=\frac{1}{2c}\big(\mu^{-1}\Lb\mu + \Lb\log(c^{-1})\big)|\slashed{d}\psi|^2 =(-\frac{1}{2c}\mu^{-1} \Lb \mu + \text{l.o.t.})|\slashed{d} \psi|^2,
\end{split}
\end{equation*}
where the lower order terms $\lot$, thanks to \eqref{bound on c}, are bounded by $\|\lot\|_{L^{\infty}(\Sigma_t)}\lesssim \delta M^2$. We rewrite the principal term $\displaystyle\int_{W_t^{\ub}}\frac{1}{2c}\mu^{-1} \Lb\mu|\slashed{d}\psi|^{2}$ as
\begin{equation}\label{a Q_1,2 0}
 \int_{W_t^{\ub}}\frac{1}{2c}\mu^{-1}\Lb\mu |\slashed{d}\psi|^2 = \int_{W_t^{\ub}\cap W_{shock}}\frac{1}{2c}\mu^{-1}\Lb\mu |\slashed{d}\psi|^2+ \int_{W_t^{\ub}\cap W_{rare}}\frac{1}{2c}\mu^{-1}\Lb\mu|\slashed{d}\psi|^2.
\end{equation}
In view of the estimates \eqref{expansion of Lb mu} and the fact that $\mu \geq \dfrac{1}{10}$ in the rarefaction wave region $W_{rare}$, for sufficiently small $\delta$, the second term is bounded by	
\begin{align*}
  \int_{W_t^{\ub}\cap W_{rare}} \frac{1}{2c} \mu^{-1}\Lb\mu|\slashed{d}\psi|^2 &\lesssim  \int_{W_t^{\ub}\cap W_{rare}} |\Lb\mu| |\slashed{d}\psi|^2= \int_{W_t^{\ub}\cap W_{rare}} |\frac{|r_0|^2}{|t|^2} \Lb\mu(-r_0,\ub,\theta)+\delta M^2| |\slashed{d}\psi|^2\\
&\lesssim  \int_{W_t^{\ub}\cap W_{rare}} |\slashed{d}\psi|^2.
\end{align*}
Since $\int_{W_t^{\ub}\cap W_{rare}} |\slashed{d}\psi|^2 \leq \int_{W_t^{\ub}} |\slashed{d}\psi|^2 =\int_{-r_0}^t \big(\int_{\Sigma_\tau} \mu |\slashed{d}\psi|^2\big)d\tau$, we obtain
\begin{equation}\label{a Q_1,2 1}
  \int_{W_t^{\ub}\cap W_{rare}} \frac{1}{2c} \mu^{-1}\Lb\mu|\slashed{d}\psi|^2 \lesssim \int_{-r_0}^t {\Eb}(\tau,\ub)d\tau.
\end{equation}

In the shock region, we make use of the key estimate \eqref{C3}. Therefore, we have
\begin{align*}
\int_{W_t^{\ub}\cap W_{shock}} \frac{1}{2c}\mu^{-1}\Lb\mu|\slashed{d}\psi|^2 &\leq -\int_{W_t^{\ub}\cap W_{shock}} \frac{1}{8c|t|}\mu^{-1}|\slashed{d}\psi|^2 \lesssim -\int_{W_t^{\ub}\cap W_{shock}} \mu^{-1}|\slashed{d}\psi|^2.
\end{align*}
We define
\begin{equation}\label{definition for K}
 K(t,\ub) = \int_{W_t^{\ub}\cap W_{shock}} \mu^{-1}|\slashed{d}\psi|^2,
\end{equation}
Therefore,
\begin{equation}\label{a Q_1,2 2}
\int_{W_t^{\ub}\cap W_{shock}} \frac{1}{2c}\mu^{-1}\Lb\mu|\slashed{d}\psi|^2 \lesssim -K(t,\ub).
\end{equation}

Finally, by combining \eqref{a Q_1,2 0}, \eqref{a Q_1,2 1} and \eqref{a Q_1,2 2}, for sufficiently small $\delta$, we obtain
\begin{equation}\label{bound on Q_1,2}
 \int_{W_t^{\ub}} Q_{1,2} \lesssim -K(t,\ub) +\int_{-r_0}^t  {\Eb}(\tau,\ub)d\tau.
\end{equation}
The negative term $-K(t,\ub)$ in above estimates plays a key role to control $\slashed{d} \psi$ in error terms. It will compensate the degeneracy of the $\mu$ factor in front of $\ds\psi$ in the energy.

\subsection{Estimates on $Q_{0,1}$}
The estimates on  $Q_{0,1}$ relies on the second key property \eqref{C2} of $\mu$. 
 We first separate the principal terms and lower order terms of $Q_{0,1}$ as follows:
\begin{equation*}
\begin{split}
  Q_{0,1} &=   \frac{1}{4c}\Big(\mu^{-1}L\mu+L\log(c^{-1}) + \Lb(c^{-1}\kappa)\Big)|\slashed{d}\psi|^2 = \frac{1}{2c}\mu^{-1}T\mu |\slashed{d}\psi|^2 + \lot,
\end{split}
\end{equation*}
where $ \lot = \frac{1}{4c} \big(c^{-2} \Lb \mu + L\log(c^{-1})+\Lb(c^{-1}\kappa)\big)|\slashed{d} \psi|^2$.
According to the estimates derived in previous sections, the terms in the parentheses are bounded by $M^2$. Hence,
\begin{equation}\label{a1}
\begin{split}
 \int_{W_t^{\ub}} \lot &\lesssim \int_{W_t^{\ub}}  M^2 |\slashed{d} \psi|^2 =M^2 \int_{0}^{\ub}\big(\int_{\Cb_{\ub'}}\mu |\slashed{d}\psi|^2\big)d\ub' \leq M^2 \int_{0}^{\ub} {F}(t,\ub') d\ub'.
 \end{split}
\end{equation}

For the principal term whose integrand is $\dfrac{1}{2c}\mu^{-1}T\mu |\slashed{d}\psi|^2 \sim \mu^{-1}T\mu |\slashed{d}\psi|^2$, due to the positivity of $|\ds \psi|^2$, one ignores the contribution from the negative part $\dfrac{1}{c}(\mu^{-1}T\mu)_{-} |\slashed{d}\psi|^2$. Therefore, it is bounded by
\begin{align*}
\lesssim \int_{W_t^{\ub}} (\mu^{-1}T\mu)_+ |\slashed{d}\psi|^2 = \int_{W_t^{\ub} \cap W_{rare}} (\mu^{-1}T\mu)_+ |\slashed{d}\psi|^2 + \int_{W_t^{\ub}\cap
  W_{shock}} (\mu^{-1}T\mu)_+ |\slashed{d}\psi|^2.
\end{align*}

In $W_{rare}$, since $\mu^{-1} \lesssim 1$, we have $\int_{W_t^{\ub}\cap W_{rare}}(\mu^{-1}T\mu)_+	 |\slashed{d}\psi|^2 \lesssim \int_{W_t^{\ub}\cap W_{rare}}|T\mu|	 |\slashed{d}\psi|^2 \leq \int_{W_t^{\ub}}|T\mu|	 |\slashed{d}\psi|^2$.
According to  \eqref{precise bound on L mu}, we have $T\mu \lesssim \delta^{-1}$, therefore,
\begin{equation}\label{a2}
\begin{split}
 \int_{W_t^{\ub}\cap W_{rare}}(\mu^{-1}T\mu)_+	 |\slashed{d}\psi|^2 \lesssim \int_{W_t^{\ub}}  \delta^{-1}|\slashed{d} \psi|^2 \leq \delta^{-1}\int_{0}^{\ub} {F}(t,\ub') d\ub'.
 \end{split}
\end{equation}

In $W_{shock}$, the argument relies on Proposition \ref{Proposition C2}. Indeed, we have
\begin{align*}
 \int_{W_t^{\ub}\cap W_{shock}}(\mu^{-1}T\mu)_+ |\slashed{d}\psi|^2 &\lesssim \int_{W_t^{\ub}\cap W_{shock}} \frac{1}{|t-s^*|^\frac{1}{2}} \delta^{-1}|\slashed{d} \psi|^2  \leq \delta^{-1} \int_{W_t^{\ub}} \frac{1}{|t-t^*|^\frac{1}{2}} |\slashed{d} \psi|^2.
\end{align*}
By definition, the last integral is equal to $\delta^{-1} \int_{-r_0}^{t}\frac{1}{|\tau-t^*|^\frac{1}{2}} \big( \int_{\Sigma\tau} \mu |\slashed{d} \psi|^2d\mu_{\slashed{g}}\big)d\tau$. In view of the definition of ${\Eb}(t,\ub)$, we obtain
\begin{equation}\label{a3}
 \int_{W_t^{\ub}\cap W_{shock}} (\mu^{-1}T\mu)_+ |\slashed{d}\psi|^2  \lesssim \delta^{-1} \int_{-r_0}^{t}\frac{1}{|\tau-t^*|^\frac{1}{2}} {\Eb}(\tau,\ub)d\tau
\end{equation}
We remark that the key feature of above estimates is that the factor $\dfrac{1}{|t-t^*|^\frac{1}{2}}$ is integrable in $t$. It will allow us to use Gronwall's inequality.

Finally, taking into account of the estimates \eqref{a1}, \eqref{a2} and \eqref{a3}, for sufficiently small $\delta$, we obtain

\begin{equation}\label{bound on Q_0,1}
 \int_{W_t^{\ub}} Q_{0,1} \lesssim \delta^{-1} \int_{-r_0}^{t}\frac{1}{|\tau-t^*|^\frac{1}{2}} {\Eb}(\tau,\ub)d\tau +\delta^{-1}\int_{0}^{\ub} {F}(t,\ub') d\ub'.
\end{equation}

\subsection{Estimates on other error terms} We deal with $Q_{0,2}$, $Q_{0,3}$, $Q_{0,4}$, $Q_{1,1}$, $Q_{1,3}$ and $Q_{1,4}$ one by one.

For $Q_{0,2}$, we have $|Q_{0,2}|= |-\frac{1}{c} X^A(c^{-2}\mu) \Lb\psi \cdot X_A(\psi)|\lesssim  |\slashed{d}\mu| |\Lb\psi||\slashed{d}\psi|$. According to \eqref{precise bound on d mu}, $|\slashed{d}\mu| \lesssim 1+ \delta M^4$. Hence, for sufficiently small $\delta$, we have
\begin{align*}
   \big|\int_{W_t^{\ub}} Q_{0,2}\big| &\lesssim  \int_{W_t^{\ub}} (1+\delta M^4)|\Lb\psi||\slashed{d}\psi| \lesssim \int_{W_t^{\ub}}|\Lb\psi|^2+|\slashed{d}\psi|^2
\end{align*}
Therefore, we obtain
\begin{equation}\label{estimates on Q_0,2}
  \big|\int_{W_t^{\ub}} Q_{0,2}\big| \lesssim   \int_{-r_0}^t {\Eb}(\tau,\ub) d\tau. 
\end{equation}

For $Q_{0,3}$,  we first recall that $|\zetab|\lesssim 1$ and $\etab = \zetab + \ds \mu$, therefore, $|\zetab| + |\etab| \lesssim 1+ \delta M^4$. We break $Q_{0,3}$ into two parts as follows:
\begin{align*}
 \big|\int_{W_t^{\ub}} Q_{0,3} \big| &\lesssim \int_{W_t^{\ub}} \mu^{-1}(|\zetab|+|\etab|)|L \psi||\slashed{d}\psi| \lesssim  \int_{W_t^{\ub}} \mu^{-1}|L \psi||\slashed{d}\psi|  = I_1 + I_2 \\
 &= \int_{W_t^{\ub}\cap W_{rare}}\mu^{-1}|L \psi||\slashed{d}\psi|+ \int_{W_t^{\ub}\cap W_{shock}} \mu^{-1}|L \psi||\slashed{d}\psi|.
\end{align*}
The integral $I_1=\int_{W_t^{\ub}\cap W_{rare}}\mu^{-1}|L \psi||\slashed{d}\psi|$ is taken in $W_{rare}$ where $\mu \geq \frac{1}{10}$, therefore, we obtain
\begin{align*}
 I_1 &\leq  \int_{W_t^{\ub}\cap W_{rare}}\!\!\!\!\!\! \mu^{-1}(|L \psi|^2+|\slashed{d}\psi|^2) \lesssim  \int_{W_t^{\ub}\cap W_{rare}}\!\!\!\!\!\!\mu^{-1}(|L\psi|^2+\mu|\slashed{d}\psi|^2)\lesssim\int_{W_t^{\ub}} \mu^{-1}(|L\psi|^2+\mu|\slashed{d}\psi|^2)\\
&\lesssim \int_{-r_0}^t {E}(\tau,\ub)d\tau + \int_{0}^{\ub} {F}(t,\ub') d\ub'.
\end{align*}
To control the integral $I_2=\int_{W_t^{\ub}\cap W_{shock}} \mu^{-1}|L\psi||\slashed{d}\psi|$,  we use the coercive term $K(t,\ub)$ from $Q_{1,2}$ to control the loss of $\mu^{-1}$:
\begin{align*}
 I_2 &\leq  \big(\int_{W_t^{\ub}\cap W_{shock}}\mu^{-1}|\slashed{d}\psi|^2\big)^{\frac{1}{2}}\big(\int_{W_t^{\ub}\cap W_{shock}} \mu^{-1}|L\psi|^2\big)^{\frac{1}{2}} \lesssim  K(t,\ub)^{\frac{1}{2}} \big(\int_{-r_0}^{t}  {E}(\tau,\ub) d\tau \big)^{\frac{1}{2}}.
\end{align*}
Hence, we obtain
\begin{equation}\label{estimates on Q_0,3}
    \big|\int_{W_t^{\ub}} Q_{0,3} \big| \lesssim \int_{-r_0}^t {E}(\tau,\ub)d\tau + \int_{0}^{\ub} {F}(t,\ub') d\ub'+ K(t,\ub)^{\frac{1}{2}} \big(\int_{-r_0}^{t}  {E}(\tau,\ub) d\tau \big)^{\frac{1}{2}}.
\end{equation}

For $Q_{0,4}$,  we have
\begin{align*}
 \big|\int_{W_t^{\ub}} Q_{0,4} \big| &\lesssim \int_{W_t^{\ub}} |\widehat{\chit}| |\slashed{d}\psi|^2 +\mu^{-1}|\tr_{\gt} {\chit}||L\psi||\Lb \psi| \lesssim \int_{W_t^{\ub}} \delta M^2 |\slashed{d}\psi|^2 +\mu^{-1}|L\psi||\Lb \psi| =I_1+I_2.
\end{align*}
The bound on $I_1$ is immediate: $I_1=  \delta M^2 \int_{W_t^{\ub}} |\slashed{d}\psi|^2 \leq \delta M^2 \int_{0}^{\ub} {F}(t,\ub')d\ub'$.
The bound on $I_2 = \int_{W_t^{\ub}}\mu^{-1}|L\psi||\Lb \psi|$ relies on the energy ${E}(t,\ub)$ and on the flux ${\Fb}(t,\ub)$:
\begin{align*}
 I_2 &\lesssim \int_{W_t^{\ub}}  \mu^{-1}|L\psi|^2 + \int_{W_t^{\ub}}\mu^{-1}|\Lb \psi|^2 =\int_{-r_0}^t {E}(\tau,\ub)d\tau + \int_{0}^{\ub} {\Fb}(t,\ub')d\ub'.
\end{align*}
Hence, we obtain
\begin{equation}\label{estimates on Q_0,4}
  \big| \int_{W_t^{\ub}} Q_{0,4} \big| \lesssim  \int_{-r_0}^t {E}(\tau,\ub)d\tau + \delta M^2 \int_{0}^{\ub} {F}(t,\ub')d\ub' +\int_{0}^{\ub} {\Fb}(t,\ub') d\ub'.
\end{equation}

For $Q_{1,1}$, according to \eqref{expansion of Lb mu}, we have $\big|Q_{1,1}\big| =\big|\frac{1}{2c\mu} \Lb(c^{-2}\mu)|\Lb \psi|^2\big| \lesssim \big|\mu^{-1}\Lb\mu|\Lb \psi|^2\big| \lesssim \mu^{-1} |\Lb \psi|^2$. Hence,
\begin{equation}\label{estimates on Q_1,1}
   \big| \int_{W_t^{\ub}} Q_{1,1} \big| \lesssim  \int_{0}^{\ub} {\Fb}(t,\ub')d\ub'.
\end{equation}

For $Q_{1,3}$, since $|\zetab| + |\etab| \lesssim 1$, we break the integral into two parts:
\begin{align*}
 \big|\int_{W_t^{\ub}} Q_{1,3} \big| &\lesssim \int_{W_t^{\ub}} \mu^{-1}|\Lb \psi||\slashed{d}\psi| = I_1 + I_2 \\
 &= \int_{W_t^{\ub}\cap W_{rare}}\mu^{-1}|\Lb \psi||\slashed{d}\psi|+ \int_{W_t^{\ub}\cap W_{shock}} \mu^{-1}|\Lb \psi||\slashed{d}\psi|.
\end{align*}
In the rarefaction wave region, since $\mu^{-1} \sim 1$,  we obtain
\begin{align*}
 I_1 &\lesssim  \int_{W_t^{\ub}\cap W_{rare}} \mu^{-1}(|\Lb \psi|^2+|\slashed{d}\psi|^2)  \approx \int_{W_t^{\ub}\cap W_{rare}} |\Lb\psi|^2+|\slashed{d}\psi|^2\\
 &\lesssim\int_{W_t^{\ub}}|\Lb\psi|^2+|\slashed{d}\psi|^2 \lesssim \int_{-r_0}^t {\Eb}(\tau,\ub)d\tau.
\end{align*}
In the shock wave region, we still use $K(t,\ub)$ to control the loss of $\mu^{-1}$:
\begin{align*}
 I_2 &\leq  \big(\int_{W_t^{\ub}\cap W_{shock}}\mu^{-1}|\slashed{d}\psi|^2\big)^{\frac{1}{2}}\big(\int_{W_t^{\ub}\cap W_{shock}} \mu^{-1}|\Lb\psi|^2\big)^{\frac{1}{2}} \lesssim  K(t,\ub)^{\frac{1}{2}} \big(\int_{0}^{\ub}  {\Fb}(t,\ub') d\ub' \big)^{\frac{1}{2}}.
\end{align*}
Hence,
\begin{equation}\label{estimates on Q_1,3}
    \big|\int_{W_t^{\ub}} Q_{1,3} \big| \lesssim \int_{-r_0}^t {\Eb}(\tau,\ub)d\tau + K(t,\ub)^{\frac{1}{2}} \big(\int_{0}^{\ub}  {\Fb}(t,\ub') d\ub' \big)^{\frac{1}{2}}.
\end{equation}

For $Q_{1,4}$,  we have
\begin{align*}
 \big|\int_{W_t^{\ub}} Q_{1,4} \big| &\lesssim \int_{W_t^{\ub}} |\widehat{\chibt}| |\slashed{d}\psi|^2 +\mu^{-1}|\tr_{\gt}{\chibt}||L\psi||\Lb \psi| \lesssim \int_{W_t^{\ub}} \delta M^2 |\slashed{d}\psi|^2 +\mu^{-1}|L\psi||\Lb \psi| \\
& \lesssim \delta M^2 \int_{0}^{\ub} {F}(t,\ub')d\ub' + \Big( \delta \int_{W_t^{\ub}}  \mu^{-1}|L\psi|^2 + \delta^{-1}\int_{W_t^{\ub}}\mu^{-1}|\Lb \psi|^2\Big).
\end{align*}
Therefore, we obtain
\begin{equation}\label{estimates on Q_1,4}
  \big| \int_{W_t^{\ub}} Q_{1,4} \big| \lesssim \delta \int_{-r_0}^t {E}(\tau,\ub)d\tau + \delta M^2 \int_{0}^{\ub} {F}(t,\ub')d\ub' + \delta^{-1}\int_{0}^{\ub} {\Fb}(t,\ub')d\ub'.
\end{equation}

\subsection{Summary}

In view of \eqref{energy identities} and \eqref{comparisons between energy and reduced energy}, since $\psi$ vanishes to infinite order on $\Cb_{0}$, we have
\begin{equation*}
\begin{split}
 \Big({E}(t,\ub) + {F}(t,\ub)\Big) &+ \delta^{-1}\Big({\Eb}(t,\ub) + {\Fb}(t,\ub)\Big) \lesssim {E}(-r_0,\ub)  + \delta^{-1}{\Eb}(-r_0,\ub)  \\
& +  \sum_{i=1}^4\int_{W_t^{\ub}} Q_{0,i} +  \delta^{-1}\sum_{i=1}^4 \int_{W_t^{\ub}}Q_{1,i} +  \Big|\int_{W_t^{\ub}}\rho\cdot L\psi \Big| + \delta^{-1}\Big|\int_{W_t^{\ub}}\rho \cdot \Lb\psi\Big|.
\end{split}
\end{equation*}
We bound sums $\sum_{i=1}^4\int_{W_t^{\ub}} Q_{0,i}$ and $\sum_{i=1}^4 \int_{W_t^{\ub}}Q_{1,i}$ by \eqref{bound on Q_1,2}, \eqref{bound on Q_0,1}, \eqref{estimates on Q_0,2}, \eqref{estimates on Q_0,3}, \eqref{estimates on Q_0,4}, \eqref{estimates on Q_1,1}, \eqref{estimates on Q_1,3} and \eqref{estimates on Q_1,4}. Therefore, we have
\begin{equation*}
\begin{split}
&\quad \Big({E}(t,\ub) + {F}(t,\ub)\Big) + \delta^{-1}\Big({\Eb}(t,\ub) + {\Fb}(t,\ub)\Big) \\
&\lesssim {E}(-r_0,\ub)  + \delta^{-1}{\Eb}(-r_0,\ub)  +  \Big|\int_{W_t^{\ub}}\rho\cdot L\psi \Big| + \delta^{-1}\Big|\int_{W_t^{\ub}}\rho \cdot \Lb\psi \Big| \\
&\quad+ \delta^{-1}\Big(\underbrace{-K(t,\ub) +\int_{-r_0}^t  {\Eb}(\tau,\ub)d\tau}_{Q_{1,2}}\Big) + \underbrace{\Big( \delta^{-1} \int_{-r_0}^{t}\frac{1}{|\tau-t^*|^\frac{1}{2}} {\Eb}(\tau,\ub)d\tau +\delta^{-1}\int_{0}^{\ub} {F}(t,\ub') d\ub'\Big)}_{Q_{0,1}}\\
&\quad + \underbrace{\Big( \int_{-r_0}^t  [{E}(\tau,\ub)+ {\Eb}(\tau,\ub)] d\tau + \int_{0}^{\ub} [{F}(t,\ub')+ {\Fb}(t,\ub')] d\ub' + K(t,\ub)^{\frac{1}{2}} \big(\int_{-r_0}^{t}  {E}(\tau,\ub) d\tau \big)^{\frac{1}{2}}\Big)}_{Q_{0,2}+Q_{0,3} +Q_{0,4}} \\
&\quad +\delta^{-1}\underbrace{\Big( \int_{-r_0}^t  [\delta {E}(\tau,\ub)+ {\Eb}(\tau,\ub) ]d\tau + \int_{0}^{\ub} [\delta M^2 {F}(t,\ub')+ \delta^{-1}{\Fb}(t,\ub')] d\ub' + K(t,\ub)^{\frac{1}{2}} \big(\int_{-r_0}^{t}  {\Fb}(t,\ub') d\ub' \big)^{\frac{1}{2}}\Big)}_{Q_{1,1}+Q_{1,3} +Q_{1,4}}.
\end{split}
\end{equation*}
For sufficiently small $\delta$, we can rewrite the estimates as
\begin{equation*}
\begin{split}
&\quad  \Big({E}(t,\ub) + {F}(t,\ub)\Big) + \delta^{-1}\Big({\Eb}(t,\ub) + {\Fb}(t,\ub)\Big) \\
&\lesssim  {E}(-r_0,\ub) +\delta^{-1}{\Eb}(-r_0,\ub)  +   \Big|\int_{W_t^{\ub}}\rho\cdot L\psi \Big| + \delta^{-1}\Big|\int_{W_t^{\ub}}\rho \cdot \Lb\psi \Big| \\
&\quad  -\delta^{-1}K(t,\ub) + \delta^{-1}\int_{-r_0}^{t}\frac{1}{|\tau-t^*|^\frac{1}{2}} {\Eb}(\tau,\ub)d\tau\\ &+\underbrace{\delta^{-1}K(t,\ub)^{\frac{1}{2}}\Big(\delta \big(\int_{-r_0}^{t}  {E}(\tau,\ub) d\tau \big)^{\frac{1}{2}}+ \big(\int_{0}^{t}  {\Fb}(t,\ub') d\ub' \big)^{\frac{1}{2}}\Big)}_{I}\\
&\quad + \Big[ \int_{-r_0}^t    [{E}(\tau,\ub)+ \delta^{-1}{\Eb}(\tau,\ub) ]d\tau + \delta^{-1}\int_{0}^{\ub}[ {F}(t,\ub')+ \delta^{-1} {\Fb}(t,\ub') ]d\ub'\Big]
\end{split}
\end{equation*}
For the term involving $I$, we use Cauchy-Schwarz inequality and put a small parameter $\varepsilon_0$ in front of the $K(t,\ub)$, i.e, $
I \lesssim \delta^{-1}\varepsilon_0 K(t,\ub) + \delta^{-1}\frac{1}{\varepsilon_0}\Big(\delta \int_{-r_0}^{t}  {E}(\tau,\ub) d\tau + \int_{-r_0}^{t}  {\Fb}(t,\ub') d\ub'\Big)$. Therefore, the resulting $K(t,\ub)$ term can be absorbed by the coercive term $-K(t,\ub)$ and we obtain
\begin{equation*}
\begin{split}
&\quad  \Big({E}(t,\ub) + {F}(t,\ub)\Big) + \delta^{-1}\Big({\Eb}(t,\ub) + {\Fb}(t,\ub)\Big) \\
&\lesssim  {E}(-r_0,\ub) + \delta^{-1}{\Eb}(-r_0,\ub)  +   \Big|\int_{W_t^{\ub}}\rho\cdot L\psi \Big| + \delta^{-1}\Big|\int_{W_t^{\ub}}\rho \cdot \Lb\psi \Big|  \\
&-\delta^{-1}K(t,\ub) + \delta^{-1}\int_{-r_0}^{t}\frac{1}{|\tau-t^*|^\frac{1}{2}} {\Eb}(\tau,\ub)d\tau \\
&\quad + \Big[ \int_{-r_0}^t    {E}(\tau,\ub)+ \delta^{-1}{\Eb}(\tau,\ub) d\tau + \delta^{-1}\int_{0}^{\ub} {F}(t,\ub')+ \delta^{-1} {\Fb}(t,\ub') d\ub'\Big]
\end{split}
\end{equation*}
By Gronwall's inequality (the factor $\frac{1}{|\tau-t^*|^\frac{1}{2}}$ is integrable in $\tau$!), we can remove all the integral terms on the last line, this proves the Fundamental Energy Estimates (\textbf{F.E.E}) for $\Box_{\gt}\psi = \rho$:
\begin{equation}
\begin{split}
&\quad  \Big({E}(t,\ub) + {F}(t,\ub)\Big) + \delta^{-1}\Big({\Eb}(t,\ub) + {\Fb}(t,\ub)+K(t,\ub)\Big) \\
&\lesssim   {E}(-r_0,\ub) + \delta^{-1}{\Eb}(-r_0,\ub)  +  \Big|\int_{W_t^{\ub}}\rho\cdot L\psi \Big| +\delta^{-1} \Big|\int_{W_t^{\ub}}\rho \cdot \Lb\psi \Big|.
\end{split}\tag{\textbf{F.E.E}}
\end{equation}
\begin{remark}
We see that the energy estimates only imply
\begin{align*}
\Eb(t,\ub)+K(t,\ub)\lesssim \delta,
\end{align*}
which does not recover the full regularity of $\Lb\psi$ and $\ds psi$ with respect to $\delta$, as indicated by bootstrap assumption. We will see finally that this full regularity in $\delta$ is recovered by using the estimates for $E(t,\ub)$ together with Lemma \ref{Calculus Inequality} and the commutation of $Q$.
\end{remark}

\section{Comparisons between Euclidean and optical geometries}\label{Appendix_Estimates on higher order derivatives of deformation tensors for Ri}

There are two different geometries coming into play on $W^{*}_{\delta}$, namely, the Minkowski geometry and the optical geometry.
The knowledge on two metrics $g_{\mu\nu}$ (or $\gt_{\mu\nu}$) and $m_{\mu\nu}$ is essentially tied to the estimates on the solution of $(\star)$.
There are many ways to compare two geometries, e.g., we may 
consider the Cartesian coordinates $x^k$ as functions of the optical coordinates $(t,\ub, \theta)$. In what follows, we also study other quantities as $y^k$, $z^k$, $\lambda_i$, etc. As a by product, we will also obtain estimates for the lower order objects, i.e., with order $< \Ntop+1$.

Given a vectorfield $V$, we define the null components of its deformation tensor as
\begin{equation}\label{definition for Zb and pislash}
{}^{(V)}\Zb_A = {}^{(V)}\pi (\Lb, X_A), \ \ {}^{(V)}Z_A = {}^{(V)}\pi (L, X_A), \ \  {}^{(V)}\slashed{\pi}_{AB} =  {}^{(V)}\pi (X_A, X_B).
\end{equation}
The projection of Lie derivative $\mathcal{L}_{V}$ to $S_{t,\ub}$ is denoted as $\slashedL_{V}$.
The shorthand notation $\slashedL_{Z_i}^{\, \alpha}$ to denote $\slashedL_{Z_{i_1}} \slashedL_{Z_{i_2}}\cdots \slashedL_{Z_{i_k}} $
for a multi-index $\alpha =(i_1,\cdots,i_k)$. We will show that,
for all $|\alpha| \leq \Ninfty$, we have $\slashedL_{Z_i}^{\, \alpha-1} \chib', \slashedL_{Z_i}^{\, \alpha-1}  {}^{(Z_j)}\Zb, \slashedL_{Z_i}^{\, \alpha-1}  {}^{(Z_j)}\slashed{\pi}, \slashedL_{Z_i}^{\, \alpha-1}  {}^{(Q)}\left(\slashed{\pi}+4\right) \in \O^{|\alpha|}_{2-2l}$,
where $l$ is the number of $T$'s in $Z_i$'s, $|\alpha|\geq 1$ and $Z_{j}\slashed{=}T$. If $Z_{j}=T$, then we have $\slashedL_{Z_i}^{\, \alpha-1}  {}^{(T)}\Zb, \slashedL_{Z_i}^{\, \alpha-1}  {}^{(T)}\slashed{\pi} \in \O^{|\alpha|}_{-2l}$.
The idea of the proof is to compare $g_{\mu\nu}$ and $m_{\mu\nu}$ via quantities such as $x^i$, $y^i$, $z^i$, $\widehat{T}^i$ and $\lambda_i$.
Similarly, we will derive $L^2$-estimates on objects of order $\leq \Nmu$. The $L^2$ estimates depend on the $L^\infty$ estimates up to order $\Ninfty+2$.
In the course of the proof, it will be clear why $\Ninfty$ is chosen to be approximately $\frac{1}{2} N_{top}$.

\subsection{$L^\infty$ estimates}
We assume (B.1): for all $\ |\alpha| \leq N_{\infty}$, $Z_i^{\alpha+2} \psi \in \Psi^{|\alpha|+2}_{1-2l}$.
\begin{proposition}\label{L infity estimates on lot}
For sufficiently small $\delta$, for all $|\alpha| \leq \Ninfty$ and $t\in [-r_0, s^*]$, we can bound $\big\{\slashedL_{Z_i}^{\, \alpha} \chib'$, $\slashedL_{Z_i}^{\, \alpha}  {}^{(Z_j)}\Zb$, $\slashedL_{Z_i}^{\, \alpha}  {}^{(Z_j)}\slashed{\pi}, \slashedL_{Z_i}^{\, \alpha}  {}^{(Q)}\left(\slashed{\pi}+4\right), Z_i^{\,\alpha+1} y^j, Z_i^{\,\alpha+1} \lambda_j\} \subset \O_{2-2l}^{|\alpha|+1}$ and\\ $\left\{
Z_i^{\,\alpha+1} x^j,  Z_i^{\,\alpha+1} \widehat{T}^j, \slashedL_{Z_i}^{\, \alpha+1} Z_j, \slashedL_{Z_i}^{\, \alpha}  {}^{(T)}\Zb, \slashedL_{Z_i}^{\, \alpha}  {}^{(T)}\slashed{\pi} \right\} \subset \O_{-2l}^{|\alpha|+1}$ in terms of $Z_{i}^{\alpha+2}\psi\in\Psi^{|\alpha|+2}_{1-2l}$. Here $l$ is the number of $T$'s in $Z_i$'s.
\end{proposition}
The estimates on $x^j$ should be viewed as a good comparison between Euclidean and optical geometries on each slice $\Sigma_t$. 
\begin{proof}
We do induction on the order. When $|\alpha|=0$, the estimates are treated in Section 3. Here we only treat the estimates when $l=0$, when $l\geq1$, we can use the structure equation \eqref{Structure Equation T chib} to reduce the problem to the estimates for $\mu$, which will be treated in Proposition \ref{prop mu L infty estimates}. Here the loss of $\delta$ in the estimates for $\slashedL_{Z_i}^{\, \alpha}  {}^{(T)}\Zb, \slashedL_{Z_i}^{\, \alpha}  {}^{(T)}\slashed{\pi}$ comes from applying $T$ to $\dfrac{2}{t-\ub}$, which is the principal part of $\tr\chib$. Given $|\alpha| \leq \Ninfty$, 
we assume that estimates hold for terms of order $\leq |\alpha|$. In particular, 
we have $\slashedL_{R_i}^{\, \beta} \chib', R_i^{\,\beta+1} y^j \in \O^{|\beta|}_{2}$ for all $|\beta| \leq |\alpha|$. We prove the proposition for $|\alpha|+1$.

\underline{Step 1 \ Bounds $R_i^{\,\alpha+1} x^j$}. \ \  Let $\delta_{\alpha+1,i}^{j} = \Omega_i^{\alpha+1} x^j - R_i^{\,\alpha+1} x^j$  where $\Omega_i$'s are the standard rotational vectorfields on Euclidean space. It is obvious that $\Omega_i^{\alpha+1} x^j$ is equal to some $x^k$, therefore, bounded by $r$ hence by a universal constant. Since $R_i = \Omega_i -\lambda_i \widehat{T}^j \partial_j$, $R_{i}x^{j}\in\O^{0}_{0}$ and by ignoring all the numerical constants, we have
\begin{equation*}
 \delta_{\alpha+1,i}^{j} =  R_{i}^{\alpha}\left(\lambda_i \widehat{T}^j\right)= R_{i}^{\alpha} \left(\lambda_i \left(\frac{x^j}{\ub-t}+y^j\right)\right).
\end{equation*}
Here the index $i$ is not a single index. It means we apply a string of different $R_{i}$s. This notation applies in the following when a string of $R_{i}$s are considered. 

Since the above expression has total order  $\leq |\alpha|$, by the induction hypothesis 
we obtain immediately that $\delta_{\alpha+1,i}^{j} \in \O_{2}^{|\alpha|}$. In view of the definition of $\delta_{\alpha+1,i}^{j}$, we then have $R_i^{\,\alpha+1} x^j \in \O_{0}^{|\alpha|+1}$.

\underline{Step 2 \ Bounds on $\slashedL_{R_i}^{\, \alpha} \chib'$}. \ \ We commute $\slashedL_{R_i}^{\, \alpha}$ with \eqref{transport equation for chib prime} to derive
\begin{equation}\label{b2}
\begin{split}
\slashedL_{\Lb} \slashedL_{R_i}^{\, \alpha}  \chib' &= \left[ \slashedL_{\Lb}, \slashedL_{R_i}^{\, \alpha} \right]\chib' + e \cdot \slashedL_{R_i}^{\, \alpha} \chib' + \chib' \cdot  \slashedL_{R_i}^{\, \alpha} \chib' + \sum_{|\beta_1|+|\beta_2| = |\alpha| \atop |\beta_2|<|\alpha|}  {R_i}^{\beta_1} e \cdot \slashedL_{R_i}^{\beta_2}\chib' \\
&\quad +   \sum_{|\beta_1|+|\beta_2| +|\beta_3| = |\alpha| \atop |\beta_1|<|\alpha|, |\beta_3|<|\alpha|}  \slashedL_{R_i}^{\beta_1} \chib' \cdot \slashedL_{R_i}^{\beta_2}\slashed{g}^{-1} \cdot \slashedL_{R_i}^{\beta_3}\chib'+ \slashedL_{R_i}^{\, \alpha}\left(\frac{e\slashed{g}_{AB}}{t-\ub} -\alphab'_{AB}\right).
\end{split}
\end{equation}
Since $e =c^{-1}\dfrac{d c}{d\rho}\Lb\rho$, ${R_i}^{\alpha} e$ is of order $\leq |\alpha|+1$. By (B.1), we have $R_{i}^\alpha e \in \Psi_{2}^{|\alpha|+1}$.
Similarly, by the explicit formula of $\alphab'_{AB}$, we have $\slashedL_{R_i}^\beta \alphab'_{AB} \in \Psi_{2}^{|\alpha|+2}$. Since ${}^{(R_i)}\pi_{AB} = 2c^{-1}\lambda_i \chib_{AB}$ and $\slashedL_{R_i}^{\beta_2}\slashed{g} = \slashedL_{R_i}^{\beta_2-1}\,{}^{(R_i)}\slashed{\pi}_{AB}$, by the estimates derived in previous sections and by the induction hypothesis, we can rewrite \eqref{b2} as
\begin{equation*}
\slashedL_{\Lb} \slashedL_{R_i}^{\, \alpha}  \chib' = \big[ \slashedL_{\Lb}, \slashedL_{R_i}^{\, \alpha} \big]\chib' +\O_{2}^{1}\cdot \slashedL_{R_i}^{\, \alpha-1} \chib'+\Psi_{2}^{\leq |\alpha|+2}.
\end{equation*}
The commutator can be computed as $\big[ \slashedL_{\Lb}, \slashedL_{R_i}^{\, \alpha} \big]\chib' = \sum_{|\beta_1|+|\beta_2| = |\alpha|-1} \slashedL_{R_i}^{\beta_1}\slashedL_{{}^{(R_i)}\Zb}\slashedL_{R_i}^{\beta_2}\chib'$.
Since ${}^{(R_i)}\Zb_A = -\chib'_{AB}R_i{}^B + \varepsilon_{ijk}z^{j}X_A{}^k + \lambda_i \slashed{d}_A (c)$ and $z^j=-\frac{(c-1)x^j}{\ub-t}-cy^j$, the commutator term is of type $\mathcal{O}^{1}_{2}\cdot\slashed{\mathcal{L}}^{\alpha}_{R_{i}}\chib'$. Therefore, we have
\begin{equation*}
\slashedL_{\Lb} \slashedL_{R_i}^{\, \alpha}  \chib' = \O_{2}^{1}\cdot \slashedL_{R_i}^{\, \alpha} \chib'+\Psi_{2}^{\leq|\alpha|+2}.
\end{equation*}
By integrating this equation from $-r_0$ to $t$, the Gronwall's inequality yields $\|\slashedL_{R_i}^{\, \alpha} \chib'\|_{L^\infty(\Sigma_t)} \lesssim_M \delta$.

\underline{Step 3 \ Bounds on $R_i^{\,\alpha+1} y^j$ and $R_i^{\,\alpha+1}\lambda_{j}$}.  Since $R_i y^j = (-c^{-1}\chib^{A}_{B}-\frac{\delta^{A}_{B}}{\ub-t}) R_{i}^A \slashed{d}_B x^j$, schematically we have
\begin{equation*}
 R_i^{\,\alpha+1} y^j = R_{i}^{\alpha}R_k y^j = R_i^{\alpha}\left( \left(c^{-1}\chib+\frac{\delta}{\ub-t}\right)\cdot R_k \cdot \slashed{d} x^j \right).
\end{equation*}
 We distribute $R_i^\alpha$ inside the parenthesis by Leibniz rule. (Here again, the index $i$ is not a single index, so we use index $k$ to distinct the last rotation vectorfield.) Therefore, a typical term would be either $\slashedLRi^{\beta_1}\chib\cdot \slashedLRi^{\beta_2}\slashed{g}^{-1}\cdot \slashedLRi^{\beta_3} R_k \cdot \slashed{d} R_i^{\beta_4} x^j$ or $\slashedLRi^{\beta_1}\slashed{g}\cdot \slashedLRi^{\beta_2}\slashed{g}^{-1}\cdot \slashedLRi^{\beta_3} R_k \cdot \slashed{d} R_i^{\beta_4} x^j$ with $|\beta_1|+|\beta_2|+|\beta_3|+|\beta_4| = |\alpha|$.

There are only two terms where are not included in the induction hypothesis: $\slashedLRi^{\beta_2} \slashed{g}_{AB}$ and $\slashedLRi^{\beta_3} R_j$. The first term is in fact easy to handle by induction hypothesis and estimates derived in Step 1 and Step 2, since $\slashedLRi \slashed{g}_{AB} = {}^{(R_i)}\slashed{\pi}_{AB} = 2\lambda_i c^{-1}\chib_{AB}$. For the second one, we use the following expression:
\begin{align*}
\slashedLRi R_j = -\sum_{k=1}^3 \varepsilon_{ijk} R_k + \lambda_i\big( \frac{c^{-1}-1}{\ub-t}\slashed{g}&-c^{-1}(\chib+\frac{\slashed{g}}{\ub-t})\big)R_j-\lambda_j\big( \frac{c^{-1}-1}{\ub-t}\slashed{g}-c^{-1}(\chib+\frac{\slashed{g}}{\ub-t})\big)R_i\\
 & -\lambda_i \varepsilon_{jkl}y^k\slashed{d}x^l \cdot \slashed{g}^{-1}+ \lambda_j \varepsilon_{ikl}y^k\slashed{d}x^l\cdot \slashed{g}^{-1}.
\end{align*}
Therefore, $\slashedLRi^{\beta_3} R_k =\slashedLRi^{\beta_3-1} \slashedLRi R_k = \O_{0}^{|\beta_3|-1} +\O_{\geq2}^{|\beta_3|-1} R_i^{\beta_3}x^j $. Finally, we obtain that
\begin{equation*}
 R_i^{\,\alpha+1} y^j = \O^{\leq |\alpha|}_{2} + \O^{\leq |\alpha|}_{2} \cdot \slashed{d} R_i^{\beta_4} x^j.
\end{equation*}
Although $ \slashed{d} R_i^{\beta_4} x^j$ and $\slashedLRi^{\beta_1}\chib$ may have order $|\alpha|+1$, they have been controlled from previous steps. This gives the bounds on $R_i^{\,\alpha+1} y^j$. Then by the fact that $\lambda_{j}=\epsilon_{jkl}x^{k}y^{l}$, the estimate for $R^{\,\alpha+1}_{i}\lambda_{j}$ follows. This completes the proof of the proposition.
\end{proof}

\begin{proposition}\label{prop mu L infty estimates}
For sufficiently small $\delta$, for all $|\alpha|\leq \Ninfty, t \in [-r_{0},s^{*}]$, we can bound $Z_i^{\alpha+1} \mu \in \O_{-2l}^{|\alpha|+1}$ in terms of $Z^{\alpha+2}_{i}\psi\in\Psi^{|\alpha|+2}_{1-2l}$.
\end{proposition}
\begin{proof}
We use an induction argument on the order of derivatives. The base case $|\alpha|=0$ has be treated in Section 3 and Section 4.  We assume the proposition holds with order of derivatives on $\mu$ at most $|\alpha|$. For $|\alpha|+1$, by commuting $Z^{\alpha+1}_{i}$ with $\Lb\mu=m+\mu e$, we have
\begin{align*}
\Lb \delta^{l} Z^{\alpha+1}_{i}\mu = (e + \leftexp{(Z_{i})}{\Zb}) \delta^{l}Z^{\alpha+1}_{i}\mu + \delta^{l}Z^{\alpha+1}_{i}m +\sum_{|\beta_1|+|\beta_2|\leq |\alpha|+1 \atop|\beta_{1}|<|\alpha|}\delta^{l_{1}}Z_{i}^{\beta_{1}}\mu \delta^{l_{2}}Z_{i}^{\beta_{2}}e
\end{align*}
where $l_{a}, a=1,2$ is the number of $T$'s in $Z^{\beta_{a}}$'s and $l$ is the number of $T$'s in $Z^{\alpha}$'s.
By the induction hypothesis, the above equation can be written as:
\begin{align*}
\Lb \delta^{l}Z^{\alpha+1}_{i}\mu=\O^{\leq1}_{2}\cdot \delta^{l} Z^{\alpha+1}_{i}\mu+\Psi^{\leq |\alpha|+2}_{0}
\end{align*}
Similar to the estimates derived in the Step 3 in previous section, we can use induction hypothesis and Gronwall's inequality to conclude that $\|Z^{\alpha+1}_{i}\mu\|_{L^{\infty}(\Sigma_{t})}\lesssim_M \delta^{-l}$. 
\end{proof}

\subsection{$L^2$ estimates}
$\Ntop$ will be the total number of derivatives commuted with $\Box_{\widetilde{g}} \psi = 0$. The highest order objects will be of order $\Ntop +1$. In this subsection, based on (B.1) and (B.2),  we will derive $L^2$ estimates on the objects of order $\leq \Ntop$ in terms of the $L^{2}$ norms of $Z^{\alpha+2}_{i}\psi\in\Psi^{|\alpha|+2}_{1-2l}$ with $|\alpha|\leq N_{top}-1$. We start with the following lemma:
\begin{lemma}
For a $\psi$ which vanishes on $\Cb_{0}$, we have
\begin{equation}\label{Calculus Inequality}
\begin{split}
\int_{S_{t,\ub}} \psi^2 &\lesssim \delta \int_{\Sigma_t^{\ub}} (L\psi)^2 +\mu(\Lb \psi)^2, \ \ \int_{\Sigma_{t}^{\ub}} \psi^2 \lesssim \delta^2 \int_{\Sigma_t^{\ub}} (L\psi)^2 +\mu(\Lb \psi)^2.
\end{split}
\end{equation}
\end{lemma}
\begin{proof} Since $\psi(t,\ub,\theta) = \displaystyle \int_{0}^{\ub} T\psi(t,\ub',\theta) d\ub'$, we have
\begin{align*}
\int_{S_{t,\ub}} \psi^2 d\mu_{\slashed{g}}&= \int_{S_{t,\ub}} \big(\int_{0}^{\ub} T\psi(t,\ub',\theta) d\ub' \big)^2 d\mu_{\slashed{g}(t,\ub)}\\ &\lesssim \delta  \int_{S_{t,\ub}} \int_{0}^{\ub} \big(T\psi(t,\ub',\theta)\big)^2 d\ub' d\mu_{\slashed{g}(t,\ub)}\\
&\lesssim\delta  \int_{S_{t,\ub}} \int_{0}^{\ub} \big(T\psi(t,\ub',\theta)\big)^2 d\mu_{\slashed{g}(t,\ub')}d\ub'
\end{align*}
Here we have used the fact: $\sqrt{\det\slashed{g}(t,\ub)}\lesssim \sqrt{\det\slashed{g}(t,\ub')}\lesssim \sqrt{\det\slashed{g}(t,\ub)} $ due to the bound of the second fundamental form $\theta$.
On the other hand,  $(T\psi) \lesssim (L\psi)^2 + \mu^2 (\Lb\psi)^2$ and $\mu \lesssim 1$, the first inequality follows immediately. The second is an immediate consequence of the first one.
\end{proof}
As a corollary, for $k\leq \Ntop-1$, we have
\begin{equation}\label{L 2 estimates on Ri alpha psi}
\sum_{|\beta| \leq k}\int_{S_{t,\ub}}(R_i{}^{\beta}\psi)^2 \lesssim \delta   E_{ \leq k+1}(t,\ub).
\end{equation}

\begin{proposition}\label{Proposition lower order L2}
 For sufficiently small $\delta$, for all $\alpha$ with $|\alpha| \leq \Ntop-1$ and $t\in [-r_0, s^*]$, the $L^2(\Sigma_t^{\ub})$ norms of all the quantities listed below
\begin{equation*}
\slashedL_{Z_i}^{\, \alpha} \chib', \slashedL_{Z_i}^{\, \alpha}  {}^{(Z_j)}\Zb, \slashedL_{Z_i}^{\, \alpha}  {}^{(Z_j)}\slashed{\pi}, Z_i^{\,\alpha+1} y^j, Z_i^{\,\alpha+1} \lambda_j, 
\end{equation*}
are bounded \footnote{The inequlaity is up to a constant depending only on the bootstrap constant $M$.} by 
$\delta^{1/2-l}\int_{-r_{0}}^{t}\mu_{m}^{-1/2}(t')\sqrt{\Eb_{\leq|\alpha|+2}(t',\ub)}dt'$, where $l$ is the number of $T$'s in $Z_i$'s.
\end{proposition}

\begin{proof}
We do induction on the order. When $|\alpha|=0$, the result follows from the estimates in Section 3 and 4. Again, here we only treat the case $l=0$. By assuming the proposition holds for terms with order $\leq |\alpha|$, 
we show it holds for $|\alpha|+1$.

\underline{Step 1 \ Bounds on $\slashedL_{R_i}^{\, \alpha} \chib'$}. \ \  By affording a $\Lb$-derivative,  we have
\begin{equation}\label{b3}
\|\slashedL_{R_i}^{\, \alpha}  \chib'\|_{L^2(\Sigma_t^{\ub})} \lesssim \|\slashedL_{R_i}^{\, \alpha}  \chib'\|_{L^2(\Sigma_{-r_0}^{\ub})} + \int_{-r_0}^t \||\chib'||\slashedL_{R_i}^{\, \alpha}  \chib'| + |\slashedL_{\Lb} \slashedL_{R_i}^{\, \alpha}  \chib'|\|_{L^2(\Sigma_\tau^{\ub})} d\tau.
\end{equation}
We use formula \eqref{b2} to replace $\slashedL_{\Lb} \slashedL_{R_i}^{\, \alpha}  \chib'$ by the  terms with lower orders. Each nonlinear term has at most one factor with order $>\Ninfty$. We bound this factor in $L^2(\Sigma_t)$ and the rest in $L^\infty$. We now indicate briefly how the estimates on the factors involving $e$ and $\alphab'$ work.

For $\alphab'$, since 
\begin{align*}
\alphab'_{AB} = c\frac{d c}{d\rho}\slashed{D}^2_{A,B}\rho + \frac{1}{2}\big[\dfrac{d^2(c^2)}{d\rho^2}-\dfrac{1}{2c^2}\big(\dfrac{d c^2}{d\rho}\big)^2 \big]X_A(\rho)X_B(\rho),
\end{align*}
 in view of the definition of $\Eb(t,\ub)$, for sufficiently small $\delta$, we have 
 \begin{align*}
\|\slashedL_{R_i}^{\alpha} \alphab'_{AB}\|_{L^2(\Sigma_t^{\ub})} \lesssim\sum_{|\alpha|\leq k}\|\ds R^{\alpha+1}_{i}\psi\|_{L^{2}(\Sigma_{t}^{\ub})}\lesssim_M \delta^\frac{1}{2}\mu_{m}^{-1/2}(t) \displaystyle\sqrt{\Eb_{\leq |\alpha|+2}(t,\ub)}.
\end{align*}

For $e$, since $e =c^{-1}\dfrac{d c}{d\rho}\Lb\rho$,  we have
\begin{align*}
\|R_{i}^{\alpha} e\|_{L^2(\Sigma_t^{\ub})} &\lesssim_M \sum_{|\beta|\leq |\alpha|} \left(\delta^{1/2} \|\ds R^{|\beta|-1}_{i}\psi\|_{L^{2}(\Sigma_{t}^{\ub})}+\delta^{1/2} \|\ds R^{|\beta|-1}_{i}\Lb \psi\|_{L^{2}(\Sigma_{t}^{\ub})}\right)\\
&\lesssim_M \delta^{1/2}\mu_{m}^{-1/2}(t) \displaystyle\sqrt{\Eb_{\leq |\alpha|+1}(t,\ub)}.
\end{align*}

By applying Gronwall's inequality to \eqref{b3}, we obtain immediately that
\begin{align*}
\|\slashedL_{R_i}^{\, \alpha} \chib'\|_{L^2(\Sigma_t^{\ub})} \lesssim_M  \delta^{1/2}\int_{-r_{0}}^{t}  \displaystyle\mu_{m}^{-1/2}(\tau)\sqrt{\Eb_{\leq |\alpha|+2}(\tau,\ub)}d\tau
\end{align*}

\underline{Step 2 \ Bounds on $R_i^{\,\alpha+1} y^j$}. \ \ By the computations in the Step 2 of the proof of Proposition 
\ref{L infity estimates on lot}, $R_i^{\,\alpha+1} y$ is a linear combination of the terms such as 
$\slashedLRi^{\beta_1}\chib\cdot \slashedLRi^{\beta_2}\slashed{g}^{-1}\cdot \slashedLRi^{\beta_3} R_j \cdot \slashed{d} R_i^{\beta_4} x^j$ 
or as $\slashedLRi^{\beta_1}\slashed{g}\cdot \slashedLRi^{\beta_2}\slashed{g}^{-1}\cdot \slashedLRi^{\beta_3} R_j \cdot \slashed{d} R_i^{\beta_4} x^j$,
 where $|\beta_1|+|\beta_2|+|\beta_3|+|\beta_4| = |\alpha|$. Similarly, we bound all factors with order $\leq \Ninfty$ by 
 the $L^\infty$ estimates in Proposition \ref{L infity estimates on lot}. By the induction hypothesis, this yields the bound on $R_i^{\,\alpha+1} y^j$ immediately. The estimates for other quantities follow from the estimates of $\chib'$ and $y^{j}$. In this process, the terms like $R_{i}^{\beta}x^{j}$ and the leading term in $\slashed{\mathcal{L}}_{R_{i}}R_{j}$, which can be bounded by a constant $C$ disregarding the order of the derivatives, are bounded in $L^{\infty}$. The rest terms in $\slashed{\mathcal{L}}_{R_{i}}R_{j}$, which depend on $\chib'$ and $y^{j}$ as well as their derivatives, are bounded in $L^{2}$ based on the $L^{2}$ estimates for $\slashed{\mathcal{L}}_{R_{i}}^{\beta}\chib'$.
\end{proof}

We also have $L^2$ estimates for derivatives of $\mu$.
\begin{proposition}\label{Proposition lower order L2 mu}
For sufficiently small $\delta$, for all $\alpha$ with $|\alpha| \leq \Ntop-1$ and $t\in [-r_0, s^*]$, we have
\begin{align*}
\delta^{l}\|Z_{i}^{\alpha+1}\mu\|_{L^{2}(\Sigma_{t})}
\lesssim_M \delta^{l}\|Z_{i}^{\alpha+1}\mu\|_{L^{2}(\Sigma_{-r_{0}}^{\ub})}+\delta^{1/2}\int_{-r_{0}}^{t}
\sqrt{E_{\leq |\alpha|+2}(\tau,\ub)}+\mu_{m}^{-1/2}(\tau)\sqrt{\Eb_{\leq |\alpha|+2}(\tau,\ub)}d\tau.
\end{align*}
\end{proposition}
\begin{proof}
According to the proof of Proposition \ref{prop mu L infty estimates}, we have
\begin{equation*}
 \delta^{l}\Lb|Z^{\alpha+1}_{i}\mu|\lesssim \delta^{l}|Z^{\alpha+1}_{i}m|+\delta^{l}(|e|+|\leftexp{(R_{i})}{\Zb}|) |Z^{\alpha+1}_{i}\mu|+\sum_{|\beta_{1}+\beta_2|\leq|\alpha|}\delta^{l_{1}}|Z_{i}^{\beta_{1}}\mu| \delta^{l_{2}}|R_{i}^{\beta_{2}}e|.
\end{equation*}
Then the result follows in the same way as Proposition 7.2.
\end{proof}

\section{Estimates on top order terms}
The highest possible order of an object in the paper will be $\Ntop+1$. The current section is devoted to the $L^2$ estimates of $\ds R_i^{\,\alpha}\tr\chib$ and $Z_i^{\,\alpha+2} \mu$ with $|\alpha| = \Ntop-1$.

\subsection{Estimates on $\tr\chib$}
We first sketch the idea of the proof. Since we deal with top order terms, we can not use the transport equation \eqref{Structure Equation Lb trchib} directly as in the previous section (which loses one derivative). Roughly speaking, we derive an elliptic system coupled with a transport equation for $\widehat{\chib}$ and $ \slashed{d}\tr\chib$:
\begin{equation*}
\Lb (\slashed{d}R_i^{\alpha}\tr\chib) = \widehat{\chib}\cdot \slashed{\nabla}\mathcal{L}_{R_i}^\alpha \widehat{\chib}+ \cdots, \ \ \slashed{\text{div}} \mathcal{L}_{R_i}^\alpha \widehat{\chib} = \slashed{d}R_i^\alpha\tr\chib + \cdots.
\end{equation*}
The new idea is using elliptic estimates and rewriting the right hand side of the transport equations to avoid the loss of derivatives.

Given $\Box_{\gt} \psi_0 =0$, since $\rho = \psi_0^2$, we can derive a wave equation for $\rho$:
\begin{equation}\label{wave equation for rho}
\Box_g \rho = \frac{d(\log(c))}{d \rho}g^{\mu\nu}\partial_\mu \rho \partial_\nu \rho + 2g^{\mu\nu}\partial_\mu \psi_0 \partial_\nu \psi_0
\end{equation}
Therefore, in the null frame, we can rewrite $\laplacianslash \rho$ as $\slashed{\Delta}\rho = \mu^{-1} \Lb \big(L\rho\big) + \lot$
where $\lot$ represents all the terms with order at most $1$. 
On the other hand, according to the definition of $\alphab_{AB}$, we can rewrite \eqref{Structure Equation Lb trchib} as $
\Lb \tr\chib = -\frac{1}{2}\frac{d (c^{2})}{d \rho}\slashed{\Delta} \rho + \lot$ where the lower order terms $\lot$ standard for terms with order at most $1$. By substituting to the previous expression on $\slashed{\Delta} \rho$, we obtain:
\begin{equation}\label{renormalized equation for tr chib}
\Lb \big(\mu \tr\chib -\check{f} \big) = 2 \Lb\mu \tr\chib -\frac{1}{2}\mu (\tr\chib)^2-\mu |\widehat{\chib}|^2+\check{g},
\end{equation}
where $\check{f} = -\dfrac{1}{2}\dfrac{d(c^2)}{d\rho}L\rho$ and $\check{g}$ is given by
\begin{equation*}
\check{g} = \left( 2 \left(\frac{d(c)}{d\rho}\right)^2 + c\frac{d^2(c)}{d \rho^2} \right)\left(\Lb\rho L\rho-\mu|\slashed{d}\rho|^2\right)+ 2 c\frac{d(c)}{d\rho}\left( \left(L\psi_0 \Lb
\psi_0 -\mu |\slashed{d}\psi_0|^2\right)+ \left(\frac{1}{4}\frac{\mu|\slashed{d}\rho|^2}{c^2}-\zetab^A\slashed{d}_A\rho\right)\right).
\end{equation*}

We observe two main features of \eqref{renormalized equation for tr chib}: the order of the righthand side terms are one less than that of the lefthand side; It is regular in $\mu$, i.e. there is no $\mu^{-1}$ factor. In order to commute $R_i$'s with \eqref{renormalized equation for tr chib} and to control the $\alpha^{\text{th}}$ derivatives of $\slashed{d}(\tr\chib)$, for a given multi-index $\alpha$, we introduce
\begin{equation*}
F_{\alpha}=  \mu\slashed{d}\left(R_i{}^\alpha\tr\chib\big)- \slashed{d}\big(R_i{}^\alpha \check{f}\right).
\end{equation*}

For $\alpha=0$ and ${F}=F_0 = \mu \slashed{d}\tr\chib - \slashed{d}\check{f}$, by commuting $\slashed{d}$ with \eqref{renormalized equation for tr chib}, we obtain
\begin{equation*}
\slashedL_{\Lb} {F} + (\tr\chib-2\mu^{-1}\Lb \mu){F} =\big(-\frac{1}{2}\tr\chib + 2\mu^{-1}\Lb\mu\big)\slashed{d}\check{f}-\mu\slashed{d}\big(|\widehat{\chib}|^2\big) + g_0
\end{equation*}
with $g_0 = \slashed{d}\check{g}-\frac{1}{2}\tr\chib \,\slashed{d}\big(\check{f}-2\Lb\mu\big)-(\slashed{d}\mu)\big(\Lb \tr\chib+|\widehat{\chib}|^2\big)$. Similarly, for $|\alpha|\neq 0$, we first commute $R_i^{\, \alpha}$ and then commute $\slashed{d}$ with \eqref{renormalized equation for tr chib}.  This leads to
\begin{equation}\label{transport equation for F alpha for chib}
\slashedL_{\Lb} F_\alpha + \left(\tr\chib-2\mu^{-1}\Lb \mu\right)F_\alpha =\left(-\frac{1}{2}\tr\chib + 2\mu^{-1}\Lb\mu\right)\slashed{d}\left(R_i{}^\alpha\check{f}\right)
-\mu\slashed{d}\left(R_i{}^{\alpha}\left(|\widehat{\chib}|^2\right)\right) + g_\alpha,
\end{equation}
with $g_\alpha$ in the following schematic expression (by setting all the numerical constants to be $1$):
\begin{align*}
g_\alpha &= \slashedL_{R_i}{}^\alpha g_0 + \sum_{|\beta_1|+|\beta_2| = |\alpha|-1}\!\!\!\!\!\!\!\!\!\slashedL_{R_i}^{\beta_1}\slashedL_{^{(R_i)}\Zb}F_{\beta_2} + \sum_{|\beta_1|+|\beta_2| = |\alpha|-1} \!\!\!\!\!\!\!\!\!\slashedL_{R_i}^{\beta_1} \Big(\big(\mu R_i\tr\chib + R_i\Lb\mu + {}^{(R_i)}\Zb \mu\big)\slashed{d}\big(R_i{}^{\beta_2}\tr\chib\big)\Big)\\
&\quad + \sum_{|\beta_1|+|\beta_2| = |\alpha|-1}\!\!\!\!\!\!\!\!\!\slashedL_{R_i}^{\beta_1} \Big(R_i \mu \big[\slashedL_{\Lb}\slashed{d}\big(R_i{}^{\beta_2}\tr\chib\big)+\tr\chib\slashed{d}\big(R_i{}^{\beta_2}\tr\chib \big) + \slashed{d}\big(R_i{}^{\beta_2}(|\widehat{\chib}|^2)\big)\big] + R_i\tr\chib \slashed{d}\big(R_{i}{}^{\beta_2}\check{f}\big)\Big).
\end{align*}
We remark that $F_\alpha$ is of order $\Ntop+1$ so that $\alpha=\Ntop-1$.

We rewrite \eqref{Structure Equation div chib} as $\divslash{\widehat{\chib}}=\frac{1}{2}\slashed{d}\tr\chib - \big(\mu^{-1}\zetab\cdot \widehat{\chib}-\frac{1}{2}\mu^{-1}\zetab \, \tr\chib\big)$. By commuting $\slashedLhat_{R_i}\!\!{}^\alpha$, we obtain the following schematic expression:
\begin{equation}\label{elliptic equation for L R_i alpha}
\divslash\big  (\slashedLhat_{R_i}{}^\alpha \widehat{\chib}\big) = \frac{1}{2}\slashed{d} \big( R_i{}^{\alpha} \tr\chib \big) +H_\alpha,
\end{equation}
with $\displaystyle H_\alpha =\big(\slashedL_{R_i}+\frac{1}{2}\tr {}^{(R_i)}\slashed{\pi}\big)^\alpha \big(\mu^{-1}\zetab\cdot \widehat{\chib}-\frac{1}{2}\mu^{-1}\zetab \, \tr\chib\big) + \!\!\!\!\!\!\!\!\!\sum_{|\beta_1|+|\beta_2| = |\alpha|-1}\!\!\!\!\!\!\!\!\!\big(\slashedL_{R_i}+\frac{1}{2}\tr {}^{(R_i)}\slashed{\pi}\big)^{\beta_1}\Big(\tr{}^{(R_i)} \slashed{\pi}\cdot \slashed{d}(R_i{}^{\beta_2}\tr\chib) + \big( \divslash {}^{(R_i)}\widehat{\slashed{\pi}}\big)\cdot \slashedLhat_{R_i}{}^{\beta_2} \widehat{\chib}\Big) +\!\!\!\!\!\!\!\!\! \sum_{|\beta_1|+|\beta_2| = |\alpha|-1}\!\!\!\!\!\!\!\!\!\big(\slashedL_{R_i}+\frac{1}{2}\tr {}^{(R_i)}\slashed{\pi}\big)^{\beta_1} \Big({}^{(R_i)}\widehat{\slashed{\pi}}\cdot \nablaslash \slashedLhat_{R_i}{}^{\beta_2} \widehat{\chib} \Big)$.
By applying the elliptic estimates \eqref{elliptic estimates} to \eqref{elliptic equation for L R_i alpha}, in view of the definition of $F_\alpha$, we obtain (the $L^2$ norms are taken at $S(t,\ub)$)
\begin{equation}\label{cc1}
\|\mu\slashed{\nabla} \slashedLhat_{R_i}{}^\alpha \widehat{\chib} \|_{L^2(S_{t,\ub})}\lesssim \|F_\alpha \|_{L^2}+ \|\slashed{d} R_i{}^\alpha \check{f}\|_{L^2}+\|\slashed{d}\mu\|_{L^\infty}\|\slashedLhat_{R_i}{}^\alpha \widehat{\chib} \|_{L^2} +\|\mu H_\alpha\|_{L^2}.
\end{equation}

For any form $\xi$, since $|\xi| \Lb |\xi| = (\xi,\slashedL_{\Lb}\xi) -\xi \cdot\widehat{\chib}\cdot \xi-\dfrac{1}{2}\tr\chib|\xi|^2$, we have $\Lb |\xi| \leq  |\slashedL_{\Lb}\xi| +|\widehat{\chib}||\xi|-\dfrac{1}{2}\tr\chib|\xi|$. Applying this inequality to \eqref{transport equation for F alpha for chib}, we obtain
\begin{equation}\label{eq 11}
\Lb |F_\alpha| \leq (\mu^{-1}\Lb \mu - \frac{3}{2}\tr\chib +|\widehat{\chib}|)|F_\alpha| + \big(2\mu^{-1}|\Lb\mu| -\tr\chib\big)|\slashed{d} R_i{}^\alpha\check{f}|+|\mu\slashed{d} R_i{}^{\alpha}\big(|\widehat{\chib}|^2\big)| + |g_\alpha|.
\end{equation}
To obtain the $L^2(\Sigma_t)$ bound on $F_\alpha$, we first integrate \eqref{eq 11} from $-r_0$ to $t$ and then take the $L^2$ norm on $\Sigma_t$. We claim that we can ignore the first term on the right hand side: The term $\big(|\tr\chib|+|\widehat{\chib}|\big)|F_\alpha|$ can be removed immediately by Gronwall's inequality. For $\mu^{-1}\Lb \mu \cdot |F_\alpha|$, for a fixed $(\ub, \theta)$, if $\mu \geq \frac{1}{10}$ for all $\ub$, it can be also removed by Gronwall's inequality; otherwise, $\mu < \frac{1}{10}$, therefore, according to the argument in Proposition \ref{Proposition C3}, the sign of $\Lb\mu$ is negative so that this term can be ignored. As a result, \eqref{eq 11} yields
\begin{equation*}
\begin{split}
\|F_\alpha \|_{L^2(\Sigma_t)}	
&\lesssim \|F_\alpha \|_{L^2(\Sigma_{-r_0})} + \int_{-r_0}^t \|(\mu^{-1}|\Lb\mu|+1)\slashed{d} R_i{}^\alpha\check{f}\|_{L^2(\Sigma_\tau)} + \| \mu\slashed{d} R_i{}^{\alpha}\big(|\widehat{\chib}|^2\big)\|_{L^2(\Sigma_\tau)}+\| g_\alpha\|_{L^2(\Sigma_\tau)} d\tau \\
&= \|F_\alpha \|_{L^2(\Sigma_{-r_0})}+ I_1+I_2+I_3.
\end{split}
\end{equation*}
where the $I_i$'s are defined in the obvious way.

We first bound $I_2$. According to the Leibniz rule, we have
\begin{align*}
I_2 
&=  \!\!\!\!\!\!\!\!\!\sum_{|\beta_1|+|\beta_2|+|\beta_3| +|\beta_4|= |\alpha|} \!\!\!\int_{-r_0}^t \| \mu\slashed{d}\big(  \slashedLhat_{R_i}{}^{\beta_1}  \slashed{g} \cdot  \slashedLhat_{R_i}{}^{\beta_2}  \slashed{g} \cdot  \slashedLhat_{R_i}{}^{\beta_3} \widehat{\chib}  \cdot \slashedLhat_{R_i}{}^{\beta_4} \widehat{\chib}\big)\|_{L^2(\Sigma_\tau)} d\tau.	
\end{align*}
Therefore, at least three indices of the $\beta_i$'s are at most $\Ninfty$. According to Proposition \ref{Proposition lower order L2}, we have
\begin{align*}
I_2 &\lesssim \int_{-r_0}^t  \| \mu\slashed{d}  \slashedLhat_{R_i}{}^{\alpha} \widehat{\chib}\|_{L^2(\Sigma_\tau)} \|\widehat{\chib}\|_{L^\infty(\Sigma_\tau)}+ \delta^{3/2}  \sqrt{\Eb_{ \leq |\alpha|+2}(\tau,\ub)}d\tau \\
&\stackrel{\eqref{cc1}}{\lesssim} \delta  \int_{-r_0}^t  \|F_\alpha \|_{L^2(\Sigma_\tau)}+ \|\slashed{d} R_i{}^\alpha \check{f}\|_{L^2(\Sigma_\tau)} +  \delta^{1/2}\sqrt{\Eb_{\leq |\alpha|+2}(\tau,\ub)}d\tau.
\end{align*}
Since $\check{f} = - \frac{1}{2} \frac{d c^2}{d\rho} L\rho$
, we have $\displaystyle
\slashed{d} R_i{}^\alpha \check{f} =\frac{d c^2}{d\rho}\psi_0 \slashed{d} R_i{}^\alpha L \psi_0  + \slashed{d}\Big(\sum_{\beta_1+\beta_2 = \alpha, \atop |\beta_1| \geq 1} R_i{}^{\beta_1}\big(\frac{d c^2}{d\rho}\psi_0 \big) R_i{}^{\beta_2} \big( L\psi_0 \big)\Big)$.
This leads to
\begin{align*}
\int_{-r_0}^t \|\slashed{d} R_i{}^\alpha \check{f}\|_{L^2(\Sigma_\tau)}
&\lesssim\int_{-r_0}^t \delta^{\frac{1}{2}} \sqrt{E_{\leq |\alpha|+2}(\tau,\ub)}d\tau.
\end{align*}
Therefore, we have the following bound for $I_2$:
\begin{align}\label{I2 chib}
I_2 \lesssim   \int_{-r_0}^t  \delta\|F_\alpha \|_{L^2(\Sigma_\tau)}+ \delta^{3/2} \sqrt{E_{\leq |\alpha|+2}(\tau,\ub)}
+\delta^{3/2}\sqrt{\Eb_{\leq |\alpha|+2}(\tau,\ub)}d\tau.
\end{align}
We remark that, as long as the terms under consideration are not top order terms, i.e. not a term of order $\Ntop+1$, we can simply use the estimates from previous section to get the estimates. The reason is as follows: each term is a product of $\O_k^{\leq l}$ with $l \leq |\alpha|$. Only one of the factor is of order $l>\Ninfty$. We can bound the rest in $L^\infty$ and the highest order one in $L^2$, thanks to the estimates derived in the previous section.

To bound $I_3$, as we pointed out above, we only have to take care of the terms appearing in $g_\alpha$ whose orders are possibly $\Ntop+1$. The rest of them can be easily bounded in $L^{2}$ by a universal constant times the sum of $\int_{-r_0}^t \delta^{\frac{1}{2}} \sqrt{{E}_{\leq |\alpha|+2}(\tau,\ub)} + \delta^{1/2}\mu_{m}^{-1/2}(t)\sqrt{\Eb_{\leq |\alpha|+2}(\tau,\ub)}d\tau$.

We now investigate the possible highest order terms in $g_{\alpha}$. There are three possiblities: the first one are the terms of the form $\displaystyle\sum_{|\beta_1|+|\beta_2| = |\alpha|-1}\!\!\!\!\!\!\!\!\!\slashedL_{R_i}^{\beta_1}\slashedL_{^{(R_i)}\Zb}F_{\beta_2}$. They can be bounded by $|\leftexp{(R_{i})}{\Zb}||F_{\alpha}|\lesssim \delta^{1/2}|F_{\alpha}|$ provided that $\delta$ is suitably small. The second possiblity is from the $(\Lb\text{tr}\chib+|\chib|^{2})$ term of $g_0$. However, equation \eqref{Structure Equation Lb trchib} says that $\Lb\text{tr}\chib+|\chib|^{2}= e-\text{tr}\alphab'$, so although it is of the highest order, the highest order part constists
only $\nablaslash$ derivatives of $\psi$ (thanks to the expression of $\alpha'$), hence can be bounded in the same way as lower order terms. The last possiblity is from the term $\Lb\psi_{0}L\psi_{0}$ appearing in $\check{g}$. They are of top orders and they can not be converted into terms involving only $\nablaslash$ derivatives. These terms contribute to $g_\alpha$ the terms of the form $\O_{0}^{\leq 1}\cdot L\psi_{0} \cdot \Lb R^{\alpha+1}_{i}\psi_{0}$ in $g_{\alpha}$. 
 \begin{align*}
 \int_{-r_{0}}^{t}\|L\psi_{0}\Lb R^{\alpha+1}_{i}\psi_{0}\|_{L^{2}(\Sigma_{\tau}^{\ub})}d\tau
 &\lesssim \delta^{-1/2}\Big(\int_{-r_{0}}^{t}\int_{0}^{\ub}\int_{S_{\tau,\ub'}}\big(\Lb R^{\alpha'+1}_{i}\psi\big)^{2}
 d\mu_{\slashed{g}}d\ub'd\tau\Big)^{1/2}\\
 & \lesssim \delta^{-1/2}\Big(\int_{0}^{\ub}
 \underline{F}_{\leq |\alpha|+2}(t,\ub')d\ub'\Big)^{1/2}
\end{align*}
Therefore, we have the following estimates for $I_{3}$:
\begin{align}\label{I3 chib}
I_3 \lesssim \int_{-r_0}^t \delta^{\frac{1}{2}} \sqrt{{E}_{\leq |\alpha|+2}(\tau)}+\delta^{1/2}\mu_{m}^{-1/2}(\tau)\sqrt{\Eb_{\leq |\alpha|+2}(t)}d\tau+\delta^{-1/2}\Big(\int_{0}^{\ub}
\Fb_{\leq |\alpha|+2}(t,\ub')d\ub'\Big)^{1/2}
\end{align}
We now study the estimates on $I_{1}$. When $\mu\geq\dfrac{1}{10}$, the estimates are straightforward. To study the case when $\mu\leq\dfrac{1}{10}$, we introduce a few notations (where $a$ is a positive constant):
\begin{equation*}
\begin{split}
 t_{0} &= \inf \big\{\tau \in [-2,t^*) \big| \mu_m (t)<\frac{1}{10}\big\},\ \ M(t)=\!\!\!\!\!\!\!\!\max_{(\ub,\theta), \atop (t,\ub,\theta)\in W_{shock}} \!\!\!\!\!\!\!\!\big|\big(\Lb(\log\mu)\big)_{-}(t,\ub,\theta)\big|, \ \ I_{a}(t)=\int_{t_{0}}^{t}\mu_{m}^{-a}(\tau)M(\tau)d\tau.
\end{split}
\end{equation*}
\begin{lemma}\label{lemma on mu to power a} 
We assume that $a$ is at least $4$.
\begin{itemize}
\item[(1)] For sufficiently large $a$ and for all $t\in[t_{0},t^{*})$, we have
\begin{align}\label{key lemma}
I_{a}(t) \lesssim a^{-1}\mu^{-a}_{m}(t). 
\end{align}

\item[(1')] For sufficiently large $a$ and for all $t\in[t_{0},t^{*})$, we have 
\begin{align*}
\int_{t_{0}}^{t}\mu^{-a-1}_{m}(t')dt'\lesssim\frac{1}{a}\mu_{m}^{-a}(t)
\end{align*}

\item[(2)] For $a\geq 4$ and sufficiently small $\delta$, there is a constant $C_0$ independent of $a$ and $\delta$, so that for all $\tau\in[-r_{0},t]$, we have
\begin{align}\label{mu is decreasing}
\mu^{a}_{m}(t)\leq C_0 \mu^{a}_{m}(\tau)
\end{align}
\end{itemize}
\end{lemma}
\begin{proof}
(1) By Proposition \ref{Proposition C3}, for $t\geq t_{0}$, the minimum of $r_{0}^{2}(\Lb\mu)(-r_{0},\ub,\theta)$ on $[0,\delta]\times\mathbb{S}^{2}$ is negative and we denote it by
\begin{align}\label{min Lbmu initial}
-\eta_{m}=\min_{(\ub,\theta)\in[0,\delta]\times \mathbb{S}^{2}}\{r_{0}^{2}(\Lb\mu)(-r_{0},\ub,\theta)\}.
\end{align}
We notice that $1\leq \eta_{m}\leq C_{m}$ where $C_{m}$ is a constant depending on the initial data. In view of the asymptotic expansion for $(\Lb\mu)(t,\ub,\theta)$ in Lemma \ref{lemma on the expansion of Lb mu}, we have
\begin{align}\label{mu expansion in use}
\begin{split}
\mu(t,\ub,\theta)=&1-\left(\frac{1}{t}+\frac{1}{r_{0}}\right)r_{0}^{2}(\Lb\mu)(-r_{0},\ub,\theta)+O\left(\delta M^{4}\right)\left(\frac{1}{t^{2}}-\frac{1}{r_{0}^{2}}\right).
\end{split}
\end{align}
We fix an $s\in (t_{0}, t^{*})$ in such a way that $t_{0}\leq t<s<t^{*}$. There exists $(\ub_{s},\theta_{s})\in[0,\delta]\times \mathbb{S}^{2}$ and $(\ub_{m},\theta_{m})\in[0,\delta]\times\mathbb{S}^{2}$ so that 
\begin{align}\label{achieve min}
\mu(s,\ub_{s},\theta_{s})=\mu_{m}(s),\quad r_{0}^{2}(\Lb\mu)(-r_{0},\ub_{m},\theta_{m})=-\eta_{m}.
\end{align}
We claim that
\begin{align}\label{mins are close}
\left|\eta_{m}+r_{0}^{2}(\Lb\mu)(-r_{0},\ub_{s},\theta_{s})\right|\leq O(\delta M^{4}).
\end{align}
Indeed, one can apply \eqref{mu expansion in use} to $\mu(s,\ub_{m},\theta_{m})$ and $\mu(s,\ub_{s},\theta_{s})$ to derive
\begin{align}
\begin{split}
\mu(s,\ub_{s},\theta_{s})=&1-\left(\frac{1}{s}+\frac{1}{r_{0}}\right)(-\eta_{m}+d_{ms})+O(\delta M^{4})\left(\frac{1}{t^{2}}-\frac{1}{r_{0}^{2}}\right)\\
\mu(s,\ub_{m},\theta_{m})=&1-\left(\frac{1}{s}+\frac{1}{r_{0}}\right)(-\eta_{m})+O(\delta M^{4})\left(\frac{1}{t^{2}}-\frac{1}{r_{0}^{2}}\right),
\end{split}
\end{align}
where the quantity $d_{ms}>0$ is defined as 
\begin{align}\label{dms}
d_{ms}:=\eta_{m}+r_{0}^{2}(\Lb\mu)(-r_{0},\ub_{s},\theta_{s}).
\end{align}
Since $\mu(s,\ub_{s},\theta_{s})\leq \mu(s,\ub_{m},\theta_{m})$, we have
\begin{align*}
0<-\left(\frac{1}{s}+\frac{1}{r_{0}}\right)d_{ms}\leq O(\delta M^{4})\left(\frac{1}{t^{2}}-\frac{1}{r_{0}^{2}}\right).
\end{align*}
Hence,
\begin{align}\label{dms esti}
d_{ms}\leq O(\delta M^{4}).
\end{align}
The constants in the above inequalities depend on $t_{0}$ therefore on $\eta_{m}$ and  they are absolute constants. With this preparation, one can derive precise upper and lower bounds for $\mu_{m}(t)$. 

We pick up a $(\ub'_{m},\theta'_{m})\in[0,\delta]\times\mathbb{S}^{2}$ in such a way that $\mu(t,\ub'_{m},\theta'_{m})=\mu_{m}(t)$.  For the lower bound, by virtue of Lemma \ref{lemma on the expansion of Lb mu}, we have
\begin{equation}\label{mu m lower bound temp 1}
\begin{split}
\mu_{m}(t)=&\mu(t,\ub'_{m},\theta'_{m})=\mu(s,\ub'_{m},\theta'_{m})+\int_{s}^{t}(\Lb\mu)(t',\ub'_{m},\theta'_{m})dt'\\
\geq &\mu_{m}(s)+\int_{s}^{t}\frac{\eta_{m}}{-t^{\prime2}}+\frac{O(\delta M^{4})}{(-t')^{3}}dt'\\
\geq &\mu_{m}(s)+\left(\eta_{m}-\frac{1}{2a}\right)\left(\frac{1}{t}-\frac{1}{s}\right).
\end{split}
\end{equation}
In the last step, we take sufficiently small $\delta$ so that $O(\delta M^{4})\leq \frac{1}{2a}$. 

For the upper bound, in view of Lemma \ref{lemma on the expansion of Lb mu} and \eqref{dms esti}, we have
\begin{align}\label{mu m upper bound temp 1}
\begin{split}
\mu_{m}(t)\leq& \mu(t,\ub_{s},\theta_{s})=\mu_{m}(s)+\int_{s}^{t}(\Lb\mu)(t',\ub_{s},\theta_{s})dt'\\
=& \mu_{m}(s)+\int_{s}^{t}\frac{\eta_{m}-d_{ms}}{-t^{\prime2}}+\frac{O(\delta M^{4})}{(-t')^{3}}dt'\\
\leq &\mu_{m}(s)+\int_{s}^{t}\frac{\eta_{m}}{-t^{\prime2}}+\frac{O(\delta M^{4})}{t^{\prime2}}dt'\\
\leq &\mu_{m}(s)+\left(\eta_{m}+\frac{1}{2a}\right)\left(\frac{1}{t}-\frac{1}{s}\right)
\end{split}.
\end{align}
In the last step, we also take sufficiently small $\delta$ so that $O(\delta M^{4})\leq \frac{1}{2a}$. 

For $I_{a}(t)$, first of all, we have
\begin{align*}
I_{a}(t)\lesssim& \int_{t_{0}}^{t}\left(\mu_{m}(s)+\left(\eta_{m}-\frac{1}{2a}\right)\left(\frac{1}{t}-\frac{1}{s}\right)\right)^{-a-1}t^{\prime-2}dt'\\
=&\int_{\tau}^{\tau_{0}}\left(\mu_{m}(s)+\left(\eta_{m}-\frac{1}{2a}\right)\left(\tau-\tau_{s}\right)\right)^{-a-1}d\tau'\\
\leq &\frac{1}{\eta_{m}-\frac{1}{2a}}\frac{1}{a}\left(\mu_{m}(s)+\left(\eta_{m}-\frac{1}{2a}\right)\left(\tau-\tau_{s}\right)\right)^{-a}.
\end{align*}
Hence,
\begin{align}\label{la temp 1}
\begin{split}
I_{a}(t) \lesssim &\frac{1}{a}\left(\mu_{m}(s)+\left(\eta_{m}-\frac{1}{2a}\right)\left(\frac{1}{t}-\frac{1}{s}\right)\right)^{-a}\\
\leq&\frac{1}{a}\frac{\left(\mu_{m}(s)+\left(\eta_{m}-\frac{1}{2a}\right)\left(\frac{1}{t}-\frac{1}{s}\right)\right)^{-a}}{\left(\mu_{m}(s)+\left(\eta_{m}+\frac{1}{2a}\right)\left(\frac{1}{t}-\frac{1}{s}\right)\right)^{-a}}\mu^{-a}_{m}(t)\\
\leq &\frac{1}{a}\frac{\left(\left(\eta_{m}-\frac{1}{2a}\right)\left(\frac{1}{t}-\frac{1}{s}\right)\right)^{-a}}{\left(\left(\eta_{m}+\frac{1}{2a}\right)\left(\frac{1}{t}-\frac{1}{s}\right)\right)^{-a}}\mu^{-a}_{m}(t).
\end{split}
\end{align}
Since as $a\rightarrow\infty$, one has
\begin{align*}
\frac{\left(\eta_{m}-\frac{1}{2a}\right)^{-a}}{\left(\eta_{m}+\frac{1}{2a}\right)^{-a}}\rightarrow e^{\frac{1}{\eta_{m}}}.
\end{align*}
The limit is an absolute constant. Therefore, \eqref{la temp 1} yields the proof for part (1) of the lemma. The proof for part (1') is exactly the same.

(2) We start with an easy observation: if $\mu(t,\ub,\theta)\leq 1-\dfrac{1}{a}$, then $\Lb \mu(t,\ub,\theta)\lesssim -a^{-1}$. In fact, we claim that $(\frac{1}{t}+\frac{1}{r_{0}})r_{0}^{2}\Lb\mu(-r_{0},\ub,\theta)\geq \frac{1}{2}a^{-1}$.
Otherwise, for sufficiently small $\delta$ (say $\delta^{1/4}\leq a^{-1}$), according to the expansion for $\mu(t,\ub,\theta)$, i.e. $\mu(t,\ub,\theta)=\mu(-r_{0},\ub,\theta)-(\frac{1}{t}+\frac{1}{r_{0}})r^{2}_{0}\Lb\mu(-r_{0},\ub,\theta)+O(\delta)$, we have
\begin{align*}
\mu(t,\ub,\theta)> 1-\frac{1}{2a}-C\delta\geq 1-\frac{1}{a}.
\end{align*}
which is a contradiction. Therefore Lemma \ref{lemma on the expansion of Lb mu} implies $\Lb \mu(t,\ub,\theta)\lesssim -a^{-1}$. In particular, this observation implies that, if there is a $t' \in[-r_{0},s^{*}]$, so that $\mu_{m}(t')\leq 1-a^{-1}$, then for all $t\geq t'$, we have $\mu_{m}(t)\leq 1-a^{-1}$. This allows us to define a time $t_1$, such that it is the minimum of all such $t'$ with $\mu_{m}(t')\leq 1-a^{-1}$.

We now prove the lemma. If $\tau \leq t_{1}$, since $\mu_m(t) \leq 2$, we have
\begin{align*}
\mu_{m}^{-a}(\tau)\leq (1-\frac{1}{a})^{-a}\leq C_0\leq C_0\mu^{-a}_{m}(t).
\end{align*}
If $\tau \geq t_{1}$, then $\mu_{m}(\tau)\leq 1-\dfrac{1}{a}$. Let $\mu_{m}(\tau)=\mu(\tau,\ub_{\tau},\theta_{\tau})$. We know that $\mu(t,\ub_\tau,\theta_\tau)$ is decreasing in $t$ for $t\geq \tau$. Therefore, we have
\begin{align*}
\mu_{m}(t)\leq \mu(t,\ub_{\tau},\theta_{\tau})\leq \mu(\tau,\ub_{\tau},\theta_{\tau})=\mu_{m}(\tau).
\end{align*}
The proof now is complete.
\end{proof}
For $I_1$, according to the above lemma with $a=b_{|\alpha|+2}$, we then have
\begin{align}\label{I1 chib}
\begin{split}
 I_1 
 &\lesssim \delta^{1/2}\sum_{\beta \leq |\alpha|}\int_{-r_0}^t \||\Lb\big(\log\mu\big)| R_i{}^{\beta+1}L\psi\|_{L^2(\Sigma_\tau)}\\
 &\lesssim \delta^{1/2} I_{b_{|\alpha|+2}}(t) \sqrt{\widetilde{E}_{\leq|\alpha|+2}(t,\ub)}\\
 &\lesssim\delta^{1/2}\mu_{m}^{-b_{|\alpha|+2}}(t)\sqrt{\widetilde{E}_{\leq|\alpha|+2}(t,\ub)}.
 \end{split}
\end{align}

Finally, the estimates \eqref{I1 chib}, \eqref{I2 chib}, \eqref{I3 chib} on $I_1$, $I_2$ and $I_3$ together yield
\begin{align}\label{L2 Falpha}
\begin{split}
\|F_{\alpha}\|_{L^{2}(\Sigma_{t})} \lesssim&\|F_{\alpha}\|_{L^{2}(\Sigma_{-r_{0}})}+
\int_{-r_{0}}^{t}\delta^{1/2}\mu_{m}^{-b_{|\alpha|+2}}(\tau)\Big( \sqrt{\Et_{\leq |\alpha|+2}(\tau,\ub)}+\mu^{-1/2}_{m}(\tau)\sqrt{\Ebt_{\leq |\alpha|+2}(\tau,\ub)} \big)d\tau\\
&+\delta^{-\frac{1}{2}}\mu_{m}^{-b_{|\alpha|+2}}(t)\sqrt{\int_{0}^{\ub}\Fbt_{\leq |\alpha|+2}(t,\ub')d\ub'}
+\delta^{1/2} \mu_{m}^{-b_{|\alpha|+2}}(t)\sqrt{\Et_{\leq|\alpha|+2}(t,\ub)}.
\end{split}
\end{align}
This implies
\begin{equation}\label{top order chib}
\begin{split}
\|\mu&\ds(R^{\alpha}_{i}\text{tr}\chib)\|_{L^{2}(\Sigma_{t})} \lesssim \delta^{1/2}\mu_{m}^{-b_{|\alpha|+2}}(t)\sqrt{\Et_{\leq |\alpha|+2}(t,\ub)}\\
&+\int_{-r_{0}}^{t}\delta^{1/2}\mu^{-b_{|\alpha|+2}-1/2}_{m}(\tau)\sqrt{\Ebt_{\leq |\alpha|+2}(\tau,\ub)}d\tau +\delta^{-1/2}\mu^{-b_{|\alpha|+2}}_{m}(t)\sqrt{\int_{0}^{\ub}
\Fbt_{\leq|\alpha|+2}(t,\ub')d\ub'}.
\end{split}
\end{equation}

\subsection{Estimates on $\mu$}
The top order estimates on $\mu$ depend on the equation $\Lb\mu=m+\mu e$. Since $m=-\frac{1}{2}\frac{d(c^{2})}{d\rho}T\rho$ and $e\mu=\frac{1}{2c^{2}}\frac{d(c^{2})}{d\rho}\Lb\rho\cdot\mu$, it is visible that $\mu$ can be bounded by the total energy  on $\psi$, i.e. the $E_{k}$'s. However, to avoid loss of derivatives, we should not integrate $\Lb \mu$ directly.

In view of the following commutation formulas,
\begin{equation*}
 \begin{split}
[\Lb,\slashed{\Delta}]\phi+\text{tr}\chib\slashed{\Delta}\phi &=-2\widehat{\chib}\cdot\widehat{\slashed{D}}^{2}\phi
-2\slashed{\text{div}}\widehat{\chib}\cdot\slashed{d}\phi,\\
[T,\slashed{\Delta}]\phi+c^{-1}\mu\text{tr}\theta
\slashed{\Delta}\phi&=-2c^{-1}\mu\widehat{\theta}
\cdot\widehat{\slashed{D}}^{2}\phi-2\slashed{\text{div}}(c^{-1}\mu\widehat{\theta})\cdot\slashed{d}\phi,
\end{split}
\end{equation*}
we have
\begin{align}\label{equaiton for Lb Laplacian mu}
\begin{split}
\Lb \slashed{\Delta} \mu = &-\dfrac{1}{2}\dfrac{d c^2}{d\rho}\slashed{\Delta} T\rho + \mu \slashed{\Delta} e + e \slashed{\Delta} \mu\\+ &\ds\mu\cdot\ds e-\tr\chib\slashed{\Delta}\mu-2\widehat{\chib}
\cdot\widehat{\slashed{D}}^{2}\mu
-2\slashed{\text{div}}\widehat{\chib}\cdot\ds\mu.
\end{split}
\end{align}
According to \eqref{wave equation for rho}, 
\begin{align*}
\Box_{g}\rho=\frac{d\log(c)}{d\rho}\left(\mu^{-1}\Lb\rho L\rho+\ds\rho\cdot\ds\rho\right)+2\mu^{-1}L\psi_{0}\Lb\psi_{0}
+2\ds\psi_{0}\cdot\ds\psi_{0}.
\end{align*}
 Therefore, by multiplying $\mu$, we have 
\begin{align*}
\mu \slashed{\Delta} \rho = \Lb(L \rho) + \frac{1}{2}\Lb\rho\tr\chi+\frac{1}{2}L\rho\tr\chib +\frac{d\log(c)}{d\rho}\left(\Lb\rho L\rho+\mu\ds\rho\cdot\ds\rho\right)+2L\psi_{0}\Lb\psi_{0}
+2\mu\ds\psi_{0}\cdot\ds\psi_{0}.
\end{align*} 
  We commute $T$ and obtain
  \begin{align}\label{Laplacian Trho}
  \begin{split}
\mu \slashed{\Delta} T \rho = &\Lb(T L \rho) + \frac{1}{2}\Lb\rho(T\tr\chi)+\frac{1}{2}L\rho(T\tr\chib)\\+&\{T,\Lb]L\rho+\frac{1}{2}\tr\chi T\Lb\rho+\frac{1}{2}TL\rho\tr\chib\\
+&T\left(\frac{d\log(c)}{d\rho}\left(\Lb\rho L\rho+\mu\ds\rho\cdot\ds\rho\right)
+2L\psi_{0}\Lb\psi_{0}
+2\mu\ds\psi_{0}\cdot\ds\psi_{0}\right)\\
+&c^{-1}\mu\tr\theta\slashed{\Delta}\rho
+2c^{-1}\mu\widehat{\theta}\cdot\widehat{\slashed{D}}^{2}\rho
+2\slashed{\text{div}}(c^{-1}\mu\widehat{\theta})\cdot\ds\rho
-(T\mu)\slashed{\Delta}\rho.
\end{split}
\end{align}
Therefore
\begin{align}\label{Laplacian m}
\begin{split}
-\frac{1}{2}\frac{dc^{2}}{d\rho}\slashed{\Delta}T\rho=&\Lb\left(-\frac{1}{2}\frac{dc^{2}}{d\rho}TL\rho\right)-\frac{1}{2}\frac{dc^{2}}{d\rho}\left(\frac{1}{2}\Lb\rho(T\tr\chi)+\frac{1}{2}L\rho(T\tr\chib)\right)\\
-&\frac{1}{2}\frac{dc^{2}}{d\rho}\left(\{T,\Lb]L\rho+\frac{1}{2}\tr\chi T\Lb\rho+\frac{1}{2}TL\rho\tr\chib\right)\\
-&\frac{1}{2}\frac{dc^{2}}{d\rho}T\left(\frac{d\log(c)}{d\rho}\left(\Lb\rho L\rho+\mu\ds\rho\cdot\ds\rho\right)
+2L\psi_{0}\Lb\psi_{0}
+2\mu\ds\psi_{0}\cdot\ds\psi_{0}\right)\\
-&\frac{1}{2}\frac{dc^{2}}{d\rho}\left(c^{-1}\mu\tr\theta\slashed{\Delta}\rho
+2c^{-1}\mu\widehat{\theta}\cdot\widehat{\slashed{D}}^{2}\rho
+2\slashed{\text{div}}(c^{-1}\mu\widehat{\theta})\cdot\ds\rho
-(T\mu)\slashed{\Delta}\rho\right)\\+&\Lb\left(\frac{1}{2}\frac{dc^{2}}{d\rho}\right)TL\rho.
\end{split}
\end{align}
In view of the commutator formula, we also have 
\begin{align}\label{Laplacian e}
\begin{split}
\mu^{2}\slashed{\Delta}e=&\Lb\left(\frac{\mu^{2}}{2c^{2}}\frac{dc^{2}}{d\rho}\slashed{\Delta}\rho\right)+\frac{\mu^{2}}{c^{2}}\frac{dc^{2}}{d\rho}\left(\underline{\chi}\cdot\slashed{D}^{2}\rho
+\slashed{\textrm{div}}\hat{\underline{\chi}}
\cdot\slashed{d}\rho\right)-\Lb\left(\frac{\mu^{2}}{2c^{2}}\frac{dc^{2}}{d\rho}\right)\slashed{\Delta}\rho\\
+&\mu^{2}\frac{d}{d\rho}\left(\frac{1}{c^{2}}\frac{dc^{2}}{d\rho}\right)\slashed{d}\rho\cdot\slashed{d}\underline{L}\rho+
\mu^{2}\left(\frac{d}{d\rho}\left(\frac{1}{2c^{2}}\frac{dc^{2}}{d\rho}\right)\slashed{\Delta}\rho+\frac{d^{2}}{d\rho^{2}}\left(\frac{1}{2c^{2}}\frac{dc^{2}}{d\rho}\right)|\slashed{d}\rho|^{2}\right)\underline{L}\rho.
\end{split}
\end{align} 
Let us define 
\begin{align}\label{modified top mu}
\check{f}':=-\frac{1}{2}\frac{dc^{2}}{d\rho}TL\rho
+\frac{\mu^{2}}{2c^{2}}\frac{dc^{2}}{d\rho}\slashed{\Delta}\rho,\quad F':=\mu\slashed{\Delta}\mu-\check{f}'.
\end{align}
Then in view of \eqref{equaiton for Lb Laplacian mu}, \eqref{Laplacian m} and \eqref{Laplacian e}, we obtain the following propagation equation for $F'$:
\begin{align}\label{transport equation for Lb laplaican mu}
\underline{L}F^{\prime}
+(\textrm{tr}\underline{\chi}
-2\mu^{-1}\underline{L}\mu)F^{\prime}=-(\frac{1}{2}\textrm{tr}\underline{\chi}
-2\mu^{-1}\underline{L}\mu)\check{f}^{\prime}
-2\mu\hat{\underline{\chi}}\cdot\hat{\slashed{D}}^{2}\mu
+\check{g}^{\prime}.
\end{align}
where
\begin{align}\label{g prime check}
\check{g}^{\prime}=\left(-\ds\mu+\frac{\mu}{c^{2}}\frac{dc^{2}}{d\rho}\ds\rho\right)\cdot(\mu\ds\tr\chib)+\Psi^{\leq2}_{\geq-2}
+\O^{\leq1}_{0}\Psi^{\leq2}_{\geq-2}+\Psi^{\leq2}_{\geq0}.
\end{align}
Here we have used the structure equation \eqref{Structure Equation T chib} to cancel the contribution from the term $\dfrac{1}{2}\Lb\rho(T\tr\chi)+\dfrac{1}{2}L\rho(T\tr\chib)$ in \eqref{Laplacian Trho} and the term $(\Lb\mu)\slashed{\Delta}\mu$ when we write $\Lb(\mu\slashed{\Delta}\mu)
=\mu\Lb(\slashed{\Delta}\mu)+\Lb(\mu)\slashed{\Delta}\mu$. We also remark that the $L^2$ norm of all derivatives on $\slashed{\textrm{div}}\chibh$ has been estimated from previous subsection. In such a sense, it can also be considered as a $\Psi_{2}^{\leq 2}$ term and we use \eqref{Structure Equation div chib} to replace $\divslash \chib$ by $\slashed{d}\tr\chib + \cdots$. The term $\Psi_{\geq-2}^{\leq2}$ comes from the contribution of $L\psi_{0}\Lb\psi_{0}$ and $\O^{\leq1}_{0}\Psi^{\leq2}_{\geq-2}$ comes from $-\dfrac{1}{4}\dfrac{dc^{2}}{d\rho}TL\rho\tr\chib$ in \eqref{Laplacian m}. Since we already applied $T$ to $L\psi_{0}\Lb\psi_{0}$ once in \eqref{Laplacian m}, instead of using flux $\Fb(t,\ub)$ as we did in the last subsection, we only need to use the energy $E(t,\ub)$ to control the contribution of this term.

We set $F'_{\alpha,l} = \mu R_i{}^{\alpha'} T^{l}\slashed{\Delta} \mu - R_i{}^{\alpha'} T^{l}\check{f'}$ and $|\alpha'|+l=|\alpha|$. According to \eqref{transport equation for Lb laplaican mu}, we have
\begin{equation}\label{transport equation for F alpha}
\Lb F'_{\alpha,l} + \left(\tr \chib - 2\mu^{-1}\Lb \mu\right)F'_{\alpha,l} = -\left(\frac{1}{2}\tr\chib - 2\mu^{-1}\Lb\mu\right)R_i{}^{\alpha'}T^{l} \check{f}'-2\mu \widehat{\chib}\cdot \slashed{\mathcal{L}}_{R_{i}}^{\alpha'}
\slashed{\mathcal{L}}_{T}^{l}\widehat{\slashed{D}^2} \mu + \check{g}'_{\alpha',l}.
\end{equation}
where $\check{g}'_{\alpha',l}$ is given by
\begin{equation*}
\begin{split}
\check{g}'_{\alpha',l} = &\left(-\slashed{d}\mu + \frac{\mu}{c^2}\frac{d c^2}{d\rho} \slashed{d}\rho\right)\cdot\mu \slashed{d}\left(R_i{}^{\alpha'}T^{l}\tr\chib\right) + l\cdot\Lambda F'_{\alpha,l-1} +\,^{(R_i)} \Zb F'_{\alpha-1,l}\\
+&\O^{\leq|\alpha'|+1}_{-k}\O^{\leq|\alpha|-|\alpha'|+1}_{\geq2-2l+k}+\Psi_{\geq -2l-2}^{\leq |\alpha|+2}+\O^{\leq|\alpha'|+1}_{-k}\Psi^{\leq|\alpha|+2-|\alpha'|}_{\geq-2l-2+k}+\Psi_{\geq-2l}^{\leq|\alpha|+2}.
\end{split}
\end{equation*}
We remark that the second term of $\check{g}'_{\alpha',l}$ vanishes when $l=0$ and $\Lambda =[\Lb, T]$.
According to \eqref{transport equation for F alpha}, we have
\begin{align*}
\|F'_{\alpha,l} \|_{L^2(\Sigma_t)}&\lesssim \|F'_{\alpha,l} \|_{L^2(\Sigma_{-r_0})} + \int_{-r_0}^t \|\mu^{-1} \Lb\mu\|_{L^{\infty}(\Sigma_\tau)}\| R_i{}^{\alpha'}T^{l}\check{f}'\|_{L^2(\Sigma_\tau)} d\tau\\
&\ \ +\int_{-r_0}^t \| \mu \widehat{\chib}\cdot \slashedLRi^{\alpha'}\slashed{\mathcal{L}}_{T}^{l}\,\widehat{\slashed{D}^2} \mu \|_{L^2(\Sigma_{\tau})} d\tau +\int_{-r_0}^t \| \check{g}'_{\alpha',l}\|_{L^2(\Sigma_{\tau})} d\tau\\
&= \|F'_\alpha \|_{L^2(\Sigma_{-r_0})}+ I_1+I_2+I_3.
\end{align*}
We remark that we must multiply both sides $\delta^{l}$ to get the correct estimates.
We first deal with $I_1$.  Since $\check{f}' = \frac{1}{2} \frac{d c^2}{d \rho}\left(-T L\rho + \frac{\mu^2}{c^2}\slashed{\Delta} \rho\right)$, therefore, we have
\begin{equation*}
R_i{}^{\alpha'}T^{l} \check{f} = -\frac{d c^2}{d \rho}\psi_0 R_i{}^{\alpha'} T^{l+1} L\psi_0 + \frac{\mu^2}{2 c^2}\frac{d c^2}{d \rho}\psi_0 R^{\alpha'}_{i}T^{l}\slashed{\Delta} \psi_0 +\Psi^{\leq |\alpha|+2}_{\geq -2(l+2)}.
\end{equation*}
Compared to the first two terms, the last term on the right hand side above is of lower order with respect to the order of derivatives. Hence,
\begin{align}\label{L2 check f'}
\delta^{l+1}\|R_i{}^{\alpha'}T^{l} \check{f}'\|_{L^{2}(\Sigma_{\tau})}
\lesssim\delta^{1/2}\sqrt{E_{\leq |\alpha|+2}(\tau,\ub)}+\delta^{3/2}\sqrt{\Eb_{\leq |\alpha|+2}(\tau,\ub)} 
\end{align}
This together with Lemma \ref{lemma on mu to power a} yields
\begin{equation}\label{I1}
\delta^{l+1}I_1 \lesssim \mu^{-b_{|\alpha|+2}}_{m}(t)\delta^{1/2}\left(\sqrt{\Et_{\leq |\alpha|+2}(t,\ub)} +\delta\sqrt{\Ebt_{\leq |\alpha|+2}(t,\ub)}\right). 
\end{equation}

For $I_2$, we can use elliptic estimates, i.e. to bound $\widehat{\slashed{D}^2} \mu$ by $\slashed{\Delta} \mu$. This leads to
\begin{align}\label{I2}
\begin{split}
\delta^{l+1}I_2  &\lesssim \delta^{l+2}\int_{-r_0}^t \| \mu R_i{}^{\alpha'} T^{l} \slashed{\Delta} \mu \|_{L^2(\Sigma_{\tau})}d\tau\\
& + \int_{-r_0}^t \delta^{3/2}\sqrt{E_{\leq|\alpha|+2}(\tau,\ub)}+\delta^{3/2}\mu^{-1/2}_{m}(\tau)
 \sqrt{\Eb_{\leq|\alpha|+2}(\tau,\ub)}d\tau\\
 &\lesssim \delta\int_{-r_0}^t \delta^{l+1}\| F_{\alpha,l}\|_{L^2(\Sigma_{\tau})} +\delta^{l+1}\| R_i{}^{\alpha'}T^{l}\check{f}\|_{L^2(\Sigma_\tau)}\\
 &+\delta^{1/2}\left(\sqrt{E_{\leq |\alpha|+2}(\tau,\ub)}+\mu^{-1/2}_{m}(\tau)
 \sqrt{\Eb_{\leq|\alpha|+2}(\tau,\ub)}\right)d\tau.
 \end{split}
\end{align}
We can skip the first two terms: The second term is already controlled by $I_1$. While the first term will be eventually absorbed by Gronwall's inequality. The last two terms come from the commutator between $\slashed{\Delta}$ and $T, R_{i}$ as well as using Proposition \ref{Proposition lower order L2 mu}.

For $I_3:=I'_{3}+I''_{3}$, we first consider the contributions from the first line in the expression of $\check{g}'_{\alpha',l}$, which are denoted by $I'_{3}$. We consider the cases $l=0$ and $l>0$ separately. For $l=0$, we have
\begin{align*}
\delta I'_3 &\lesssim \delta \int_{-r_0}^t \| \left(-\slashed{d}\mu + \frac{\mu}{c^2}\frac{d c^2}{d\rho} \slashed{d}\rho\right)\left(\mu \slashed{d}\left(R_i{}^\alpha\tr\chib\right) \right) +\left(|\,^{(R_i)} \Zb|+|\Lambda| \right)\slashed{d}F'_{|\alpha|-1} \|_{L^2(\Sigma_{\tau})}d\tau\\
&\lesssim  \delta\int_{-r_0}^t \|\mu \slashed{d}\left(R_i{}^\alpha\tr\chib\right) \|_{L^2(\Sigma_{\tau})} +\delta\|\slashed{d} F'_{|\alpha|-1} \|_{L^2(\Sigma_{\tau})} d\tau\\
& +\int_{-r_{0}}^{t}\left(\delta^{3/2}\sqrt{E_{\leq |\alpha|+2}(\tau,\ub)}+\delta^{3/2}\mu^{-1/2}_{m}(\tau)\sqrt{\Eb_{\leq |\alpha|+2}(\tau,\ub)}\right)d\tau
\end{align*}
The first two terms are bound by the top order estimates on $\tr\chib$ in the previous subsection, therefore, for $l=0$, we have
\begin{align}\label{I'3 l=0}
\begin{split}
\delta I'_3 \lesssim &\int_{-r_{0}}^{t} \delta^{3/2} \mu_{m}^{-b_{|\alpha|+2}}(\tau)\sqrt{\Et_{\leq |\alpha|+2}(\tau,\ub)}+\delta^{3/2}\mu_{m}^{-b_{|\alpha|+2}-1/2}(\tau)\sqrt{\Ebt_{\leq |\alpha|+2}(\tau,\ub)}d\tau\\
+&\delta^{1/2}\int_{-r_{0}}^{t}\mu_{m}^{-b_{|\alpha|+2}}(\tau)\sqrt{\int_{0}^{\ub}
\Fbt_{\leq |\alpha|+2}(\tau,\ub')d\ub'}d\tau.
\end{split}
\end{align}

For the $l\geq1$ case, we use \eqref{Structure Equation T chib}, i.e. $\slashed{\mathcal{L}}_{T}\chib=
\slashed{\nabla}\widehat{\otimes}\etab
+\mu^{-1}\zetab\widehat{\otimes}\etab-c^{-1}\Lb
(c^{-1}\mu)\chib +c^{-1}\mu\theta\widehat{\otimes}\chib$
to rewrite $T\tr\chib$. By taking the trace in \eqref{Structure Equation T chib}, we obtain
\begin{align*}
\mu\slashed{d}(R_{i}^{\alpha'-1}T^{l}\text{tr}\chib)=\mu R_{i}^{\alpha'}T^{l-1}\slashed{\Delta}\mu+\O_{\geq -2l+2}^{\leq |\alpha|+1},
\end{align*}
where
\begin{align*}
\delta^{l+1}\|\O_{\geq -2l+2}^{\leq |\alpha|+1}\|_{L^{2}(\Sigma_{\tau})}\lesssim\delta^{3/2}\sqrt{{E}_{\leq|\alpha|+2}(t',\ub)}+\delta^{3/2}\mu_{m}^{-1/2}(\tau)\sqrt{\Eb_{\leq|\alpha|+2}(\tau,\ub)}.
\end{align*}
We then conclude that (for $l\geq1$)
\begin{align}\label{I'3}
\begin{split}
\delta^{l+1}I'_{3}
\lesssim \int_{-r_{0}}^{t}\delta^{3/2}\mu_{m}^{-b_{|\alpha|+2}}(\tau)\sqrt{\Et_{\leq |\alpha|+2}(\tau,\ub)}+\delta^{3/2}\mu_{m}^{-b_{|\alpha|+2}-\frac{1}{2}}(\tau)\sqrt{\Ebt_{\leq|\alpha|+2}(\tau,\ub)}d\tau.	
\end{split}
\end{align}
Now we discuss the contributions from the last three terms in the expression of $\check{g}'_{\alpha',l}$, which are denoted by $I''_{3}$. In view of Proposition \ref{Proposition lower order L2} and \ref{Proposition lower order L2 mu}, the terms $\O^{\leq|\alpha'|+1}_{-k}\O^{\leq|\alpha|-|\alpha'|+1}_{2-k}$ can be absorbed by $\O^{\leq|\alpha'|+1}_{-k}\Psi^{\leq|\alpha|+2-|\alpha'|}_{\geq-2l-2+k}$. We can bound the last three terms of $\check{g}'_{\alpha',l}$ as follows
\begin{align}\label{last three terms in checkg'alpha}
\begin{split}
\delta^{l+1}\|\Psi_{\geq -2l-2}^{\leq |\alpha|+2}\|_{L^{2}(\Sigma_{t})}\lesssim&\delta^{1/2}\sqrt{{E}_{\leq |\alpha|+2}(t,\ub)},\\
\delta^{l+1}\|\O^{\leq|\alpha'|+1}_{-k}\Psi^{\leq|\alpha|+2-|\alpha'|}_{\geq-2l-2+k}\|_{L^{2}(\Sigma_{t}^{\ub})}\lesssim&\delta^{1/2}\sqrt{E_{\leq|\alpha|+2}(t,\ub)}+\delta^{1/2}\int_{-r_{0}}^{t}\mu_{m}^{-1/2}(t')\sqrt{\Eb_{\leq|\alpha|+2}(t',\ub)}dt',\\
 \delta^{l+1}\|\Psi_{\geq -2l}^{\leq |\alpha|+2}\|_{L^{2}(\Sigma_{t})}\lesssim &\delta^{3/2}\sqrt{{E}_{\leq |\alpha|+2}(t,\ub)}+\delta^{3/2}\mu_{m}^{-1/2}(t)\sqrt{\Eb_{\leq|\alpha|+2}(t,\ub)}.
 \end{split}
\end{align}
Therefore we have the following estimates for $I''_{3}$:
\begin{align}\label{I''3}
\delta^{l+1}I''_{3}\lesssim\delta^{1/2}\int_{-r_{0}}^{t}\mu_{m}^{-b_{|\alpha|+2}}(t')\sqrt{\Et_{\leq|\alpha|+2}(t',\ub)}dt'+\delta^{1/2}\int_{-r_{0}}^{t}\mu^{-b_{|\alpha|+2}-1/2}_{m}(t')\sqrt{\Ebt_{\leq|\alpha|+2}(t',\ub)}dt'.
\end{align}
By combining the estimate \eqref{L2 check f'} for $\check{f}'$ and the estimates for \eqref{I1}, \eqref{I2}, \eqref{I'3 l=0}, 
\eqref{I'3} and \eqref{I''3} for $I_{1}, I_{2}, I_{3}$, we obtain 
\begin{equation}\label{top order mu}
\begin{split}
 \delta^{l+1}\|R^{\alpha'}_{i}T^{l}\slashed{\Delta}\mu\|_{L^{2}(\Sigma_{t})}&\lesssim \delta^{l+1}\|F_{\alpha,l}\|_{L^{2}(\Sigma_{-r_0})}\\
& +\delta^{1/2}\mu_{m}^{-b_{|\alpha|+2}}(t)\left(\sqrt{\Et_{\leq |\alpha|+2}(t,\ub)}+\mu_{m}^{-b_{|\alpha|+2}}(t)\sqrt{\Ebt_{\leq|\alpha|+2}(t,\ub)}\right) \\
& + \int_{-r_{0}}^{t}\delta^{1/2}\mu_{m}^{-b_{|\alpha|+2}}(\tau)\sqrt{\Et_{\leq |\alpha|+2}(\tau,\ub)}\\
&+\delta^{1/2}\mu_{m}^{-b_{|\alpha|+2}-1/2}(\tau)\sqrt{\Ebt_{\leq|\alpha|+2}(\tau,\ub)}d\tau\\
& +\delta^{1/2}\int_{-r_{0}}^{t}\mu_{m}^{-b_{|\alpha|+2}}(\tau)\sqrt{\int_{0}^{\ub}
\Fbt_{\leq |\alpha|+2}(\tau,\ub')d\ub'}d\tau.
\end{split}
\end{equation}

\section{Commutator estimates}
In this section, we shall estimate the error spacetime integrals for the contributions of commutators.

Let $\psi$ be a solution of the inhomogeneous wave equation $\Box_{\gt} \psi = \rho$ and $Z$ be a vector field, one can commute $Z$ with the equation to derive
\begin{equation}\label{commute an arbitary vector field with inhomogeneous wave equation}
\Box_{\gt} \left( Z \psi \right) = Z \rho + \frac{1}{2}\tr_{\gt}{}^{(Z)}\widetilde{\pi} \cdot \rho + c^2 \text{div}_{g} \, {}^{(Z)} J
\end{equation}
where the vector field ${}^{(Z)}J$ is defined by
\begin{equation*}
{}^{(Z)}J^{\mu} = \left( {}^{(Z)}\widetilde{\pi}^{\mu\nu} - \frac{1}{2}g^{\mu\nu} \tr_{g}{}^{(Z)}\widetilde{\pi} \right)\partial_\nu \psi.
\end{equation*}
We remark that the raising indices for ${}^{(Z)}\widetilde{\pi}^{\mu\nu}$ are with respect to the optic metric $g$.

In applications, we use the above formulas for homogeneous wave equations $\Box_{\gt} \psi =0$ and commute some commutation vector fields $Z_i$'s several times. Therefore, we need the following recursion formulas:
\begin{equation}\label{commute vector fields with equations}
\begin{split}
\Box_{\gt} \psi_n &= \rho_n. \ \ \psi_n = Z \psi_{n-1}, \ \ \rho_1 = 0,\\
\rho_n &= Z \rho_{n-1} + \frac{1}{2}\tr_{\gt}{}^{(Z)}\widetilde{\pi} \cdot \rho_{n-1} + c^2 \text{div}_{g} \, {}^{(Z)} J_{n-1},\\
{}^{(Z)}J_{n-1}^{\mu}&=\left( {}^{(Z)}\widetilde{\pi}^{\mu\nu} - \frac{1}{2}g^{\mu\nu} \tr_{g}{}^{(Z)}\widetilde{\pi} \right)\partial_\nu \psi_{n-1}.
\end{split}
\end{equation}

\begin{remark}
When we derive energy estimates for $\Box_{\gt} \psi_n = \rho_n$, due to the volume form of the conformal optic metric $\gt$, the integrands $\widetilde{\rho}_n$ appearing in the error terms is slightly different from $\rho_n$. The rescaled source terms $\widetilde{\rho}_n$ are defined as follows:
\begin{equation}\label{rescaled source terms}
\begin{split}
\widetilde{\rho}_n &= \frac{1}{c^2}\mu \rho_n = Z \widetilde{\rho}_{n-1} +{}^{(Z)} \delta \cdot \widetilde{\rho}_{n-1} + {}^{(Z)} \sigma_{n-1}, \\
 \widetilde{\rho}_1 &=0, \ \ {}^{(Z)} \sigma_{n-1} = \mu\cdot \text{div}_{g} \, {}^{(Z)} J_{n-1}, \ \  {}^{(Z)} \delta =\frac{1}{2}\tr_{\gt}{}^{(Z)}\widetilde{\pi}-\mu^{-1}Z\mu +2Z\left(\log(c)\right).
\end{split}
\end{equation}
In view of \eqref{deformation tensor of Q}, \eqref{deformation tensor of Q}, \eqref{deformation tensor of Rotational R_i in T,Lb,X_A} as well as the formula
\begin{align*}
\widetilde{\tr}{}^{(Z)}\widetilde{\pi}=c\tr{}^{(Z)}\widetilde{\pi}
=c\left(-\mu^{-1}{}^{(Z)}{\widetilde{\pi}}_{L\Lb}+\tr{}^{(Z)}\widetilde{\slashed{\pi}}\right)
\end{align*}
we have:
\begin{align}\label{boundedness for delta tensor}
\left|\leftexp{(T)}{\delta}\right|\lesssim 1,\quad \left|\leftexp{(Q)}{\delta}\right|\lesssim 1,\quad \left|\leftexp{(R_{j})}{\delta}\right|\lesssim \delta.
\end{align}
\end{remark}
Then error spacetime integral corresponding to $K_{0}=L$ and $K_{1}=\Lb$ containing $\rho_{n}$ are as follows:
\begin{align*}
-\int_{W^{t}_{\ub}}\frac{1}{c^{2}}\rho_{n}L\psi_{n}d\mu_{g}
=-\int_{W^{t}_{\ub}}\widetilde{\rho}_{n}L\psi_{n}dtd\ub d\mu_{\slashed{g}}\\
-\int_{W^{t}_{\ub}}\frac{1}{c^{2}}\rho_{n}\Lb\psi_{n}d\mu_{g}
=-\int_{W^{t}_{\ub}}\widetilde{\rho}_{n}\Lb\psi_{n}dtd\ub d\mu_{\slashed{g}}
\end{align*}

We first consider the contribution of $\leftexp{(Z)}{\sigma}_{n-1}$ in $\widetilde{\rho}_{n}$. We write $\leftexp{(Z)}{\sigma}_{n-1}$ in null frame $(\Lb,L,\dfrac{\partial}{\partial\theta^{A}})$:
\begin{align*}
\leftexp{(Z)}{\sigma}_{n-1}=&-\frac{1}{2}L(\leftexp{(Z)}{J}_{n-1,\underline{L}})-\frac{1}{2}\underline{L}
(\leftexp{(Z)}{J}_{n-1,L})+\slashed{\textrm{div}}(\mu\leftexp{(Z)}{\slashed{J}}_{n-1})\\\notag
&-\frac{1}{2}\underline{L}(c^{-2}\mu)\leftexp{(Z)}{J}_{n-1,\underline{L}}-\frac{1}{2}\textrm{tr}\chi
\leftexp{(Z)}{J}_{n-1,\underline{L}}-\frac{1}{2}\textrm{tr}\underline{\chi}\leftexp{(Z)}{J}_{n-1,L},
\end{align*}
Then with the following expressions for  the components of $\leftexp{(Z)}{J}_{n-1}$ in the null frame:
\begin{align*}
\leftexp{(Z)}{J}_{n-1,\underline{L}}=&-\frac{1}{2}\textrm{tr}\leftexp{(Z)}{\tilde{\slashed{\pi}}}
(\underline{L}\psi_{n-1})+\leftexp{(Z)}{\tilde{\underline{Z}}}\cdot\slashed{d}\psi_{n-1}\\
\leftexp{(Z)}{J}_{n-1,L}=&-\frac{1}{2}\textrm{tr}\leftexp{(Z)}{\tilde{\slashed{\pi}}}(L\psi_{n-1})
+\leftexp{(Z)}{\tilde{Z}}\cdot\slashed{d}\psi_{n-1}-\frac{1}{2\mu}\leftexp{(Z)}{\tilde{\pi}}_{LL}
(\underline{L}\psi_{n-1})\\
\mu\leftexp{(Z)}{\slashed{J}}^{A}_{n-1}=&-\frac{1}{2}\leftexp{(Z)}{\tilde{Z}}^{A}(\underline{L}\psi_{n-1})
-\frac{1}{2}\leftexp{(Z)}{\tilde{\underline{Z}}}^{A}(L\psi_{n-1})\\\notag
&+\frac{1}{2}(\leftexp{(Z)}{\tilde{\pi}}_{L\underline{L}}-\mu\textrm{tr}\leftexp{(Z)}{\tilde{\slashed{\pi}}})
\slashed{d}^{A}\psi_{n-1}+\mu\leftexp{(Y)}{\tilde{\slashed{\pi}}}^{A}_{B}\slashed{d}^{B}\psi_{n-1}
\end{align*}
Based on the above expressions, we decompose:
\begin{align*}
\leftexp{(Z)}{\sigma}_{n-1}=\leftexp{(Z)}{\sigma}_{1,n-1}+\leftexp{(Z)}{\sigma}_{2,n-1}+\leftexp{(Z)}{\sigma}_{3,n-1}
\end{align*}
where $\leftexp{(Z)}{\sigma}_{1,n-1}$ contains the products of components of $\leftexp{(Z)}{\widetilde{\pi}}$ with the 2nd derivatives of $\psi_{n-1}$, $\leftexp{(Z)}{\sigma}_{2,n-1}$ contains the products of the 1st derivatives of $\leftexp{(Z)}{\widetilde{\pi}}$ with the 1st derivatives of
$\psi_{n-1}$, and $\leftexp{(Z)}{\sigma}_{3,n-1}$ contains the other lower order terms. More specifically, we have:
\begin{align}\label{sigma1}
\begin{split}
\leftexp{(Z)}{\sigma}_{1,n-1}=&\frac{1}{2}\text{tr}\leftexp{(Z)}{\tilde{\slashed{\pi}}}(\Lb L\psi_{n-1}+\frac{1}{2}\text{tr}_{\tilde{\slashed{g}}}\widetilde{\chib}L\psi_{n-1})\\
&+\frac{1}{4}(\mu^{-1}\leftexp{(Z)}{\tilde{\pi}}_{LL})\Lb^{2}\psi_{n-1}\\
&-\leftexp{(Z)}{\tilde{Z}}\cdot\ds\Lb\psi_{n-1}-\leftexp{(Z)}{\tilde{\Zb}}\cdot\ds L\psi_{n-1}\\
&+\frac{1}{2}\leftexp{(Z)}{\tilde{\pi}}_{L\Lb}\slashed{\Delta}\psi_{n-1}+\mu\leftexp{(Z)}{\widehat{\tilde{\slashed{\pi}}}}\cdot\slashed{D}^{2}\psi_{n-1}
\end{split}
\end{align}
\begin{align}\label{sigma2}
\begin{split}
\leftexp{(Z)}{\sigma}_{2,n-1}=&\frac{1}{4}\Lb(\text{tr}\leftexp{(Z)}{\tilde{\slashed{\pi}}})L\psi_{n-1}+\frac{1}{4}L(\text{tr}\leftexp{(Z)}{\tilde{\slashed{\pi}}})\Lb\psi_{n-1}\\
&+\frac{1}{4}\Lb(\mu^{-1}\leftexp{(Z)}{\tilde{\pi}}_{LL})\Lb\psi_{n-1}\\
&-\frac{1}{2}\slashed{\mathcal{L}}_{\Lb}\leftexp{(Z)}{\tilde{Z}}\cdot\ds\psi_{n-1}-\frac{1}{2}\slashed{\mathcal{L}}_{L}\leftexp{(Z)}{\tilde{\Zb}}\cdot\ds\psi_{n-1}\\
&-\frac{1}{2}\slashed{\text{div}}\leftexp{(Z)}{\tilde{Z}}\Lb\psi_{n-1}-\frac{1}{2}\slashed{\text{div}}\leftexp{(Z)}{\tilde{\Zb}}L\psi_{n-1}\\
&+\frac{1}{2}\ds\leftexp{(Z)}{\tilde{\pi}}_{L\Lb}\ds\psi_{n-1}+\slashed{\text{div}}(\mu\leftexp{(Z)}{\widehat{\tilde{\slashed{\pi}}}})\cdot\ds\psi_{n-1}
\end{split}
\end{align}
and
\begin{align}\label{sigma3}
\begin{split}
\leftexp{(Z)}{\sigma}_{3,n-1}=&\big(\frac{1}{4}\text{tr}\chi\text{tr}\leftexp{(Z)}{\tilde{\slashed{\pi}}}+\frac{1}{4}\text{tr}\chib(\mu^{-1}\leftexp{(Z)}{\tilde{\pi}}_{LL})\\&+\frac{1}{2}\leftexp{(Z)}{\tilde{\Zb}}\cdot\ds(c^{-2}\mu)\big)\Lb\psi_{n-1}-\frac{1}{4}(\Lb\log(c^{-1}))\text{tr}\leftexp{(Z)}{\tilde{\slashed{\pi}}}L\psi_{n-1}\\
&-\big((\frac{1}{2}\text{tr}\chi+\Lb(c^{-2}\mu))\leftexp{(Z)}{\tilde{\Zb}}+\frac{1}{2}\text{tr}\chib\leftexp{(Z)}{\tilde{Z}}\big)\ds\psi_{n-1}
\end{split}
\end{align}
With these expressions for $\leftexp{(Z)}{\sigma}_{n-1}$, we are able to investigate the structure of $\widetilde{\rho}_{n}$. Basically, we want to use the recursion formulas in \eqref{rescaled source terms} to obtain a relatively explicit expression for $\widetilde{\rho}_{n}$.

On the other hand, for the energy estimates, we consider the following possible $\psi_{n}$:
\begin{align*}
\psi_{n}=R_{i}^{\alpha+1}\psi, \quad \psi_{n}=R^{\alpha'}_{i}T^{l+1}\psi,\quad \psi_{n}=QR^{\alpha'}_{i}T^{l}\psi
\end{align*}
Here $\psi_{n}$ is the $n$th order variation and $n=|\alpha|+1=|\alpha'|+l+1$. While $\psi$ is any first order variation. The reason that we can always first apply $T$, then $R_{i}$, and finally a possible $Q$ is that the commutators $[R_{i},T]$, $[R_{i},Q]$ and $[T,Q]$ are one order lower than $R_{i}T, TR_{i}$; $Q R_{i}, R_{i}Q$; $QT,TQ$ respectively. Moreover, the commutators $[R_{i},T], [R_{i},Q]$ and $[T,Q]$ are tangent to $S_{t,\ub}$. Since we let $Q$ be the last possible commutator, there will be no $Q$'s in $\psi_{n-1}$ in the second term on the right hand side of \eqref{sigma1}. Therefore we only need to commute $Q$ once. 
Now suppose that we consider the $n=|\alpha|+2$th order variations of the following form:
\begin{align*}
\psi_{|\alpha|+2}:=Z_{|\alpha|+1}...Z_{1}\psi
\end{align*}
We have the inhomogeneous wave equation:
\begin{align*}
\Box_{\tilde{g}}\psi_{|\alpha|+2}=\rho_{|\alpha|+2}
\end{align*}
As we pointed out in Remark 9.1, we define:
\begin{align*}
\widetilde{\rho}_{|\alpha|+2}=\frac{\mu}{c^{2}}\rho_{|\alpha|+2}
\end{align*}
Then by a induction argument, the corresponding inhomogeneous term $\widetilde{\rho}_{|\alpha|+2}$ is given by:
\begin{align}\label{rho tilde n}
\widetilde{\rho}_{|\alpha|+2}=\sum_{k=0}^{|\alpha|}\big(Z_{|\alpha|+1}+\leftexp{(Z_{|\alpha|+1})}{\delta}\big)...\big(Z_{|\alpha|-k+2}+\leftexp{(Z_{|\alpha|-k+2})}{\delta}\big)\leftexp{(Z_{|\alpha|-k+1})}{\sigma}_{|\alpha|-1+k}
\end{align}

\subsection{Error Estimates for the lower order terms}
Consider an arbitrary term in this sum. There is a total of $k$ derivatives with respect to the commutators acting on $\leftexp{(Z)}{\sigma}_{|\alpha|-1+k}$. In view of the fact that $\leftexp{(Z)}{\sigma}_{|\alpha|-1+k}$ has the structure described in \eqref{sigma1}, \eqref{sigma2} and \eqref{sigma3}, in considering the partial contribution of each term in $\leftexp{(Z)}{\sigma}_{1,|\alpha|-1+k}$, if the factor which is a component of $\leftexp{(Z)}{\tilde{\pi}}$ receives more than $\big[\dfrac{|\alpha|+1}{2}\big]$ derivatives with respect to the commutators, then the factor which is a $2$nd order derivative of $\psi_{|\alpha|+1-k}$ receives at most $k-\big[\dfrac{|\alpha|+1}{2}\big]-1$ order derivatives of commutators, thus corresponds to a derivative of the $\psi$ of order at most: $k-\big[\dfrac{|\alpha|+1}{2}\big]+1+|\alpha|-k=\big[\dfrac{|\alpha|}{2}\big]+1$, therefore this factor is bounded in $L^{\infty}(\Sigma_{t}^{\ub})$ by the bootstrap assumption. Also, in considering the partial contribution of each
term in
$\leftexp{(Z)}{\sigma}_{2,|\alpha|+1-k}$, if the factor which is a 1st derivative of $\leftexp{(Z)}{\tilde{\pi}}$ receives more than $\big[\dfrac{|\alpha|+1}{2}\big]-1$ derivatives with respect to the
commutators, then the factor which is a 1st derivative of $\psi_{|\alpha|+1-k}$ receives at most $k-\big[\dfrac{|\alpha|+1}{2}\big]$ derivatives with respect to the commutators, thus
corresponds to a derivative of the $\psi_{\alpha}$ of order at most $ k-\big[\dfrac{|\alpha|+1}{2}\big]+1+|\alpha|-k=\big[\dfrac{|\alpha|}{2}\big]+1$, therefore this factor is again bounded in $L^{\infty}(\Sigma_{t}^{\ub})$ by the bootstrap assumption. Similar considerations apply to $\leftexp{(Z)}{\sigma}
_{3,|\alpha|+1-k}$. We conclude that for all the terms in the sum in \eqref{rho tilde n} of which one factor is a derivative of the $\leftexp{(Z)}{\tilde{\pi}}$ of order more than
$\big[\dfrac{|\alpha|+1}{2}\big]$, the other factor is then a derivative of the $\psi_{\alpha}$ of order at most $\big[\dfrac{|\alpha|}{2}\big]+1$ and is thus bounded in $L^{\infty}(\Sigma_{t}^{\ub})$ by
the bootstrap assumption. Of these terms we shall estimate the contribution of those containing the top order spatial derivatives of the optical entities in the next subsection. Before we give the estimates for the contribution of the lower order optical terms to the spacetime integrals:
\begin{align}\label{spacetime error lower order}
-\delta^{2k}\int_{W^{t}_{\ub}}\widetilde{\rho}_{\leq|\alpha|+2}L\psi_{\leq|\alpha|+2}dtd\ub d\mu_{\slashed{g}},\quad -\delta^{2k}\int_{W^{t}_{\ub}}\widetilde{\rho}_{\leq|\alpha|+2}\Lb\psi_{\leq|\alpha|+2}dtd\ub d\mu_{\slashed{g}},
\end{align}
we investigate the behavior of these integrals with respect to $\delta$. Here $k$ is the number of $T$s in string of commutators. For the multiplier $K_{1}=\Lb$, the associated energy inequality is 
\begin{align}\label{energy inequality delta behavior K1}
\Eb_{\leq|\alpha|+2}(t,\ub)+\Fb_{\leq|\alpha|+2}(t,\ub)+K_{\leq|\alpha|+2}(t,\ub)\lesssim\Eb_{\leq|\alpha|+2}(-r_{0},\ub)+\int_{W^{t}_{\ub}}\widetilde{Q}_{1,\leq|\alpha|+2}.
\end{align}
The quantities $K_{\leq|\alpha|+2}(t,\ub)$ 
are defined similar as $K(t,\ub)$:
\begin{align*}
K_{\leq|\alpha|+2}(t,\ub):=\sum_{|\alpha'|\leq |\alpha|+1}\delta^{l'}K(t,\ub)[Z^{\alpha'}\psi],
\end{align*}
Again, $l'$ is the number of $T$'s in $Z^{\alpha'}$.

In $\widetilde{Q}_{1,\leq|\alpha|+2}$, there are contributions from the deformation tensors of two multipliers, which has been treated in Section 6. There are also contributions from the deformation tensors of commutators, which are given by \eqref{rho tilde n}. Now we investigate the terms which are not top order optical terms, namely, the terms containing $\chib$ and $\mu$ of order less than $|\alpha|+2$. In view of the discussion in Section 6, the left hand side of \eqref{energy inequality delta behavior K1} is of order $\delta$, so we expect these lower order terms in the second integral of \eqref{spacetime error lower order} is of order $\delta$. In fact, the integration on $W^{t}_{\ub}$ gives us a $\delta$ and the multiplier $\delta^{k}\Lb\psi_{\leq|\alpha|+2}$ is of order $\delta^{1/2}$. 
To see the behavior of $\delta^{k}\sigma$, we look at $\sigma_{1}$ as an example. Let $k'$ be the number of $T$s applied to $\left(\Lb L\psi_{l}+\frac{1}{2}\widetilde{\tr}\widetilde{\chib}L\psi_{l}\right)$. \eqref{box psi_0 in null frame} implies
\begin{align*}
\delta^{k'}\left(\Lb L\psi_{l}+\frac{1}{2}\widetilde{\tr}\widetilde{\chib}L\psi_{l}\right)\sim\delta^{1/2}
+\delta^{k'}\rho_{l}.
\end{align*}
Since $\rho_{1}=0$, an induction argument implies that 
 \begin{align*}
 \delta^{k'}\left(\Lb L\psi_{l}+\frac{1}{2}\widetilde{\tr}\widetilde{\chib}L\psi_{l}\right)\sim\delta^{1/2}.
 \end{align*}
 Then in view of \eqref{deformation tensor of T}, \eqref{deformation tensor of Q} and \eqref{deformation tensor of Rotational R_i in T,Lb,X_A}, the first term in $\sigma_{1}$ behaves like $\delta^{1/2}$. Following the same procedure, one sees straightforwardly that all the other terms in $\sigma_{1}, \sigma_{2}$ and $\sigma_{3}$ behave like $\delta^{1/2}$ (one keeps in mind that if $Z=T$, then we multiplier a $\delta$ with the corresponding deformation tensor.) except
the term $\Lb\left(\tr{}^{(Z)}\widetilde{\slashed{\pi}}\right)L\psi_{l}$. For this term we use the argument deriving \eqref{Improved estimate for trQ} and Proposition \ref{L infity estimates on lot} to see actually we have:
\begin{align*}
\|\Lb\big(\text{tr}\leftexp{(Q)}{\tilde{\slashed{\pi}}}\big)\|_{L^{\infty}(\Sigma_{t}^{\ub})}\lesssim\delta.
\end{align*}
This completes the discussions for $\sigma$ associated to $K_{1}$.

The same argument applies to the energy inequality associated to $K_{0}$: 
\begin{align}\label{energy inequality delta behavior K0}
E_{\leq|\alpha|+2}(t,\ub)+F_{\leq|\alpha|+2}(t,\ub)\lesssim E_{\leq|\alpha|+2}(-r_{0},\ub)+\int_{W^{t}_{\ub}}\widetilde{Q}_{0,\leq|\alpha|+2}.
\end{align}
and we conclude that the lower order optical terms in the error spacetime integrals have one more power in $\delta$ than the energies on the left hand side. 

Now we summarize the spacetime error estimates for the terms which come from the $L^{2}$ norms of the lower order optical quantities. 
In the proof of Proposition \ref{Proposition lower order L2} and Proposition \ref{Proposition lower order L2 mu}, we use $\mu_{m}^{-1/2}(t)\sqrt{\Eb_{\leq|\alpha|+2}(t,\ub)}$ to control the $L^{2}$ norm $\sum_{\alpha'\leq|\alpha|+1}\|\ds Z^{\alpha'}\psi\|_{L^{2}(\Sigma_{t}^{\ub})}$. Now we just keep this $L^{2}$ norm as it is. This together with the bootstrap assumptions on the $L^{\infty}$ norms of the variations implies that the contributions from the $L^{2}$ norms of lower order optical terms are bounded as:
\begin{align*}
&\int_{-r_{0}}^{t}\left(\sum_{|\alpha'|\leq |\alpha|+1}\delta^{1/2+l'}\|\ds Z^{\alpha'}_{i}\psi\|_{L^{2}(\Sigma_{t'}^{\ub})}+\delta\sqrt{E_{\leq|\alpha|+2}(t',\ub)}\right)dt'\cdot\delta^{1/2}
\int_{-r_{0}}^{t}\|\Lb\psi_{|\alpha|+2}\|_{L^{2}(\Sigma_{t'}^{\ub})}dt'\\
&\lesssim\int_{-r_{0}}^{t}\left(\sum_{|\alpha'|\leq|\alpha|+1}\delta^{2+2l'}\|\ds Z^{\alpha'}_{i}\psi\|^{2}_{L^{2}(\Sigma_{t'}^{\ub})}+\delta^{2} E_{\leq|\alpha|+2}(t',\ub)\right)dt'+\int_{0}^{\ub}\Fb_{\leq|\alpha|+2}(t,\ub')d\ub'
\quad
\text{for} \quad K_{1}
\end{align*}
and
\begin{align*}
&\int_{-r_{0}}^{t}\left(\sum_{|\alpha'|\leq |\alpha|+1}\delta^{1/2+l'}\|\ds
Z^{\alpha'}_{i}\psi\|_{L^{2}(\Sigma_{t'}^{\ub})}+\delta\sqrt{E_{\leq|\alpha|+2}(t',\ub)}\right)dt'\cdot\delta^{1/2}
\int_{-r_{0}}^{t}\|L\psi_{|\alpha|+2}\|_{L^{2}(\Sigma^{\ub}_{t'})}dt'\\
&\lesssim \int_{-r_{0}}^{t}\left(\sum_{|\alpha'|\leq|\alpha|+1}\delta^{1+2l'}\|\ds Z^{\alpha'}_{i}\psi\|^{2}_{L^{2}(\Sigma_{t'}^{\ub})}+\delta E_{\leq|\alpha|+2}(t',\ub)\right)dt'+\delta\int_{-r_{0}}^{t}E_{\leq|\alpha|+2}(t',\ub)dt'\quad
\text{for} \quad K_{0}
\end{align*}
where $l'$ is the number of $T$'s in the string of $Z^{\alpha'}_{i}$.
Therefore we obtain the following error estimates for the lower order optical terms:
\begin{align}\label{lower order optical error estimates K1 a}
\int_{-r_{0}}^{t}\delta^{2} E_{\leq|\alpha|+2}(t',\ub)dt'+\int_{0}^{\ub}\Fb_{\leq|\alpha|+2}(t,\ub')d\ub'+\delta^{2} K_{\leq|\alpha|+2}(t,\ub)\quad
\text{for}\quad K_{1}
\end{align}
and
\begin{align}\label{lower order optical error estimates K0 a}
\int_{-r_{0}}^{t}\delta E_{\leq|\alpha|+2}(t',\ub)dt'+\delta K_{\leq|\alpha|+2}(t,\ub)\quad
\text{for}\quad K_{0}
\end{align}
Next we consider the case in which the deformation tensors receive less derivatives with respect to the commutators than the variations in the expression for $\leftexp{(Z)}{\sigma}_{l}$. More specifically, we consider the terms in the sum \eqref{rho tilde n} in which there are at most $\big[\dfrac{|\alpha|+1}{2}\big]$ derivatives hitting the deformation tensor $\leftexp{(Z)}{\tilde{\slashed{\pi}}}$, thus the spatial derivatives on $\chib$ is at most $\big[\dfrac{|\alpha|+1}{2}\big]$ and the spatial derivatives on $\mu$ is at most $\big[\dfrac{|\alpha|+1}{2}\big]+1$, which are bounded in $L^{\infty}(\Sigma_{t}^{\ub})$ by virtue of Proposition 7.1 and Proposition 7.2. 
Using the inequality $ab\leq \epsilon a^{2}+\frac{1}{\epsilon}b^{2}$,
we have the following estimates for these contributions:
\begin{align}\label{lower order optical terms separately K1 a}
\begin{split}
\int_{-r_{0}}^{t}\delta^{2} E_{\leq|\alpha|+2}(t',\ub)dt'+\delta^{-1/2}\int_{0}^{\ub}\Fb_{\leq|\alpha|+2}(t,\ub')d\ub'\\
+ \delta^{1/2}K_{\leq|\alpha|+2}(t,\ub)+\int_{-r_{0}}^{t}\Eb_{\leq|\alpha|+2}(t',\ub)dt'\quad
\text{for}\quad K_{1}
\end{split}
\end{align}
\begin{align}\label{lower order optical terms separately K0 a}
\int_{-r_{0}}^{t}\delta^{1/2}E_{\leq|\alpha|+2}(t',\ub)dt'+\delta^{-1/2}\int_{0}^{\ub}\Fb_{\leq|\alpha|+2}(t,\ub')d\ub'+ \delta^{1/2}K_{\leq|\alpha|+2}(t,\ub)\quad
\text{for}\quad K_{0}
\end{align}
Here we estimate the terms involving $\Lb^{2}\psi_{n-1}$ and $\Lb\psi_{n-1}$ in terms of flux.

\subsection{Top Order Optical Estimates}
Now we estimate the contributions from the top order optical terms to the error spacetime integrals. In estimating the top order optical terms, we need to choose the power of $\mu_{m}(t)$ large enough. Therefore from this subsection on, we will use $C$ to denote an absolute positive constant so that one can see the largeness of the power of $\mu_{m}(t)$ more clearly.

The top order optical terms come from the term in which all the commutators hit the deformation tensors in the expression of $\leftexp{(Z_{1})}{\sigma}_{1}$, namely, the term:
\begin{align*}
(Z_{|\alpha|+1}+\leftexp{(Z_{|\alpha|+1})}{\delta})...(Z_{2}+\leftexp{(Z_{2})}{\delta})\leftexp{(Z_{1})}{\sigma}_{1}
\end{align*}
more precisely, in:
\begin{align*}
Z_{|\alpha|+1}...Z_{2}\big(-\frac{1}{2}L(\leftexp{(Z_{1})}{J}_{1,\Lb})-\frac{1}{2}\Lb(\leftexp{(Z_{1})}{J}_{1,L})+\slashed{\text{div}}(\mu\leftexp{(Z_{1})}{\slashed{J}})\big)
\end{align*}
when the operators $L, \Lb, \slashed{\text{div}}$ hit the deformation tensors in the expression of $\leftexp{(Z_{1})}{J}$.

Now we consider the top order variations:
\begin{align*}
R_{i}^{\alpha+1}\psi,\quad R_{i}^{\alpha'}T^{l+1}\psi,\quad QR^{\alpha'}_{i}T^{l}\psi
\end{align*}
where $|\alpha|=\Ntop-1$ and $|\alpha'|+l+1=|\alpha|+1$. 
Then the corresponding principal optical terms are:
\begin{align*}
\widetilde{\rho}_{|\alpha|+2}(R_{i}^{\alpha+1}\psi):=\frac{1}{c}(R_{i}^{\alpha+1}\text{tr}\chib' )\cdot T\psi,\quad \widetilde{\rho}_{|\alpha|+2}(R_{i}^{\alpha'}T^{l+1}\psi):=\frac{1}{c}(R_{i}^{\alpha'}T^{l}\slashed{\Delta}\mu)\cdot T\psi\\
\widetilde{\rho}_{|\alpha|+2}(QR^{\alpha'}_{i}T^{l}\psi):=\frac{t\mu}{c}\big(\slashed{d}R^{\alpha'}_{i}\text{tr}
\chib'\big)\cdot\ds\psi+\frac{t\mu}{c}(\ds R^{\alpha'}_{i}\slashed{\Delta}\mu)\cdot\Lb\psi,\quad\text{if}\quad l=0\\
\widetilde{\rho}_{|\alpha|+2}(QR^{\alpha'}_{i}T^{l}\psi):=\frac{t\mu}{c}\big(\slashed{d}R^{\alpha'}_{i}T^{l-1}
\slashed{\Delta}\mu\big)\cdot\ds\psi+\frac{t\mu}{c}(\ds R^{\alpha'}_{i}T^{l}\slashed{\Delta}\mu)\cdot\Lb\psi,\quad\text{if}\quad l\geq1
\end{align*}
Here we used the structure equation \eqref{Structure Equation T chib}
\begin{align*}
\slashed{\mathcal{L}}_{T}\text{tr}\chib
=\slashed{\Delta}\mu+\mathcal{O}^{\leq1}_{\geq0}
\end{align*}
Now we briefly investigate the behavior of the above terms with respect to $\delta$. Note that $R^{\alpha}_{i}Q\psi$ has the same behavior as $R^{\alpha+1}_{i}\psi$ with respect to $\delta$, while for their corresponding top order optical terms $\dfrac{t\mu}{c}\big(\ds R^{\alpha}_{i}\text{tr}\chib'\big)
\cdot\ds\psi$ and $\dfrac{1}{c}\big(R^{\alpha+1}_{i}\text{tr}\chib'
\big)\cdot T\psi$, the former behaves better than the latter with respect to $\delta$:
\begin{align*}
|\dfrac{t\mu}{c}\big(\ds R^{\alpha}_{i}\text{tr}\chib'\big)
\cdot\ds\psi|\sim |\mu \big(R^{\alpha+1}_{i}\text{tr}\chib'\big)|\delta^{1/2},\quad |\frac{1}{c}\big(R^{\alpha+1}_{i}\text{tr}
\chib'\big)\cdot T\psi|\sim|\big(R^{\alpha+1}_{i}\text{tr}\chib'\big)|\delta^{-1/2}
\end{align*}
 We see that not only the former behaves better with respect to $\delta$, but also has an extra $\mu$, which makes the behavior even better when $\mu$ is small. This means that we only need to estimate the contribution of $\dfrac{1}{c}\big(R^{\alpha+1}_{i}\text{tr}\chib'
\big) \cdot T\psi$. The same analysis applies to the terms involving $\Lb\psi$ as well as the comparison between $R^{\alpha'}_{i}T\psi$ and $QR^{\alpha'}_{i}T\psi$, which correspond to $\dfrac{1}{c}\big(R^{\alpha}_{i}\slashed{\Delta}
\mu\big)\cdot T\psi$ and $\dfrac{t\mu}{c}\big(R^{\alpha'+1}_{i}
\slashed{\Delta}\mu\big)\cdot\ds\psi$. So in the following, we do not need to estimate the contributions corresponding to the variations containing a $Q$.
\subsubsection{Contribution of $K_{0}$}
In this subsection we first estimate the spacetime integral:
\begin{align}\label{spacetime K_0 chib}
&\int_{W^{t}_{\ub}}\frac{1}{c}|R_{i}^{\alpha+1}\text{tr}\chib'||T\psi||LR^{\alpha+1}_{i}\psi|dt'd\ub' d\mu_{\slashed{g}}\lesssim\\\notag &\int_{-r_{0}}^{t}\sup_{\Sigma_{t'}^{\ub}}(\mu^{-1}|T\psi|)\|\mu\ds R_{i}^{\alpha}\text{tr}\chib'\|_{L^{2}(\Sigma_{t'}^{\ub})}\|LR^{\alpha+1}_{i}\psi\|_{L^{2}(\Sigma_{t'}^{\ub})}dt'
\end{align}
By \eqref{top order chib}, we have:
\begin{align*}
\|\mu\slashed{d}R^{\alpha}_{i}\text{tr}\chib'\|_{L^{2}(\Sigma^{\ub}_{t})}&\lesssim\delta^{1/2}\mu_{m}^{-b_{|\alpha|+2}}(t)\sqrt{\widetilde{E}_{\leq|\alpha|+2}(t,\ub)}\\&+\int_{-r_{0}}^{t}\delta^{1/2}\mu_{m}^{-b_{|\alpha|+2}-1/2}(\tau)\sqrt{\widetilde{\Eb}_{\leq |\alpha|+2}(\tau,\ub)}d\tau\\
&+\delta^{-1/2}\mu^{-b_{|\alpha|+2}}_{m}(t)\sqrt{\int_{0}^{\ub}\widetilde{\Fb}_{\leq|\alpha|+2}(t,\ub')d\ub'}.
\end{align*}
By the monotonicity of $ \widetilde{\Eb}_{\leq |\alpha|+2}(t,\ub)$ in $t$, we have
\begin{align}\label{final top chib}
\|\mu\slashed{d}R^{\alpha}_{i}\text{tr}\chib'\|_{L^{2}(\Sigma_{t}^{\ub})}
&\lesssim\delta^{1/2}\mu_{m}^{-b_{|\alpha|+2}}(t)\sqrt{\widetilde{E}_{\leq|\alpha|+2}(t,\ub)}\\\notag &+\delta^{1/2}\sqrt{\widetilde{\Eb}_{\leq |\alpha|+2}(t,\ub)}\int_{-r_{0}}^{t}\mu_{m}^{-b_{|\alpha|+2}-1/2}(t')dt'\\\notag
&+\delta^{-1/2}\sqrt{\int_{0}^{\ub}\widetilde{\Fb}_{\leq|\alpha|+2}(t,\ub')d\ub'}\mu_{m}^{-b_{|\alpha|+2}}(t)
\end{align}
Without loss of generality, here we assume that there is a $t_{0}\in[-r_{0},t^{*})$ such that $\mu_{m}(t_{0})=\frac{1}{10}$ and $\mu_{m}(t')\geq \frac{1}{10}$ for $t\leq t_{0}$. If there is no such $t_{0}$ in $[-r_{0},t^{*})$, then $\mu_{m}(t)$ has an absolute positive lower bound for all $[-r_{0},t^{*})$ and it is clear to see that the following argument simplifies and also works in this case. In view of part (1') and part (2) in Lemma \ref{lemma on mu to power a} and the fact $\mu_{m}(t')\geq \frac{1}{10}$ for $t'\in[-r_{0},t_{0}]$, we have

\begin{align*}
\int_{-r_{0}}^{t_{0}}\mu_{m}^{-b_{|\alpha|+2}-1/2}(t')dt'\lesssim&\mu_{m}^{-b_{|\alpha|+2}+1/2}(t_{0})\leq \mu_{m}^{-b_{|\alpha|+2}+1/2}(t),\\
\int_{t_{0}}^{t}\mu_{m}^{-b_{|\alpha|+2}-1/2}(t')dt'\lesssim&\frac{1}{\left(b_{|\alpha|+2}-1/2\right)}\mu_{m}^{-b_{|\alpha|+2}+1/2}(t).
\end{align*}

Therefore the second term in \eqref{final top chib} are bounded by:
\begin{align*}
\delta^{1/2}\mu_{m}^{-b_{|\alpha|+2}+1/2}(t)\sqrt{\widetilde{\Eb}_{\leq |\alpha|+2}(t,\ub)}
\end{align*}

Substituting this in \eqref{spacetime K_0 chib}, and using the fact that $|T\psi|\lesssim\delta^{-1/2}$, we see that the spacetime integral \eqref{spacetime K_0 chib} is bounded by:
\begin{align*}
&\int_{-r_{0}}^{t}\mu^{-b_{|\alpha|+2}-1}_{m}(t')\sqrt{\widetilde{E}_{\leq|\alpha|+2}(t',\ub)}\|LR^{\alpha+1}_{i}\psi\|_{L^{2}(\Sigma_{t'}^{\ub})}dt'\\&+\int_{-r_{0}}^{t}\mu^{-b_{|\alpha|+2}-1/2}_{m}(t')\sqrt{\widetilde{\Eb}_{\leq|\alpha|+2}(t',\ub)}\|LR^{\alpha+1}_{i}\psi\|_{L^{2}(\Sigma_{t'}^{\ub})}dt'\\\notag
&+\int_{-r_{0}}^{t}\delta^{-1}\mu^{-b_{|\alpha|+2}-1}_{m}(t')\sqrt{\int_{0}^{\ub}\widetilde{\Fb}_{\leq|\alpha|+2}(t',\ub')d\ub'}\|LR^{\alpha+1}_{i}\psi\|_{L^{2}(\Sigma_{t'}^{\ub})}dt'
\end{align*}
For the factor $\|LR^{\alpha+1}_{i}\psi\|_{L^{2}(\Sigma_{t}^{\ub})}$, we bound it as:
\begin{align*}
\|LR^{\alpha+1}_{i}\psi\|_{L^{2}(\Sigma_{t}^{\ub})}\leq \sqrt{E_{|\alpha|+2}(t,\ub)}\leq \mu_{m}^{-b_{|\alpha|+2}}(t)\sqrt{\widetilde{E}_{\leq|\alpha|+2}(t,\ub)}
\end{align*}
Then the spacetime integral \eqref{spacetime K_0 chib} is bounded by:
\begin{align*}
&\int_{-r_{0}}^{t}\mu^{-2b_{|\alpha|+2}-1}_{m}(t')\widetilde{E}_{\leq|\alpha|+2}(t',\ub)dt'\\
&+\int_{-r_{0}}^{t}\mu^{-2b_{|\alpha|+2}-1/2}_{m}(t')\sqrt{\widetilde{\Eb}_{\leq|\alpha|+2}(t',\ub)}\sqrt{\widetilde{E}_{\leq|\alpha|+2}(t',\ub)}dt'\\
&+\delta^{-1}\int_{-r_{0}}^{t}\mu^{-2b_{|\alpha|+2}-1}_{m}(t')\sqrt{\int_{0}^{\ub}\widetilde{\Fb}_{\leq|\alpha|+2}(t',\ub')d\ub'}\sqrt{\widetilde{E}_{\leq|\alpha|+2}(t',\ub)}dt'
\end{align*}

Splitting the above integrals as $\int_{-r_{0}}^{t_{0}}+\int_{t_{0}}^{t}$, we see that
the ``non-shock" part $\int_{-r_{0}}^{t_{0}}$ is bounded by:
\begin{align}\label{top gronwall 1}
\int_{-r_{0}}^{t_{0}}\mu^{-2b_{|\alpha|+2}}_{m}(t')\Big(\widetilde{E}_{\leq|\alpha|+2}(t',\ub)+\widetilde{\Eb}_{\leq|\alpha|+2}(t',\ub)+\delta^{-2}\int_{0}^{\ub}\widetilde{\Fb}_{\leq|\alpha|+2}(t',\ub')d\ub'\Big)dt'
\end{align}
using part (2) of Lemma \ref{lemma on mu to power a}.
\begin{align}\label{top optical 1}
\boxed{\frac{C}{2b_{|\alpha|+2}}\mu_{m}^{-2b_{|\alpha|+2}}(t)\widetilde{E}_{\leq|\alpha|+2}(t,\ub)}&+\frac{C}{\big(2b_{|\alpha|+2}-1/2\big)}\mu_{m}^{-2b_{|\alpha|+2}+1/2}(t)\widetilde{\Eb}_{\leq |\alpha|+2}(t,\ub)\\\notag&+\delta^{-2}\frac{C}{\big(2b_{|\alpha|+2}\big)}\mu^{-2b_{|\alpha|+2}}_{m}(t)\int_{0}^{\ub}\widetilde{\Fb}_{\leq|\alpha|+2}(t,\ub')d\ub'
\end{align}
using part (1') of Lemma \ref{lemma on mu to power a}.\\

\begin{remark}\label{a dependence}
The boxed term is from the estimates for the top order term $F_{\alpha}$. In view of \eqref{transport equation for F alpha for chib}, the number of top order terms contributed by the variations is independent of $\delta$ and $|\alpha|$, so is the constant $C$ in the boxed term. Later on in the top order energy estimates we will choose $b_{|\alpha|+2}$ in such a way that $\frac{C}{b_{top}}\leq \frac{1}{10}$. (The purpose of doing this is to make sure that this term can be absorbed by the left hand side the energy inequality.)  Therefore we can choose $b_{top}=\{10, \frac{C}{20}\}$. In particular $b_{top}$ is independent of $\delta$. 
\end{remark}

Next we consider the spacetime integral:
\begin{align}\label{spacetime K0 mu}
&\delta^{2l+2}\int_{W^{t}_{\ub}}\frac{1}{c}|R^{\alpha'}_{i}T^{l}\slashed{\Delta}\mu||T\psi||LR^{\alpha'}T^{l+1}\psi|dt'd\ub'd\mu_{\slashed{g}}\lesssim\\\notag
&\delta^{2l+2}\int_{-r_{0}}^{t}\sup_{\Sigma_{t'}^{\ub}}\big(\mu^{-1}|T\psi|\big)\|\mu R^{\alpha'}_{i}T^{l}\slashed{\Delta}\mu\|_{L^{2}(\Sigma_{t'}^{\ub})}\|L R^{\alpha'}T^{l+1}\psi\|_{L^{2}(\Sigma_{t'}^{\ub})}dt'
\end{align}
By \eqref{top order mu} and the monotonicity of $\widetilde{E}_{\leq|\alpha|+2}(t,\ub)$, $\widetilde{\Eb}_{\leq|\alpha|+2}(t,\ub)$ and $\widetilde{\Fb}_{\leq|\alpha|+2}(t,\ub)$ in $t$, we have:
\begin{align}\label{first bound for K0 mu}
\delta^{l+1}\|\mu R^{\alpha'}_{i}T^{l}\slashed{\Delta}\mu\|_{L^{2}(\Sigma_{t'}^{\ub})}&\leq C\delta^{l+1}\|F_{\alpha,l}\|_{L^{2}(\Sigma_{-r_{0}}^{\ub})}\\\notag
&+\delta^{1/2}\mu^{-b_{|\alpha|+2}}_{m}(t)\sqrt{\widetilde{E}_{\leq|\alpha|+2}(t,\ub)}+\delta^{1/2}\mu^{-b_{|\alpha|+2}}_{m}(t)\sqrt{\widetilde{\Eb}_{\leq|\alpha|+2}(t,\ub)} \\\notag
&+\int_{-r_{0}}^{t}\Big(\delta^{1/2}\mu_{m}^{-b_{|\alpha|+2}}(t')\sqrt{\widetilde{E}_{\leq|\alpha|+2}(t',\ub)}\\\notag 
&+\delta^{1/2}\mu_{m}^{-b_{|\alpha|+2}-1/2}(t')\sqrt{\widetilde{\Eb}_{\leq|\alpha|+2}(t',\ub)}\Big)dt'\\\notag
&+\delta^{1/2}\int_{-r_{0}}^{t}\mu_{m}^{-b_{|\alpha|+2}}(t')\sqrt{\int_{0}^{\ub}
\widetilde{\Fb}_{\leq|\alpha|+2}(t',\ub')d\ub'}dt'
\end{align}
As before, we have:
\begin{align*}
\int_{-r_{0}}^{t}\mu_{m}^{-b_{|\alpha|+2}}(t')dt'\leq C\mu_{m}^{-b_{|\alpha|+2}+1}(t),\quad \int_{-r_{0}}^{t}\mu^{-b_{|\alpha|+2}-1/2}_{m}(t')dt'\leq C\mu^{-b_{|\alpha|+2}+1/2}_{m}(t)
\end{align*}
Therefore
\begin{align}\label{top mu final}
\delta^{l+1}\|\mu R^{\alpha'}_{i}T^{l}\slashed{\Delta}\mu\|_{L^{2}(\Sigma_{t'}^{\ub})}&\leq C\delta^{l+1}\|F_{\alpha,l}\|_{L^{2}(\Sigma_{-r_{0}}^{\ub})}\\\notag
&+C\delta^{1/2}\mu_{m}^{-b_{|\alpha|+2}}(t)\sqrt{\widetilde{E}_{\leq|\alpha|+2}(t,\ub)}\\\notag&+C\delta^{1/2}\mu_{m}^{-b_{|\alpha|+2}}(t)\sqrt{\widetilde{\Eb}_{\leq|\alpha|+2}(t,\ub)}\\\notag
&+C\delta^{1/2}\mu_{m}^{-b_{|\alpha|+2}+1}(t)\sqrt{\int_{0}^{\ub}\widetilde{\Fb}_{\leq|\alpha|+2}(t,\ub')d\ub'}
\end{align}
Again, we estimate the other factor $\|LR^{\alpha'}_{i}T^{l+1}\psi\|_{L^{2}(\Sigma_{t}^{\ub})}$ by:
\begin{align*}
\delta^{l+1}\|LR^{\alpha'}_{i}T^{l+1}\psi\|_{L^{2}(\Sigma_{t}^{\ub})}\leq \sqrt{E_{\leq|\alpha|+2}(t,\ub)}\leq \mu^{-b_{|\alpha|+2}}_{m}(t)\sqrt{\widetilde{E}_{\leq|\alpha|+2}(t,\ub)}
\end{align*}
Here we also split the spacetime integral \eqref{spacetime K0 mu} into "shock part $\int_{t_{0}}^{t}$" and "non-shock part $\int_{-r_{0}}^{t_{0}}$". We first estimate the "shock part". 
Due to the estimate $|T\psi|\lesssim\delta^{-1/2}$, the contribution of the first term in \eqref{first bound for K0 mu} to the spacetime integral \eqref{spacetime K0 mu} is bounded by:
\begin{align*}
&\int_{t_{0}}^{t}\mu_{m}^{-b_{|\alpha|+2}-1}(t')\sqrt{\widetilde{E}_{\leq|\alpha|+2}(t,\ub)}\|LR^{\alpha'}_{i}T^{l+1}\psi\|_{L^{2}(\Sigma_{t'}^{\ub})}dt'\\
&\leq C\int_{t_{0}}^{t}\mu_{m}^{-2b_{|\alpha|+2}-1}(t')\widetilde{E}_{\leq|\alpha|+2}(t',\ub)dt'\leq C\widetilde{E}_{\leq|\alpha|+2}(t,\ub)\int_{t_{0}}^{t}\mu^{-2b_{|\alpha|+2}-1}_{m}(t')dt'\\
&\leq\frac{C}{2b_{|\alpha|+2}}\mu_{m}^{-2b_{|\alpha|+2}}(t)\widetilde{E}_{\leq|\alpha|+2}(t,\ub)
\end{align*}
Therefore, the contributions from the first two terms in \eqref{first bound for K0 mu} are bounded by:
\begin{align*}
\frac{C}{2b_{|\alpha|+2}}\mu_{m}^{-2b_{|\alpha|+2}}(t)\big[\widetilde{E}_{\leq|\alpha|+2}(t,\ub)+\widetilde{\Eb}_{\leq|\alpha|+2}(t,\ub)\big]
\end{align*}
and the contributions from the rest terms in \eqref{first bound for K0 mu} are bounded by:
\begin{align*}
&\frac{C}{\big(2b_{|\alpha|+2}-1/2\big)}\mu^{-2b_{|\alpha|+2}+1/2}_{m}(t)\widetilde{\Eb}_{\leq|\alpha|+2}(t,\ub)\\&+\frac{C}{\big(2b_{|\alpha|+2}-1\big)}\mu_{m}^{-2b_{|\alpha|+2}+1}(t)\big[\widetilde{E}_{\leq|\alpha|+2}(t,\ub)+\int_{0}^{\ub}
\widetilde{\Fb}_{\leq|\alpha|+2}(t,\ub')d\ub'\big]
\end{align*}
Therefore "shock-part" of the spacetime integral \eqref{spacetime K0 mu} is bounded by:
\begin{align}\label{top optical K0 mu-shock}
\frac{C}{2b_{|\alpha|+2}}\mu_{m}^{-2b_{|\alpha|+2}}(t)\big[\widetilde{E}_{\leq|\alpha|+2}(t,\ub)+\widetilde{\Eb}_{\leq|\alpha|+2}(t,\ub)\big]
+\frac{C}{\left(2b_{|\alpha|+2}-1\right)}\mu_{m}^{-2b_{|\alpha|+2}+1}(t)\int_{0}^{\ub}\widetilde{\Fb}_{\leq|\alpha|+2}(t,\ub')d\ub'
\end{align}
As before, for the "non-shock part", which mean $t'\in[-r_{0},t_{0}]$,  we have a positive lower bound for $\mu_{m}(t')$, then the spacetime integral \eqref{spacetime K0 mu} is bounded by:
\begin{align}\label{top optical K0 mu-non-shock}
\int_{-r_{0}}^{t_{0}}\mu^{-2b_{|\alpha|+2}}_{m}(t')\big[\widetilde{E}_{\leq|\alpha|+2}(t,\ub)+\widetilde{\Eb}_{\leq|\alpha|+2}(t',\ub)+\int_{0}^{\ub}\widetilde{\Fb}_{\leq|\alpha|+2}(t',\ub')d\ub'\big]dt'
\end{align}
which will be treated by Gronwall. 

Finally, using the inequality $ab\leq\frac{1}{2}b^{2}+\frac{1}{2}b^{2}$, the initial contribution $\|F_{\alpha,l}\|_{L^{2}(\Sigma_{-r_{0}})}$ is bounded by 
\begin{align}\label{initial top mu}
\delta^{2l+2}\|F_{\alpha,l}\|^{2}_{L^{2}(\Sigma_{-r_{0}})}+\int_{-r_{0}}^{t}\mu^{-2b_{|\alpha|+2}-1}_{m}(t')\widetilde{E}_{\leq|\alpha|+2}(t',\ub)dt'.
\end{align}
The second term above has already been estimated.

\subsubsection{Contribution of $K_{1}$}
In this subsection we estimate the contributions of top order optical terms associated to $K_{1}$. We start with the following absolute value of a spacetime integral:
\begin{align}\label{spacetime K1 chib}
\Big|\int_{W^{t}_{\ub}}\frac{1}{c}(R^{\alpha+1}_{i}\text{tr}\chib')\cdot
(T\psi)\cdot(\Lb R^{\alpha+1}_{i}\psi)dt'd\ub' d\mu_{\slashed{g}}\Big|
\end{align}
Since
\begin{align*}
\tilde{\slashed{g}}=\frac{1}{c}\slashed{g},\quad d\mu_{\tilde{\slashed{g}}}=\frac{1}{c}d\mu_{\slashed{g}}
\end{align*}
the above spacetime integral can be written as:
\begin{align}\label{spacetime integral K1}
\int_{W^{t}_{\ub}}(R^{\alpha+1}_{i}\text{tr}\chib')\cdot
(T\psi)\cdot(\Lb R^{\alpha+1}_{i}\psi)dt'd\ub' d\mu_{\tilde{\slashed{g}}}
\end{align}
Let us set:
\begin{align*}
F(t,\ub)=\int_{S_{t,\ub}}fd\mu_{\tilde{\slashed{g}}}
\end{align*}
then
\begin{align*}
\frac{\partial F}{\partial t}=\int_{S_{t,\ub}}\big(\Lb f+\widetilde{\text{tr}}\chib f\big)d\mu_{\tilde{\slashed{g}}}
\end{align*}
So we may convert the spacetime integral into two hypersurface integrals:
\begin{align*}
\int_{W^{t}_{\ub}}\big(\Lb f+\widetilde{\text{tr}}\chib f\big)d\mu_{\tilde{\slashed{g}}}d\ub'dt'
=\int_{\Sigma_{t}^{\ub}}f
d\mu_{\tilde{\slashed{g}}}d\ub'
-\int_{\Sigma_{-r_{0}}^{\ub}}f
d\mu_{\tilde{\slashed{g}}}d\ub'
\end{align*}
So we write the spacetime integral \eqref{spacetime integral K1} as:
\begin{align*}
&\int_{W^{t}_{\ub}}(R^{\alpha+1}_{i}\text{tr}\chib')\cdot
(T\psi)\cdot(\Lb R^{\alpha+1}_{i}\psi
+\widetilde{\text{tr}}\chib R^{\alpha+1}_{i}\psi)dt'd\ub' d\mu_{\tilde{\slashed{g}}}\\
&-\int_{W^{t}_{\ub}}\widetilde{\text{tr}}
\chib(R^{\alpha+1}_{i}\text{tr}\chib')
\cdot(T\psi)\cdot (R^{\alpha+1}_{i}\psi)dt'd\ub' d\mu_{\tilde{\slashed{g}}}\\
&=\int_{W^{t}_{\ub}}\big(\Lb
+\widetilde{\text{tr}}\chib\big)\Big[
(R^{\alpha+1}_{i}\text{tr}\chib')
\cdot(T\psi)\cdot(R^{\alpha+1}_{i}\psi)\Big]
dt'd\ub'd\mu_{\tilde{\slashed{g}}}\\&-
\int_{W^{t}_{\ub}}(\Lb+\widetilde{\text{tr}}\chib)
\Big[(R^{\alpha+1}_{i}\text{tr}\chib')
\cdot(T\psi)\Big]\cdot (R^{\alpha+1}_{i}\psi)dt'd\ub' d\mu_{\tilde{\slashed{g}}}
\end{align*}
We conclude that \eqref{spacetime integral K1} equals:
\begin{align*}
&\int_{\Sigma_{t}^{\ub}}(R^{\alpha+1}_{i}\text{tr}\chib')
\cdot(T\psi)\cdot(R^{\alpha+1}_{i}\psi)d\ub'
d\mu_{\tilde{\slashed{g}}}\\
&-\int_{\Sigma_{-r_{0}}^{\ub}}(R^{\alpha+1}_{i}\text{tr}\chib')
\cdot(T\psi)\cdot(R^{\alpha+1}_{i}\psi)d\ub'
d\mu_{\tilde{\slashed{g}}}\\
&-\int_{W^{t}_{\ub}}(\Lb+\widetilde{\text{tr}}\chib)
\Big[(R^{\alpha+1}_{i}\text{tr}\chib')
\cdot(T\psi)\Big]\cdot (R^{\alpha+1}_{i}\psi)dt'd\ub' d\mu_{\tilde{\slashed{g}}}
\end{align*}
We first consider the hypersurface integral which, integrating by parts, equals:
\begin{align*}
-H_{0}-H_{1}-H_{2}
\end{align*}
where:
\begin{align*}
H_{0}&=\int_{\Sigma_{t}^{\ub}}(R^{\alpha}_{i}\text{tr}\chib)\cdot(T\psi)
\cdot(R^{\alpha+2}_{i}\psi)
d\ub'd\mu_{\tilde{\slashed{g}}}\\
H_{1}&=\int_{\Sigma_{t}^{\ub}}(R^{\alpha}_{i}\text{tr}\chib')
\cdot(R_{i}T\psi)
\cdot(R^{\alpha+1}_{i}\psi)
d\ub'd\mu_{\tilde{\slashed{g}}}\\
H_{2}&=\int_{\Sigma_{t}^{\ub}}(T\psi)\cdot(R^{\alpha+1}_{i}\psi)
\cdot(R^{\alpha}\text{tr}\chib')(\frac{1}{2}\text{tr}\leftexp{(R_{i})}{\tilde{\slashed{\pi}}})
d\ub'd\mu_{\tilde{\slashed{g}}}
\end{align*}
Since we shall bound both $T\psi$ and $R_{i}T\psi$ in $L^{\infty}$ norm, compared to $H_{0}$, $H_{1}$ is a lower order term with respect to the order of derivatives. While for $H_{2}$, we use the estimate:
\begin{align*}
|\text{tr}\leftexp{(R_{i})}{\tilde{\slashed{\pi}}}|\leq C\delta
\end{align*}
to see that it is a lower order term with respect to both the behavior of $\delta$ and the order of derivatives compared to $H_{0}$. This analysis tells us that we only need to estimate $H_{0}$.
\begin{align*}
|H_{0}|\leq \int_{\Sigma_{t}^{\ub}}|T\psi||R^{\alpha}_{i}\text{tr}\chib'||\ds R^{\alpha+1}_{i}\psi|d\ub'd\mu_{\tilde{\slashed{g}}}
\leq C\delta^{-1/2}\|R^{\alpha}_{i}\text{tr}\chib'\|_{L^{2}(\Sigma_{t}^{\ub})}\|\slashed{d}R^{\alpha+1}_{i}\psi\|_{L^{2}(\Sigma_{t}^{\ub})}\\
\leq C\int_{-r_{0}}^{t}\Big(\mu_{m}^{-1/2}(t')\sqrt{\underline{E}_{|\alpha|+2}(t',\ub)}\Big)dt'\cdot\mu^{-1/2}_{m}(t)\sqrt{\underline{E}_{|\alpha|+2}(t,\ub)}
\end{align*}
by Proposition 7.4. Then in terms of modified energies, we have:
\begin{align*}
|H_{0}|\leq C\int_{-r_{0}}^{t}\mu_{m}^{-b_{|\alpha|+2}-1/2}(t')\sqrt{\widetilde{\Eb}_{\leq|\alpha|+2}(t',\ub)}dt'\cdot\mu^{-b_{|\alpha|+2}-1/2}_{m}(t)\sqrt{\widetilde{\Eb}_{\leq |\alpha|+2}(t,\ub)}
\end{align*}
Again, we consider the ``shock part $\int_{-r_{0}}^{t_{0}}$" and ``non-shock part $\int_{t_{0}}^{t}$", which are denoted by $H^{S}_{0}$ and $H^{N}_{0}$, separately.

In the ``shock part", by the monotonicity of $\widetilde{\Eb}_{\leq |\alpha|+2}(t)$, we have:
\begin{align*}
|H^{S}_{0}|&\leq C\sqrt{\widetilde{\Eb}_{\leq |\alpha|+2}(t,\ub)}\int_{t_{0}}^{t}\mu^{-b_{|\alpha|+2}-1/2}_{m}(t')dt'\cdot\mu_{m}^{-b_{|\alpha|+2}-1/2}(t)\sqrt{\widetilde{\Eb}_{\leq |\alpha|+2}(t,\ub)}\\
&\leq C\widetilde{\Eb}_{\leq |\alpha|+2}(t,\ub)\mu_{m}^{-b_{|\alpha|+2}-1/2}(t)\int_{t_{0}}^{t}\mu_{m}^{-b_{|\alpha|+2}-1/2}(t')dt'\\
&\leq \frac{C}{\big(b_{|\alpha|+2}-1/2\big)}\mu_{m}^{-b_{|\alpha|+2}-1/2}(t)\widetilde{\Eb}_{\leq |\alpha|+2}(t,\ub)\cdot\mu_{m}^{-b_{|\alpha|+2}+1/2}(t)\\&= \frac{C}{\big(b_{|\alpha|+2}-1/2\big)}\mu_{m}^{-2b_{|\alpha|+2}}(t)\widetilde{\Eb}_{\leq |\alpha|+2}(t,\ub)
\end{align*}

The estimates for the ``non-shock part" is more delicate.
\begin{align*}
|H_{0}^{N}|\leq C\int_{-r_{0}}^{t_{0}}\mu_{m}^{-b_{|\alpha|+2}-1/2}(t')\sqrt{\widetilde{\Eb}_{\leq|\alpha|+2}(t',\ub)}dt'\cdot\mu_{m}^{-b_{|\alpha|+2}-1/2}(t)\sqrt{\widetilde{\Eb}_{\leq |\alpha|+2}(t,\ub)}
\end{align*}
Now for $t'\in[-r_{0},t_{0}]$, $\mu_{m}^{-1}(t')\leq 10$. This means we have the following estimate for the above integral:
\begin{align*}
&\int_{-r_{0}}^{t_{0}}\mu_{m}^{-b_{|\alpha|+2}-1/2}(t')\sqrt{\widetilde{\Eb}_{\leq|\alpha|+2}(t',\ub)}dt'=\int_{-r_{0}}^{t_{0}}\mu_{m}^{-1}(t')\mu_{m}^{-b_{|\alpha|+2}+1/2}(t')\sqrt{\widetilde{\Eb}_{\leq|\alpha|+2}(t',\ub)}dt'\\
&\leq C\int_{-r_{0}}^{t_{0}}\mu_{m}^{-b_{|\alpha|+2}+1/2}(t')\sqrt{\widetilde{\Eb}_{\leq|\alpha|+2}(t',\ub)}dt'\leq C\int_{-r_{0}}^{t_{0}}\sqrt{\widetilde{\Eb}_{\leq|\alpha|+2}(t',\ub)}dt'\cdot\mu^{-b_{|\alpha|+2}+1/2}_{m}(t)
\end{align*}
where in the last step, we have used part (2) of Lemma \ref{lemma on mu to power a}.

Therefore
\begin{align*}
|H^{N}_{0}|&\leq C\int_{-r_{0}}^{t_{0}}\sqrt{\widetilde{\Eb}_{\leq|\alpha|+2}(t',\ub)}dt'\cdot \mu^{-2b_{|\alpha|+2}}_{m}(t)\sqrt{\widetilde{\Eb}_{\leq |\alpha|+2}(t,\ub)}\\
&\leq \epsilon\mu^{-2b_{|\alpha|+2}}_{m}(t)\widetilde{\Eb}_{\leq |\alpha|+2}(t,\ub)+C_{\epsilon}\mu^{-2b_{|\alpha|+2}}_{m}(t)\int_{-r_{0}}^{t_{0}}\widetilde{\Eb}_{\leq|\alpha|+2}(t',\ub)dt'
\end{align*}

Here $\epsilon$ is a small absolute constant to be determined later.

We obtain the following estimate for $|H_{0}|$:
\begin{align}\label{estimates for H0}
|H_{0}|&\leq  \frac{C}{\big(b_{|\alpha|+2}-1/2\big)}\mu_{m}^{-2b_{|\alpha|+2}}(t)\widetilde{\Eb}_{\leq |\alpha|+2}(t,\ub)\\\notag&+C_{\epsilon}\mu^{-2b_{|\alpha|+2}}_{m}(t)\int_{-r_{0}}^{t_{0}}\widetilde{\Eb}_{\leq|\alpha|+2}(t',\ub)dt'+\epsilon\mu^{-2b_{|\alpha|+2}}_{m}(t)\widetilde{\Eb}_{\leq |\alpha|+2}(t,\ub)
\end{align}

Next we consider the spacetime integral:

\begin{align*}
\int_{W^{t}_{\ub}}(\Lb+\widetilde{\text{tr}}\chib)
\Big[(R^{\alpha+1}_{i}\text{tr}\chib)\cdot
(T\psi)\Big]\cdot(R^{\alpha+1}_{i}\psi)
dt'd\ub'd\mu_{\tilde{\slashed{g}}}\\
=\int_{W^{t}_{\ub}}\big((\Lb
+\widetilde{\text{tr}}\chib)(R^{\alpha+1}_{i}\text{tr}\chib')\big)
\cdot(T\psi)\cdot(R^{\alpha+1}_{i}\psi)dt'
d\ub'd\mu_{\tilde{\slashed{g}}}\\
+\int_{W^{t}_{\ub}}(R^{\alpha+1}_{i}\text{tr}\chib')
\cdot\Lb(T\psi)\cdot(R^{\alpha+1}_{i}\psi)
dt'd\ub'd\mu_{\tilde{\slashed{g}}}:=I+II
\end{align*}
For the second term in the above sum, by the fact that:
\begin{align*}
|\Lb(T\psi)|\leq C\delta^{-1/2}
\end{align*}
we have:
\begin{align*}
|II|\leq C\delta^{-1/2}\int_{-r_{0}}^{t}\|\ds R^{\alpha}_{i}\text{tr}\chib'\|_{L^{2}(\Sigma_{t'}^{\ub})}\|R^{\alpha+1}_{i}\psi\|_{L^{2}(\Sigma_{t'}^{\ub})}dt'
\end{align*}
On the other hand, by Lemma 7.3, we have:
\begin{align*}
\|R^{\alpha+1}_{i}\psi\|_{L^{2}(\Sigma_{t'}^{\ub})}\leq C\delta\big(E_{|\alpha|+2}(t',\ub)+\underline{E}_{|\alpha|+2}(t',\ub)\big)
\end{align*}
which gives:
\begin{align*}
|II|\leq C\delta^{1/2}\int_{-r_{0}}^{t}\|\ds R^{\alpha}_{i}\text{tr}\chib'\|_{L^{2}(\Sigma_{t'}^{\ub})}\big(E_{|\alpha|+2}(t',\ub)+\underline{E}_{|\alpha|+2}(t',\ub)\big)dt'
\end{align*}
This has a similar form as \eqref{spacetime K_0 chib}, and moreover, it has an extra $\delta$, which is consistent with the order of $\Eb_{|\alpha|+2}(t,\ub)$. (We already took into account the effect of $\Lb(T\psi)$.) So this term is already handled.

To estimate $|I|$, we first rewrite:
\begin{align*}
(\Lb+\widetilde{\text{tr}}\chib)(R^{\alpha+1}_{i}\text{tr}\chib')=R_{i}(\Lb+\widetilde{\text{tr}}\chib)
(R^{\alpha}_{i}\text{tr}\chib')
+\leftexp{(R_{i})}{\Zb}R^{\alpha}_{i}\text{tr}\chib'-R_{i}(\text{tr}\chib')R^{\alpha}_{i}\text{tr}\chib'+\textrm{l.o.t.}
\end{align*}
The contribution of the second term is:
\begin{align*}
\int_{W^{t}_{\ub}}|T\psi||\leftexp{(R_{i})}{\Zb}||\ds R^{\alpha}_{i}\text{tr}\chib'||R^{\alpha+1}_{i}\psi|dt'd\ub'd\mu_{\tilde{\slashed{g}}}
\end{align*}
By the estimates:
\begin{align*}
|\leftexp{(R_{i})}{\Zb}|\leq C\delta,\quad |T\psi|\leq C\delta^{-1/2}
\end{align*}
this contribution is bounded by:
\begin{align*}
\delta^{1/2}\int_{-r_{0}}^{t}\|\ds R^{\alpha}_{i}\text{tr}\chib'\|_{L^{2}(\Sigma_{t'}^{\ub})}\|R^{\alpha+1}_{i}\psi\|_{L^{2}(\Sigma_{t'}^{\ub})}dt'
\end{align*}
which is similar to the estimate for $|II|$ and has an extra $\delta$, so this is a lower order term.

By the estimate
\begin{align*}
|R_{i}\text{tr}\chib|\leq C\delta,
\end{align*}
the contribution of the third term is bounded by:
\begin{align*}
C\delta\int_{W^{t}_{\ub}}|R^{\alpha}_{i}\text{tr}\chib'||T\psi||R^{\alpha+1}_{i}\psi|dt'd\ub'd\mu_{\tilde{\slashed{g}}}\leq C\delta^{1/2}\int_{-r_{0}}^{t}\|R^{\alpha}_{i}\text{tr}\chib'\|_{L^{2}(\Sigma_{t'}^{\ub})}\|R^{\alpha+1}_{i}\psi\|_{L^{2}(\Sigma_{t'}^{\ub})}dt'
\end{align*}
By proposition 7.4 and Lemma 7.3, this is bounded in terms of modified energies by:
\begin{align}\label{spacetime II K1 lower order}
&\delta^{2}\int_{-r_{0}}^{t}\mu^{-b_{|\alpha|+2}-1/2}_{m}(t')\sqrt{\widetilde{\Eb}_{\leq|\alpha|+2}(t',\ub)}\cdot\mu^{-b_{|\alpha|+2}}_{m}(t')\big(\sqrt{\widetilde{E}_{\leq|\alpha|+2}(t',\ub)}+\sqrt{\widetilde{\Eb}_{\leq|\alpha|+2}(t',\ub)}\big)dt'\\\notag
&\leq C\delta^{2}\big(\widetilde{\Eb}_{\leq |\alpha|+2}(t,\ub)+\widetilde{E}_{\leq|\alpha|+2}(t,\ub)\big)\int_{-r_{0}}^{t}\mu_{m}^{-2a-1/2}(t')dt'\\\notag
&\leq C\delta^{2}\mu^{-2b_{|\alpha|+2}+1/2}_{m}(t)\big(\widetilde{\Eb}_{\leq |\alpha|+2}(t,\ub)+\widetilde{E}_{\leq|\alpha|+2}(t,\ub)\big)
\end{align}
Again, we used Lemma \ref{lemma on mu to power a} in the last step.

Now we are left with the spacetime integral:
\begin{align*}
\int_{W^{t}_{\ub}}\Big(R_{i}(\Lb+\widetilde{\text{tr}}\chib)(R^{\alpha}_{i}\text{tr}\chib')\Big)\cdot(T\psi)
\cdot(R^{\alpha+1}_{i}\psi)dt'd\ub'
d\mu_{\tilde{\slashed{g}}}
\end{align*}
Integrating by parts, this equals:
\begin{align*}
-\int_{W^{t}_{\ub}}\Big((\Lb
+\widetilde{\text{tr}}\chib)(R^{\alpha}_{i}\text{tr}\chib')\Big)
\cdot(T\psi)\cdot(R^{\alpha+2}_{i}\psi)
dt'd\ub'd\mu_{\tilde{\slashed{g}}}\\
-\int_{W^{t}_{\ub}}\Big((\Lb
+\widetilde{\text{tr}}\chib)(R^{\alpha}_{i}\text{tr}\chib')\Big)
\cdot(R_{i}T\psi+\frac{1}{2}\text{tr}\leftexp{(R_{i})}{\tilde{\slashed{\pi}}})\cdot(R^{\alpha+1}_{i}\psi)
dt'd\ub'd\mu_{\tilde{\slashed{g}}}\\
:=-V_{1}-V_{2}
\end{align*}
By the estimate:
\begin{align*}
|\text{tr}\leftexp{(R_{i})}{\tilde{\slashed{\pi}}}|\leq C\delta
\end{align*}
we see that:
\begin{align*}
|R_{i}T\psi+\frac{1}{2}\text{tr}\leftexp{(R_{i})}{\tilde{\slashed{\pi}}}|\leq C\delta^{-1/2},\quad |T\psi|\leq C\delta^{-1/2}
\end{align*}
So compared to $V_{1}$, $V_{2}$ is a lower order term, and we use Lemma 7.3 to estimate:
\begin{align*}
\|R^{\alpha+1}_{i}\psi\|_{L^{2}(\Sigma_{t}^{\ub})}\leq C\delta\Big(\sqrt{E_{|\alpha|+2}(t,\ub)}+\sqrt{\underline{E}_{|\alpha|+2}(t,\ub)}\Big)
\end{align*}
So we only need to estimate $|V_{1}|$, which is bounded by:
\begin{align*}
&\int_{W^{t}_{\ub}}\Big|(\Lb+\frac{2}{t-\ub})(R^{\alpha}_{i}\text{tr}\chib')\Big||T\psi||R^{\alpha+2}_{i}\psi|dt'd\ub'd\mu_{\tilde{\slashed{g}}}\\
&+\int_{W^{t}_{\ub}}|\widetilde{\text{tr}}\chib'||R^{\alpha}_{i}\text{tr}\chib'||T\psi||R^{\alpha+2}_{i}\psi|dt'd\ub'd\mu_{\tilde{\slashed{g}}}:
=V_{11}+V_{12}
\end{align*}
By the estimates:
\begin{align*}
|\widetilde{\text{tr}}\chib'|\leq C\delta, \quad |T\psi|\leq C\delta^{-1/2}
\end{align*}
and proposition 7.4, $V_{12}$ is bounded in terms of modified energies by:
\begin{align}\label{V12}
&\delta^{1/2}\int_{-r_{0}}^{t}\|R^{\alpha}_{i}\text{tr}\chib'\|_{L^{2}(\Sigma_{t'}^{\ub})}\|\ds R^{\alpha+1}_{i}\psi\|_{L^{2}(\Sigma_{t'}^{\ub})}dt'\\\notag
&\leq C\delta\int_{-r_{0}}^{t}\mu_{m}^{-b_{|\alpha|+2}-1/2}(t')\sqrt{\widetilde{\Eb}_{\leq|\alpha|+2}(t',\ub)}\mu^{-b_{|\alpha|+2}-1/2}_{m}(t')\sqrt{\widetilde{\Eb}_{\leq|\alpha|+2}(t',\ub)}dt'\\\notag
&\leq C\delta\widetilde{\Eb}_{\leq |\alpha|+2}(t,\ub)\int_{-r_{0}}^{t}\mu^{-2b_{|\alpha|+2}-1}_{m}(t')dt'\leq C\delta\mu_{m}^{-2b_{|\alpha|+2}}(t)\widetilde{\Eb}_{\leq |\alpha|+2}(t,\ub).
\end{align}
In the last step we used Lemma \ref{lemma on mu to power a}.

To estimate $V_{11}$, we recall the propagation equation for $\text{tr}\chib'$:
\begin{align*}
\Lb\text{tr}\chib'+\frac{2}{t-\ub}\text{tr}\chib'=e\text{tr}\chib'-|\chib'|^{2}+\frac{2e}{t-\ub}-\text{tr}\alphab':=\rho_{0}
\end{align*}
Applying $R^{\alpha}_{i}$ to this equation gives:
\begin{align*}
\Lb R^{\alpha}_{i}\text{tr}\chib'+\frac{2}{t-\ub}R^{\alpha}_{i}\text{tr}\chib'=
\sum_{|\beta|\leq|\alpha|}R^{\beta}_{i}\leftexp{(R_{i})}{\Zb}R^{\alpha-\beta-1}_{i}\text{tr}\chib'
+R^{\alpha}_{i}\rho_{0}
\end{align*}
Again, by the estimates
\begin{align*}
|\leftexp{(R_{i})}{\Zb}|\leq C\delta,\quad |T\psi|\leq C\delta^{-1/2}
\end{align*}
as well as Proposition 7.4, the contribution of the first term on the right hand side is bounded in terms of modified energies by:
\begin{align}\label{V11 lower order}
\begin{split}
&\delta^{1/2}\int_{-r_{0}}^{t}\|R^{\alpha}_{i}\text{tr}\chib'\|_{L^{2}(\Sigma_{t'}^{\ub})}\|R^{\alpha+2}_{i}\psi\|_{L^{2}(\Sigma_{t'}^{\ub})}dt'\\
&\leq C\delta\int_{-r_{0}}^{t}\mu^{-2b_{|\alpha|+2}-1}_{m}(t')\widetilde{\Eb}_{\leq|\alpha|+2}(t',\ub)dt'\leq C\delta\mu_{m}^{-2b_{|\alpha|+2}}(t)\widetilde{\Eb}_{\leq |\alpha|+2}(t,\ub).
\end{split}
\end{align}
The following terms in $R^{\alpha}\rho_{0}$ also enjoy the above estimates:
\begin{align*}
R^{\alpha}_{i}\big(|\chib'|^{2}\big), \quad (R^{\beta}_{i}e)\cdot(R^{\alpha-\beta}_{i}
\text{tr}\chib')\quad \text{where}\quad |\beta|\leq |\alpha|-|\beta|
\end{align*}
While the other contributions from $R^{\alpha}_{i}e\cdot\text{tr}\chib'$, $\dfrac{2e}{t-\ub}$ and the lower order terms in $\text{tr}\alphab'$ can be bounded by:
\begin{align}\label{V11 lower order 1}
C\mu^{-2b_{|\alpha|+2}+1/2}_{m}(t)\Big(\delta^{3/2}\widetilde{E}_{|\alpha|+2}(t,\ub)+\delta^{3/2}\widetilde{\Eb}_{\leq |\alpha|+2}(t,\ub)\Big)
\end{align}
in view of Lemma 7.3.

Now we estimate the contribution from the principal term in $\text{tr}\alphab'$, which is:
\begin{align*}
\frac{dc^{2}}{d\rho}\psi_{0}
\slashed{\Delta}R^{\alpha}_{i}\psi_{0}
\end{align*}
In view of the estimates:
\begin{align*}
|\psi_{0}|\leq C\delta^{1/2},\quad |T\psi|\leq C\delta^{-1/2}
\end{align*}
this contribution is bounded by:
\begin{align*}
\int_{-r_{0}}^{t}\|\ds R^{\alpha+1}_{i}\psi_{0}\|_{L^{2}(\Sigma_{t'}^{\ub})}\|R^{\alpha+2}_{i}\psi\|_{L^{2}(\Sigma_{t'}^{\ub})}dt'
&\leq C\int_{-r_{0}}^{t}\|\ds R^{\alpha+1}_{i}\psi\|_{L^{2}(\Sigma_{t'}^{\ub})}^{2}dt'\\\notag
&\leq C\int_{-r_{0}}^{t}\mu^{-2b_{|\alpha|+2}-1}_{m}(t')\widetilde{\Eb}_{\leq|\alpha|+2}(t',\ub)dt'
\end{align*}
Again, considering the``shock part" and ``non-shock part" in regard to this integral, we obtain that it is bounded by:
\begin{align}\label{V11}
\frac{C}{2b_{|\alpha|+2}}\mu^{-2b_{|\alpha|+2}}_{m}(t)\widetilde{\Eb}_{\leq |\alpha|+2}(t,\ub)+\int_{-r_{0}}^{t_{0}}\mu_{m}^{-2b_{|\alpha|+2}}(t')\widetilde{\Eb}_{\leq|\alpha|+2}(t',\ub)dt'
\end{align}
This completes the estimates for the spacetime integral \eqref{spacetime K1 chib}.\vspace{3mm}

Next we consider the top order optical contribution of the variation $R^{\alpha'}T^{l+1}\psi$, where $|\alpha'|+l+1=|\alpha|+1$, which is the following spacetime integral:
\begin{align}\label{top optical K1 mu}
\delta^{2l+2}\Big|\int_{W^{t}_{\ub}}(T\psi)\cdot(R^{\alpha'}_{i}T^{l}
\slashed{\Delta}\mu)\cdot(\Lb R^{\alpha'}_{i}
T^{l+1}\psi)dt'd\ub'
d\mu_{\tilde{\slashed{g}}}\Big|
\end{align}
Again, we rewrite the above spacetime integral as:
\begin{align*}
&\delta^{2l+2}\int_{W^{t}_{\ub}}(T\psi)\cdot(R^{\alpha'}_{i}T^{l}
\slashed{\Delta}\mu)\cdot
\big((\Lb+\widetilde{\text{tr}}\chib)(R^{\alpha'}_{i}T^{l+1}\psi)\big)
dt'd\ub'd\mu_{\tilde{\slashed{g}}}
\\
&-\delta^{2l+2}\int_{W^{t}_{\ub}}(T\psi)\cdot(R^{\alpha'}_{i}T^{l}
\slashed{\Delta}\mu)\cdot
\big(\widetilde{\text{tr}}
\chib(R^{\alpha'}_{i}T^{l+1}\psi)\big)
dt'd\ub'd\mu_{\tilde{\slashed{g}}}
\end{align*}
which is:
\begin{align*}
&\delta^{2l+2}\int_{W^{t}_{\ub}}(\Lb+\widetilde{\text{tr}}\chib)
\Big((T\psi)(R^{\alpha'}_{i}T^{l}\slashed{\Delta}\mu)(R^{\alpha'}_{i}T^{l+1}\psi)\Big)dt' d\ub'  d\mu_{\tilde{\slashed{g}}}\\
&-\delta^{2l+2}\int_{W^{t}_{\ub}}(\Lb T\psi)(R^{\alpha'}_{i}T^{l}\slashed{\Delta}\mu)(R^{\alpha'}_{i}T^{l+1}\psi)dt'd\ub'
d\mu_{\tilde{\slashed{g}}}\\
&-\delta^{2l+2}\int_{W^{t}_{\ub}}(T\psi)\big((\Lb+
\widetilde{\text{tr}}\chib)(R^{\alpha'}_{i}T^{l}\slashed{\Delta}\mu)
\big)(R^{\alpha'}_{i}T^{l+1}\psi)dt'd\ub'
d\mu_{\tilde{\slashed{g}}}
\end{align*}
As before, the spacetime integral in the first line above can be written as:
\begin{align*}
\delta^{2l+2}\int_{\Sigma_{t}^{\ub}}(T\psi)(R^{\alpha'}_{i}T^{l}\slashed{\Delta}\mu)(R^{\alpha'}_{i}T^{l+1}\psi)d\ub'
d\mu_{\tilde{\slashed{g}}}-
\delta^{2l+2}\int_{\Sigma_{-r_{0}}^{\ub}}(T\psi)(R^{\alpha'}_{i}T^{l}\slashed{\Delta}\mu)(R^{\alpha'}_{i}T^{l+1}\psi)d\ub'
d\mu_{\tilde{\slashed{g}}}
\end{align*}
We shall only estimate the integral on $\Sigma_{t}^{\ub}$, which can be written as:
\begin{align*}
&-\delta^{2l+2}\int_{\Sigma_{t}^{\ub}}(T\psi)(R^{\alpha'-1}_{i}T^{l}\slashed{\Delta}
\mu)(R^{\alpha'+1}_{i}T^{l+1}\psi)d\ub'
d\mu_{\tilde{\slashed{g}}}
\\
&-\delta^{2l+2}\int_{\Sigma_{t}^{\ub}}\big((R_{i}T\psi)+\frac{1}{2}\text{tr}\leftexp{(R_{i})}{\tilde{\slashed{\pi}}}\big)(R^{\alpha'-1}_{i}T^{l}\slashed{\Delta}\mu)(R^{\alpha'}_{i}T^{l+1}\psi)d\ub'
d\mu_{\tilde{\slashed{g}}}\\
&:=-H'_{0}-H'_{1}
\end{align*}
By the estimates:
\begin{align*}
|R_{i}T\psi|\leq C\delta^{-1/2},\quad |\text{tr}\leftexp{(R_{i})}{\tilde{\slashed{\pi}}}|\leq C\delta
\end{align*}
we see that compared to $H'_{0}$, $H'_{1}$ is a lower order term, so here we only give the estimates for $H'_{0}$:
\begin{align*}
|H'_{0}|\leq C\delta^{-1/2+2l+2}\|R^{\alpha'-1}T^{l}\slashed{\Delta}\mu\|_{L^{2}(\Sigma_{t}^{\ub})}\|\ds R^{\alpha'}_{i}T^{l+1}\psi\|_{L^{2}(\Sigma_{t}^{\ub})}
\end{align*}
In view of Proposition \ref{Proposition lower order L2 mu}, this is bounded in terms of modified energies by:
\begin{align*}
|H'_{0}|&\leq C\int_{-r_{0}}^{t}\delta\Big(\mu_{m}^{-b_{|\alpha|+2}-1/2}(t')\sqrt{\widetilde{\Eb}_{\leq|\alpha|+2}(t',\ub)}+\mu_{m}^{-b_{|\alpha|+2}}(t')\sqrt{\widetilde{E}_{\leq|\alpha|+2}(t',\ub)}\Big)dt'\\
&\cdot\mu_{m}^{-b_{|\alpha|+2}-1/2}(t)\sqrt{\widetilde{\Eb}_{\leq|\alpha|+2}(t,\ub)}
\end{align*}
Again, we need to consider the ``shock part $H^{\prime S}_{0}$" and the ``non-shock part $H^{\prime N}_{0}$" separately.

For $H^{\prime S}_{0}$, we have:
\begin{align*}
|H^{\prime S}_{0}|&\leq C\mu^{-b_{|\alpha|+2}-1/2}_{m}(t)\sqrt{\widetilde{\Eb}_{\leq|\alpha|+2}(t,\ub)}\cdot\\
&\delta\Big(\sqrt{\widetilde{\Eb}_{\leq|\alpha|+2}(t,\ub)}\int_{t_{0}}^{t}\mu^{-b_{|\alpha|+2}-1/2}_{m}(t')dt'+\sqrt{\widetilde{E}_{\leq|\alpha|+2}(t,\ub)}\int_{t_{0}}^{t}\mu^{-b_{|\alpha|+2}}_{m}(t')dt'\Big)\\
&\leq \frac{C\delta}{\big(b_{|\alpha|+2}-1/2\big)}\mu^{-2b_{|\alpha|+2}}_{m}(t)\widetilde{\Eb}_{\leq|\alpha|+2}(t,\ub)\\
&+\frac{C\delta}{\big(b_{|\alpha|+2}-1\big)}\mu^{-2b_{|\alpha|+2}+1/2}_{m}(t)\sqrt{\widetilde{E}_{\leq|\alpha|+2}(t,\ub)}\sqrt{\widetilde{\Eb}_{\leq|\alpha|+2}(t,\ub)}
\end{align*}

For $H^{\prime N}_{0}$, we use the same
argument as we did in estimating \eqref{spacetime integral K1}. Since $\mu_{m}^{-1}(t')\leq C$ for $t'\in[-r_{0},t_{0}]$ with an absolute constant $C$, we have, by \ref{mu is decreasing}:
\begin{align*}
\int_{-r_{0}}^{t_{0}}\mu_{m}^{-b_{|\alpha|+2}-1/2}(t')\sqrt{\widetilde{\Eb}_{\leq|\alpha|+2}(t',\ub)}&\leq C\int_{-r_{0}}^{t_{0}}\mu_{m}^{-b_{|\alpha|+2}+1/2}(t')\sqrt{\widetilde{\Eb}_{\leq|\alpha|+2}(t',\ub)}dt'\\
&\leq C\int_{-r_{0}}^{t_{0}}\sqrt{\widetilde{\Eb}_{\leq|\alpha|+2}(t',\ub)}dt'\cdot\mu_{m}^{-b_{|\alpha|+2}+1/2}(t)
\end{align*}
The same argument applies to the integral involving $\delta\sqrt{\widetilde{E}_{\leq|\alpha|+2}(t,\ub)}$, so finally we obtain:
\begin{align*}
|H^{\prime N}_{0}|&\leq C\delta\int_{-r_{0}}^{t_{0}}\sqrt{\widetilde{\Eb}_{\leq|\alpha|+2}(t',\ub)}dt'\cdot\mu^{-2b_{|\alpha|+2}}_{m}(t)\sqrt{\widetilde{\Eb}_{\leq|\alpha|+2}(t,\ub)}\\
&+C\int_{-r_{0}}^{t_{0}}\delta\sqrt{\widetilde{E}_{\leq|\alpha|+2}(t',\ub)}dt'\cdot\mu^{-2b_{|\alpha|+2}}_{m}(t)\sqrt{\widetilde{\Eb}_{\leq|\alpha|+2}(t,\ub)}\\
&\leq C\delta\mu_{m}^{-2b_{|\alpha|+2}}(t)\Big(\widetilde{\Eb}_{\leq|\alpha|+2}(t,\ub)+\sqrt{\widetilde{\Eb}_{\leq|\alpha|+2}(t,\ub)}\sqrt{E_{\leq|\alpha|+2}(t,\ub)}\Big)
\end{align*}
We finally obtain the estimates for $|H_{0}'|$:
\begin{align}\label{H'0}
|H'_{0}|\leq C\delta\mu_{m}^{-2b_{|\alpha|+2}}(t)\Big(\widetilde{\Eb}_{\leq|\alpha|+2}(t,\ub)+\sqrt{\widetilde{\Eb}_{\leq|\alpha|+2}(t,\ub)}\sqrt{E_{\leq|\alpha|+2}(t,\ub)}\Big)
\end{align}

Finally, we estimate the spacetime integrals:
\begin{align*}
&-\delta^{2l+2}\int_{W^{t}_{\ub}}(\Lb T\psi)(R^{\alpha'}_{i}T^{l}\slashed{\Delta}\mu)(R^{\alpha'}_{i}T^{l+1}\psi)dt'd\ub'
d\mu_{\tilde{\slashed{g}}}\\
&-\delta^{2l+2}\int_{W^{t}_{\ub}}(T\psi)\big((\Lb+\frac{2}{t-\ub})(R^{\alpha'}_{i}T^{l}\slashed{\Delta}\mu)
\big)
(R^{\alpha'}_{i}T^{l+1}\psi)dt'd\ub'
d\mu_{\tilde{\slashed{g}}}\\
&-\delta^{2l+2}\int_{W^{t}_{\ub}}(T\psi)\cdot\widetilde{\text{tr}}\chib'
(R^{\alpha'}_{i}T^{l}\slashed{\Delta}\mu)\cdot
(R^{\alpha'}_{i}T^{l+1}\psi)dt'
d\ub'd\mu_{\tilde{\slashed{g}}}\\
&:=-V'_{1}-V'_{2}-V'_{3}
\end{align*}
By the estimates:
\begin{align*}
|T\psi|\leq C\delta^{-1/2},\quad |\Lb T\psi|\leq C\delta^{-1/2},\quad |\widetilde{\text{tr}}\chib'|\leq C\delta
\end{align*}
we see that, compared to $V'_{1}$, $V'_{3}$ is a lower order term. By Lemma 7.3 and \eqref{top mu final}, $V'_{1}$ is bounded by:
\begin{align*}
|V'_{1}|&\leq C\delta^{-1/2}\int_{-r_{0}}^{t}\sup_{\Sigma_{t'}^{\ub}}(\mu^{-1})
\delta^{l+1}\|\mu R^{\alpha'}_{i}T^{l}\slashed{\Delta}\mu\|_{L^{2}(\Sigma_{t'}^{\ub})}\\
&\cdot\delta\Big(\mu^{-b_{|\alpha|+2}}_{m}(t')\sqrt{\widetilde{E}_{\leq|\alpha|+2}(t,\ub)}+\mu^{-b_{|\alpha|+2}}_{m}(t')\sqrt{\widetilde{\Eb}_{\leq |\alpha|+2}(t',\ub)}\Big)dt'\\
&\leq C\int_{-r_{0}}^{t}\mu^{-2b_{|\alpha|+2}-1}_{m}(t')\Big(\sqrt{\widetilde{\Eb}_{\leq|\alpha|+2}(t',\ub)}+\sqrt{\widetilde{E}_{\leq|\alpha|+2}(t)}\Big)\\
&\cdot\delta\Big(\sqrt{\widetilde{E}_{\leq|\alpha|+2}(t,\ub)}+\sqrt{\widetilde{\Eb}_{\leq |\alpha|+2}(t',\ub)}\Big)dt'\\
&+C\int_{-r_{0}}^{t}\mu^{-2b_{|\alpha|+2}}_{m}(t')\sqrt{\int_{0}^{\ub}\widetilde{\Fb}_{\leq|\alpha|+2}(t,\ub')d\ub'}\Big(\sqrt{\widetilde{\Eb}_{\leq |\alpha|+2}(t',\ub)}+\sqrt{\widetilde{E}_{\leq|\alpha|+2}(t,\ub)}\Big)dt'
\end{align*}
Again, we consider the ``shock part" and ``non-shock part" separately. When $t'\in[-r_{0},t_{0}]$, since $\mu^{-1}_{m}(t')\leq C$, the above integrals are bounded by:
\begin{align*}
C\delta\int_{-r_{0}}^{t_{0}}\mu_{m}^{-2b_{|\alpha|+2}}(t')\big(\widetilde{E}_{\leq|\alpha|+2}(t,\ub)+\widetilde{\Eb}_{\leq |\alpha|+2}(t',\ub)\big)dt'\\
+C\int_{-r_{0}}^{t_{0}}\delta^{-1}\mu_{m}^{-2b_{|\alpha|+2}}(t')\big(\int_{0}^{\ub}
\widetilde{\Fb}_{\leq|\alpha|+2}(t',\ub')d\ub'\big)dt'
\end{align*}
While when $t'\in[t_{0},t]$, we have the following estimates:
\begin{align*}
&C\delta\big(\widetilde{\Eb}_{\leq|\alpha|+2}(t,\ub)+\widetilde{E}_{\leq|\alpha|+2}(t,\ub)\big)\int_{t_{0}}^{t}\mu^{-2b_{|\alpha|+2}-1}_{m}(t')dt'\\
&+\delta^{-1}\big(\int_{0}^{\ub}
\widetilde{\Fb}_{\leq|\alpha|+2}(t,\ub')d\ub'\big)\int_{t_{0}}^{t}\mu^{-2b_{|\alpha|+2}}_{m}(t')dt'\\
&\leq \frac{C\delta}{2b_{|\alpha|+2}}\mu^{-2b_{|\alpha|+2}}_{m}(t)\big(\widetilde{\Eb}_{\leq|\alpha|+2}(t,\ub)+\widetilde{E}_{\leq|\alpha|+2}(t,\ub)\big)\\
&+\frac{C\delta^{-1}}{\big(2b_{|\alpha|+2}-1\big)}\mu^{-2b_{|\alpha|+2}+1}_{m}(t)\int_{0}^{\ub}\widetilde{\Fb}_{\leq|\alpha|+2}(t,\ub')d\ub'
\end{align*}
Therefore $V'_{1}$ is bounded by:
\begin{align}\label{V'1}
&\frac{C\delta}{2b_{|\alpha|+2}}\mu^{-2b_{|\alpha|+2}}_{m}(t)\big(\widetilde{\Eb}_{\leq|\alpha|+2}(t,\ub)+\widetilde{E}_{\leq|\alpha|+2}(t,\ub)\big)\\\notag
&+\frac{C\delta^{-1}}{\big(2b_{|\alpha|+2}-1\big)}\mu^{-2b_{|\alpha|+2}+1}_{m}(t)\int_{0}^{\ub}\widetilde{\Fb}_{\leq|\alpha|+2}(t,\ub')d\ub'\\\notag
&+C\delta\int_{-r_{0}}^{t_{0}}\mu_{m}^{-2b_{|\alpha|+2}}(t')\big(\widetilde{E}_{\leq|\alpha|+2}(t,\ub)+\widetilde{\Eb}_{\leq |\alpha|+2}(t',\ub)\big)dt'\\\notag
&+C\int_{-r_{0}}^{t_{0}}\delta^{-1}\mu^{-2b_{|\alpha|+2}}_{m}(t')\big(\int_{0}^{\ub}
\widetilde{\Fb}_{\leq|\alpha|+2}(t',\ub')d\ub'\big)dt'
\end{align}

Now we estimate $V'_{2}$. First we write $V'_{2}$ as:
\begin{align*}
&\delta^{2l+2}\int_{W^{t}_{\ub}}(T\psi)\big(R_{i}(\Lb+\frac{2}{t-\ub})R^{\alpha'-1}_{i}T^{l}
\slashed{\Delta}\mu)\big)(R^{\alpha'}_{i}T^{l+1}\psi)dt'd\ub'
d\mu_{\tilde{\slashed{g}}}\\
&+\delta^{2l+2}\int_{W^{t}_{\ub}}(T\psi)(\leftexp{(R_{i})}{\Zb}R^{\alpha'-1}_{i}T^{l}
\slashed{\Delta}\mu)(R^{\alpha'}_{i}T^{l+1}\psi)dt'd\ub'
d\mu_{\tilde{\slashed{g}}}:=V'_{21}
+V'_{22}
\end{align*}
By the estimate:
\begin{align*}
|\leftexp{(R_{i})}{\Zb}|\leq C\delta
\end{align*}
$V'_{22}$ has the identical structure with $V'_{3}$, therefore is a lower order term compared to $V'_{1}$.

While for $V'_{21}$, integrating by parts, we have:
\begin{align*}
V'_{21}&=-\delta^{2l+2}\int_{W^{t}_{\ub}}
\big((R_{i}T\psi)+\frac{1}{2}\text{tr}\leftexp{(R_{i})}{\tilde{\slashed{\pi}}}\big)\big((\Lb+\frac{2}{t-\ub})R^{\alpha'-1}_{i}
T^{l}\slashed{\Delta}\mu)\big)(R^{\alpha'}_{i}T^{l+1}\psi)dt'd\ub'
d\mu_{\tilde{\slashed{g}}}\\
&-\delta^{2l+2}\int_{W^{t}_{\ub}}(T\psi)\big((\Lb+\frac{2}{t-\ub})R^{\alpha'-1}_{i}
T^{l}\slashed{\Delta}\mu)\big)(R^{\alpha'+1}_{i}T^{l+1}\psi)
dt'd\ub'd\mu_{\tilde{\slashed{g}}}:=-V'_{211}-V'_{212}
\end{align*}
By the estimates:
\begin{align*}
|T\psi|\leq C\delta^{-1/2},\quad |R_{i}T\psi+\frac{1}{2}\leftexp{(R_{i})}{\text{tr}\tilde{\slashed{\pi}}}|\leq C\delta^{-1/2}
\end{align*}
as well as Lemma 7.3, $V'_{211}$ is a lower order term compared to $V'_{212}$. To estimate $V'_{212}$, we first consider:
\begin{align}\label{V'212 1}
\delta^{2l+2}\int_{W^{t}_{\ub}}\frac{2}{t-\ub}(T\psi)(R^{\alpha'-1}_{i}T^{l}
\slashed{\Delta}\mu)(R^{\alpha'+1}_{i}T^{l+1}\psi)
dt'd\ub'd\mu_{\tilde{\slashed{g}}}
\end{align}
By proposition 7.5, this is bounded by:
\begin{align*}
C\int_{-r_{0}}^{t}\delta\Big(\mu_{m}^{-b_{|\alpha|+2}-1/2}(t')\sqrt{\widetilde{\Eb}_{\leq|\alpha|+2}(t',\ub)}+\mu_{m}^{-b_{|\alpha|+2}}(t')\sqrt{\widetilde{E}_{\leq|\alpha|+2}(t',\ub)}\Big)\\
\cdot\Big(\mu^{-b_{|\alpha|+2}-1/2}_{m}(t')\sqrt{\widetilde{\Eb}_{\leq |\alpha|+2}(t',\ub)}\Big)dt'
\end{align*}
Again, we consider the ``shock part" and the ``non-shock part" separately. When $t'\in[-r_{0},t_{0}]$, the above integrals are bounded by:
\begin{align*}
C\int_{-r_{0}}^{t_{0}}\mu_{m}^{-2b_{|\alpha|+2}}(t')\big(\delta^{2}\widetilde{E}_{\leq|\alpha|+2}(t',\ub)+\widetilde{\Eb}_{\leq |\alpha|+2}(t',\ub)\big)dt'
\end{align*}
While when $t'\in[t_{0},t]$, the above integrals are bounded by:
\begin{align*}
&C\widetilde{\Eb}_{\leq|\alpha|+2}(t,\ub)\int_{t_{0}}^{t}\mu^{-2b_{|\alpha|+2}-1}_{m}(t')dt'+C\delta^{2}\widetilde{E}_{\leq|\alpha|+2}(t,\ub)\int_{t_{0}}^{t}\mu^{-2b_{|\alpha|+2}}_{m}(t')dt'\\
&\leq \frac{C}{2b_{|\alpha|+2}}\mu^{-2b_{|\alpha|+2}}_{m}(t)\widetilde{\Eb}_{\leq|\alpha|+2}(t,\ub)+\frac{C}{\big(2b_{|\alpha|+2}-1\big)}\mu^{-2b_{|\alpha|+2}+1}_{m}(t)\delta^{2}\widetilde{E}_{\leq|\alpha|+2}(t,\ub)
\end{align*}
Therefore the spacetime integral \eqref{V'212 1} is bounded by:
\begin{align}\label{estimates for V'212 1}
&\frac{C}{2b_{|\alpha|+2}}\mu^{-2b_{|\alpha|+2}}_{m}(t)\widetilde{\Eb}_{\leq|\alpha|+2}(t,\ub)+\frac{C}{\big(2b_{|\alpha|+2}-1\big)}\mu^{-2b_{|\alpha|+2}+1}_{m}(t)\delta^{2}\widetilde{E}_{\leq|\alpha|+2}(t,\ub)\\\notag
&+C\int_{-r_{0}}^{t_{0}}\mu_{m}^{-2b_{|\alpha|+2}}(t')\big(\delta^{2}\widetilde{E}_{\leq|\alpha|+2}(t',\ub)+\widetilde{\Eb}_{\leq |\alpha|+2}(t',\ub)\big)dt'
\end{align}

We end this section by estimating the spacetime integral:
\begin{align}\label{V'212 2}
\delta^{2l+2}\int_{W^{t}_{\ub}}(T\psi)(\Lb R^{\alpha'-1}_{i}T^{l}\slashed{\Delta}
\mu)(R^{\alpha'+1}_{i}T^{l+1}\psi)
dt'd\ub'd\mu_{\tilde{\slashed{g}}}
\end{align}
which is bounded by:
\begin{align*}
C\delta^{-1/2}\int_{-r_{0}}^{t}\delta^{l+1}\|\Lb R^{\alpha'-1}_{i}T^{l}\slashed{\Delta}\mu
\|_{L^{2}(\Sigma_{t'}^{\ub})}\delta^{l+1}\|R^{\alpha'+1}_{i}T^{l+1}\psi\|_{L^{2}(\Sigma_{t'}^{\ub})}dt'
\end{align*}
By the propagation equation:
\begin{align*}
\Lb \mu=m+\mu e
\end{align*}
we have:
\begin{align*}
\Lb R^{\alpha'-1}_{i}T^{l}
\slashed{\Delta}\mu=-\frac{dc^{2}}{d\rho}\psi_{0}
TR^{\alpha'-1}_{i}T^{l}
\slashed{\Delta}\psi_{0}+\mu\psi_{0}\Lb
R^{\alpha'-1}_{i}T^{l}\slashed{\Delta}
\psi_{0}+\lot
\end{align*}
Here the lower order terms $\lot$ can be
bounded in the same fashion as \eqref{estimates for V'212 1}. Now we are going to bound $\delta^{l+1}\|TR^{\alpha'-1}_{i}T^{l}
\slashed{\Delta}\psi_{0}\|_{L^{2}(\Sigma_{t'}^{\ub})}$ in terms of $\sqrt{E_{\leq|\alpha|+2}(t')}$ and $\delta^{l+1}\|\mu\Lb R^{\alpha'-1}_{i}T^{l}
\slashed{\Delta}\psi_{0}\|_{L^{2}(\Sigma_{t'}^{\ub})}$ in terms of $\sqrt{\underline{E}_{\leq|\alpha|+2}(t')}$, so the latter is a lower order term with respect to the behavior of $\delta$. We only estimate the contribution from the former one.

In view of the estimate:
\begin{align*}
|\psi_{0}|\leq C\delta^{1/2}
\end{align*}
This contribution to \eqref{V'212 2} is bounded by:
\begin{align*}
&C\int_{-r_{0}}^{t}\delta\sqrt{E_{\leq|\alpha|+2}(t',\ub)}\mu^{-1/2}_{m}(t')\sqrt{\underline{E}_{\leq|\alpha|+2}(t',\ub)}dt'\\
&\leq C\int_{-r_{0}}^{t}\mu^{-2a-1/2}_{m}(t')\delta\sqrt{\widetilde{E}_{\leq|\alpha|+2}(t',\ub)}\sqrt{\widetilde{\Eb}_{\leq |\alpha|+2}(t',\ub)}dt'
\end{align*}
Again, considering the ``shock part" and ``non-shock part" separately, we have the finally estimates for \eqref{V'212 2}:
\begin{align}\label{estimates for V'212 2}
&C\int_{-r_{0}}^{t_{0}}\mu^{-2b_{|\alpha|+2}}_{m}(t')\big(\delta^{2}\widetilde{E}_{\leq|\alpha|+2}(t',\ub)+\widetilde{\Eb}_{\leq |\alpha|+2}(t',\ub)\big)dt'\\\notag
&+\frac{C}{\big(2b_{|\alpha|+2}-1/2\big)}\mu_{m}^{-2b_{|\alpha|+2}+1/2}(t)\big(\delta^{2}\widetilde{E}_{\leq|\alpha|+2}(t,\ub)+\widetilde{\Eb}_{\leq|\alpha|+2}(t,\ub)\big)
\end{align}
This completes the error estimates for the top order optical terms. As we mentioned at the beginning of this subsection, although we only considered the variations $R^{\alpha+1}_{i}\psi$ and $R^{\alpha'}_{i}T^{l+1}\psi$ with $|\alpha'|+l=|\alpha|$, all the estimates in this subsection are also true for the variations $Z^{\alpha+1}_{i}\psi$ and $Z^{\alpha'}_{i}T^{l+1}\psi$ respectively. Here $Z_{i}$ is either $R_{i}$ or $Q$.

\section{Top Order Energy Estimates}
With the estimates for the contributions from top order optical terms as well as lower order optical terms, we are ready to complete the top order energy estimates, namely, the energy estimates for the variations of order up to $|\alpha|+2$. As we have pointed out, we allow the top order energies to blow up as shocks form. So in this section, we shall prove that the modified energies $\widetilde{E}_{\leq|\alpha|+2}(t,\ub)$, $\widetilde{\Eb}_{\leq |\alpha|+2}(t,\ub)$ and $\widetilde{\Fb}_{\leq|\alpha|+2}(t,\ub)$, $\widetilde{F}_{\leq |\alpha|+2}(t,\ub)$ are bounded by initial data. Therefore we obtain a rate for the possible blow up of the top order energies.

\subsection{Estimates associated to $K_{1}$}
We start with the energy inequality for $Z^{\alpha'+1}_{i}\psi$ as we obtained in Section 6. Here $Z_{i}$ is any one of $R_{i}$ $Q$ and $T$.
\begin{align*}
\sum_{|\alpha'|\leq |\alpha|}\delta^{2l'}\Big(\Eb[Z^{\alpha'+1}_{i}\psi](t,\ub)+\Fb[Z^{\alpha'+1}_{i}\psi](t,\ub)+K[Z^{\alpha'+1}_{i}\psi](t,\ub)\Big)\\\leq C\sum_{|\alpha'|\leq|\alpha|}\delta^{2l'}\Eb[Z^{\alpha'+1}_{i}\psi](-r_{0},\ub)+C\sum_{|\alpha'|\leq |\alpha|}\delta^{2l'}\int_{W^{t}_{\ub}}c^{-2}
\widetilde{Q}_{1,|\alpha'|+2}
\end{align*}
where $l'$ is the number of $T$s' appearing in the string of $Z^{\alpha'}_{i}$.
In the spacetime integral $\int_{W^{t}_{\ub}}c^{-2}\widetilde{Q}_{1,\alpha'+2}$ we have the contributions from the deformation tensor of $K_{1}$,
which have been investigated in Section 6. Actually, if we choose $N_{\text{top}}$ to be large enough, then we can bound $\|\slashed{d}Z^{\beta}_{i}\mu\|_{L^{\infty}(\Sigma_{t}^{\ub})}$ in terms of initial data by using the same argument as in Section 4.2 for $|\beta|\leq N_{\infty}+1$.

Another contribution of the spacetime integral $\int_{W^{t}_{\ub}}c^{-2}\widetilde{Q}_{1,|\alpha'|+2}$ comes from $\int_{W^{t}_{\ub}}\dfrac{1}{c}\widetilde{\rho}_{|\alpha'|+2}\cdot \Lb Z^{\alpha'+1}_{i}\psi$, namely, the deformation tensor of commutators, which has been studied intensively in the
last section. We first consider the lower order optical contributions, which are bounded by (See \eqref{lower order optical terms separately K1 a}):
\begin{align}\label{K1 RHS lower order}
&C\int_{-r_{0}}^{t}\delta^{2} E_{\leq|\alpha|+2}(t',\ub)dt'+C\delta^{-1/2}\int_{0}^{\ub}\Fb_{\leq|\alpha|+2}(t,\ub')d\ub'+C\delta^{1/2} K_{\leq|\alpha|+2}(t,\ub)\\\notag
&\leq C\mu^{-2b_{|\alpha|+2}}_{m}(t)\left(\int_{-r_{0}}^{t}\delta^{2}\widetilde{E}_{\leq|\alpha|+2}(t',\ub)dt'+\delta^{-1/2}\int_{0}^{\ub}\widetilde{\Fb}_{\leq|\alpha|+2}(t,\ub')d\ub'+\delta^{1/2}\widetilde{K}_{\leq|\alpha|+2}(t,\ub)\right).
\end{align}
Here we define the following non-decreasing quantity in $t$:
\begin{align*}
\widetilde{K}_{\leq|\alpha|+2}(t,\ub):=\sup_{t'\in[-r_{0},t]}\{\mu_{m}^{2b_{|\alpha|+2}}(t')K_{\leq|\alpha|+2}(t',\ub)\}
\end{align*}

By \eqref{estimates for H0}, \eqref{spacetime II K1 lower order}, \eqref{V12}, \eqref{V11}, \eqref{V11 lower order}, \eqref{V11 lower order 1}, \eqref{H'0}, \eqref{V'1}, \eqref{estimates for V'212 1} and \eqref{estimates for V'212 2} the contribution of top order optical terms are bounded as (provided that $\delta$ is sufficiently small.):
\begin{align}\label{K1 RHS top order}
&\frac{C}{\big(b_{|\alpha|+2}-1/2\big)}\mu_{m}^{-2b_{|\alpha|+2}}(t)\widetilde{\Eb}_{\leq |\alpha|+2}(t,\ub)+C_{\epsilon}\mu^{-2b_{|\alpha|+2}}_{m}(t)\int_{-r_{0}}^{t_{0}}\widetilde{\Eb}_{\leq|\alpha|+2}(t',\ub)dt'\\\notag&+\epsilon\mu^{-2b_{|\alpha|+2}}_{m}(t)\widetilde{\Eb}_{\leq |\alpha|+2}(t,\ub)
+\frac{C\delta^{-1}}{\big(2b_{|\alpha|+2}-1\big)}\mu_{m}^{-2b_{|\alpha|+2}+1}(t)\int_{0}^{\ub}\widetilde{\Fb}_{\leq|\alpha|+2}(t,\ub')d\ub'\\\notag
&+C\delta\int_{-r_{0}}^{t_{0}}\mu_{m}^{-2b_{|\alpha|+2}}(t')\Big(\widetilde{E}_{\leq|\alpha|+2}(t',\ub)+\delta^{-2}\int_{0}^{\ub}
\widetilde{\Fb}_{\leq|\alpha|+2}(t',\ub')d\ub'\Big)dt'\\\notag
&+C\delta\mu_{m}^{-2b_{|\alpha|+2}}(t)\widetilde{\Eb}_{\leq |\alpha|+2}(t,\ub)+C\mu^{-2b_{|\alpha|+2}+1/2}_{m}(t)\Big(\delta^{3/2}\widetilde{E}_{\leq|\alpha|+2}(t,\ub)+\delta^{3/2}\widetilde{\Eb}_{\leq |\alpha|+2}(t,\ub)\Big)\\\notag
&+\frac{C}{2b_{|\alpha|+2}}\mu^{-2b_{|\alpha|+2}}_{m}(t)\widetilde{\Eb}_{\leq |\alpha|+2}(t,\ub)+\frac{C\delta}{2b_{|\alpha|+2}}\mu_{m}^{-2b_{|\alpha|+2}}(t)\widetilde{E}_{\leq|\alpha|+2}(t,\ub)\\\notag
\end{align}
Substituting these contributions into the energy inequality, and use the fact that $\mu_{m}(t)\leq 1$, we obtain:
\begin{align*}
&\mu_{m}^{2b_{|\alpha|+2}}(t)\sum_{|\alpha'|\leq |\alpha|}\delta^{2l'}\Big(\Eb[Z^{\alpha'+1}_{i}\psi](t,\ub)+\Fb[Z^{\alpha'+1}_{i}\psi](t,\ub)+K[Z^{\alpha'+1}_{i}\psi](t,\ub)\Big)\\
&\leq C\mu^{2b_{|\alpha|+2}}_{m}(t)\sum_{|\alpha'|\leq|\alpha|}\delta^{2l'}\Eb[Z^{\alpha'+1}_{i}\psi](-r_{0},\ub)
+\frac{C}{\big(b_{|\alpha|+2}-1/2\big)}\widetilde{\Eb}_{\leq |\alpha|+2}(t,\ub)\\&+C_{\epsilon}\int_{-r_{0}}^{t}\widetilde{\Eb}_{\leq|\alpha|+2}(t',\ub)dt'
+\epsilon\widetilde{\Eb}_{\leq |\alpha|+2}(t,\ub)
+\frac{C\delta^{-1}}{\big(2b_{|\alpha|+2}-1\big)}\int_{0}^{\ub}\widetilde{\Fb}_{\leq|\alpha|+2}(t,\ub')d\ub'\\\notag
&+C\delta\int_{-r_{0}}^{t}\Big(\widetilde{E}_{\leq|\alpha|+2}(t',\ub)+\delta^{-2}\int_{0}^{\ub}
\widetilde{\Fb}_{\leq|\alpha|+2}(t',\ub')d\ub'\Big)dt'
+C\delta\widetilde{\Eb}_{\leq |\alpha|+2}(t,\ub)\\
&+C\Big(\delta^{3/2}\widetilde{E}_{\leq|\alpha|+2}(t,\ub)+\delta^{3/2}\widetilde{\Eb}_{\leq |\alpha|+2}(t,\ub)\Big)
+\frac{C}{2b_{|\alpha|+2}}\widetilde{\Eb}_{\leq |\alpha|+2}(t,\ub)\\
&+\frac{C\delta}{2b_{|\alpha|+2}}\widetilde{E}_{\leq|\alpha|+2}(t,\ub)+C\delta\widetilde{K}_{\leq|\alpha|+2}(t,\ub)\\\notag
\end{align*}
Now the right hand side of the above inequality is non-decreasing in $t$, so the above inequality is also valid if we replace ``$t$" by any $t'\in[-r_{0},t]$ on the left hand side:
\begin{align*}
&\mu_{m}^{2b_{|\alpha|+2}}(t')\sum_{|\alpha'|\leq |\alpha|}\Big(\Eb[Z^{\alpha'+1}_{i}\psi](t',\ub)+\Fb[Z^{\alpha'+1}_{i}\psi](t',\ub)+K[Z^{\alpha'+1}_{i}\psi](t',\ub)\Big)\\&\leq C\sum_{|\alpha'|\leq|\alpha|}\delta^{2l'}\Eb[Z^{\alpha'+1}_{i}\psi](-r_{0},\ub)
+\frac{C}{\big(b_{|\alpha|+2}-1/2\big)}\widetilde{\Eb}_{\leq |\alpha|+2}(t,\ub)\\
&+C\int_{-r_{0}}^{t}\widetilde{\Eb}_{\leq|\alpha|+2}(t',\ub)dt'+\epsilon\widetilde{\Eb}_{\leq |\alpha|+2}(t,\ub)
+\frac{C\delta^{-1}}{\big(2b_{|\alpha|+2}-1\big)}\int_{0}^{\ub}\widetilde{\Fb}_{\leq|\alpha|+2}(t,\ub')d\ub'\\\notag
&+C\delta\int_{-r_{0}}^{t}\Big(\widetilde{E}_{\leq|\alpha|+2}(t',\ub)+\delta^{-2}\int_{0}^{\ub}
\widetilde{\Fb}_{\leq|\alpha|+2}(t',\ub')d\ub'\Big)dt'
+C\delta\widetilde{\Eb}_{\leq |\alpha|+2}(t,\ub)\\&+C\Big(\delta^{3/2}\widetilde{E}_{\leq|\alpha|+2}(t,\ub)+\delta^{3/2}\widetilde{\Eb}_{\leq |\alpha|+2}(t,\ub)\Big)
+\frac{C}{2b_{|\alpha|+2}}\widetilde{\Eb}_{\leq |\alpha|+2}(t,\ub)\\
&+\frac{C\delta}{2b_{|\alpha|+2}}\widetilde{E}_{\leq|\alpha|+2}(t,\ub)+C\delta\widetilde{K}_{\leq|\alpha|+2}(t,\ub)\\\notag
\end{align*}
For each term in the sum on the left hand side of above inequality, we keep it on the left hand side and drop all the other terms. Then taking supremum of the term we kept with respect to $t'\in[-r_{0},t]$. Repeat this process for all the terms on the left hand side, we finally obtain:
\begin{align*}
&\widetilde{\Eb}_{\leq |\alpha|+2}(t,\ub)+\widetilde{\Fb}_{\leq|\alpha|+2}(t,\ub)+\widetilde{K}_{\leq|\alpha|+2}(t,\ub)\\&\leq C\sum_{|\alpha'|\leq|\alpha|}\delta^{2l'}\Eb[Z^{\alpha'+1}_{i}\psi](-r_{0},\ub)
+\frac{C}{\big(b_{|\alpha|+2}-1/2\big)}\widetilde{\Eb}_{\leq |\alpha|+2}(t,\ub)\\
&+C_{\epsilon}\int_{-r_{0}}^{t}\widetilde{\Eb}_{\leq|\alpha|+2}(t',\ub)dt'+\epsilon\widetilde{\Eb}_{\leq |\alpha|+2}(t,\ub)
+\frac{C\delta^{-1}}{\big(2b_{|\alpha|+2}-1\big)}\int_{0}^{\ub}\widetilde{\Fb}_{\leq|\alpha|+2}(t,\ub')d\ub'\\\notag
&+C\delta\int_{-r_{0}}^{t}\Big(\widetilde{E}_{\leq|\alpha|+2}(t',\ub)+\delta^{-2}\int_{0}^{\ub}
\widetilde{\Fb}_{\leq|\alpha|+2}(t',\ub')d\ub'\Big)dt'
+C\delta\widetilde{\Eb}_{\leq |\alpha|+2}(t,\ub)\\
&+C\Big(\delta^{3/2}\widetilde{E}_{\leq |\alpha|+2}(t,\ub)+\delta^{3/2}\widetilde{\Eb}_{\leq |\alpha|+2}(t,
ub)\Big)\\
&+\boxed{\frac{C}{2b_{|\alpha|+2}}\widetilde{\Eb}_{\leq |\alpha|+2}(t,\ub)}+\frac{C\delta}{2b_{|\alpha|+2}}\widetilde{E}_{\leq|\alpha|+2}(t,\ub)+C\delta\widetilde{K}_{\leq|\alpha|+2}(t,\ub)\\\notag
\end{align*}
The control on the boxed term relies on Remark \ref{a dependence}--since $\frac{C}{b_{|\alpha|+2}}$ is suitably small, the boxed term can be absorbed by the left hand side.
So if we choose $\epsilon$ and $\delta$ small enough, we obtain:
\begin{align}\label{G1 Gronwall}
&\widetilde{\Eb}_{\leq |\alpha|+2}(t,\ub)+\widetilde{\Fb}_{\leq|\alpha|+2}(t,\ub)+\widetilde{K}_{\leq|\alpha|+2}(t,\ub)\\\notag
&\leq C\sum_{|\alpha'|\leq|\alpha|}\delta^{2l'}\Eb[Z^{\alpha+1}_{i}\psi](-r_{0},\ub)
+C_{\epsilon}\int_{-r_{0}}^{t}\widetilde{\Eb}_{\leq|\alpha|+2}(t',\ub)dt'\\\notag
&+C\delta^{-1}\int_{0}^{\ub}\widetilde{\Fb}_{\leq|\alpha|+2}(t,\ub')d\ub'+C\delta\int_{-r_{0}}^{t}\widetilde{E}_{\leq|\alpha|+2}(t',\ub)dt'
+C\delta^{3/2}\widetilde{E}_{\leq|\alpha|+2}(t,\ub)
\end{align}
Now we only keep $\widetilde{\Eb}_{\leq |\alpha|+2}(t,\ub)$ on the left hand side of \eqref{G1 Gronwall}
\begin{align*}
\widetilde{\Eb}_{\leq |\alpha|+2}(t,\ub)&\leq C\sum_{|\alpha'|\leq|\alpha|}\delta^{2l'}\Eb[Z^{\alpha'+1}_{i}\psi](-r_{0},\ub)
+C_{\epsilon}\int_{-r_{0}}^{t}\widetilde{\Eb}_{\leq|\alpha|+2}(t',\ub)dt'\\\notag&+C\delta^{-1}\int_{0}^{\ub}\widetilde{\Fb}_{\leq|\alpha|+2}(t,\ub')d\ub'+C\delta\int_{-r_{0}}^{t}\widetilde{E}_{\leq|\alpha|+2}(t',\ub)dt'
+C\delta^{3/2}\widetilde{E}_{\leq|\alpha|+2}(t,\ub)
\end{align*}
Then by using Gronwall we have:
\begin{align}\label{G1 Gronwall 1}
\widetilde{\Eb}_{\leq |\alpha|+2}(t,\ub)&\leq C\sum_{|\alpha'|\leq|\alpha|}\delta^{2l'}\Eb[Z^{\alpha'+1}_{i}\psi](-r_{0},\ub)+C\delta^{-1}\int_{0}^{\ub}\widetilde{\Fb}_{\leq|\alpha|+2}(t,\ub')d\ub'\\\notag&+C\delta\int_{-r_{0}}^{t}\widetilde{E}_{\leq|\alpha|+2}(t',\ub)dt'
+C\delta^{3/2}\widetilde{E}_{\leq|\alpha|+2}(t,\ub)
\end{align}
Keeping only $\widetilde{\Fb}_{\leq|\alpha|+2}(t,\ub)$ on the left hand side of \eqref{G1 Gronwall} and substituting the above estimates for $\widetilde{\Eb}_{\leq |\alpha|+2}(t,\ub)$ gives us:
\begin{align*}
\widetilde{\Fb}_{\leq|\alpha|+2}(t,\ub)&\leq C\sum_{|\alpha'|\leq|\alpha|}\delta^{2l'}\Eb[Z^{\alpha'+1}_{i}\psi](-r_{0},\ub)
+C\delta^{-1}\int_{0}^{\ub}\widetilde{\Fb}_{\leq|\alpha|+2}(t,\ub')d\ub'\\\notag&+C\delta\int_{-r_{0}}^{t}\widetilde{E}_{\leq|\alpha|+2}(t',\ub)dt'
+C\delta^{3/2}\widetilde{E}_{\leq|\alpha|+2}(t,\ub)
\end{align*}
Then again by using Gronwall we obtain:
\begin{align}\label{Final H1}
\widetilde{\Fb}_{\leq|\alpha|+2}(t,\ub)&\leq C\sum_{|\alpha'|\leq|\alpha|}\delta^{2l'}\Eb[Z^{\alpha'+1}_{i}\psi](-r_{0},\ub)\\\notag
&+C\delta\int_{-r_{0}}^{t}\widetilde{E}_{\leq|\alpha|+2}(t',\ub)dt'
+C\delta^{3/2}\widetilde{E}_{\leq|\alpha|+2}(t,\ub)
\end{align}
Note that since $0\leq \ub\leq \delta$, all the constants $C$ on the above do not depend on $\delta$. Since the right hand side of \eqref{Final H1} is increasing in $\ub$, substituting this in \eqref{G1 Gronwall 1} gives us:
\begin{align}\label{Final G1}
\widetilde{\Eb}_{\leq |\alpha|+2}(t,\ub)&\leq C\sum_{|\alpha'|\leq|\alpha|}\delta^{2l'}\Eb[Z^{\alpha'+1}_{i}\psi](-r_{0},\ub)\\\notag
&+C\delta\int_{-r_{0}}^{t}\widetilde{E}_{\leq|\alpha|+2}(t',\ub)dt'
+C\delta^{3/2}\widetilde{E}_{\leq|\alpha|+2}(t,\ub)
\end{align}
Substituting \eqref{Final H1} and \eqref{Final G1} in \eqref{G1 Gronwall} gives us:
\begin{align}\label{K1 top final}
&\widetilde{\Eb}_{\leq |\alpha|+2}(t,\ub)+\widetilde{\Fb}_{\leq|\alpha|+2}(t,\ub)+\widetilde{K}_{\leq|\alpha|+2}(t,\ub)\\\notag
&\leq C\sum_{|\alpha'|\leq|\alpha|}\delta^{2l'}\Eb[Z^{\alpha'+1}_{i}\psi](-r_{0},\ub)
+C\delta\int_{-r_{0}}^{t}\widetilde{E}_{\leq|\alpha|+2}(t',\ub)dt'
+C\delta^{3/2}\widetilde{E}_{\leq|\alpha|+2}(t,\ub)
\end{align}
This completes the top order energy estimates associated to $K_{1}$.

\subsection{Estimates associated to $K_{0}$}
Now we turn to the top order energy estimates for $K_{0}$.

We first start with the energy identity for $Z^{\alpha'+1}_{i}\psi$, where $Z_{i}$ is any one of $R_{i}$, $Q$ and $T$:
\begin{align*}
&\sum_{|\alpha'|\leq|\alpha|}\delta^{2l'}\left(E[Z^{\alpha'+1}_{i}\psi](t,\ub)+F[Z^{\alpha'+1}_{i}\psi](t,\ub)\right)\\&\leq C\sum_{|\alpha'|\leq |\alpha|}\delta^{2l'}E[Z^{\alpha'+1}_{i}\psi](-r_{0},\ub)+C\sum_{|\alpha'|\leq |\alpha|}\int_{W^{t}_{\ub}}c^{-2}
\widetilde{Q}_{0,|\alpha'|+2}
\end{align*}
Again, $l'$ is the number of $T$s' in the string of $Z^{\alpha'+1}_{i}$.

In the spacetime integral $\int_{W^{t}_{\ub}}c^{-2}\widetilde{Q}_{0,\alpha'+2}$ we have the contributions from the 
deformation tensor of $K_{0}$, which have been investigated in section 6 and also the contribution of the spacetime integral 
from $\int_{W^{t}_{\ub}}\dfrac{1}{c}\widetilde{\rho}_{|\alpha'|+2}\cdot LZ^{\alpha'+1}_{i}\psi$, namely, the deformation tensor of commutators, which has been studied intensively in the
last section. We first consider the lower order optical contributions, which are bounded by (See \eqref{lower order optical terms separately K0 a} and \eqref{K1 top final} and provided that $\delta$ is sufficiently small):
\begin{align}\label{K0 RHS lower order optical}
&C\int_{-r_{0}}^{t}\delta^{1/2}E_{\leq|\alpha|+2}(t',\ub)dt'+C\delta^{1/2} K_{\leq|\alpha|+2}(t,\ub)+\delta^{-1/2}\int_{0}^{\ub}\Fb_{\leq|\alpha|+2}(t,\ub')d\ub'\\\notag
&\leq C\Eb_{\leq|\alpha|+2}(-r_{0},\ub)+C\delta^{1/2}\mu^{-2b_{|\alpha|+2}}_{m}(t)\int_{-r_{0}}^{t}\widetilde{E}_{\leq|\alpha|+2}(t',\ub)dt'.
\end{align}

By \eqref{top gronwall 1}, \eqref{top optical 1}, \eqref{K1 top final}, \eqref{top optical K0 mu-shock} and \eqref{top optical K0 mu-non-shock}, the top order optical contributions are bounded by (provided that $\delta$ is sufficiently small and $b_{|\alpha|+2}$ is large enough):
\begin{align}\label{K0 RHS top optical}
&\int_{-r_{0}}^{t}\mu_{m}^{-2b_{|\alpha|+2}}(t')\widetilde{E}_{\leq|\alpha|+2}(t',\ub)dt'+\delta\mu^{-2b_{|\alpha|+2}}_{m}(t)\widetilde{E}_{\leq|\alpha|+2}(t,\ub)\\\notag
&+\frac{C}{2b_{|\alpha|+2}}\mu_{m}^{-2b_{|\alpha|+2}}(t)\widetilde{E}_{\leq|\alpha|+2}(t,\ub)\\\notag&+\frac{C}{\big(2b_{|\alpha|+2}-1\big)}\mu^{-2b_{|\alpha|+2}+1}_{m}(t)\Big(\int_{-r_{0}}^{t}\widetilde{E}_{\leq|\alpha|+2}(t',\ub)dt'+\widetilde{E}_{\leq|\alpha|+2}(t,\ub)\Big)\\\notag
&+C\delta^{-1+2l'}\sum_{|\alpha'|\leq|\alpha|}\Eb[Z^{\alpha'+1}_{i}\psi](-r_{0},\ub)
\end{align}
Substituting \eqref{K0 RHS lower order optical} and \eqref{K0 RHS top optical} into the energy inequality, and use the fact that $\mu_{m}(t)\leq 1$, we obtain:
\begin{align*}
&\sum_{|\alpha'|\leq|\alpha|}\mu_{m}^{2b_{|\alpha|+2}}(t')\delta^{2l'}\Big(E[Z^{\alpha'+1}_{i}\psi](t,\ub)+F[Z^{\alpha'+1}_{i}\psi](t,\ub)\Big)\leq\\& C\sum_{|\alpha'|\leq|\alpha|}\delta^{2l'}E[Z^{\alpha'+1}_{i}\psi](-r_{0},\ub)
+C\sum_{|\alpha'|\leq|\alpha|}\delta^{-1+2l'}\Eb[Z^{\alpha'+1}_{i}\psi](-r_{0},\ub)+\\ &C\int_{-r_{0}}^{t}\widetilde{E}_{\leq|\alpha|+2}(t',\ub)dt'+\delta\widetilde{E}_{\leq|\alpha|+2}(t,\ub)
+\frac{C}{2b_{|\alpha|+2}}\widetilde{E}_{\leq|\alpha|+2}(t,\ub)
\end{align*}
Since the right hand side of the above is non-decreasing in $t$, the inequality is true if we replace $t$ by any $t'\in[-r_{0},t]$ on the left hand side:
\begin{align*}
&\sum_{|\alpha'|\leq|\alpha|}\delta^{2l'}\mu_{m}^{2b_{|\alpha|+2}}(t')\Big(E[Z^{\alpha'+1}_{i}\psi](t',\ub)+F[Z^{\alpha'+1}_{i}\psi](t',\ub)\Big)\leq \\\notag &C\sum_{|\alpha'|\leq|\alpha|}\delta^{2l'}E[Z^{\alpha'+1}_{i}\psi](-r_{0},\ub)+C\delta^{-1+2l'}\sum_{|\alpha'|\leq|\alpha|}\Eb[Z^{\alpha'+1}_{i}\psi](-r_{0},\ub)\\
&\leq C\int_{-r_{0}}^{t}\widetilde{E}_{\leq|\alpha|+2}(t',\ub)dt'+\delta\widetilde{E}_{\leq|\alpha|+2}(t,\ub)
+\frac{C}{2b_{|\alpha|+2}}\widetilde{E}_{\leq|\alpha|+2}(t,\ub)
\end{align*}
As before, taking supremum on the left hand side with respect to $t'\in[-r_{0},t]$, we obtain:
\begin{align*}
&\widetilde{E}_{\leq|\alpha|+2}(t,\ub)+\widetilde{F}_{\leq|\alpha|+2}(t,\ub)\\&\leq C
\sum_{|\alpha'|\leq|\alpha|}\delta^{2l'}E[Z^{\alpha'+1}_{i}\psi](-r_{0},\ub)+C\delta^{-1+2l'}\sum_{|\alpha'|\leq|\alpha|}\Eb[Z^{\alpha'+1}_{i}\psi](-r_{0},\ub)\\
&+ C\int_{-r_{0}}^{t}\widetilde{E}_{\leq|\alpha|+2}(t',\ub)dt'+\delta\widetilde{E}_{\leq|\alpha|+2}(t)
+\boxed{\frac{C}{2b_{|\alpha|+2}}\widetilde{E}_{\leq|\alpha|+2}(t,\ub)}
\end{align*}
Similarly, the control on the boxed term relies on Remark \ref{a dependence}--since $\frac{C}{b_{|\alpha|+2}}$ is suitably small, the boxed term can be absorbed by the left hand side. By choosing $\delta$ sufficiently small, we have:
\begin{align*}
\widetilde{E}_{\leq|\alpha|+2}(t)+\widetilde{F}_{\leq|\alpha|+2}(t,\ub)&\leq C
\sum_{|\alpha'|\leq|\alpha|}\delta^{2l'}E[Z^{\alpha'+1}_{i}\psi](-r_{0},\ub)
\\&+C\delta^{-1+2l'}\sum_{|\alpha'|\leq|\alpha|}\Eb[Z^{\alpha'+1}_{i}\psi](-r_{0},\ub)+C\int_{-r_{0}}^{t}\widetilde{E}_{\leq|\alpha|+2}(t',\ub)dt'
\end{align*}
Then keeping only $\widetilde{E}_{\leq|\alpha|+2}(t)$ on the left hand side gives us:
\begin{align*}
\widetilde{E}_{\leq|\alpha|+2}(t)&\leq C\sum_{|\alpha'|\leq|\alpha|}\delta^{2l'}E[Z^{\alpha'+1}_{i}\psi](-r_{0},\ub)
+\\&C\delta^{-1+2l'}\sum_{|\alpha'|\leq|\alpha|}\Eb[Z^{\alpha'+1}_{i}\psi](-r_{0},\ub)
+ C\int_{-r_{0}}^{t}\widetilde{E}_{\leq|\alpha|+2}(t',\ub)dt'
\end{align*}
Then using Gronwall, we have:
\begin{align*}
\widetilde{E}_{\leq|\alpha|+2}(t)\leq C\sum_{|\alpha'|\leq|\alpha|}\delta^{2l'}E[Z^{\alpha'+1}_{i}\psi](-r_{0},\ub)+C\delta^{-1+2l'}\sum_{|\alpha'|\leq|\alpha|}\Eb[Z^{\alpha'+1}_{i}\psi](-r_{0},\ub)
\end{align*}
Therefore
\begin{align}\label{K0 top final}
&\widetilde{E}_{\leq|\alpha|+2}(t,\ub)+\widetilde{F}_{\leq|\alpha|+2}(t,\ub)\leq \\\notag &C\sum_{|\alpha'|\leq|\alpha|}\delta^{2l'}E[Z^{\alpha'+1}_{i}\psi](-r_{0},\ub)+C\delta^{-1+2l'}\sum_{|\alpha'|\leq|\alpha|}\Eb[Z^{\alpha'+1}_{i}\psi](-r_{0},\ub)
\end{align}
Now we substitute this to \eqref{K1 top final} for $\widetilde{E}_{\leq|\alpha|+2}(t)$:
\begin{align}
&\widetilde{\Eb}_{\leq |\alpha|+2}(t,\ub)+\widetilde{\Fb}_{\leq|\alpha|+2}(t,\ub)+\widetilde{K}_{\leq|\alpha|+2}(t,\ub)\\\notag
&\leq C\sum_{|\alpha'|\leq|\alpha|}\delta^{2l'}\Eb[Z^{\alpha'+1}_{i}\psi](-r_{0},\ub)+C \sum_{|\alpha'|\leq|\alpha|}\delta^{2l'+1}E[Z^{\alpha'+1}_{i}\psi](-r_{0},\ub)
\end{align}
If we denote:
\begin{align*}
\mathcal{D}^{\ub}_{|\alpha|+2}:=&\sum_{|\alpha'|+l'\leq|\alpha|}\delta^{2l'}E[Z^{\alpha'+1}_{i}\psi](-r_{0},\ub)+\delta^{-1+2l'}\sum_{|\alpha'|+l'\leq|\alpha|}\Eb[Z^{\alpha'+1}_{i}\psi](-r_{0},\ub)\\
+&\delta^{2l+2}\|F_{\alpha,l}\|_{L^{2}(\Sigma_{-r_{0}}^{\ub})}+\sum_{|\alpha'|+l'\leq|\alpha|+1}\delta^{2l'}\|Z^{\alpha'}_{i}T^{l'}\mu\|_{L^{2}(\Sigma_{-r_{0}}^{\ub})},\quad Z_{i}=R_{i}, Q.
\end{align*}
we can write the final top order energy estimates as:
\begin{align}\label{top energy final}
\widetilde{\Eb}_{\leq |\alpha|+2}(t,\ub)+\widetilde{\Fb}_{\leq|\alpha|+2}(t,\ub)+\widetilde{K}_{\leq|\alpha|+2}(t,\ub)&\leq C\delta\mathcal{D}^{\ub}_{|\alpha|+2}\\\notag
\widetilde{E}_{\leq|\alpha|+2}(t,\ub)+\widetilde{F}_{\leq|\alpha|+2}(t,\ub)&\leq C\mathcal{D}^{\ub}_{|\alpha|+2}
\end{align}

\section{Descent Scheme}
In the previous section, we have shown that the modified energies $\widetilde{E}_{\leq|\alpha|+2}(t), \widetilde{\Eb}_{\leq |\alpha|+2}(t)$ for the top order variations are bounded by the initial energies $\mathcal{D}^{\ub}_{|\alpha|+2}$. According to the definition, the modified energies go to zero when $\mu_{m}(t)$ goes to zero. This means the energy estimates obtained in the last section are not sufficient for us to close the argument when shock forms. However, based on those estimates, we shall show in this section, that if the order of derivative decreases, the power of $\mu_{m}(t)$ needed in the definition of modified energies also decreases. The key point is that after several steps, this power could be zero and finally we can bound the energies without any weights.

\subsection{Next-to-top order error estimates}
We first investigate the estimates associated to $K_{1}$. To improve the energy estimates for the next-to-the-top variations, we consider the spacetime integral (Keep in mind that the top order quantities are of order $|\alpha|+2$):
\begin{align}\label{K1 next to top spacetime}
&\int_{W^{t}_{\ub}}|T\psi||Z^{\alpha}_{i}\text{tr}\chib'||\Lb Z^{\alpha}_{i}\psi|dt'd\ub'd\mu_{\tilde{g}}\\\notag
&\leq C\delta^{-1/2}\int_{W^{t}_{\ub}}|Z^{\alpha}_{i}\text{tr}\chib'||\Lb Z^{\alpha}_{i}\psi|dt'd\ub'd\mu_{\tilde{g}}\\\notag
&\leq C\delta^{-1/2}\Big(\int_{W^{t}_{\ub}}|Z^{\alpha}_{i}\text{tr}\chib'|^{2}dt'd\ub'
d\mu_{\tilde{\slashed{g}}}\Big)^{1/2}
\cdot\Big(\int_{W^{t}_{\ub}}|\Lb Z^{\alpha}_{i}\psi|^{2}dt'd\ub'd\mu_{\tilde{\slashed{g}}}
\Big)^{1/2}\\\notag
&\leq C\delta^{-1/2}\Big(\int_{-r_{0}}^{t}\|Z^{\alpha}_{i}\text{tr}\chib'\|^{2}_{L^{2}(\Sigma_{t'}^{\ub})}dt'\Big)^{1/2}\cdot
\Big(\int_{0}^{\ub}\Fb[Z^{\alpha}_{i}\psi](t,\ub')d\ub'\Big)^{1/2}
\end{align}
Throughout this subsection, $Z_{i}$ is either $R_{i}$ or $Q$.
By Proposition 7.4:
\begin{align*}
\|Z^{\alpha}_{i}\text{tr}\chib'\|_{L^{2}(\Sigma_{t}^{\ub})}&\leq C\delta^{1/2}\int_{-r_{0}}^{t}\mu^{-1/2}_{m}(t')\sqrt{\underline{E}_{|\alpha|+2}(t',\ub)}dt'\\
&\leq C\delta^{1/2}\int_{-r_{0}}^{t}\mu^{-1/2-b_{|\alpha|+2}}_{m}(t')\sqrt{\widetilde{\Eb}_{\leq|\alpha|+2}(t',\ub)}dt'\\
&\leq C\delta^{1/2}\sqrt{\widetilde{\Eb}_{\leq |\alpha|+2}(t,\ub)}\int_{-r_{0}}^{t}\mu_{m}^{-b_{|\alpha|+2}-1/2}(t')dt'\\
&\leq C\delta^{1/2}\mu^{-b_{|\alpha|+2}+1/2}_{m}(t)\sqrt{\widetilde{\Eb}_{\leq |\alpha|+2}(t,\ub)}
\end{align*}
Then by the top order energy estimates obtained in the last section, the integral in the first factor of \eqref{K1 next to top spacetime} is bounded by:
\begin{align*}
C\delta\int_{-r_{0}}^{t}\mu^{-2b_{|\alpha|+2}+1}_{m}(t')\widetilde{\Eb}_{\leq|\alpha|+2}(t',\ub)dt'&\leq C\delta\widetilde{\Eb}_{\leq |\alpha|+2}(t,\ub)\int_{-r_{0}}^{t}\mu^{-2b_{|\alpha|+2}+1}_{m}(t')dt'\\
&\leq C\delta^{2}\mu^{-2b_{|\alpha|+2}+2}_{m}(t)\mathcal{D}^{\ub}_{|\alpha|+2}
\end{align*}
On the other hand, the second factor in \eqref{K1 next to top spacetime} is bounded by:
\begin{align*}
&\int_{0}^{\ub}\Fb[Z^{\alpha}_{i}\psi](t,\ub')d\ub'\leq \mu^{-2b_{|\alpha|+1}}_{m}(t)\int_{0}^{\ub}\sup_{t'\in[-r_{0},t]}\{\mu_{m}^{2b_{|\alpha|+1}}(t')\Fb[Z^{\alpha}_{i}\psi](t',\ub')\}d\ub'
\end{align*}
where $b_{|\alpha|+1}=b_{|\alpha|+2}-1$. Therefore \eqref{K1 next to top spacetime} is bounded by:
\begin{align}\label{K1 next to top estimates}
&C\delta^{1/2}\mu^{-2b_{|\alpha|+1}}_{m}(t)\sqrt{\mathcal{D}^{\ub}_{|\alpha|+2}}\sqrt{\int_{0}^{\ub}
\widetilde{\Fb}_{\leq|\alpha|+1}(t,\ub')d\ub'}\\\notag
\leq &C\delta^{2}\mu^{-2b_{|\alpha|+1}}_{m}(t)\mathcal{D}^{\ub}_{|\alpha|+2}+C\delta^{-1}\mu_{m}^{-2b_{|\alpha|+1}}(t)\int_{0}^{\ub}\widetilde{\Fb}_{\leq|\alpha|+1}(t,\ub')d\ub'
\end{align}
Next we consider the spacetime integral:
\begin{align}\label{K1T next to top spacetime}
&\delta^{2l'+2}\int_{W^{t}_{\ub}}|T\psi||Z^{\alpha'}_{i}T^{l'}\slashed{\Delta}\mu||\Lb Z^{\alpha'}_{i}T^{l'+1}\psi|dt'd\ub'd\mu_{\tilde{\slashed{g}}}\\\notag
&\leq C\delta^{-1/2}\Big(\int_{-r_{0}}^{t}\delta^{l'+1}\|Z^{\alpha'}_{i}T^{l'}\slashed{\Delta}\mu\|^{2}_{L^{2}(\Sigma_{t'}^{\ub})}\Big)^{1/2}
\Big(\int_{0}^{\ub}\delta^{l'+1}
\Fb[Z^{\alpha'}_{i}T^{l'+1}\psi](t,\ub')d\ub'\Big)^{1/2}
\end{align}
with $|\alpha'|+l'\leq |\alpha|-1$.

By Proposition 7.5:
\begin{align*}
&\delta^{l'+1}\|Z_{i}^{\alpha'}T^{l'}\slashed{\Delta}\mu\|_{L^{2}(\Sigma_{t}^{\ub})}\\
&\leq C\delta^{1/2}\int_{-r_{0}}^{t}\sqrt{E_{\leq|\alpha|+2}(t',\ub)}+\mu_{m}^{-1/2}(t')\sqrt{\underline{E}_{\leq|\alpha|+2}(t',\ub)}dt'\\
&\leq C\delta^{1/2}\int_{-r_{0}}^{t}\mu^{-b_{|\alpha|+2}}_{m}(t')\sqrt{\widetilde{E}_{\leq|\alpha|+2}(t',\ub)}+\mu_{m}^{-b_{|\alpha|+2}-1/2}(t')\sqrt{\widetilde{\Eb}_{\leq|\alpha|+2}(t',\ub)}dt'\\
&\leq C\delta^{1/2}\Big(\sqrt{\widetilde{E}_{\leq|\alpha|+2}(t,\ub)}\int_{-r_{0}}^{t}\mu_{m}^{-b_{|\alpha|+2}}(t')dt'+\sqrt{\widetilde{\Eb}_{\leq|\alpha|+2}(t,\ub)}\int_{-r_{0}}^{t}\mu^{-b_{|\alpha|+2}-1/2}_{m}(t')dt'\Big)\\
&\leq C\delta^{1/2}\mu^{-b_{|\alpha|+2}+1/2}_{m}(t)\sqrt{\widetilde{\Eb}_{\leq|\alpha|+2}(t,\ub)}+C\delta^{1/2}\mu_{m}^{-b_{|\alpha|+2}+1}(t)\sqrt{\widetilde{E}_{\leq|\alpha|+2}(t,\ub)}
\end{align*}
Then by the top order energy estimates obtained in the last section, the integral in the first factor of \eqref{K1T next to top spacetime} is bounded by ($\mu_{m}(t)\leq1$):
\begin{align*}
&C\delta\int_{-r_{0}}^{t}\mu^{-2b_{|\alpha|+2}+1}_{m}(t')\Big(\widetilde{E}_{\leq|\alpha|+2}(t',\ub)+\widetilde{\Eb}_{\leq|\alpha|+2}(t',\ub)\Big)dt'\\
&\leq C\delta\Big(\widetilde{E}_{\leq|\alpha|+2}(t)+\widetilde{\Eb}_{\leq|\alpha|+2}(t,\ub)\Big)\int_{-r_{0}}^{t}\mu^{-2b_{|\alpha|+2}+1}_{m}(t')dt'\\
&\leq C\delta\mu^{-2b_{|\alpha+2}+2}_{m}(t)\Big(\widetilde{E}_{\leq|\alpha|+2}(t,\ub)+\widetilde{\Eb}_{\leq|\alpha|+2}(t,\ub)\Big)\\
&\leq C\delta\mu^{-2b_{|\alpha+2}+2}_{m}(t)\mathcal{D}^{\ub}_{|\alpha|+2}
\end{align*}
Then again, with $b_{|\alpha|+1}=b_{|\alpha|+2}-1$, the spacetime integral \eqref{K1T next to top spacetime} is bounded by:
\begin{align}\label{K1T next to top estimates}
&C\mu^{-2b_{|\alpha|+1}}_{m}(t)\sqrt{\mathcal{D}^{\ub}_{|\alpha|+2}}\sqrt{\int_{0}^{\ub}\widetilde{\Fb}_{\leq|\alpha|+1}(t,\ub')d\ub'}\\\notag
&\leq C\delta\mu_{m}^{-2b_{|\alpha|+1}}(t)\mathcal{D}^{\ub}_{|\alpha|+2}+C\delta^{-1}\mu^{-2b_{|\alpha|+1}}_{m}(t)\int_{0}^{\ub}\widetilde{\Fb}_{\leq|\alpha|+1}(t,\ub')d\ub'
\end{align}

We proceed to consider the spacetime error integral associated to $K_{0}$. We first consider the spacetime integral:
\begin{align*}
&\int_{W^{t}_{\ub}}|T\psi||Z^{\alpha}_{i}\text{tr}\chib'||L Z^{\alpha}_{i}\psi|dt'd\ub'd\mu_{\tilde{\slashed{g}}}\\
&\leq C\delta^{-1/2}\int_{-r_{0}}^{t}\|Z^{\alpha}_{i}\text{tr}\chib'\|_{L^{2}(\Sigma_{t'}^{\ub})}\|LZ^{\alpha}_{i}\psi\|_{L^{2}(\Sigma_{t'}^{\ub})}dt'
\end{align*}
Substituting the estimates:
\begin{align*}
\|Z^{\alpha}_{i}\text{tr}\chib'\|_{L^{2}(\Sigma_{t'}^{\ub})}&\leq C\delta^{1/2}\mu^{-b_{|\alpha|+2}+1/2}_{m}(t')\sqrt{\widetilde{\Eb}_{\leq|\alpha|+2}(t',\ub)}\\&\leq C\delta\mu^{-b_{|\alpha|+2}+1/2}_{m}(t)\sqrt{\mathcal{D}^{\ub}_{|\alpha|+2}},\\
\|L Z^{\alpha}_{i}\psi\|_{L^{2}(\Sigma_{t'}^{\ub})}&\leq C\mu^{-b_{|\alpha|+1}}_{m}(t')\sqrt{\widetilde{E}_{\leq|\alpha|+1}(t',\ub)}
\end{align*}
with $b_{|\alpha|+1}=b_{|\alpha|+2}-1$, and using the fact that $\widetilde{E}_{\leq|\alpha|+1}(t)$ are non-decreasing in $t$, we see that the spacetime integral is bounded by ($\mu_{m}(t)\leq1$):
\begin{align}\label{K0 next to top estimates}
&C\delta^{1/2}\sqrt{\mathcal{D}^{\ub}_{|\alpha|+2}}\sqrt{\widetilde{E}_{\leq|\alpha|+1}(t,\ub)}\int_{-r_{0}}^{t}\mu^{-2b_{|\alpha|+1}-1/2}_{m}(t')dt'\\\notag
&\leq C\delta^{1/2}\mu^{-2b_{|\alpha|+1}+1/2}_{m}(t)\sqrt{\mathcal{D}^{\ub}_{|\alpha|+2}}\sqrt{\widetilde{E}_{\leq|\alpha|+1}(t,\ub)}\\\notag
&\leq C\mu_{m}^{-2b_{|\alpha|+1}}(t)\mathcal{D}^{\ub}_{|\alpha|+2}+C\delta\mu^{-2b_{|\alpha|+1}}_{m}(t)\widetilde{E}_{\leq|\alpha|+1}(t,\ub).
\end{align}

Finally, we consider the spacetime integral:
\begin{align}\label{K0T next to top spacetime}
&\delta^{2l'+2}\int_{W^{t}_{\ub}}|Z^{\alpha'}_{i}T^{l'}\slashed{\Delta}\mu||T\psi||LZ^{\alpha'}_{i}T^{l'+1}\psi|dt'd\ub'd\mu_{\tilde{\slashed{g}}}\\\notag
&\leq C\delta^{2l'+2-1/2}\int_{-r_{0}}^{t}\|Z^{\alpha'}_{i}T^{l'}\slashed{\Delta}\mu\|_{L^{2}(\Sigma_{t'}^{\ub})}\|LZ^{\alpha'}_{i}T^{l'+1}\psi\|_{L^{2}(\Sigma_{t'}^{\ub})}dt'
\end{align}
for $|\alpha'|+l'\leq |\alpha|-1$.
Again, substituting the estimates ($\mu_{m}(t)\leq1$):
\begin{align*}
\delta^{l'+1}\|Z^{\alpha'}_{i}T^{l'}\slashed{\Delta}\mu\|_{L^{2}(\Sigma_{t'}^{\ub})}\leq C\delta^{1/2}\mu^{-b_{|\alpha|+2}+1/2}_{m}(t')\Big(\sqrt{\widetilde{E}_{\leq|\alpha|+2}(t',\ub)}+\sqrt{\widetilde{\Eb}_{\leq|\alpha|+2}(t',\ub)}\Big)
\end{align*}
with $b_{|\alpha|+1}=b_{|\alpha|+2}-1$, the same argument implies that the spacetime integral is bounded by:
\begin{align}\label{K0T next to top estimates}
&C\sqrt{\mathcal{D}^{\ub}_{|\alpha|+2}}\sqrt{\widetilde{E}_{\leq|\alpha|+1}(t,\ub)}\int_{-r_{0}}^{t}\mu^{-2b_{|\alpha|+1}-1/2}_{m}(t')dt'\\\notag
&\leq C\mu^{-2b_{|\alpha|+1}+1/2}_{m}(t)\sqrt{\mathcal{D}^{\ub}_{|\alpha|+2,l'}}\sqrt{\widetilde{E}_{\leq|\alpha|+1}(t,\ub)}\\\notag
&\leq C_{\epsilon}\mu^{-2b_{|\alpha|+1}}_{m}(t)\mathcal{D}^{\ub}_{|\alpha|+2}+\epsilon\widetilde{E}_{\leq|\alpha|+1}(t,\ub)
\end{align}
Here $\epsilon$ is a small absolute positive constant.
\subsection{Energy estimates for next-to-top order}
Throughout this subsection $Z_{i}$ could be $R_{i}, Q$ and $T$. Now we consider the other contributions from the spacetime error integrals associated to $K_{1}$. For the variations $Z^{\alpha'}_{i}\psi$ where $|\alpha'|\leq |\alpha|$, the other contributions are bounded by (see \eqref{lower order optical terms separately K1 a}):
\begin{align}\label{lower optical K1 descent}
&C\delta^{2}\int_{-r_{0}}^{t}E_{\leq|\alpha|+1}(t',\ub)dt'+C\delta^{-1/2}\int_{0}^{\ub}\Fb_{\leq|\alpha|+1}(t,\ub')d\ub'+C\delta^{1/2} K_{\leq|\alpha|+1}(t,\ub)\\\notag
&\leq C\delta^{2}\int_{-r_{0}}^{t}\mu^{-2b_{|\alpha|+1}}_{m}(t',\ub)\widetilde{E}_{\leq|\alpha|+1}(t',\ub)dt'\\\notag
&+C\delta^{-1/2}\int_{0}^{\ub}\mu_{m}^{-2b_{|\alpha|+1}}(t)\widetilde{\Fb}_{\leq|\alpha|+1}(t,\ub')d\ub'
+C\delta^{1/2}\mu^{-2b_{|\alpha|+1}}_{m}(t)\widetilde{K}_{\leq|\alpha|+1}(t,\ub)
\end{align}
In view of \eqref{K1 next to top estimates}, \eqref{K1T next to top estimates}, \eqref{lower optical K1 descent} and multiplying $\mu_{m}^{2b_{|\alpha|+1}}(t)$ on both sides of the energy inequality associated to $K_{1}$ for $Z^{\alpha'}_{i}\psi$ with $|\alpha'|\leq |\alpha|$ 
give us:
\begin{align*}
&\sum_{|\alpha'|\leq |\alpha|}\mu^{2b_{|\alpha|+1}}_{m}(t)\delta^{2l'}\Big(\Eb[Z^{\alpha'}_{i}\psi](t,\ub)+\Fb[Z^{\alpha'}_{i}\psi](t,\ub)+K[Z^{\alpha'}_{i}\psi](t,\ub)\Big)\\
&\leq C\sum_{|\alpha'|\leq |\alpha|}\delta^{2l'}\Eb[Z^{\alpha'}_{i}\psi](-r_{0},\ub)
+C\delta^{-1}\int_{0}^{\ub}\widetilde{\Fb}_{\leq|\alpha|+1}(t,\ub')d\ub'\\
&+C\delta^{2}\int_{-r_{0}}^{t}\widetilde{E}_{\leq|\alpha|+1}(t',\ub)dt'
+C\delta^{1/2}\widetilde{K}_{\leq|\alpha|+1}(t,\ub)+C\delta\mathcal{D}^{\ub}_{|\alpha|+2}
\end{align*}
Since the right hand side of the above is non-decreasing in $t$, the above inequality is still true if we substitute $t$ by any $t'\in[-r_{0},t]$ on the left hand side:
\begin{align*}
&\sum_{|\alpha'|\leq |\alpha|}\mu^{2b_{|\alpha|+1}}_{m}(t')\delta^{2l'}\Big(\Eb[Z^{\alpha'}_{i}\psi](t',\ub)+\Fb[Z^{\alpha'}_{i}\psi](t',\ub)+K[Z^{\alpha'}_{i}\psi](t',\ub)\Big)\\
&\leq C\sum_{|\alpha'|\leq |\alpha|}\delta^{2l'}\Eb[Z^{\alpha'}_{i}\psi](-r_{0},\ub)
+C\delta^{-1}\int_{0}^{\ub}\widetilde{\Fb}_{\leq|\alpha|+1}(t,\ub')d\ub'\\
&+C\delta^{2}\int_{-r_{0}}^{t}\widetilde{E}_{\leq|\alpha|+1}(t',\ub)dt'
+C\delta\mathcal{D}^{\ub}_{|\alpha|+2}+C\delta^{1/2}\widetilde{K}_{\leq|\alpha|+1}(t,\ub)
\end{align*}
As in the previous section, taking the supremum with respect to $t'\in[-r_{0},t]$ we obtain:
\begin{align*}
&\widetilde{\Eb}_{\leq|\alpha|+1}(t,\ub)+\widetilde{\Fb}_{\leq|\alpha|+1}(t,\ub)+\widetilde{K}_{\leq|\alpha|+1}(t,\ub)\\
&\leq C\delta\mathcal{D}^{\ub}_{|\alpha|+2}
+C\delta^{-1}\int_{0}^{\ub}\widetilde{\Fb}_{\leq|\alpha|+1}(t,\ub')d\ub'\\
&+C\delta^{2}\int_{-r_{0}}^{t}\widetilde{E}_{\leq|\alpha|+1}(t',\ub)dt'+C\delta^{1/2}\widetilde{K}_{\leq|\alpha|+1}(t,\ub)
\end{align*}
Choosing $\delta$ sufficiently small, we have:
\begin{align*}
&\widetilde{\Eb}_{\leq|\alpha|+1}(t,\ub)+\widetilde{\Fb}_{\leq|\alpha|+1}(t,\ub)+\widetilde{K}_{\leq|\alpha|+1}(t,\ub)\leq C\delta\mathcal{D}^{\ub}_{|\alpha|+2}\\
&+C\delta^{-1}\int_{0}^{\ub}\widetilde{\Fb}_{\leq|\alpha|+1}(t,\ub')d\ub'+C\delta^{2}\int_{-r_{0}}^{t}\widetilde{E}_{\leq|\alpha|+1}(t',\ub)dt'
\end{align*}
Keeping only $\widetilde{\Fb}_{\leq|\alpha|+1}(t,\ub)$ and we have:
\begin{align*}
&\widetilde{\Fb}_{\leq|\alpha|+1}(t,\ub)
\leq C\delta\mathcal{D}^{\ub}_{|\alpha|+2}
+C\delta^{2}\int_{-r_{0}}^{t}\widetilde{E}_{\leq|\alpha|+1}(t',\ub)dt'+C\delta^{-1}\int_{0}^{\ub}\widetilde{\Fb}_{\leq|\alpha|+1}(t,\ub')d\ub'
\end{align*}
By using Gronwall, we obtain:
\begin{align*}
\widetilde{\Fb}_{\leq|\alpha|+1}(t,\ub)
\leq C\delta\mathcal{D}^{\ub}_{|\alpha|+2}
+C\delta^{2}\int_{-r_{0}}^{t}\widetilde{E}_{\leq|\alpha|+1}(t',\ub)dt'
\end{align*}
This together with the fact that $\widetilde{E}_{\leq|\alpha|+1}(t), \mathcal{D}^{\ub}_{|\alpha|+2}$ are non-decreasing in $\ub$ implies:
\begin{align}\label{K1 descent}
\widetilde{\Eb}_{\leq|\alpha|+1}(t)+\widetilde{\Fb}_{\leq|\alpha|+1}(t,\ub)+\widetilde{K}_{\leq|\alpha|+1}(t,\ub)\leq C\delta\mathcal{D}^{\ub}_{|\alpha|+2}
+C\delta^{2}\int_{-r_{0}}^{t}\widetilde{E}_{\leq|\alpha|+1}(t',\ub)dt'
\end{align}

Next we consider the energy estimates associated to $K_{0}$. We start with the variation $Z^{\alpha'}_{i}\psi$ with $|\alpha'|\leq |\alpha|$. The other contributions to the spacetime error integral is bounded by (see \eqref{lower order optical terms separately K0 a} and \eqref{K1 descent}):
\begin{align}\label{lower optical K0 descent}
&C\delta^{1/2}\int_{-r_{0}}^{t}E_{\leq|\alpha|+1}(t',\ub)dt'+C\delta^{1/2} K_{\leq|\alpha|+1}(t,\ub)+C\delta^{-1/2}\int_{0}^{\ub}\Fb_{\leq|\alpha|+1}(t,\ub')d\ub'\\\notag
&\leq C\delta^{1/2}\int_{-r_{0}}^{t}\mu^{-2b_{|\alpha|+1}}_{m}(t')\widetilde{E}_{\leq|\alpha|+1}(t',\ub)dt'+C\delta^{1/2}\mu^{-2b_{|\alpha|+1}}_{m}(t)\mathcal{D}^{\ub}_{|\alpha|+2}
\end{align}
Without loss of generality, we can choose $\epsilon\geq \delta$. Then in view of this and \eqref{K0 next to top estimates} and \eqref{K0T next to top estimates}, we have the following energy inequality:
\begin{align*}
&\sum_{|\alpha'|\leq |\alpha|}\mu_{m}^{2b_{|\alpha|+1}}(t)\delta^{2l'}\Big(E[Z^{\alpha'}_{i}\psi](t,\ub)+F[Z^{\alpha'}_{i}\psi](t,\ub)\Big)\\
&\leq C\mathcal{D}^{\ub}_{|\alpha|+2}+C\epsilon\widetilde{E}_{\leq|\alpha|+1}(t,\ub)+C\delta^{1/2}\int_{-r_{0}}^{t}\widetilde{E}_{\leq|\alpha|+1}(t',\ub)dt'
\end{align*}
Then similar as before, substituting $t$ by $t'\in[-r_{0},t]$ on the left hand side and taking the supremum with respect to $t'\in[-r_{0},t]$, we obtain:
\begin{align*}
&\widetilde{E}_{\leq|\alpha|+1}(t,\ub)+\widetilde{F}_{\leq|\alpha|+1}(t,\ub)\\
&\leq C\mathcal{D}^{\ub}_{|\alpha|+2}
+C\epsilon\widetilde{E}_{\leq|\alpha|+1}(t,\ub)+C\delta^{1/2}\int_{-r_{0}}^{t}\widetilde{E}_{\leq|\alpha|+1}(t',\ub)dt'
\end{align*}
Choosing $\epsilon$ sufficiently small and using Gronwall, we finally have:
\begin{align}\label{K0 descent}
\widetilde{E}_{\leq|\alpha|+1}(t,\ub)+\widetilde{F}_{\leq|\alpha|+1}(t,\ub)\leq C\mathcal{D}^{\ub}_{|\alpha|+2}
\end{align}

Now substituting \eqref{K0 descent} to the right hand side of \eqref{K1 descent}, we have:
\begin{align*}
\widetilde{\Eb}_{\leq|\alpha|+1}(t,\ub)+\widetilde{\Fb}_{\leq|\alpha|+1}(t,\ub)+\widetilde{K}_{\leq|\alpha|+1}(t,\ub)\leq C\delta\mathcal{D}^{\ub}_{|\alpha|+2}
\end{align*}

Summarizing, we have:
\begin{align}\label{energy estimates next to top}
\widetilde{\Eb}_{\leq|\alpha|+1}(t,\ub)+\widetilde{\Fb}_{\leq|\alpha|+1}(t,\ub)+\widetilde{K}_{\leq|\alpha|+1}(t,\ub)\leq C\delta\mathcal{D}^{\ub}_{|\alpha|+2}\\\notag
\widetilde{E}_{\leq|\alpha|+1}(t,\ub)+\widetilde{F}_{\leq|\alpha|+1}(t,\ub)\leq
C\mathcal{D}^{\ub}_{|\alpha|+2}
\end{align}

\subsection{Descent scheme}
We proceed in this way taking at the $n$th step:
\begin{align*}
b_{|\alpha|+2-n}=b_{|\alpha|+2}-n,\quad b_{|\alpha|+1-n}=b_{|\alpha|+2}-n-1
\end{align*}
in the role of $b_{|\alpha|+2}$ and $b_{|\alpha|+1}$ respectively, the argument beginning in the paragraph containing \eqref{K1 next to top spacetime} and concluding with \eqref{energy estimates next to top} being step $0$. The $n$th step is exactly the same as the $0$th step as above, as long as $b_{|\alpha|+1-n}>0$, that is, as long as $n\leq [b_{|\alpha|+2}]-1$. Here we choose $b_{|\alpha|+2}$ as:
\begin{align*}
b_{|\alpha|+2}=[b_{|\alpha|+2}]+\frac{3}{4}
\end{align*}
where $[b_{|\alpha|+2}]$ is the integer part of $b_{|\alpha|+2}$.
 For each of such $n$, we need to estimate the integrals:
\begin{align*}
\int_{-r_{0}}^{t}\mu_{m}^{-b_{|\alpha|+2-n}-1/2}(t')dt',\quad \int_{-r_{0}}^{t}\mu_{m}^{-2b_{|\alpha|+2-n}+1}(t')dt'
\end{align*}
As in the last section, we consider two different cases: $t'\in[-r_{0},t_{0}]$ and $t'\in[t_{0},t]$, where $\mu_{m}(t_{0})=\dfrac{1}{10}$. If $t'\in[-r_{0},t_{0}]$, we have:
\begin{align*}
\int_{-r_{0}}^{t_{0}}\mu^{-b_{|\alpha|+2-n}-1/2}_{m}(t')dt'&\leq C\int_{-r_{0}}^{t_{0}}\mu^{-b_{|\alpha|+2-n}+1/2}_{m}(t')dt'\leq C\mu_{m}^{-b_{|\alpha|+2-n}+1/2}(t)\\
\int_{-r_{0}}^{t_{0}}\mu^{-2b_{|\alpha|+2-n}+1}_{m}(t')dt'&\leq C
\int_{-r_{0}}^{t_{0}}\mu_{m}^{-2b_{|\alpha|+2-n}+2}(t')dt'\leq C
\mu_{m}^{-2b_{|\alpha|+2-n}+2}(t)
\end{align*}
Here we have used the fact that $\mu_{m}(t')\geq\dfrac{1}{10}$ for $t'\in[-r_{0},t_{0}]$. In regard to the estimate for $t'\in[t_{0},t]$, since
\begin{align*}
b_{|\alpha|+2-n}=[b_{|\alpha|+2-n}]+\frac{3}{4}\geq 1+\frac{3}{4}=\frac{7}{4},
\end{align*}
by Lemma 8.1, we have:
\begin{align*}
\int_{t_{0}}^{t}\mu^{-b_{|\alpha|+2-n}-1/2}_{m}(t')dt'&\leq C\mu^{-b_{|\alpha|+2-n}+1/2}_{m}(t)\\
\int_{t_{0}}^{t}\mu^{-2b_{|\alpha|+2-n}+1}_{m}(t')dt'&\leq C\mu^{-2b_{|\alpha|+2-n}+2}_{m}(t)
\end{align*}
So indeed, we can repeat the process of $0$th step for $n=1,...,[b_{|\alpha|+2}]-1$. Therefore we have the following estimates:
\begin{align}\label{descent energy estimates}
\widetilde{E}_{\leq|\alpha|+1-n}(t,\ub)+\widetilde{F}_{\leq|\alpha|+1-n}(t,\ub)&\leq C\mathcal{D}^{\ub}_{|\alpha|+2}\\\notag
\widetilde{\Eb}_{\leq|\alpha|+1-n}(t,\ub)+\widetilde{\Fb}_{\leq|\alpha|+1-n}(t,\ub)+\widetilde{K}_{\leq |\alpha|+1-n}(t,\ub)&\leq C\delta\mathcal{D}^{\ub}_{|\alpha|+2}
\end{align}
We now consider the final step $n=[b_{|\alpha|+2}]$.  In this case we have $b_{|\alpha|+2-n}=\frac{3}{4}$. Using the same process as in $0$th step, the contributions of the optical terms are bounded by:
\begin{align*}
\|Z^{\alpha'}\text{tr}\chib'\|_{L^{2}(\Sigma_{t}^{\ub})}&\leq C\delta\mu^{-1/4}_{m}(t)\sqrt{\mathcal{D}^{\ub}_{|\alpha|+2}}\quad\text{with}\quad |\alpha'|+2\leq |\alpha|+1-[b_{|\alpha|+2}]\\
\|Z^{\alpha'}_{i}T^{l'}\slashed{\Delta}\mu\|_{L^{2}(\Sigma_{t}^{\ub})}&\leq C\delta^{1/2}\mu^{-1/4}_{m}(t)\sqrt{\mathcal{D}^{\ub}_{|\alpha|+2}}\quad\text{with}\quad |\alpha'|+l'+2\leq |\alpha|+1-[b_{|\alpha|+2}]
\end{align*}
with $Z_{i}=R_{i}$ or $Q$.
As before, in order to bound the corresponding integrals:
\begin{align*}
\int_{-r_{0}}^{t}\|Z^{\alpha'}_{i}\text{tr}\chib'\|^{2}_{L^{2}(\Sigma_{t'}^{\ub})}dt',\quad\int_{-r_{0}}^{t}\|Z^{\alpha'}_{i}T^{l'}\slashed{\Delta}\mu\|^{2}_{L^{2}(\Sigma_{t'}^{\ub})}dt'
\end{align*}
we need to consider the integral:
\begin{align*}
\int_{-r_{0}}^{t}\mu^{-1/2}_{m}(t')dt'\leq \int_{-r_{0}}^{t_{0}}\mu^{-1/2}_{m}(t')dt'+\int_{t_{0}}^{t}\mu^{-1/2}_{m}(t')dt'\quad\text{with}\quad \mu_{m}(t_{0})=\frac{1}{10}
\end{align*}
For the ``non-shock part $\int_{-r_{0}}^{t_{0}}$", since $\mu_{m}(t_{0})\geq\dfrac{1}{10}$,
\begin{align*}
\int_{-r_{0}}^{t_{0}}\mu_{m}^{-1/2}(t')dt'\leq C
\end{align*}
For the ``shock part $\int_{t_{0}}^{t}$, as in the proof for Lemma 8.1,
\begin{align*}
\int_{t_{0}}^{t}\mu_{m}^{-1/2}(t')dt'\leq C\mu^{1/2}_{m}(t)\leq C
\end{align*}
So we have the following bounds:
\begin{align*}
\Big(\int_{-r_{0}}^{t}\|Z^{\alpha'}\text{tr}\chib'\|^{2}_{L^{2}(\Sigma_{t'}^{\ub})}dt'\Big)^{1/2}&\leq C\delta\sqrt{\mathcal{D}^{\ub}_{|\alpha|+2}}\quad\text{with}\quad |\alpha'|+2\leq |\alpha|+1-[b_{|\alpha|+2}]\\
\Big(\int_{-r_{0}}^{t}\|Z^{\alpha'}_{i}T^{l'}\slashed{\Delta}\mu\|_{L^{2}(\Sigma_{t'}^{\ub})}dt'\Big)^{1/2}&\leq C\delta^{1/2}\sqrt{\mathcal{D}^{\ub}_{|\alpha|+2}}\quad\text{with}\quad |\alpha'|+l'+2\leq |\alpha|+1-[b_{|\alpha|+2}]
\end{align*}
Therefore we can set:
\begin{align*}
b_{|\alpha|+1-n}=b_{|\alpha|+1-[b_{|\alpha|+2}]}=0
\end{align*}
in this step. Then we can proceed exactly the same as in the preceding steps. We thus arrive at the estimates:
\begin{align}\label{descent estimates final}
\widetilde{E}_{\leq|\alpha|+1-[b_{|\alpha|+2}]}(t,\ub)+\widetilde{F}_{\leq|\alpha|+1-[b_{|\alpha|+2}]}(t,\ub)&\leq C\mathcal{D}^{\ub}_{|\alpha|+2}\\\notag
\widetilde{\Eb}_{\leq|\alpha|+1-[b_{|\alpha|+2}]}(t,\ub)+\widetilde{\Fb}_{\leq|\alpha|+1-[b_{|\alpha|+2}]}(t,\ub)+\widetilde{K}_{\leq |\alpha|+1-[b_{|\alpha|+2}]}(t,\ub)&\leq C\delta\mathcal{D}^{\ub}_{|\alpha|+2}
\end{align}
These are the desired estimates, because from the definitions:
\begin{align}\label{definitions for finite energy}
\widetilde{\Eb}_{\leq|\alpha|+1-[b_{|\alpha|+2}]}(t,\ub):&=\sup_{t'\in[-r_{0},t]}\{\underline{E}_{\leq|\alpha|+1-[b_{|\alpha|+2}]}(t',\ub)\}\\\notag
\widetilde{E}_{\leq|\alpha|+1-[b_{|\alpha|+2}]}(t,\ub):&=\sup_{t'\in[-r_{0},t]}\{E_{\leq|\alpha|+1-[b_{|\alpha|+2}]}(t',\ub)\}\\\notag
\widetilde{\Fb}_{\leq|\alpha|+1-[b_{|\alpha|+2}]}(t,\ub):&=\sup_{t'\in[-r_{0},t]}\{\Fb_{\leq|\alpha|+1-[b_{|\alpha|+2}]}(t',\ub)\}\\\notag
\widetilde{F}_{\leq|\alpha|+1-[b_{|\alpha|+2}]}(t,\ub):&=\sup_{t'\in[-r_{0},t]}\{F_{\leq|\alpha|+1-[b_{|\alpha|+2}]}(t',\ub)\}\\\notag
\widetilde{K}_{\leq|\alpha|+1-[b_{|\alpha|+2}]}(t,\ub):&=\sup_{t'\in[-r_{0},t]}\{K_{\leq|\alpha|+1-[b_{|\alpha|+2}]}(t',\ub)\}
\end{align}
the weight $\mu_{m}(t')$ has been eliminated.

\section{Completion of Proof}
Let us define:
\begin{align*}
\mathcal{S}_{2}[\phi]:=\int_{S_{t,\ub}}\Big(|\phi|^{2}+|R_{i_{1}}\phi|^{2}+|R_{i_{1}}R_{i_{2}}\phi|^{2}\Big)d\mu_{\slashed{g}}
\end{align*}
And also let us denote by $\mathcal{S}_{n}(t,\ub)$ the integral on $S_{t,\ub}$ (with respect to $d\mu_{\slashed{g}}$) of the sum of the square of all the variation $\psi=\delta^{l'}Z^{\alpha'}_{i}\psi_{\gamma}$ up to order $|\alpha|+1-[b_{|\alpha|+2}]$, where $l'$ is the number of $T$'s in the string of $Z^{\alpha'}_{i}$ and $\gamma=0,1,2,3$. Then by Lemma 7.3 we have:
\begin{align*}
\mathcal{S}_{|\alpha|-[b_{|\alpha|+2}]}(t,\ub)\leq C\delta\left( E_{|\alpha|+1-[b_{|\alpha|+2}]}(t)+\Eb_{|\alpha|+1-[b_{|\alpha|+2}]}(t,\ub)\right) \quad\text{for all}\quad (t,\ub)\in[-2,t^{*})\times [0,\delta].
\end{align*}
Hence, in view of \eqref{descent estimates final} and \eqref{definitions for finite energy},
\begin{align}\label{surface integral bounded by initial data}
\mathcal{S}_{|\alpha|-[b_{|\alpha|+2}]}(t,\ub)\leq C\delta\mathcal{D}^{\ub}_{|\alpha|+2} \quad\text{for all}\quad (t,\ub)\in[-2,t^{*})\times [0,\delta].
\end{align}
Then for any variations $\psi$ of order up to $|\alpha|-2-[b_{|\alpha|+2}]$ we have:
\begin{align}\label{transaction of surface integral}
\mathcal{S}_{2}[\psi]\leq \mathcal{S}_{|\alpha|-[b_{|\alpha|+2}]}(t,\ub)
\end{align}
Then by the Sobolev inequality introduced in \eqref{Sobolev}, \eqref{surface integral bounded by initial data} and \eqref{transaction of surface integral}, we have:
\begin{align}\label{Linfty on surface}
\delta^{l'}\sup_{S_{t,\ub}}|Z^{\alpha'}_{i}\psi_{\alpha}|=\sup_{S_{t,\ub}}|\psi|\leq C\delta^{1/2}\sqrt{\mathcal{D}^{\ub}_{|\alpha|+2}}\leq C_{0}\delta^{1/2}
\end{align}
where $C_{0}$ depends on the initial energy $\mathcal{D}^{\ub}_{|\alpha|+2}$, the constant in the isoperimetric inequality and the constant in Lemma 7.3 as well as the constants in \eqref{descent estimates final}, which are absolute constants.

If we choose $|\alpha|$ large enough such that
\begin{align*}
\Big[\frac{|\alpha|+1}{2}\Big]+3\leq |\alpha|-2-[b_{|\alpha|+2}]
\end{align*}
then \eqref{Linfty on surface} recovers the bootstrap assumption (B.1) for $(t,\ub)\in[-2,t^{*})\times[0,\delta]$. 

\bigskip

To complete the proof of Theorem 3.1, it remains to show that the smooth solution exists for $t\in[-2,s^{*})$, i.e. $t^{*}=s^{*}$. More precisely, we will prove that either $\mu_{m}(t^{*})=0$ if shock forms before $t=-1$ or otherwise $t^{*}=-1$.

If $t^{*}<s^{*}$, then $\mu$ would be positive on $\Sigma_{t^{*}}^{\delta}$. In particular $\mu$ has a positive lower bound on $\Sigma_{t^{*}}^{\delta}$. Therefore by Remark \ref{geometric meaning of mu}, the Jacobian $\triangle$ of the transformation from optical coordinates to rectangular coordinates has a positive lower bound on $\Sigma_{t^{*}}^{\delta}$. This implies that the inverse transformation from rectangular coordinates to optical coordinates is regular. On the other hand, in the course of recovering bootstrap assumption we have proved that all the derivatives of the first order variations $\psi_{\alpha}$ extend smoothly in optical coordinates to $\Sigma_{t^{*}}^{\delta}$. Since the inverse transformation is regular, $\psi_{\alpha}$ also extend smoothly to $\Sigma_{t^{*}}^{\delta}$ in rectangular coordinates. Once $\psi_{\alpha}$ extend to functions of rectangular coordinates on $\Sigma_{t^{*}}^{\delta}$ belonging to some Sobolev space $H^{3}$, then the standard local existence theorem (which is stated and proved in rectangular coordinates) applies and we obtain an extension of the solution to a development containing an extension of all null hypersurface $\Cb_{\ub}$ for $\ub\in[0,\delta]$, up to a value $t_{1}$ of $t$ for some $t_{1}>t^{*}$, 
which contradicts with the definition of $t^{*}$ and therefore $t^{*}=s^{*}$. 
This completes the proof of Theorem 3.1 and the main theorem of the paper.

\section*{Acknowledgement}
\noindent The authors are grateful to three anonymous referees, who carefully read a previous version of this paper and suggested many valuable improvements and corrections. S. Miao is supported by NSF grant DMS-1253149 to The University of Michigan. P. Yu is supported by  NSFC 11101235 and NSFC 11271219. The research of S. Miao was in its initial phase supported by ERC Advanced Grant 246574 ``\emph{Partial Differential Equations of Classical Physics}".


\end{document}